\documentclass[11pt]{amsart}

\makeatletter
\renewcommand\part{%
   \if@noskipsec \leavevmode \fi
   \par
   \addvspace{4ex}%
   \@afterindentfalse
   \secdef\@part\@spart}

\def\@part[#1]#2{%
    \ifnum \c@secnumdepth >\m@ne
      \refstepcounter{part}%
      \addcontentsline{toc}{part}{\thepart\hspace{1em}#1}%
    \else
      \addcontentsline{toc}{part}{#1}%
    \fi
    {\parindent \z@ \raggedright
     \interlinepenalty \@M
     \normalfont
     \ifnum \c@secnumdepth >\m@ne
       \Large\bfseries \partname\nobreakspace\thepart
       \par\nobreak
     \fi
     \huge \bfseries #2%
     \par}%
    \nobreak
    \vskip 3ex
    \@afterheading}
\def\@spart#1{%
    {\parindent \z@ \raggedright
     \interlinepenalty \@M
     \normalfont
     \huge \bfseries #1\par}%
     \nobreak
     \vskip 3ex
     \@afterheading}
\makeatother

\usepackage[english]{babel}
\usepackage[utf8]{inputenc}  
\usepackage[T1]{fontenc}
\usepackage{enumitem}
\usepackage{xcolor}
\usepackage{tabularx}
\usepackage{framed}

\usepackage{amssymb,latexsym}
\usepackage{amsmath}
\usepackage{wrapfig}
\usepackage{tikz-cd}
\usepackage{subcaption}
\usepackage{graphicx}
\graphicspath{{./dessins/}}
\usepackage{mathrsfs}
\usepackage{bm}
\usepackage{amscd}
\usepackage[hypertexnames=false]{hyperref}
\usepackage{bbm}
\usepackage{tikz}
\usepackage{mathtools}

\usepackage{graphicx}
\graphicspath{ {./dessinsinkscape/}{./croquismain/}{./photoshop/} }

\hypersetup{
    colorlinks=false
}

\usepackage[left=2.5cm,right=2.5cm,top=4cm,bottom=3cm,headsep=1cm]{geometry}

\theoremstyle{definition}
\newtheorem{definition}{Definition}

\theoremstyle{definition}
\newtheorem{definition-proposition}{Definition-Proposition}

\theoremstyle{theorem}
\newtheorem{theorem}{Theorem}

\theoremstyle{theorem}
\newtheorem{conjecture}{Conjecture}

\theoremstyle{theorem}
\newtheorem{lemma}{Lemma}

\theoremstyle{theorem}
\newtheorem{corollary}{Corollary}

\theoremstyle{theorem}
\newtheorem{proposition}{Proposition}

\theoremstyle{remark}

\numberwithin{equation}{section}

\begin{document}

\title{Higher algebra of \Ainf\ and \ombas -algebras in Morse theory II}
\author{Thibaut Mazuir}

\newcommand{\Hom}{\ensuremath{\mathrm{Hom}}}
\newcommand{\N}{\ensuremath{\mathbb{N}}}
\newcommand{\R}{\ensuremath{\mathbb{R}}}
\newcommand{\Z}{\ensuremath{\mathbb{Z}}}
\newcommand{\Sp}{\ensuremath{\mathbb{S}}}
\newcommand{\Lac}{\ensuremath{\mathrm{L}}}
\newcommand{\ide}{\ensuremath{\mathrm{id}}}
\newcommand{\calA}{\ensuremath{\mathcal{A}}}
\newcommand{\Ainf}{\ensuremath{A_\infty}}
\newcommand{\infmor}{\ensuremath{\Ainf - \mathrm{Morph}}}
\newcommand{\infmorn}{\ensuremath{n - \Ainf - \mathrm{Morph}}}
\newcommand{\inprod}{$\infty$-inner product}
\newcommand{\infcat}{$\infty$-category}
\newcommand{\degr}{\ensuremath{\mathrm{deg}}}
\newcommand{\CM}{\ensuremath{\mathrm{CM}}}

\newcommand{\ombas}{\ensuremath{\Omega B As}}
\newcommand{\ombasm}{\ensuremath{\Omega B As - \mathrm{Morph}}}

\newcommand{\rouge}[1]{\textcolor{red}{#1}}
\newcommand{\bleu}[1]{\textcolor{blue}{#1}}
\newcommand{\verde}[1]{\textcolor{green}{#1}}
\newcommand{\violet}[1]{\textcolor{violet}{#1}}
\newcommand{\jaune}[1]{\textcolor{yellow}{#1}}

\newcommand*{\inserim}[1]{%
  \raisebox{-.3\baselineskip}{%
    \includegraphics[
      width=0.08\textwidth
    ]{#1}%
  }%
}

\newcommand*{\inserimi}[1]{%
  \raisebox{-.3\baselineskip}{%
    \includegraphics[
      width=0.15\textwidth
    ]{#1}%
  }%
}

\newcommand{\arbreop}[1]{
\begin{tikzpicture}[scale=#1,baseline=(pt_base.base)]

\coordinate (pt_base) at (0,0) ;

\draw (-2,1.25) node[above]{\tiny 1} ;
\draw (-1,1.25) node[above]{\tiny 2} ;
\draw (1,1.25) node[above]{};
\draw (2,1.25) node[above]{\tiny $n$} ;

\draw (-2,1.25) -- (0,0) ;
\draw (-1,1.25) -- (0,0) ;
\draw (1,1.25) -- (0,0) ;
\draw (2,1.25) -- (0,0) ;
\draw (0,-1.25) -- (0,0) ;

\draw [densely dotted] (-0.5,0.9) -- (0.5,0.9) ;

\end{tikzpicture}}


\newcommand{\arbreopdeux}{
\begin{tikzpicture}[scale=0.15,baseline=(pt_base.base)]

\coordinate (pt_base) at (0,-0.5) ;

\draw (-1.25,1.25) -- (0,0) ;
\draw (1.25,1.25) -- (0,0) ;
\draw (0,-1.25) -- (0,0) ;

\end{tikzpicture}
}


\newcommand{\arbreoptrois}{
\begin{tikzpicture}[scale=0.15,baseline=(pt_base.base)]

\coordinate (pt_base) at (0,-0.5) ;

\draw (0,1.25) -- (0,0) ;
\draw (-1.25,1.25) -- (0,0) ;
\draw (1.25,1.25) -- (0,0) ;
\draw (0,-1.25) -- (0,0) ;

\end{tikzpicture}
}


\newcommand{\arbreopquatre}{
\begin{tikzpicture}[scale=0.15,baseline=(pt_base.base)]

\coordinate (pt_base) at (0,-0.5) ;

\draw (-1.5,1.25) -- (0,0) ;
\draw (1.5,1.25) -- (0,0) ;
\draw (-0.5,1.25) -- (0,0) ;
\draw (0.5,1.25) -- (0,0) ;
\draw (0,-1.25) -- (0,0) ;

\end{tikzpicture}
}


\newcommand{\eqainf}{
\begin{tikzpicture}[scale=0.2,baseline=(pt_base.base)]

\coordinate (pt_base) at (0,1.25) ;

\draw (-4,1.25) node[above]{\tiny 1} ;
\draw (-3,1.25) node[above]{};
\draw (4,1.25) node[above]{\tiny $d_1$} ;
\draw (3,1.25) node[above]{};
\draw (0,1.75) node[right]{\tiny $k$} ;

\draw (-2,4) node[above]{\tiny 1} ;
\draw (-1,4) node[above]{};
\draw (1,4) node[above]{} ;
\draw (2,4) node[above]{\tiny $d_2$};

\draw (-4,1.25) -- (0,0) ;
\draw (-3,1.25) -- (0,0) ;
\draw (4,1.25) -- (0,0) ;
\draw (3,1.25) -- (0,0) ;
\draw (0,1.25) -- (0,0) ;
\draw(0,-1.25) -- (0,0) ; 

\draw (-2,4) -- (0,2.75) ;
\draw (-1,4) -- (0,2.75) ;
\draw (1,4) -- (0,2.75) ;
\draw (2,4) -- (0,2.75) ;
\draw (0,1.5) -- (0,2.75) ;

\draw [densely dotted] (-0.5,3.65) -- (0.5,3.65) ;
\draw [densely dotted] (-0.5,0.9) -- (-1.5,0.9) ;
\draw [densely dotted] (0.5,0.9) -- (1.5,0.9) ;

\end{tikzpicture}}


\newcommand{\arbreoprouge}[1]{
\begin{tikzpicture}[scale=#1,baseline=(pt_base.base)]

\coordinate (pt_base) at (0,-0.5) ;

\draw[red] (-2,1.25) -- (0,0) ;
\draw[red] (-1,1.25) -- (0,0) ;
\draw[red] (1,1.25) -- (0,0) ;
\draw[red] (2,1.25) -- (0,0) ;
\draw[red] (0,-1.25) -- (0,0) ;

\draw[red] [densely dotted] (-0.5,0.9) -- (0.5,0.9) ;

\end{tikzpicture}}

\newcommand{\arbreopbleu}[1]{
\begin{tikzpicture}[scale=#1,baseline=(pt_base.base)]

\coordinate (pt_base) at (0,-0.5) ;

\draw[blue] (-2,1.25) -- (0,0) ;
\draw[blue] (-1,1.25) -- (0,0) ;
\draw[blue] (1,1.25) -- (0,0) ;
\draw[blue] (2,1.25) -- (0,0) ;
\draw[blue] (0,-1.25) -- (0,0) ;

\draw[blue] [densely dotted] (-0.5,0.9) -- (0.5,0.9) ;

\end{tikzpicture}}


\newcommand{\arbreopmorph}[1]{
\begin{tikzpicture}[scale=#1,baseline=(pt_base.base)]

\coordinate (pt_base) at (0,-0.5) ;

\draw[blue] (-2,1.25) -- (-0.625,0.5) ;
\draw[blue] (-1,1.25) -- (-0.3125,0.5) ;
\draw[blue] (1,1.25) -- (0.3125,0.5) ;
\draw[blue] (2,1.25) -- (0.625,0.5) ;

\draw[red] (-0.625,0.5) -- (0,0) ;
\draw[red] (-0.3125,0.5) -- (0,0) ;
\draw[red] (0.625,0.5) -- (0,0) ;
\draw[red] (0.3125,0.5) -- (0,0) ;
\draw[red] (0,-1.25) -- (0,0) ;

\draw[blue] [densely dotted] (-0.5,0.9) -- (0.5,0.9) ;

\draw  (-2,0.5) -- (2,0.5) ;

\end{tikzpicture}}


\usetikzlibrary{positioning}
\usetikzlibrary{calc}

\newcommand{\arbreopunmorph}{
\begin{tikzpicture}[scale=0.15,baseline=(pt_base.base)]

\coordinate (pt_base) at (0,-0.5) ;

\draw[blue] (0,1.25) -- (0,0.5) ;
\draw[red] (0,-1.25) -- (0,0.5) ;

\draw (-1.25,0.5) -- (1.25,0.5) ;

\end{tikzpicture}
}

\newcommand{\arbreopunmorphsubun}{
\begin{tikzpicture}[scale=0.1,baseline=(pt_base.base)]

\coordinate (pt_base) at (0,-0.5) ;

\draw[blue] (0,1.25) -- (0,0.5) ;
\draw[orange] (0,-1.25) -- (0,0.5) ;

\draw (-1.25,0.5) -- (1.25,0.5) ;

\end{tikzpicture}
}

\newcommand{\arbreopunmorphsubdeux}{
\begin{tikzpicture}[scale=0.1,baseline=(pt_base.base)]

\coordinate (pt_base) at (0,-0.5) ;

\draw[orange] (0,1.25) -- (0,0.5) ;
\draw[red] (0,-1.25) -- (0,0.5) ;

\draw (-1.25,0.5) -- (1.25,0.5) ;

\end{tikzpicture}
}


\newcommand{\arbreopdeuxmorph}{
\begin{tikzpicture}[scale=0.15,baseline=(pt_base.base)]

\coordinate (pt_base) at (0,-0.5) ;

\draw[blue] (-1.25,1.25) -- (-0.5,0.5) ;
\draw[blue] (1.25,1.25) -- (0.5,0.5) ;

\draw[red] (-0.5,0.5) -- (0,0) ;
\draw[red] (0.5,0.5) -- (0,0) ;
\draw[red] (0,-1.25) -- (0,0) ;

\draw (-1.25,0.5) -- (1.25,0.5) ;

\end{tikzpicture}
}


\newcommand{\arbreoptroismorph}{
\begin{tikzpicture}[scale=0.15,baseline=(pt_base.base)]

\coordinate (pt_base) at (0,-0.5) ;

\draw[blue] (0,1.25) -- (0,0.5) ;
\draw[blue] (-1.25,1.25) -- (-0.5,0.5) ;
\draw[blue] (1.25,1.25) -- (0.5,0.5) ;

\draw[red] (0,0.5) -- (0,0) ;
\draw[red] (-0.5,0.5) -- (0,0) ;
\draw[red] (0.5,0.5) -- (0,0) ;
\draw[red] (0,-1.25) -- (0,0) ;

\draw (-1.25,0.5) -- (1.25,0.5) ;

\end{tikzpicture}
}


\newcommand{\arbreopquatremorph}{
\begin{tikzpicture}[scale=0.15,baseline=(pt_base.base)]

\coordinate (pt_base) at (0,-0.5) ;

\draw[blue] (-1.5,1.25) -- (-0.6,0.5) ;
\draw[blue] (1.5,1.25) -- (0.6,0.5) ;
\draw[blue] (-0.5,1.25) -- (-0.2,0.5) ;
\draw[blue] (0.5,1.25) -- (0.2,0.5) ;

\draw[red] (-0.6,0.5) -- (0,0) ;
\draw[red] (0.6,0.5) -- (0,0) ;
\draw[red] (0.2,0.5) -- (0,0) ;
\draw[red] (-0.2,0.5) -- (0,0) ;
\draw[red] (0,-1.25) -- (0,0) ;

\draw (-1.5,0.5) -- (1.5,0.5) ;

\end{tikzpicture}
}


\newcommand{\arbreopdeuxcol}[1]{
\begin{tikzpicture}[scale=0.15,baseline=(pt_base.base)]

\coordinate (pt_base) at (0,-0.5) ;

\draw[#1] (-1.25,1.25) -- (0,0) ;
\draw[#1] (1.25,1.25) -- (0,0) ;
\draw[#1] (0,-1.25) -- (0,0) ;

\end{tikzpicture}
}


\newcommand{\arbreoptroiscol}[1]{
\begin{tikzpicture}[scale=0.15,baseline=(pt_base.base)]

\coordinate (pt_base) at (0,-0.5) ;

\draw[#1] (0,1.25) -- (0,0) ;
\draw[#1] (-1.25,1.25) -- (0,0) ;
\draw[#1] (1.25,1.25) -- (0,0) ;
\draw[#1] (0,-1.25) -- (0,0) ;

\end{tikzpicture}
}


\newcommand{\arbreopquatrecol}[1]{
\begin{tikzpicture}[scale=0.15,baseline=(pt_base.base)]

\coordinate (pt_base) at (0,-0.5) ;

\draw[#1] (-1.5,1.25) -- (0,0) ;
\draw[#1] (1.5,1.25) -- (0,0) ;
\draw[#1] (-0.5,1.25) -- (0,0) ;
\draw[#1] (0.5,1.25) -- (0,0) ;
\draw[#1] (0,-1.25) -- (0,0) ;

\end{tikzpicture}
}


\newcommand{\eqainfmorphun}{
\begin{tikzpicture}[scale=0.2,baseline=(pt_base.base)]

\coordinate (pt_base) at (0,1.25) ;


\draw (-4,1.25) node[above]{\tiny 1} ;
\draw (-3,1.25) node[above]{};
\draw (4,1.25) node[above]{\tiny $k$} ;
\draw (3,1.25) node[above]{};
\draw (0,1.75) node[right]{\tiny $i$} ;

\draw (-2,4) node[above]{\tiny 1} ;
\draw (-1,4) node[above]{};
\draw (1,4) node[above]{} ;
\draw (2,4) node[above]{\tiny $h$};


\draw[blue] (-4,1.25) -- (-1.6,0.5) ;
\draw[blue] (-3,1.25) -- (-1.2,0.5) ;
\draw[blue] (4,1.25) -- (1.6,0.5) ;
\draw[blue] (3,1.25) -- (1.2,0.5) ;
\draw[blue] (0,1.25) -- (0,0.5) ;

\draw[red] (-1.6,0.5) -- (0,0) ;
\draw[red] (-1.2,0.5) -- (0,0) ;
\draw[red] (1.6,0.5) -- (0,0) ;
\draw[red] (1.2,0.5) -- (0,0) ;
\draw[red] (0,0.5) -- (0,0) ;
\draw[red] (0,-1.25) -- (0,0) ; 


\draw[blue] (-2,4) -- (0,2.75) ;
\draw[blue] (-1,4) -- (0,2.75) ;
\draw[blue] (1,4) -- (0,2.75) ;
\draw[blue] (2,4) -- (0,2.75) ;
\draw[blue] (0,1.5) -- (0,2.75) ;


\draw[blue] [densely dotted] (-0.5,3.65) -- (0.5,3.65) ;
\draw[blue] [densely dotted] (-0.5,0.9) -- (-1.5,0.9) ;
\draw[blue] [densely dotted] (0.5,0.9) -- (1.5,0.9) ;

\draw (-4,0.5) -- (4,0.5) ;

\end{tikzpicture}}


\newcommand{\eqainfmorphdeux}{
\begin{tikzpicture}[scale=0.2,baseline=(pt_base.base)]

\coordinate (pt_base) at (0,1.25) ;


\draw (2.5,4) node[above]{\tiny 1} ;
\draw (5.5,4) node[above]{\tiny $i_s$};
\draw (-2.5,4) node[above]{\tiny $i_1$} ;
\draw (-5.5,4) node[above]{\tiny 1};


\draw[red] (-4,1.25) -- (0,0) ;
\draw[red] (-2,1.25) -- (0,0) ;
\draw[red] (4,1.25) -- (0,0) ;
\draw[red] (2,1.25) -- (0,0) ;
\draw[red] (0,-1.25) -- (0,0) ; 


\draw[blue] (2.5,4) -- (3.4,3.25) ;
\draw[blue] (3.5,4) -- (3.8,3.25) ;
\draw[blue] (4.5,4) -- (4.2,3.25) ;
\draw[blue] (5.5,4) -- (4.6,3.25) ;

\draw[red] (3.4,3.25) -- (4,2.75) ;
\draw[red] (3.8,3.25) -- (4,2.75) ;
\draw[red] (4.2,3.25) -- (4,2.75) ;
\draw[red] (4.6,3.25) -- (4,2.75) ;
\draw[red] (4,1.5) -- (4,2.75) ;

\draw (2.5,3.25) -- (5.5,3.25) ;


\draw[blue] (-2.5,4) -- (-3.4,3.25) ;
\draw[blue] (-3.5,4) -- (-3.8,3.25) ;
\draw[blue] (-4.5,4) -- (-4.2,3.25) ;
\draw[blue] (-5.5,4) -- (-4.6,3.25) ;

\draw[red] (-3.4,3.25) -- (-4,2.75) ;
\draw[red] (-3.8,3.25) -- (-4,2.75) ;
\draw[red] (-4.2,3.25) -- (-4,2.75) ;
\draw[red] (-4.6,3.25) -- (-4,2.75) ;
\draw[red] (-4,1.5) -- (-4,2.75) ;

\draw (-2.5,3.25) -- (-5.5,3.25) ;


\draw [densely dotted,red] (-1,0.9) -- (1,0.9) ;
\draw [densely dotted] (2,2.75) -- (-2,2.75) ; 
\draw [densely dotted,blue] (3.8,3.7) -- (4.2,3.7) ;
\draw [densely dotted,blue] (-3.8,3.7) -- (-4.2,3.7) ;

\end{tikzpicture}}



\newcommand{\associaedreun}{
\begin{tikzpicture}[ baseline = (pt_base.base)]
\coordinate (pt_base) at (0,0) ;
\node[scale = 0.1] at (0,0) {\pointbullet} ;
\node[above left] at (0,0) {\arbreopdeux} ;
\end{tikzpicture}
}


\newcommand{\arbreopdeuxun}{
\begin{tikzpicture}

\begin{scope}
\draw (-1.25,1.25) -- (0,0) ;
\draw (1.25,1.25) -- (0,0) ;
\draw (0,-1.25) -- (0,0) ;
\end{scope}

\begin{scope}[shift = {(-1.25,3)}]
\draw (-1.25,1.25) -- (0,0) ;
\draw (1.25,1.25) -- (0,0) ;
\draw (0,-1.25) -- (0,0) ;
\end{scope}

\end{tikzpicture}
}

\newcommand{\arbreopdeuxdeux}{
\begin{tikzpicture}
\begin{tikzpicture}

\begin{scope}
\draw (-1.25,1.25) -- (0,0) ;
\draw (1.25,1.25) -- (0,0) ;
\draw (0,-1.25) -- (0,0) ;
\end{scope}

\begin{scope}[shift = {(1.25,3)}]
\draw (-1.25,1.25) -- (0,0) ;
\draw (1.25,1.25) -- (0,0) ;
\draw (0,-1.25) -- (0,0) ;
\end{scope}

\end{tikzpicture}
\end{tikzpicture}
}

\newcommand{\associaedredeux}{
\begin{tikzpicture}[ baseline = (pt_base.base)]
\coordinate (pt_base) at (0,0) ;
\node[scale = 0.1] (A) at (-1,0) {\pointbullet} ;
\node[scale = 0.1] (B) at (1,0) {\pointbullet} ;
\node[above left = 2 pt,scale = 0.1] at (A) {\arbreopdeuxun} ;
\node[above right = 2 pt,scale = 0.1] at (B) {\arbreopdeuxdeux} ;
\draw (A.center) -- (B.center) ;
\node[above] at (0,0) {\arbreoptrois} ;
\end{tikzpicture}
}


\newcommand{\arbreoptdBC}{
\begin{tikzpicture}
\begin{scope}
\draw (0,1.25) -- (0,0) ;
\draw (-1.25,1.25) -- (0,0) ;
\draw (1.25,1.25) -- (0,0) ;
\draw (0,-1.25) -- (0,0) ;
\end{scope}

\begin{scope}[shift = {(1.25,3)}]
\draw (-1.25,1.25) -- (0,0) ;
\draw (1.25,1.25) -- (0,0) ;
\draw (0,-1.25) -- (0,0) ;
\end{scope}
\end{tikzpicture}
}

\newcommand{\arbreoptdDE}{
\begin{tikzpicture}
\begin{scope}
\draw (0,1.25) -- (0,0) ;
\draw (-1.25,1.25) -- (0,0) ;
\draw (1.25,1.25) -- (0,0) ;
\draw (0,-1.25) -- (0,0) ;
\end{scope}

\begin{scope}[shift = {(0,3)}]
\draw (-1.25,1.25) -- (0,0) ;
\draw (1.25,1.25) -- (0,0) ;
\draw (0,-1.25) -- (0,0) ;
\end{scope}
\end{tikzpicture}
}

\newcommand{\arbreoptdAB}{
\begin{tikzpicture}
\begin{scope}
\draw (0,1.25) -- (0,0) ;
\draw (-1.25,1.25) -- (0,0) ;
\draw (1.25,1.25) -- (0,0) ;
\draw (0,-1.25) -- (0,0) ;
\end{scope}

\begin{scope}[shift = {(-1.25,3)}]
\draw (-1.25,1.25) -- (0,0) ;
\draw (1.25,1.25) -- (0,0) ;
\draw (0,-1.25) -- (0,0) ;
\end{scope}
\end{tikzpicture}
}

\newcommand{\arbreoptdCD}{
\begin{tikzpicture}

\begin{scope}
\draw (-1.25,1.25) -- (0,0) ;
\draw (1.25,1.25) -- (0,0) ;
\draw (0,-1.25) -- (0,0) ;
\end{scope}

\begin{scope}[shift = {(1.25,3)}]
\draw (0,1.25) -- (0,0) ;
\draw (-1.25,1.25) -- (0,0) ;
\draw (1.25,1.25) -- (0,0) ;
\draw (0,-1.25) -- (0,0) ;
\end{scope}

\end{tikzpicture}
}

\newcommand{\arbreoptdEA}{
\begin{tikzpicture}
\begin{scope}
\draw (-1.25,1.25) -- (0,0) ;
\draw (1.25,1.25) -- (0,0) ;
\draw (0,-1.25) -- (0,0) ;
\end{scope}

\begin{scope}[shift = {(-1.25,3)}]
\draw (0,1.25) -- (0,0) ;
\draw (-1.25,1.25) -- (0,0) ;
\draw (1.25,1.25) -- (0,0) ;
\draw (0,-1.25) -- (0,0) ;
\end{scope}
\end{tikzpicture}
}

\newcommand{\arbreopdddA}{
\begin{tikzpicture}

\begin{scope}
\draw (-1.25,1.25) -- (0,0) ;
\draw (1.25,1.25) -- (0,0) ;
\draw (0,-1.25) -- (0,0) ;
\end{scope}

\begin{scope}[shift = {(-1.25,3)}]
\draw (-1.25,1.25) -- (0,0) ;
\draw (1.25,1.25) -- (0,0) ;
\draw (0,-1.25) -- (0,0) ;
\end{scope}

\begin{scope}[shift = {(-2.5,6)}]
\draw (-1.25,1.25) -- (0,0) ;
\draw (1.25,1.25) -- (0,0) ;
\draw (0,-1.25) -- (0,0) ;
\end{scope}
\end{tikzpicture}
}

\newcommand{\arbreopdddB}{
\begin{tikzpicture}

\begin{scope}
\draw (-1.25,1.25) -- (0,0) ;
\draw (1.25,1.25) -- (0,0) ;
\draw (0,-1.25) -- (0,0) ;
\end{scope}

\begin{scope}[shift = {(-1.4,3)}]
\draw (-1.25,1.25) -- (0,0) ;
\draw (1.25,1.25) -- (0,0) ;
\draw (0,-1.25) -- (0,0) ;
\end{scope}

\begin{scope}[shift = {(1.4,3)}]
\draw (-1.25,1.25) -- (0,0) ;
\draw (1.25,1.25) -- (0,0) ;
\draw (0,-1.25) -- (0,0) ;
\end{scope}
\end{tikzpicture}
}

\newcommand{\arbreopdddC}{
\begin{tikzpicture}

\begin{scope}
\draw (-1.25,1.25) -- (0,0) ;
\draw (1.25,1.25) -- (0,0) ;
\draw (0,-1.25) -- (0,0) ;
\end{scope}

\begin{scope}[shift = {(1.25,3)}]
\draw (-1.25,1.25) -- (0,0) ;
\draw (1.25,1.25) -- (0,0) ;
\draw (0,-1.25) -- (0,0) ;
\end{scope}

\begin{scope}[shift = {(2.5,6)}]
\draw (-1.25,1.25) -- (0,0) ;
\draw (1.25,1.25) -- (0,0) ;
\draw (0,-1.25) -- (0,0) ;
\end{scope}
\end{tikzpicture}
}

\newcommand{\arbreopdddD}{
\begin{tikzpicture}

\begin{scope}
\draw (-1.25,1.25) -- (0,0) ;
\draw (1.25,1.25) -- (0,0) ;
\draw (0,-1.25) -- (0,0) ;
\end{scope}

\begin{scope}[shift = {(1.25,3)}]
\draw (-1.25,1.25) -- (0,0) ;
\draw (1.25,1.25) -- (0,0) ;
\draw (0,-1.25) -- (0,0) ;
\end{scope}

\begin{scope}[shift = {(0,6)}]
\draw (-1.25,1.25) -- (0,0) ;
\draw (1.25,1.25) -- (0,0) ;
\draw (0,-1.25) -- (0,0) ;
\end{scope}
\end{tikzpicture}
}

\newcommand{\arbreopdddE}{
\begin{tikzpicture}

\begin{scope}
\draw (-1.25,1.25) -- (0,0) ;
\draw (1.25,1.25) -- (0,0) ;
\draw (0,-1.25) -- (0,0) ;
\end{scope}

\begin{scope}[shift = {(-1.25,3)}]
\draw (-1.25,1.25) -- (0,0) ;
\draw (1.25,1.25) -- (0,0) ;
\draw (0,-1.25) -- (0,0) ;
\end{scope}

\begin{scope}[shift = {(0,6)}]
\draw (-1.25,1.25) -- (0,0) ;
\draw (1.25,1.25) -- (0,0) ;
\draw (0,-1.25) -- (0,0) ;
\end{scope}
\end{tikzpicture}
}

\newcommand{\associaedretrois}{
\begin{tikzpicture}[scale = 0.5 , baseline = (pt_base.base)]

\coordinate (pt_base) at (-1,0) ; 

\node[scale = 0.1] (A) at (0,3) {} ;
\node[scale=0.1,above = 5pt] at (A) {\arbreopdddA} ;
\node[scale = 0.1] (B) at (2,0) {} ;
\node[scale = 0.1,right = 5 pt] at (B) {\arbreopdddB} ;
\node[scale = 0.1] (C) at (0,-3) {} ;
\node[scale = 0.1 , below right = 5 pt] at (C) {\arbreopdddC} ;
\node[scale = 0.1] (D) at (-4,-1.5) {} ;
\node[scale = 0.1,below left = 7 pt] at (D) {\arbreopdddD} ;
\node[scale = 0.1] (E) at (-4,1.5) {} ;
\node[scale = 0.1,above left = 7 pt] at (E) {\arbreopdddE} ;

\node at (-1,0) {\arbreopquatre} ;

\draw[fill,opacity = 0.1] (A.center) -- (B.center) -- (C.center) -- (D.center) -- (E.center) -- (A.center) ;

\draw (A.center) -- (B.center) -- (C.center) -- (D.center) -- (E.center) -- (A.center) ;

\node[scale=0.1] at (A) {\pointbullet} ;
\node[scale=0.1] at (B) {\pointbullet} ;
\node[scale=0.1] at (C) {\pointbullet} ;
\node[scale=0.1] at (D) {\pointbullet} ;
\node[scale=0.1] at (E) {\pointbullet} ;

\node[scale = 0.1,above right = 3 pt] at ($(A)!0.5!(B)$) {\arbreoptdAB} ;
\node[scale = 0.1,below right = 3 pt] at ($(B)!0.5!(C)$) {\arbreoptdBC} ;
\node[scale = 0.1,below left = 6 pt] at ($(C)!0.5!(D)$) {\arbreoptdCD} ;
\node[scale = 0.1,left = 2 pt] at ($(D)!0.5!(E)$) {\arbreoptdDE} ;
\node[scale = 0.1,above left = 6 pt] at ($(E)!0.5!(A)$) {\arbreoptdEA} ;

\end{tikzpicture}
}


\newcommand{\multiplaedreun}{
\begin{tikzpicture}[ baseline = (pt_base.base)]
\coordinate (pt_base) at (0,0) ;
\node[scale = 0.1] at (0,0) {\pointbullet} ;
\node[above left] at (0,0) {\arbreopunmorph} ;
\end{tikzpicture}
}


\newcommand{\arbreopdeuxunmorph}{
\begin{tikzpicture}

\begin{scope}[red]
\draw (-1.25,1.25) -- (0,0) ;
\draw (1.25,1.25) -- (0,0) ;
\draw (0,-1.25) -- (0,0) ;
\end{scope}

\begin{scope}[shift = {(-1.25,3)}]

\draw[blue] (0,1.25) -- (0,0.5) ;
\draw[red] (0,-1.25) -- (0,0.5) ;

\draw (-1.25,0.5) -- (0.75,0.5) ;

\end{scope}

\begin{scope}[shift = {(1.25,3)}]

\draw[blue] (0,1.25) -- (0,0.5) ;
\draw[red] (0,-1.25) -- (0,0.5) ;

\draw (-0.75,0.5) -- (1.25,0.5) ;

\end{scope}

\end{tikzpicture}
}

\newcommand{\arbreopdeuxunmorphbis}{
\begin{tikzpicture}[scale=0.10,baseline=(pt_base.base)]

\coordinate (pt_base) at (0,0) ;

\begin{scope}
\draw[blue] (0,1.25) -- (0,0.5) ;
\draw[red] (0,-1.25) -- (0,0.5) ;

\draw (-1.25,0.5) -- (1.25,0.5) ;
\end{scope}

\begin{scope}[blue, shift = {(0,3)}]
\draw (-1.25,1.25) -- (0,0) ;
\draw (1.25,1.25) -- (0,0) ;
\draw (0,-1.25) -- (0,0) ;
\end{scope}

\end{tikzpicture}
}

\newcommand{\arbreopdeuxdeuxmorphbis}{
\begin{tikzpicture}[scale=0.10,baseline=(pt_base.base)]

\coordinate (pt_base) at (0,0) ;

\begin{scope}[red]
\draw (-1.25,1.25) -- (0,0) ;
\draw (1.25,1.25) -- (0,0) ;
\draw (0,-1.25) -- (0,0) ;
\end{scope}

\begin{scope}[shift = {(-1.25,3)}]

\draw[blue] (0,1.25) -- (0,0.5) ;
\draw[red] (0,-1.25) -- (0,0.5) ;

\draw (-1.25,0.5) -- (0.75,0.5) ;

\end{scope}

\begin{scope}[shift = {(1.25,3)}]

\draw[blue] (0,1.25) -- (0,0.5) ;
\draw[red] (0,-1.25) -- (0,0.5) ;

\draw (-0.75,0.5) -- (1.25,0.5) ;

\end{scope}

\end{tikzpicture}
}

\newcommand{\arbreopdeuxdeuxmorph}{
\begin{tikzpicture}

\begin{scope}
\draw[blue] (0,1.25) -- (0,0.5) ;
\draw[red] (0,-1.25) -- (0,0.5) ;

\draw (-1.25,0.5) -- (1.25,0.5) ;
\end{scope}

\begin{scope}[blue, shift = {(0,3)}]
\draw (-1.25,1.25) -- (0,0) ;
\draw (1.25,1.25) -- (0,0) ;
\draw (0,-1.25) -- (0,0) ;
\end{scope}

\end{tikzpicture}
}

\newcommand{\multiplaedredeux}{
\begin{tikzpicture}[ baseline = (pt_base.base)]
\coordinate (pt_base) at (0,0) ;
\node[scale = 0.1] (A) at (-1,0) {\pointbullet} ;
\node[scale = 0.1] (B) at (1,0) {\pointbullet} ;
\node[above left = 2 pt,scale = 0.1] at (A) {\arbreopdeuxunmorph} ;
\node[above right = 2 pt,scale = 0.1] at (B) {\arbreopdeuxdeuxmorph} ;
\draw (A.center) -- (B.center) ;
\node[above] at (0,0) {\arbreopdeuxmorph} ;
\end{tikzpicture}
}


\newcommand{\multiplaedretroisA}{
\begin{tikzpicture}

\begin{scope}
\draw[blue] (0,1.25) -- (0,0.5) ;
\draw[red] (0,-1.25) -- (0,0.5) ;

\draw (-1.25,0.5) -- (1.25,0.5) ;
\end{scope}

\begin{scope}[shift = {(0,3)},blue]
\draw (-1.25,1.25) -- (0,0) ;
\draw (1.25,1.25) -- (0,0) ;
\draw (0,-1.25) -- (0,0) ;
\end{scope}

\begin{scope}[shift = {(1.25,6)},blue]
\draw (-1.25,1.25) -- (0,0) ;
\draw (1.25,1.25) -- (0,0) ;
\draw (0,-1.25) -- (0,0) ;
\end{scope}
\end{tikzpicture}
}

\newcommand{\multiplaedretroisB}{
\begin{tikzpicture}

\begin{scope}[red]
\draw (-1.25,1.25) -- (0,0) ;
\draw (1.25,1.25) -- (0,0) ;
\draw (0,-1.25) -- (0,0) ;
\end{scope}

\begin{scope}[shift = {(-1.25,3)}]
\draw[blue] (0,1.25) -- (0,0.5) ;
\draw[red] (0,-1.25) -- (0,0.5) ;

\draw (-1.25,0.5) -- (0.75,0.5) ;
\end{scope}

\begin{scope}[shift = {(1.25,3)}]
\draw[blue] (0,1.25) -- (0,0.5) ;
\draw[red] (0,-1.25) -- (0,0.5) ;

\draw (-0.75,0.5) -- (1.25,0.5) ;
\end{scope}

\begin{scope}[shift = {(1.25,6)},blue]
\draw (-1.25,1.25) -- (0,0) ;
\draw (1.25,1.25) -- (0,0) ;
\draw (0,-1.25) -- (0,0) ;
\end{scope}
\end{tikzpicture}
}

\newcommand{\multiplaedretroisC}{
\begin{tikzpicture}

\begin{scope}[red]
\draw (-1.25,1.25) -- (0,0) ;
\draw (1.25,1.25) -- (0,0) ;
\draw (0,-1.25) -- (0,0) ;
\end{scope}

\begin{scope}[shift = {(1.25,3)},red]
\draw (-1.25,1.25) -- (0,0) ;
\draw (1.25,1.25) -- (0,0) ;
\draw (0,-1.25) -- (0,0) ;
\end{scope}

\begin{scope}[shift = {(2.5,6)}]
\draw[blue] (0,1.25) -- (0,0.5) ;
\draw[red] (0,-1.25) -- (0,0.5) ;

\draw (-0.75,0.5) -- (1.25,0.5) ;
\end{scope}

\begin{scope}[shift = {(0,6)}]
\draw[blue] (0,1.25) -- (0,0.5) ;
\draw[red] (0,-1.25) -- (0,0.5) ;

\draw (-1.25,0.5) -- (0.75,0.5) ;
\end{scope}

\begin{scope}[shift = {(-1.25,3)}]
\draw[blue] (0,1.25) -- (0,0.5) ;
\draw[red] (0,-1.25) -- (0,0.5) ;

\draw (-1.25,0.5) -- (1.25,0.5) ;
\end{scope}

\end{tikzpicture}
}

\newcommand{\multiplaedretroisD}{
\begin{tikzpicture}

\begin{scope}[red]
\draw (-1.25,1.25) -- (0,0) ;
\draw (1.25,1.25) -- (0,0) ;
\draw (0,-1.25) -- (0,0) ;
\end{scope}

\begin{scope}[shift = {(-1.25,3)},red]
\draw (-1.25,1.25) -- (0,0) ;
\draw (1.25,1.25) -- (0,0) ;
\draw (0,-1.25) -- (0,0) ;
\end{scope}

\begin{scope}[shift = {(-2.5,6)}]
\draw[blue] (0,1.25) -- (0,0.5) ;
\draw[red] (0,-1.25) -- (0,0.5) ;

\draw (-1.25,0.5) -- (0.75,0.5) ;
\end{scope}

\begin{scope}[shift = {(0,6)}]
\draw[blue] (0,1.25) -- (0,0.5) ;
\draw[red] (0,-1.25) -- (0,0.5) ;

\draw (-0.75,0.5) -- (1.25,0.5) ;
\end{scope}

\begin{scope}[shift = {(1.25,3)}]
\draw[blue] (0,1.25) -- (0,0.5) ;
\draw[red] (0,-1.25) -- (0,0.5) ;

\draw (-1.25,0.5) -- (1.25,0.5) ;
\end{scope}

\end{tikzpicture}
}

\newcommand{\multiplaedretroisE}{
\begin{tikzpicture}

\begin{scope}[red]
\draw (-1.25,1.25) -- (0,0) ;
\draw (1.25,1.25) -- (0,0) ;
\draw (0,-1.25) -- (0,0) ;
\end{scope}

\begin{scope}[shift = {(-1.25,3)}]
\draw[blue] (0,1.25) -- (0,0.5) ;
\draw[red] (0,-1.25) -- (0,0.5) ;

\draw (-1.25,0.5) -- (0.75,0.5) ;
\end{scope}

\begin{scope}[shift = {(1.25,3)}]
\draw[blue] (0,1.25) -- (0,0.5) ;
\draw[red] (0,-1.25) -- (0,0.5) ;

\draw (-0.75,0.5) -- (1.25,0.5) ;
\end{scope}

\begin{scope}[shift = {(-1.25,6)},blue]
\draw (-1.25,1.25) -- (0,0) ;
\draw (1.25,1.25) -- (0,0) ;
\draw (0,-1.25) -- (0,0) ;
\end{scope}
\end{tikzpicture}
}

\newcommand{\multiplaedretroisF}{
\begin{tikzpicture}

\begin{scope}
\draw[blue] (0,1.25) -- (0,0.5) ;
\draw[red] (0,-1.25) -- (0,0.5) ;

\draw (-1.25,0.5) -- (1.25,0.5) ;
\end{scope}

\begin{scope}[shift = {(0,3)},blue]
\draw (-1.25,1.25) -- (0,0) ;
\draw (1.25,1.25) -- (0,0) ;
\draw (0,-1.25) -- (0,0) ;
\end{scope}

\begin{scope}[shift = {(-1.25,6)},blue]
\draw (-1.25,1.25) -- (0,0) ;
\draw (1.25,1.25) -- (0,0) ;
\draw (0,-1.25) -- (0,0) ;
\end{scope}
\end{tikzpicture}
}

\newcommand{\multiplaedretroisAB}{
\begin{tikzpicture}

\begin{scope}
\draw[blue] (-1.25,1.25) -- (-0.5,0.5) ;
\draw[blue] (1.25,1.25) -- (0.5,0.5) ;

\draw[red] (-0.5,0.5) -- (0,0) ;
\draw[red] (0.5,0.5) -- (0,0) ;
\draw[red] (0,-1.25) -- (0,0) ;

\draw (-1.25,0.5) -- (1.25,0.5) ;
\end{scope}

\begin{scope}[shift = {(1.25,3)},blue]
\draw (-1.25,1.25) -- (0,0) ;
\draw (1.25,1.25) -- (0,0) ;
\draw (0,-1.25) -- (0,0) ;
\end{scope}
\end{tikzpicture}
}

\newcommand{\multiplaedretroisBC}{
\begin{tikzpicture}

\begin{scope}[shift = {(1.25,3)}]
\draw[blue] (-1.25,1.25) -- (-0.5,0.5) ;
\draw[blue] (1.25,1.25) -- (0.5,0.5) ;

\draw[red] (-0.5,0.5) -- (0,0) ;
\draw[red] (0.5,0.5) -- (0,0) ;
\draw[red] (0,-1.25) -- (0,0) ;

\draw (-1.25,0.5) -- (1.25,0.5) ;
\end{scope}

\begin{scope}[red]
\draw (-1.25,1.25) -- (0,0) ;
\draw (1.25,1.25) -- (0,0) ;
\draw (0,-1.25) -- (0,0) ;
\end{scope}

\begin{scope}[shift = {(-1.25,3)}]
\draw[blue] (0,1.25) -- (0,0.5) ;
\draw[red] (0,-1.25) -- (0,0.5) ;

\draw (-1.25,0.5) -- (0.75,0.5) ;
\end{scope}
\end{tikzpicture}
}

\newcommand{\multiplaedretroisDE}{
\begin{tikzpicture}

\begin{scope}[shift = {(-1.25,3)}]
\draw[blue] (-1.25,1.25) -- (-0.5,0.5) ;
\draw[blue] (1.25,1.25) -- (0.5,0.5) ;

\draw[red] (-0.5,0.5) -- (0,0) ;
\draw[red] (0.5,0.5) -- (0,0) ;
\draw[red] (0,-1.25) -- (0,0) ;

\draw (-1.25,0.5) -- (1.25,0.5) ;
\end{scope}

\begin{scope}[red]
\draw (-1.25,1.25) -- (0,0) ;
\draw (1.25,1.25) -- (0,0) ;
\draw (0,-1.25) -- (0,0) ;
\end{scope}

\begin{scope}[shift = {(1.25,3)}]
\draw[blue] (0,1.25) -- (0,0.5) ;
\draw[red] (0,-1.25) -- (0,0.5) ;

\draw (-0.75,0.5) -- (1.25,0.5) ;
\end{scope}
\end{tikzpicture}
}

\newcommand{\multiplaedretroisEF}{
\begin{tikzpicture}

\begin{scope}
\draw[blue] (-1.25,1.25) -- (-0.5,0.5) ;
\draw[blue] (1.25,1.25) -- (0.5,0.5) ;

\draw[red] (-0.5,0.5) -- (0,0) ;
\draw[red] (0.5,0.5) -- (0,0) ;
\draw[red] (0,-1.25) -- (0,0) ;

\draw (-1.25,0.5) -- (1.25,0.5) ;
\end{scope}

\begin{scope}[shift = {(-1.25,3)},blue]
\draw (-1.25,1.25) -- (0,0) ;
\draw (1.25,1.25) -- (0,0) ;
\draw (0,-1.25) -- (0,0) ;
\end{scope}
\end{tikzpicture}
}

\newcommand{\multiplaedretroisCD}{
\begin{tikzpicture}

\begin{scope}[shift = {(-1.25,3)}]
\draw[blue] (0,1.25) -- (0,0.5) ;
\draw[red] (0,-1.25) -- (0,0.5) ;

\draw (-1.25,0.5) -- (0.5,0.5) ;
\end{scope}

\begin{scope}[shift = {(0,3)}]
\draw[blue] (0,1.25) -- (0,0.5) ;
\draw[red] (0,-1.25) -- (0,0.5) ;

\draw (-0.5,0.5) -- (0.5,0.5) ;
\end{scope}

\begin{scope}[shift = {(1.25,3)}]
\draw[blue] (0,1.25) -- (0,0.5) ;
\draw[red] (0,-1.25) -- (0,0.5) ;

\draw (-0.5,0.5) -- (1.25,0.5) ;
\end{scope}

\begin{scope}[red]
\draw (0,1.25) -- (0,0) ;
\draw (-1.25,1.25) -- (0,0) ;
\draw (1.25,1.25) -- (0,0) ;
\draw (0,-1.25) -- (0,0) ;
\end{scope}

\end{tikzpicture}
}

\newcommand{\multiplaedretroisFA}{
\begin{tikzpicture}

\begin{scope}
\draw[blue] (0,1.25) -- (0,0.5) ;
\draw[red] (0,-1.25) -- (0,0.5) ;

\draw (-1.25,0.5) -- (1.25,0.5) ;
\end{scope}

\begin{scope}[blue, shift = {(0,3)}]
\draw (0,1.25) -- (0,0) ;
\draw (-1.25,1.25) -- (0,0) ;
\draw (1.25,1.25) -- (0,0) ;
\draw (0,-1.25) -- (0,0) ;
\end{scope}

\end{tikzpicture}
}

\newcommand{\multiplaedretrois}{
\begin{tikzpicture}[scale = 1 , baseline = (pt_base.base)]

\coordinate (pt_base) at (-1,0) ; 

\node[scale = 0.1] (A) at (0.85,-1.4) {} ;
\node[scale=0.1,below = 5pt] at (A) {\multiplaedretroisA} ;
\node[scale = 0.1] (B) at (1.65,0) {} ;
\node[scale = 0.1,right = 5 pt] at (B) {\multiplaedretroisB} ;
\node[scale = 0.1] (C) at (0.85,1.35) {} ;
\node[scale = 0.1 , above = 5 pt] at (C) {\multiplaedretroisC} ;
\node[scale = 0.1] (D) at (-0.75,1.35) {} ;
\node[scale = 0.1,above = 5 pt] at (D) {\multiplaedretroisD} ;
\node[scale = 0.1] (E) at (-1.5,0) {} ;
\node[scale = 0.1,left = 5 pt] at (E) {\multiplaedretroisE} ;
\node[scale = 0.1] (F) at (-0.75,-1.4) {} ;
\node[scale = 0.1,below = 5 pt] at (F) {\multiplaedretroisF} ;

\node at (0,0) {\arbreoptroismorph} ;

\draw[fill,opacity = 0.1] (A.center) -- (B.center) -- (C.center) -- (D.center) -- (E.center) -- (F.center) -- (A.center) ;

\draw (A.center) -- (B.center) -- (C.center) -- (D.center) -- (E.center) -- (F.center) -- (A.center) ;

\node[scale=0.1] at (A) {\pointbullet} ;
\node[scale=0.1] at (B) {\pointbullet} ;
\node[scale=0.1] at (C) {\pointbullet} ;
\node[scale=0.1] at (D) {\pointbullet} ;
\node[scale=0.1] at (E) {\pointbullet} ;
\node[scale=0.1] at (F) {\pointbullet} ;

\node[scale = 0.1,below right = 5 pt] at ($(A)!0.5!(B)$) {\multiplaedretroisAB} ;
\node[scale = 0.1,above right = 5 pt] at ($(B)!0.5!(C)$) {\multiplaedretroisBC} ;
\node[scale = 0.1,above = 5 pt] at ($(C)!0.5!(D)$) {\multiplaedretroisCD} ;
\node[scale = 0.1,above left = 5 pt] at ($(D)!0.5!(E)$) {\multiplaedretroisDE} ;
\node[scale = 0.1,below left = 5 pt] at ($(E)!0.5!(F)$) {\multiplaedretroisEF} ;
\node[scale = 0.1,below = 5 pt] at ($(F)!0.5!(A)$) {\multiplaedretroisFA} ;

\end{tikzpicture}
}


\newcommand{\exampleribbontree}{
\begin{tikzpicture}[scale = 0.4]
\node (A) at (-1,1) {} ;
\node (B) at (-1,0) {} ;
\node (C) at (0,-1) {} ;
\node (D) at (1,0) {} ;
\node (E) at (0,1) {} ;
\node (F) at (2,1) {} ;
\node (G) at (0,-2) {} ;

\draw (A.center) -- (B.center) -- (C.center) -- (D.center) -- (E.center) ;
\draw (D.center) -- (F.center) ;
\draw (C.center) -- (G.center) ;

\end{tikzpicture}}

\newcommand{\examplemetricribbontree}{
\begin{tikzpicture}[scale = 0.4]
\node (A) at (-1,1) {} ;
\node (B) at (-1,0) {} ;
\node (C) at (0,-1) {} ;
\node (D) at (1,0) {} ;
\node (E) at (0,1) {} ;
\node (F) at (2,1) {} ;
\node (G) at (0,-2) {} ;

\draw (A.center) -- (B.center) -- (C.center) -- (D.center) -- (E.center) ;
\draw (D.center) -- (F.center) ;
\draw (C.center) -- (G.center) ;

\node[below left = 0.2 pt] at ($(B)!0.5!(C)$) {\tiny $l_1$} ;
\node[below right = 0.2 pt] at ($(C)!0.5!(D)$) {\tiny $l_2$} ;

\end{tikzpicture}}

\newcommand{\examplestablemetricribbontree}{
\begin{tikzpicture}[scale = 0.4]
\node (A) at (-1.5,1) {} ;
\node (B) at (0,0) {} ;
\node (C) at (1.5,1) {} ;
\node (D) at (0,-1) {} ;
\node (E) at (0.25,2) {} ;
\node (F) at (2.75,2) {} ;
\node (G) at (-0.25,2) {} ;
\node (H) at (-1.5,2) {} ;
\node (I) at (-2.75,2) {} ;

\draw (I.center) -- (A.center) -- (B.center) -- (C.center) -- (E.center) ;
\draw (F.center) -- (C.center) ;
\draw (H.center) -- (A.center) ;
\draw (G.center) -- (A.center) ;
\draw (B.center) -- (D.center) ;

\node[below left = 0.2 pt] at ($(A)!0.5!(B)$) {\tiny $l_1$} ;
\node[below right = 0.2 pt] at ($(C)!0.5!(B)$) {\tiny $l_2$} ;

\end{tikzpicture}
}

\newcommand{\examplebinarymetricribbontree}{
\begin{tikzpicture}[scale = 0.4]
\node (A) at (-1.5,1) {} ;
\node (B) at (0,0) {} ;
\node (C) at (1.5,1) {} ;
\node (D) at (0,-1) {} ;
\node (E) at (0.25,2) {} ;
\node (F) at (2.75,2) {} ;
\node (G) at (-0.25,2) {} ;
\node (I) at (-2.75,2) {} ;

\draw (I.center) -- (A.center) -- (B.center) -- (C.center) -- (E.center) ;
\draw (F.center) -- (C.center) ;
\draw (G.center) -- (A.center) ;
\draw (B.center) -- (D.center) ;

\node[below left = 0.2 pt] at ($(A)!0.5!(B)$) {\tiny $l_1$} ;
\node[below right = 0.2 pt] at ($(C)!0.5!(B)$) {\tiny $l_2$} ;

\end{tikzpicture}
}


\newcommand{\segm}{
\begin{tikzpicture}[very thick]
\draw (0,-1) -- (0,1) ;
\draw (-0.5,1) -- (0.5,1) ;
\draw (-0.5,-1) -- (0.5,-1) ;
\end{tikzpicture}}

\newcommand{\examplediffcobarbarA}{\begin{tikzpicture}[scale = 0.15 ,baseline=(pt_base.base)]

\coordinate (pt_base) at (0,-0.5) ;
\node (A) at (-1.5,1) {} ;
\node (B) at (0,0) {} ;
\node (C) at (1.5,1) {} ;
\node (D) at (0,-1) {} ;
\node (E) at (0.25,2) {} ;
\node (F) at (2.75,2) {} ;
\node (G) at (-0.25,2) {} ;
\node (I) at (-2.75,2) {} ;

\draw (I.center) -- (A.center) -- (B.center) -- (C.center) -- (E.center) ;
\draw (F.center) -- (C.center) ;
\draw (G.center) -- (A.center) ;
\draw (B.center) -- (D.center) ;
\end{tikzpicture}}

\newcommand{\examplediffcobarbarB}{
\begin{tikzpicture}[scale = 0.15 ,baseline=(pt_base.base)]

\coordinate (pt_base) at (0,-0.5) ;
\begin{scope}
\node (K) at (0.625,0.625) {} ;
\node (L) at (0,0) {} ;
\draw (0.625,0.625) -- (0,1.25) ;
\draw (-1.25,1.25) -- (0,0) ;
\draw (1.25,1.25) -- (0,0) ;
\draw (0,-1.25) -- (0,0) ;
\node (O) at (0,-1.25) {} ;
\end{scope}

\begin{scope}[shift = {(-1.25,2.1)}]
\draw (-0.625,0.625) -- (0,0) ;
\draw (0.625,0.625) -- (0,0) ;
\draw (0,-0.625) -- (0,0) ;
\end{scope}
\end{tikzpicture}}

\newcommand{\examplediffcobarbarC}{
\begin{tikzpicture}[scale = 0.15,baseline=(pt_base.base)]

\coordinate (pt_base) at (0,-0.5) ;
\begin{scope}
\node (K) at (-0.625,0.625) {} ;
\node (L) at (0,0) {} ;
\draw (-0.625,0.625) -- (0,1.25) ;
\draw (-1.25,1.25) -- (0,0) ;
\draw (1.25,1.25) -- (0,0) ;
\draw (0,-1.25) -- (0,0) ;
\node (O) at (0,-1.25) {} ;
\end{scope}

\begin{scope}[shift = {(1.25,2.1)}]
\draw (-0.625,0.625) -- (0,0) ;
\draw (0.625,0.625) -- (0,0) ;
\draw (0,-0.625) -- (0,0) ;
\end{scope}
\end{tikzpicture}}

\newcommand{\examplediffcobarbarD}{
\begin{tikzpicture}[scale = 0.15 ,baseline=(pt_base.base)]

\coordinate (pt_base) at (0,-0.5) ;
\node (A) at (-1.5,1) {} ;
\node (B) at (0,0) {} ;
\node (C) at (1.5,1) {} ;
\node (D) at (0,-1) {} ;
\node (E) at (1,2) {} ;
\node (F) at (2.75,2) {} ;
\node (G) at (-0.25,2) {} ;
\node (I) at (-2.75,2) {} ;

\draw (I.center) -- (A.center) -- (B.center) -- (C.center) ;
\draw (F.center) -- (C.center) ;
\draw (G.center) -- (A.center) ;
\draw (B.center) -- (D.center) ;
\draw (B.center) -- (E.center) ;
\end{tikzpicture}}

\newcommand{\examplediffcobarbarE}{
\begin{tikzpicture}[scale = 0.15 ,baseline=(pt_base.base)]

\coordinate (pt_base) at (0,-0.5) ;
\node (A) at (-1.5,1) {} ;
\node (B) at (0,0) {} ;
\node (C) at (1.5,1) {} ;
\node (D) at (0,-1) {} ;
\node (E) at (0.25,2) {} ;
\node (F) at (2.75,2) {} ;
\node (G) at (-1,2) {} ;
\node (I) at (-2.75,2) {} ;

\draw (I.center) -- (A.center) -- (B.center) -- (C.center) -- (E.center) ;
\draw (F.center) -- (C.center) ;
\draw (B.center) -- (D.center) ;
\draw (B.center) -- (G.center) ;

\end{tikzpicture}}

\newcommand{\premiertermecobarbarA}{
\begin{tikzpicture}[scale=0.15,baseline=(pt_base.base)]

\coordinate (pt_base) at (0,-0.5) ;

\draw (-1.25,1.25) -- (0,0) ;
\draw (1.25,1.25) -- (0,0) ;
\draw (0,-1.25) -- (0,0) ;

\end{tikzpicture}
}

\newcommand{\premiertermecobarbarB}{
\begin{tikzpicture}[scale = 0.15,baseline = (O.base)]
\node (K) at (-0.625,0.625) {} ;
\node (L) at (0,0) {} ;
\draw (-0.625,0.625) -- (0,1.25) ;
\draw (-1.25,1.25) -- (0,0) ;
\draw (1.25,1.25) -- (0,0) ;
\draw (0,-1.25) -- (0,0) ;
\node (O) at (0,-0.5) {} ;
\end{tikzpicture}
}

\newcommand{\premiertermecobarbarC}{
\begin{tikzpicture}[scale = 0.15,baseline = (O.base)]
\node (K) at (0.625,0.625) {} ;
\node (L) at (0,0) {} ;
\draw (0.625,0.625) -- (0,1.25) ;
\draw (-1.25,1.25) -- (0,0) ;
\draw (1.25,1.25) -- (0,0) ;
\draw (0,-1.25) -- (0,0) ;
\node (O) at (0,-0.5) {} ;
\end{tikzpicture}
}

\newcommand{\premiertermecobarbarD}{
\begin{tikzpicture}[scale=0.15,baseline=(pt_base.base)]

\coordinate (pt_base) at (0,-0.5) ;

\draw (0,1.25) -- (0,0) ;
\draw (-1.25,1.25) -- (0,0) ;
\draw (1.25,1.25) -- (0,0) ;
\draw (0,-1.25) -- (0,0) ;

\end{tikzpicture}
}

\newcommand{\arbrebicoloreA}{
\begin{tikzpicture}[scale = 0.15 ,baseline=(pt_base.base)]

\coordinate (pt_base) at (0,-0.5) ;
\node (A) at (-1.5,1) {} ;
\node (B) at (0,0) {} ;
\node (C) at (1.5,1) {} ;
\node (D) at (0,-1) {} ;
\node (E) at (0.25,2) {} ;
\node (F) at (2.75,2) {} ;
\node (G) at (-0.25,2) {} ;
\node (I) at (-2.75,2) {} ;
\node (O) at (0.9,0.6) {} ;
\node (P) at (-0.9,0.6) {} ;

\node (K) at (-2.75,0.6) {} ;
\node (L) at (2.75,0.6) {} ;

\node (M) at (3,0.6) {} ;
\node (N) at (3,0) {} ;

\draw[blue] (F.center) -- (C.center) ;
\draw[blue] (E.center) -- (C.center) ;
\draw[blue] (O.center) -- (C.center) ;
\draw[blue] (A.center) -- (I.center) ;
\draw[blue] (A.center) -- (G.center) ;
\draw[blue] (A.center) -- (P.center) ;
\draw[red] (B.center) -- (P.center) ;
\draw[red] (B.center) -- (O.center) ;
\draw[red] (B.center) -- (D.center) ;

\draw (K.center) -- (L.center) ;

\end{tikzpicture}
}

\newcommand{\arbrebicoloreB}{
\begin{tikzpicture}[scale = 0.15 ,baseline=(pt_base.base)]

\coordinate (pt_base) at (0,-0.5) ;
\node (A) at (0,0) {} ;
\node (B) at (0,-1) {} ;
\node (C) at (1.25,1.25) {} ;
\node (D) at (-1.25,1.25) {} ;
\node (E) at (0,1.25) {} ;
\node (F) at (-0.625,0.625) {} ;
\node (G) at (-0.9,0.9) {} ;
\node (H) at (-0.35,0.9) {} ;
\node (I) at (0.9,0.9) {}  ;
\node (J) at (-1.25,0.9) {} ;
\node (K) at (1.25,0.9) {} ;
\node (L) at (1.5,0) {} ;
\node (M) at (1.5,0.9) {} ;

\draw[blue] (D.center) -- (G.center) ;
\draw[blue] (E.center) -- (H.center) ;
\draw[blue] (C.center) -- (I.center) ;

\draw[red] (G.center) -- (F.center) ;
\draw[red] (H.center) -- (F.center) ;
\draw[red] (F.center) -- (A.center) ;
\draw[red] (A.center) -- (I.center) ;
\draw[red] (A.center) -- (B.center) ;

\draw (J.center) -- (K.center) ;

\end{tikzpicture}}

\newcommand{\arbrebicoloreC}{
\begin{tikzpicture}[scale = 0.15 ,baseline=(pt_base.base)]

\coordinate (pt_base) at (0,-0.5) ;
\node (A) at (-1.5,1) {} ;
\node (B) at (0,0) {} ;
\node (C) at (1.5,1) {} ;
\node (D) at (0,-1) {} ;
\node (E) at (0.25,2) {} ;
\node (F) at (2.75,2) {} ;
\node (G) at (-0.25,2) {} ;
\node (I) at (-2.75,2) {} ;
\node (O) at (0.9,0.6) {} ;
\node (P) at (-0.9,0.6) {} ;

\node (K) at (-2.75,0.6) {} ;
\node (L) at (2.75,0.6) {} ;

\node (M) at (3,0.6) {} ;
\node (N) at (3,0) {} ;

\draw (F.center) -- (C.center) ;
\draw (E.center) -- (C.center) ;
\draw (O.center) -- (C.center) ;
\draw (A.center) -- (I.center) ;
\draw (A.center) -- (G.center) ;
\draw (A.center) -- (P.center) ;
\draw (B.center) -- (P.center) ;
\draw (B.center) -- (O.center) ;
\draw (B.center) -- (D.center) ;

\end{tikzpicture}
}

\newcommand{\arbrebicoloreD}{
\begin{tikzpicture}[scale = 0.15 ,baseline=(pt_base.base)]

\coordinate (pt_base) at (0,-0.5) ;
\node (A) at (-1.5,1) {} ;
\node (B) at (0,0) {} ;
\node (C) at (0.9,0.6) {} ;
\node (D) at (0,-1) {} ;
\node (E) at (0.25,2) {} ;
\node (F) at (2.75,2) {} ;
\node (G) at (-0.25,2) {} ;
\node (I) at (-2.75,2) {} ;
\node (O) at (0.9,0.6) {} ;
\node (P) at (-0.9,0.6) {} ;

\node (K) at (-2.75,0.6) {} ;
\node (L) at (2.75,0.6) {} ;

\node (M) at (3,0.6) {} ;
\node (N) at (3,0) {} ;

\draw[blue] (F.center) -- (C.center) ;
\draw[blue] (E.center) -- (C.center) ;
\draw[blue] (O.center) -- (C.center) ;
\draw[blue] (A.center) -- (I.center) ;
\draw[blue] (A.center) -- (G.center) ;
\draw[blue] (A.center) -- (P.center) ;
\draw[red] (B.center) -- (P.center) ;
\draw[red] (B.center) -- (O.center) ;
\draw[red] (B.center) -- (D.center) ;

\draw (K.center) -- (L.center) ;

\end{tikzpicture}
}

\newcommand{\arbrebicoloreE}{
\begin{tikzpicture}[scale = 0.15 ,baseline=(pt_base.base)]

\coordinate (pt_base) at (0,-0.5) ;
\node (A) at (-0.9,0.6) {} ;
\node (B) at (0,0) {} ;
\node (C) at (1.5,1) {} ;
\node (D) at (0,-1) {} ;
\node (E) at (0.25,2) {} ;
\node (F) at (2.75,2) {} ;
\node (G) at (-0.25,2) {} ;
\node (I) at (-2.75,2) {} ;
\node (O) at (0.9,0.6) {} ;
\node (P) at (-0.9,0.6) {} ;

\node (K) at (-2.75,0.6) {} ;
\node (L) at (2.75,0.6) {} ;

\node (M) at (3,0.6) {} ;
\node (N) at (3,0) {} ;

\draw[blue] (F.center) -- (C.center) ;
\draw[blue] (E.center) -- (C.center) ;
\draw[blue] (O.center) -- (C.center) ;
\draw[blue] (A.center) -- (I.center) ;
\draw[blue] (A.center) -- (G.center) ;
\draw[blue] (A.center) -- (P.center) ;
\draw[red] (B.center) -- (P.center) ;
\draw[red] (B.center) -- (O.center) ;
\draw[red] (B.center) -- (D.center) ;

\draw (K.center) -- (L.center) ;

\end{tikzpicture}
}

\newcommand{\arbrebicoloreF}{
\begin{tikzpicture}[scale = 0.15 ,baseline=(pt_base.base)]

\coordinate (pt_base) at (0,-0.5) ;
\node (A) at (-1.5,1) {} ;
\node (B) at (0,0) {} ;
\node (C) at (1.5,1) {} ;
\node (D) at (0,-1) {} ;
\node (E) at (0.25,2) {} ;
\node (F) at (2.75,2) {} ;
\node (G) at (-0.25,2) {} ;
\node (I) at (-2.75,2) {} ;

\node (K) at (-2.75,0) {} ;
\node (L) at (2.75,0) {} ;

\draw[blue] (F.center) -- (C.center) ;
\draw[blue] (E.center) -- (C.center) ;
\draw[blue] (B.center) -- (C.center) ;
\draw[blue] (A.center) -- (I.center) ;
\draw[blue] (A.center) -- (G.center) ;
\draw[blue] (A.center) -- (B.center) ;
\draw[red] (B.center) -- (D.center) ;

\draw (K.center) -- (L.center) ;

\end{tikzpicture}
}

\newcommand{\arbrebicoloreI}{
\begin{tikzpicture}[scale = 0.18,baseline=(pt_base.base)]

\coordinate (pt_base) at (0,-0.5) ;

\begin{scope}
\node (K) at (-0.625,0.625) {} ;
\node (L) at (0,0) {} ;
\draw[blue] (-0.625,0.625) -- (0,1.25) ;
\draw[blue] (-1.25,1.25) -- (-0.3,0.3) ;
\draw[blue] (1.25,1.25) -- (0.3,0.3) ;
\draw[red] (0,0) -- (-0.3,0.3) ;
\draw[red] (0,0) -- (0.3,0.3) ;
\draw[red] (0,-1.25) -- (0,0) ;

\draw (-1.25,0.3) -- (1.25,0.3) ;
\end{scope}

\begin{scope}[shift = {(1.25,2.1)},blue]
\draw (-0.625,0.625) -- (0,0) ;
\draw (0.625,0.625) -- (0,0) ;
\draw (0,-0.625) -- (0,0) ;
\end{scope}
\end{tikzpicture}
}

\newcommand{\arbrebicoloreJ}{
\begin{tikzpicture}[scale = 0.18,baseline=(pt_base.base)]

\coordinate (pt_base) at (0,-0.5) ;

\begin{scope}
\node (K) at (-0.625,0.625) {} ;
\node (L) at (0,0) {} ;
\draw[blue] (0.625,0.625) -- (0,1.25) ;
\draw[blue] (-1.25,1.25) -- (-0.3,0.3) ;
\draw[blue] (1.25,1.25) -- (0.3,0.3) ;
\draw[red] (0,0) -- (-0.3,0.3) ;
\draw[red] (0,0) -- (0.3,0.3) ;
\draw[red] (0,-1.25) -- (0,0) ;

\draw (-1.25,0.3) -- (1.25,0.3) ;
\end{scope}

\begin{scope}[shift = {(-1.25,2.1)},blue]
\draw (-0.625,0.625) -- (0,0) ;
\draw (0.625,0.625) -- (0,0) ;
\draw (0,-0.625) -- (0,0) ;
\end{scope}
\end{tikzpicture}
}

\newcommand{\arbrebicoloreK}{
\begin{tikzpicture}[scale = 0.1,baseline=(pt_base.base)]

\coordinate (pt_base) at (0,-0.5) ;

\begin{scope}[red]
\draw (-1.6,1.25) -- (0,0) ;
\draw (1.6,1.25) -- (0,0) ;
\draw (0,-1.25) -- (0,0) ;
\end{scope}

\begin{scope}[shift = {(-1.6,2.8)}]
\draw[blue] (-1.25,1.25) -- (0,0) ;
\draw[blue] (1.25,1.25) -- (0,0) ;
\draw[blue] (0,-0.625) -- (0,0) ;
\draw[red] (0,-0.625) -- (0,-1.25) ;
\draw (-1.25,-0.625) -- (1.25,-0.625) ;
\end{scope}

\begin{scope}[shift = {(1.6,2.8)}]
\draw[blue] (-1.25,1.25) -- (0,0) ;
\draw[blue] (1.25,1.25) -- (0,0) ;
\draw[blue] (0,-0.625) -- (0,0) ;
\draw[red] (0,-0.625) -- (0,-1.25) ;
\draw (-1.25,-0.625) -- (1.25,-0.625) ;
\end{scope}
\end{tikzpicture}
}

\newcommand{\arbrebicoloreLsub}{
\begin{tikzpicture}[scale=0.1,baseline=(pt_base.base)]

\coordinate (pt_base) at (0,-0.5) ;

\draw[blue] (-1.25,1.25) -- (-0.5,0.5) ;
\draw[blue] (1.25,1.25) -- (0.5,0.5) ;

\draw[orange] (-0.5,0.5) -- (0,0) ;
\draw[orange] (0.5,0.5) -- (0,0) ;
\draw[orange] (0,-1.25) -- (0,0) ;

\draw (-1.25,0.5) -- (1.25,0.5) ;

\end{tikzpicture}
}

\newcommand{\arbrebicoloreL}{
\begin{tikzpicture}[scale=0.15,baseline=(pt_base.base)]

\coordinate (pt_base) at (0,-0.5) ;

\draw[blue] (-1.25,1.25) -- (-0.5,0.5) ;
\draw[blue] (1.25,1.25) -- (0.5,0.5) ;

\draw[red] (-0.5,0.5) -- (0,0) ;
\draw[red] (0.5,0.5) -- (0,0) ;
\draw[red] (0,-1.25) -- (0,0) ;

\draw (-1.25,0.5) -- (1.25,0.5) ;

\end{tikzpicture}
}

\newcommand{\arbrebicoloreMsubun}{
\begin{tikzpicture}[scale=0.1,baseline=(pt_base.base)]

\coordinate (pt_base) at (0,-0.5) ;

\draw[orange] (-1.25,1.25) -- (0,0) ;
\draw[orange] (1.25,1.25) -- (0,0) ;

\draw[orange] (0,-0.5) -- (0,0) ;
\draw[red] (0,-0.5) -- (0,-1.25) ;

\draw (-1.25,-0.5) -- (1.25,-0.5) ;

\end{tikzpicture}
}

\newcommand{\arbrebicoloreMsubdeux}{
\begin{tikzpicture}[scale=0.09,baseline=(pt_base.base)]

\coordinate (pt_base) at (0,-0.5) ;

\draw[blue] (-1.25,1.25) -- (0,0) ;
\draw[blue] (1.25,1.25) -- (0,0) ;

\draw[blue] (0,-0.5) -- (0,0) ;
\draw[orange] (0,-0.5) -- (0,-1.25) ;

\draw (-1.25,-0.5) -- (1.25,-0.5) ;

\end{tikzpicture}
}

\newcommand{\arbrebicoloreMsub}{
\begin{tikzpicture}[scale=0.1,baseline=(pt_base.base)]

\coordinate (pt_base) at (0,-0.5) ;

\draw[blue] (-1.25,1.25) -- (0,0) ;
\draw[blue] (1.25,1.25) -- (0,0) ;

\draw[blue] (0,-0.5) -- (0,0) ;
\draw[red] (0,-0.5) -- (0,-1.25) ;

\draw (-1.25,-0.5) -- (1.25,-0.5) ;

\end{tikzpicture}
}

\newcommand{\arbrebicoloreM}{
\begin{tikzpicture}[scale=0.15,baseline=(pt_base.base)]

\coordinate (pt_base) at (0,-0.5) ;

\draw[blue] (-1.25,1.25) -- (0,0) ;
\draw[blue] (1.25,1.25) -- (0,0) ;

\draw[blue] (0,-0.5) -- (0,0) ;
\draw[red] (0,-0.5) -- (0,-1.25) ;

\draw (-1.25,-0.5) -- (1.25,-0.5) ;

\end{tikzpicture}
}

\newcommand{\arbrebicoloreN}{
\begin{tikzpicture}[scale=0.15,baseline=(pt_base.base)]

\coordinate (pt_base) at (0,-0.5) ;

\draw[blue] (-1.25,1.25) -- (0,0) ;
\draw[blue] (1.25,1.25) -- (0,0) ;

\draw[red] (0,-1.25) -- (0,0) ;

\draw (-1.25,0) -- (1.25,0) ;

\end{tikzpicture}
}

\newcommand{\examplestratumCTBC}{
\begin{tikzpicture}

\begin{scope}[red]
\draw (-0.625,0.625) -- (0,0) ;
\draw (0.625,0.625) -- (0,0) ;
\draw (0,-0.625) -- (0,0) ;
\end{scope}

\begin{scope}[shift = {(-0.625,1.5)}]
\draw[blue] (-0.625,0.625) -- (-0.3,0.3) ;
\draw[blue] (0.625,0.625) -- (0.3,0.3) ;
\draw[red] (-0.3,0.3) -- (0,0) ;
\draw[red] (0.3,0.3) -- (0,0) ;
\draw[red] (0,-0.625) -- (0,0) ;

\draw (-0.625,0.3) -- (0.625,0.3) ;

\draw[densely dotted,>=stealth,<->] (-0.7,0.3) -- (-0.7,0) node[midway,left] {$\lambda$} ;
\end{scope}

\begin{scope}[shift = {(0.625,1.5)}]
\draw[red] (0,-0.625) -- (0,0.3) ;
\draw[blue] (0,0.625) -- (0,0.3) ;

\draw (-0.3,0.3) -- (0.3,0.3) ;
\end{scope}
\end{tikzpicture}}

\newcommand{\examplestratumCTCD}{
\begin{tikzpicture}
\begin{scope}
\node (K) at (-0.625,0.625) {} ;
\node (L) at (0,0) {} ;
\draw[red] (-0.625,0.625) -- (0,1.25) ;
\draw[red] (-1.25,1.25) -- (0,0) ;
\draw[red] (1.25,1.25) -- (0,0) ;
\draw[red] (0,-1.25) -- (0,0) ;

\node[below left = 0.2 pt] at ($(K)!0.5!(L)$) { $l$} ;
\end{scope}

\begin{scope}[shift = {(0,1.9)}]
\draw[blue] (0,0.5) -- (0,0) ;
\draw[red] (0,-0.5) -- (0,0) ;

\draw (-0.6,0) -- (0.6,0) ;
\end{scope}

\begin{scope}[shift = {(1.25,1.9)}]
\draw[blue] (0,0.5) -- (0,0) ;
\draw[red] (0,-0.5) -- (0,0) ;

\draw (-0.3,0) -- (0.6,0) ;
\end{scope}

\begin{scope}[shift = {(-1.25,1.9)}]
\draw[blue] (0,0.5) -- (0,0) ;
\draw[red] (0,-0.5) -- (0,0) ;

\draw (-0.6,0) -- (0.3,0) ;
\end{scope}
\end{tikzpicture}}

\newcommand{\examplestratumCTAB}{
\begin{tikzpicture}
\node (A) at (0,0) {} ;
\node (B) at (0,-1) {} ;
\node (C) at (1.25,1.25) {} ;
\node (D) at (-1.25,1.25) {} ;
\node (E) at (0,1.25) {} ;
\node (F) at (-0.625,0.625) {} ;
\node (G) at (-0.625,0.625) {} ;
\node (H) at (-0.625,0.625) {} ;
\node (I) at (0.625,0.625) {}  ;
\node (J) at (-1.25,0.625) {} ;
\node (K) at (1.25,0.625) {} ;
\node (L) at (1.5,0) {} ;
\node (M) at (1.5,0.625) {} ;

\draw[blue] (D.center) -- (G.center) ;
\draw[blue] (E.center) -- (H.center) ;
\draw[blue] (C.center) -- (I.center) ;

\draw[red] (G.center) -- (F.center) ;
\draw[red] (H.center) -- (F.center) ;
\draw[red] (F.center) -- (A.center) ;
\draw[red] (A.center) -- (I.center) ;
\draw[red] (A.center) -- (B.center) ;

\draw (J.center) -- (K.center) ;

\draw[densely dotted,>=stealth,<->] (L.center) -- (M.center) ;

\node[right] at ($(L)!0.5!(M)$) { $\lambda = l$} ;
\node[below left] at ($(A)!0.5!(F)$) { $l$} ;
\end{tikzpicture}}

\newcommand{\examplestratumCTDA}{
\begin{tikzpicture}
\node (A) at (0,0) {} ;
\node (B) at (0,-1) {} ;
\node (C) at (1.25,1.25) {} ;
\node (D) at (-1.25,1.25) {} ;
\node (E) at (0,1.25) {} ;
\node (F) at (-0.625,0.625) {} ;
\node (G) at (-0.625,0.625) {} ;
\node (H) at (-0.625,0.625) {} ;
\node (I) at (0.625,0.625) {}  ;
\node (J) at (-1.25,0.625) {} ;
\node (K) at (1.25,0.625) {} ;
\node (L) at (1.5,0) {} ;
\node (M) at (1.5,0.625) {} ;

\node (N) at (0,0.625) {} ;

\draw[blue] (D.center) -- (G.center) ;
\draw[blue] (E.center) -- (N.center) ;
\draw[blue] (C.center) -- (I.center) ;

\draw[red] (G.center) -- (F.center) ;
\draw[red] (N.center) -- (A.center) ;
\draw[red] (F.center) -- (A.center) ;
\draw[red] (A.center) -- (I.center) ;
\draw[red] (A.center) -- (B.center) ;

\draw (J.center) -- (K.center) ;

\draw[densely dotted,>=stealth,<->] (L.center) -- (M.center) ;

\node[right] at ($(L)!0.5!(M)$) { $\lambda$} ;
\end{tikzpicture}}

\newcommand{\examplestratumCT}{
\begin{tikzpicture}[scale = 1.5]
\node (A) at (0,0) {} ;
\node (B) at (2,2) {} ;
\node (C) at (1,2.3) {} ;
\node (D) at (0,2) {} ;
\node (E) at (2.3,0) {} ;
\node (AB) at ($(A)!0.5!(B)$) {} ;
\node (BC) at ($(B)!0.5!(C)$) {} ;
\node (CD) at ($(C)!0.5!(D)$) {} ;
\node (DA) at ($(D)!0.5!(A)$) {} ;
\node (AC) at ($(A)!0.6!(C)$) {} ;

\draw[dotted,fill,opacity = 0.1] (A.center) -- (B.center) -- (C.center) -- (D.center) -- (A.center) ;
\draw (A.center) -- (B.center) ;
\draw[dotted] (B.center) -- (C.center) ;
\draw[dotted] (C.center) -- (D.center) ;
\draw[>=stealth,->] (A.center) -- (D.center) ;
\draw[>=stealth,->] (A.center) -- (E.center) ;

\node[below] at (E) {\tiny $l$} ;
\node[left] at (D) {\tiny $\lambda$} ;

\node[scale = 0.4 , below right = 2 pt] at (AB) {\examplestratumCTAB} ;
\node[scale = 0.4 , xshift = 10 pt, yshift = 40 pt] at (BC) {\examplestratumCTBC} ;
\node[scale = 0.3 , xshift = -30 pt, yshift = 60 pt] at (CD) {\examplestratumCTCD} ;
\node[scale = 0.4 , left = 8 pt ] at (DA) {\examplestratumCTDA} ;
\node[scale = 0.4,xshift = 4 pt] at (AC) {\CTtroisstratD} ;
\end{tikzpicture}}

\newcommand{\arbredegradientbicoloreexemple}{
\begin{tikzpicture}[scale = 2]

\node (A) at (-2,2) {};
\node (B) at (-1,2) {} ;
\node (C) at (1,2)  {};
\node (D) at (2,2) {} ;
\node (E) at (-1,1)  {};
\node (F) at (1,1) {} ;
\node (G) at (-0.5,0.5)  {};
\node (H) at (0.5,0.5) {};
\node (I) at (0,0) {} ;
\node (J) at (0,-1)  {};
\node (K) at (-2,0.5) {} ;
\node (L) at (2,0.5)  {};
\node (a) at ($(A)!0.7!(E)$) {} ;
\node (b) at ($(B)!0.7!(E)$)  {};
\node (c) at ($(C)!0.7!(F)$) {};
\node (d) at ($(D)!0.7!(F)$) {} ;
\node (e) at ($(G)!0.7!(E)$)  {};
\node (f) at ($(H)!0.7!(F)$) {} ;
\node (g) at ($(E)!0.7!(G)$)  {};
\node (gg) at ($(I)!0.7!(G)$) {} ;
\node (h) at ($(F)!0.7!(H)$)  {};
\node (hh) at ($(I)!0.7!(H)$) {};
\node (i) at ($(G)!0.7!(I)$) {} ;
\node (ii) at ($(H)!0.7!(I)$)  {};
\node (iii) at ($(J)!0.7!(I)$) {} ;

\draw (A.center) -- (a.center) node[near end,left = 3 pt]{\tiny $- \nabla f$} node[midway,sloped,allow upside down,scale=0.1]{\midarrow} ;
\draw (B.center) -- (b.center) node[near end,right = 3 pt]{\tiny $- \nabla f$} node[midway,sloped,allow upside down,scale=0.1]{\midarrow} ;
\draw (C.center) -- (c.center) node[near end,left = 3 pt]{\tiny $- \nabla f$} node[midway,sloped,allow upside down,scale=0.1]{\midarrow} ;
\draw (D.center) -- (d.center) node[near end,right = 3 pt]{\tiny $- \nabla f$} node[midway,sloped,allow upside down,scale=0.1]{\midarrow} ;

\draw [line width = 1.5 pt,blue] (a.center) -- (E.center) node[below left,blue] {\tiny $ - \nabla f + \mathbb{X}_f$} ;
\draw [line width = 1.5 pt,blue] (b.center) -- (E.center) ;
\draw [line width = 1.5 pt,blue] (c.center) -- (F.center) node[below right,blue] {\tiny $ - \nabla f + \mathbb{X}_f$} ;
\draw [line width = 1.5 pt,blue] (d.center) -- (F.center)  ;

\draw [line width = 1.5 pt,blue] (e.center) -- (E.center)  ;
\draw [line width = 1.5 pt,blue] (f.center) -- (F.center)  ;

\draw (e.center) -- (g.center) node[near end,left = 5 pt]{\tiny $- \nabla f$} node[midway,sloped,allow upside down,scale=0.1]{\midarrow} ;
\draw (f.center) -- (h.center) node[near end,right = 3 pt]{\tiny $- \nabla f$} node[midway,sloped,allow upside down,scale=0.1]{\midarrow} ;

\draw [line width = 1.5 pt,green] (g.center) -- (G.center) ;
\draw [line width = 1.5 pt,green] (h.center) -- (H.center)  ;

\node[green] at (0,0.65) {\tiny $ - \nabla f + \mathbb{Y}$} ;

\draw [line width = 1.5 pt,green] (hh.center) -- (H.center) ;
\draw [line width = 1.5 pt,green] (gg.center) -- (G.center) ;

\node[green] at (0,0.4) {\tiny $ - \nabla g + \mathbb{Y}$} ;

\draw (gg.center) -- (i.center) node[near end,left = 3 pt]{\tiny $- \nabla g$} node[midway,sloped,allow upside down,scale=0.1]{\midarrow} ;
\draw (hh.center) -- (ii.center) node[near end,right = 3 pt]{\tiny $- \nabla g$} node[midway,sloped,allow upside down,scale=0.1]{\midarrow} ;

\draw [line width = 1.5 pt,red] (i.center) -- (I.center) node[below left,red] {\tiny $ - \nabla g + \mathbb{X}_g$} ;
\draw [line width = 1.5 pt,red] (ii.center) -- (I.center) ;
\draw [line width = 1.5 pt,red] (iii.center) -- (I.center) ;

\draw (iii.center) -- (J.center) node[near end,left = 3 pt]{\tiny $- \nabla g$} node[midway,sloped,allow upside down,scale=0.1]{\midarrow} ;

\draw[line width = 1.5 pt] (K) -- (L) ;

\node[above=2 pt] at (A) {\small $x_1$}  ; 
\node[above,scale = 0.1] at (A) {\cross}  ;

\node[above=2 pt] at (B) {\small $x_2$}  ; 
\node[above,scale = 0.1] at (B) {\cross}  ;

\node[above=2 pt] at (C) {\small $x_3$}  ; 
\node[above,scale = 0.1] at (C) {\cross}  ;

\node[above=2 pt] at (D) {\small $x_4$}  ; 
\node[above,scale = 0.1] at (D) {\cross}  ;

\node[below=2 pt] at (J) {\small $y$}  ; 
\node[below,scale = 0.1] at (J) {\cross}  ;

\end{tikzpicture}}

\newcommand{\exampleperturbationCTbreakun}{\begin{tikzpicture}

\begin{scope}[yshift = -0.3cm]

\node (A) at (-1.25,1.25) {} ;
\node (B) at (0,1.25) {} ;
\node (C) at (1.25,1.25) {} ;
\node (D) at (-0.625,0.625) {} ;
\node (E) at (0,0) {} ;
\node (F) at (0,-1) {} ;

\node (G) at (-0.3,0.3) {} ;
\node (H) at (0.3,0.3) {} ;

\node (I) at (-1.25,0.3) {} ;
\node (J) at (1.25,0.3) {} ;

\node at (-1.5,-1) {$t_g$} ;

\draw (A.center) -- (D.center) ;
\draw (B.center) -- (D.center) ;
\draw (D.center) -- (G.center) ;
\draw (G.center) -- (E.center) ;
\draw (C.center) -- (H.center) ;
\draw (H.center) -- (E.center) ;
\draw (E.center) -- (F.center) ;

\draw[line width = 1.5 pt] (I.center) -- (J.center) ;

\node (d) at ($(A)!0.7!(D)$) {} ;

\draw [line width = 1.5 pt,green] (d.center) -- (D.center) ;

\node[green,xshift = -20 pt, yshift = -17 pt] at ($(A)!0.5!(D)$) { $\mathbb{Y}_{e,t_g}$} ;

\end{scope}

\begin{scope}[scale = 0.3,xshift = 12 cm]
\draw[->=latex] (-4,0) node{} -- (4,0) node{} node[midway,below = 2pt,align = center]{$\lim \textcolor{green}{\mathbb{Y}_{e,t_g}} = \textcolor{blue}{\mathbb{X}_{e,t^1}}$} ;
\end{scope}

\begin{scope}[xshift = 7cm,yshift = -0.7 cm,scale = 0.5]

\begin{scope}
\node (A) at (-1.25,1.25) {} ;
\node (B) at (1.25,1.25) {} ;
\node (C) at (0,0) {} ;
\node (D) at (0,-1) {} ;
\node (E) at (-1.25,0.5) {};
\node (F) at (1.25,0.5) {};

\draw (A.center) -- (C.center) ;
\draw (B.center) -- (C.center) ;
\draw (D.center) -- (C.center) ;

\draw[line width = 1.5 pt] (E.center) -- (F.center) ;

\node at (-1.5,-1) {$t^2_g$} ;
\end{scope}

\begin{scope}[shift = {(-1.25,2.6)}]
\node (A) at (-1.25,1.25) {} ;
\node (B) at (1.25,1.25) {} ;
\node (C) at (0,0) {} ;
\node (D) at (0,-1) {} ;

\node (c) at ($(A)!0.7!(C)$) {} ;

\draw (A.center) -- (C.center) ;
\draw (B.center) -- (C.center) ;
\draw (D.center) -- (C.center) ;

\node at (-1.5,-0.5) {$t^1$} ;

\draw [line width = 1.5 pt,blue] (c.center) -- (C.center) ;

\end{scope}

\end{scope}

\end{tikzpicture}}

\newcommand{\exampleperturbationCTbreakdeux}{\begin{tikzpicture}

\begin{scope}[yshift = -0.3cm]

\node (A) at (-1.25,1.25) {} ;
\node (B) at (0,1.25) {} ;
\node (C) at (1.25,1.25) {} ;
\node (D) at (-0.625,0.625) {} ;
\node (E) at (0,0) {} ;
\node (F) at (0,-1) {} ;

\node (G) at (-0.3,0.3) {} ;
\node (H) at (0.3,0.3) {} ;

\node (I) at (-1.25,0.3) {} ;
\node (J) at (1.25,0.3) {} ;

\node at (-1.5,-1) {$t_g$} ;

\draw (A.center) -- (D.center) ;
\draw (B.center) -- (D.center) ;
\draw (D.center) -- (G.center) ;
\draw (G.center) -- (E.center) ;
\draw (C.center) -- (H.center) ;
\draw (H.center) -- (E.center) ;
\draw (E.center) -- (F.center) ;

\node (g) at ($(E)!0.7!(G)$) {} ;
\node (e) at ($(G)!0.7!(E)$) {} ;

\draw [line width = 1.5 pt,green] (g.center) -- (G.center) ;
\draw [line width = 1.5 pt,green] (e.center) -- (E.center) ;

\node[green,xshift = -20 pt, yshift = -17 pt] at ($(E)!0.5!(G)$) { $\mathbb{Y}_{e,t_g}$} ;

\draw[line width = 1.5 pt] (I.center) -- (J.center) ;

\end{scope}

\begin{scope}[scale = 0.3,xshift = 12 cm]
\draw[->=latex] (-4,0) node{} -- (4,0) node{} node[midway,below = 2pt, align = center]{$\lim \textcolor{green}{\mathbb{Y}_{e,t_g}} = \textcolor{green}{\mathbb{Y}_{e,t^2_g}}$} ;
\end{scope}

\begin{scope}[xshift = 7cm,yshift = -0.7 cm,scale = 0.5]

\begin{scope}
\node (A) at (-1.25,1.25) {} ;
\node (B) at (1.25,1.25) {} ;
\node (C) at (0,0) {} ;
\node (D) at (0,-1) {} ;
\node (E) at (-1.25,0.5) {};
\node (F) at (1.25,0.5) {};
\node (G) at (-0.5,0.5) {} ;

\draw (A.center) -- (C.center) ;
\draw (B.center) -- (C.center) ;
\draw (D.center) -- (C.center) ;

\node (g) at ($(C)!0.7!(G)$) {} ;
\node (c) at ($(G)!0.7!(C)$) {} ;

\draw [line width = 1.5 pt,green] (c.center) -- (C.center) ;
\draw [line width = 1.5 pt,green] (g.center) -- (G.center) ;

\draw[line width = 1.5 pt] (E.center) -- (F.center) ;

\node at (-1.5,-1) {$t^2_g$} ;
\end{scope}

\begin{scope}[shift = {(-1.25,2.6)}]
\node (A) at (-1.25,1.25) {} ;
\node (B) at (1.25,1.25) {} ;
\node (C) at (0,0) {} ;
\node (D) at (0,-1) {} ;

\draw (A.center) -- (C.center) ;
\draw (B.center) -- (C.center) ;
\draw (D.center) -- (C.center) ;

\node at (-1.5,-0.5) {$t^1$} ;

\end{scope}

\end{scope}

\end{tikzpicture}}

\newcommand{\exampleperturbationCTbreaktrois}{\begin{tikzpicture}

\begin{scope}[yshift = -0.3cm]

\node (A) at (-1.25,1.25) {} ;
\node (B) at (0,1.25) {} ;
\node (C) at (1.25,1.25) {} ;
\node (D) at (-0.625,0.625) {} ;
\node (E) at (0,0) {} ;
\node (F) at (0,-1) {} ;

\node (G) at (-0.3,0.3) {} ;
\node (H) at (0.3,0.3) {} ;

\node (I) at (-1.25,0.3) {} ;
\node (J) at (1.25,0.3) {} ;

\node at (-1.5,-1) {$t_g$} ;

\draw (A.center) -- (D.center) ;
\draw (B.center) -- (D.center) ;
\draw (D.center) -- (G.center) ;
\draw (G.center) -- (E.center) ;
\draw (C.center) -- (H.center) ;
\draw (H.center) -- (E.center) ;
\draw (E.center) -- (F.center) ;

\node (g) at ($(D)!0.7!(G)$) {} ;
\node (d) at ($(G)!0.7!(D)$) {} ;

\draw [line width = 1.5 pt,green] (g.center) -- (G.center) ;
\draw [line width = 1.5 pt,green] (d.center) -- (D.center) ;

\node[green,xshift = -20 pt, yshift = -17 pt] at ($(D)!0.5!(G)$) { $\mathbb{Y}_{e,t_g}$} ;

\draw[line width = 1.5 pt] (I.center) -- (J.center) ;

\end{scope}

\begin{scope}[scale = 0.3,xshift = 12 cm]
\draw[->=latex] (-4,0) node{} -- (4,0) node{} node[midway,below = 20pt, align = center]{$\lim _{t^2_g}\textcolor{green}{\mathbb{Y}_{e,t_g}} = \textcolor{green}{\mathbb{Y}_{e,t^2_g}}$} node{} node[midway,below = 2pt, align = center]{$\lim _{t^1}\textcolor{green}{\mathbb{Y}_{e,t_g}} = \textcolor{blue}{\mathbb{X}_{e,t^1}}$} ;
\end{scope}

\begin{scope}[xshift = 7cm,yshift = -0.7 cm,scale = 0.5]

\begin{scope}
\node (A) at (-1.25,1.25) {} ;
\node (B) at (1.25,1.25) {} ;
\node (C) at (0,0) {} ;
\node (D) at (0,-1) {} ;
\node (E) at (-1.25,0.5) {};
\node (F) at (1.25,0.5) {};
\node (G) at (-0.5,0.5) {} ;

\draw (A.center) -- (C.center) ;
\draw (B.center) -- (C.center) ;
\draw (D.center) -- (C.center) ;

\node (g) at ($(A)!0.7!(G)$) {} ;

\draw [line width = 1.5 pt,green] (g.center) -- (G.center) ;

\draw[line width = 1.5 pt] (E.center) -- (F.center) ;

\node at (-1.5,-1) {$t^2_g$} ;
\end{scope}

\begin{scope}[shift = {(-1.25,2.6)}]
\node (A) at (-1.25,1.25) {} ;
\node (B) at (1.25,1.25) {} ;
\node (C) at (0,0) {} ;
\node (D) at (0,-1) {} ;

\draw (A.center) -- (C.center) ;
\draw (B.center) -- (C.center) ;
\draw (D.center) -- (C.center) ;

\node (c) at ($(D)!0.7!(C)$) {} ;

\draw [line width = 1.5 pt,blue] (c.center) -- (C.center) ;

\node at (-1.5,-0.5) {$t^1$} ;

\end{scope}

\end{scope}

\end{tikzpicture}}

\newcommand{\exampleperturbationCTbreakquatre}{\begin{tikzpicture}

\begin{scope}[yshift = -0.3cm]

\node (A) at (-1.25,1.25) {} ;
\node (B) at (0,1.25) {} ;
\node (C) at (1.25,1.25) {} ;
\node (D) at (-0.625,0.625) {} ;
\node (E) at (0,0) {} ;
\node (F) at (0,-1) {} ;

\node (G) at (-0.3,0.3) {} ;
\node (H) at (0.3,0.3) {} ;

\node (I) at (-1.25,0.3) {} ;
\node (J) at (1.25,0.3) {} ;

\node at (-1.5,-1) {$t_g$} ;

\draw (A.center) -- (D.center) ;
\draw (B.center) -- (D.center) ;
\draw (D.center) -- (G.center) ;
\draw (G.center) -- (E.center) ;
\draw (C.center) -- (H.center) ;
\draw (H.center) -- (E.center) ;
\draw (E.center) -- (F.center) ;

\node (d) at ($(A)!0.7!(D)$) {} ;

\draw [line width = 1.5 pt,green] (d.center) -- (D.center) ;

\node[green,xshift = -20 pt, yshift = -17 pt] at ($(D)!0.5!(A)$) { $\mathbb{Y}_{e,t_g}$} ;

\draw[line width = 1.5 pt] (I.center) -- (J.center) ;

\end{scope}

\begin{scope}[scale = 0.3,xshift = 12 cm]
\draw[->=latex] (-4,0) node{} -- (4,0) node{} node[midway,below = 2pt, align = center]{$\lim \textcolor{green}{\mathbb{Y}_{e,t_g}} = \textcolor{green}{\mathbb{Y}_{e,t^1_g}}$} ;
\end{scope}

\begin{scope}[xshift = 7cm,yshift = -0.7 cm,scale = 0.5]

\begin{scope}
\node (A) at (-1.25,1.25) {} ;
\node (B) at (1.25,1.25) {} ;
\node (C) at (0,0) {} ;
\node (D) at (0,-1) {} ;

\draw (A.center) -- (C.center) ;
\draw (B.center) -- (C.center) ;
\draw (D.center) -- (C.center) ;

\node[below left] at (C) {$t^2$} ;

\end{scope}

\begin{scope}[shift = {(-1.25,2.6)}]
\node (A) at (-1.25,1.25) {} ;
\node (B) at (1.25,1.25) {} ;
\node (C) at (0,0) {} ;
\node (D) at (0,-1) {} ;
\node (E) at (-1.25,-0.5) {};
\node (F) at (1.25,-0.5) {};
\node (G) at (-0.5,-0.5) {} ;

\draw (A.center) -- (C.center) ;
\draw (B.center) -- (C.center) ;
\draw (D.center) -- (C.center) ;

\node (c) at ($(A)!0.7!(C)$) {} ;

\draw [line width = 1.5 pt,green] (c.center) -- (C.center) ;

\draw[line width = 1.5 pt] (E.center) -- (F.center) ;

\node[below left] at (E) {$t^1_g$} ;

\end{scope}

\begin{scope}[shift = {(1.25,2.6)}]
\node (A) at (0,1.25) {} ;
\node (B) at (0,-1) {} ;
\node (C) at (-0.6,-0.5) {};
\node (D) at (0.6,-0.5) {};
\node (E) at (0,-0.5) {} ;

\draw (A.center) -- (B.center) ;

\draw[line width = 1.5 pt] (C.center) -- (D.center) ;

\node[below right] at (D) {$t^2_g$} ;

\end{scope}

\end{scope}

\end{tikzpicture}}

\newcommand{\exampleperturbationCTbreakcinq}{\begin{tikzpicture}

\begin{scope}[yshift = -0.3cm]

\node (A) at (-1.25,1.25) {} ;
\node (B) at (0,1.25) {} ;
\node (C) at (1.25,1.25) {} ;
\node (D) at (-0.625,0.625) {} ;
\node (E) at (0,0) {} ;
\node (F) at (0,-1) {} ;

\node (G) at (-0.3,0.3) {} ;
\node (H) at (0.3,0.3) {} ;

\node (I) at (-1.25,0.3) {} ;
\node (J) at (1.25,0.3) {} ;

\node at (-1.5,-1) {$t_g$} ;

\draw (A.center) -- (D.center) ;
\draw (B.center) -- (D.center) ;
\draw (D.center) -- (G.center) ;
\draw (G.center) -- (E.center) ;
\draw (C.center) -- (H.center) ;
\draw (H.center) -- (E.center) ;
\draw (E.center) -- (F.center) ;

\node (e) at ($(F)!0.7!(E)$) {} ;

\draw [line width = 1.5 pt,green] (e.center) -- (E.center) ;

\node[green,xshift = -20 pt, yshift = -17 pt] at ($(F)!0.5!(E)$) { $\mathbb{Y}_{e,t_g}$} ;

\draw[line width = 1.5 pt] (I.center) -- (J.center) ;

\end{scope}

\begin{scope}[scale = 0.3,xshift = 12 cm]
\draw[->=latex] (-4,0) node{} -- (4,0) node{} node[midway,below = 2pt, align = center]{$\lim \textcolor{green}{\mathbb{Y}_{e,t_g}} = \textcolor{red}{\mathbb{X}_{e,t^2}}$} ;
\end{scope}

\begin{scope}[xshift = 7cm,yshift = -0.7 cm,scale = 0.5]

\begin{scope}
\node (A) at (-1.25,1.25) {} ;
\node (B) at (1.25,1.25) {} ;
\node (C) at (0,0) {} ;
\node (D) at (0,-1) {} ;

\draw (A.center) -- (C.center) ;
\draw (B.center) -- (C.center) ;
\draw (D.center) -- (C.center) ;

\node[below left] at (C) {$t^2$} ;

\node (c) at ($(D)!0.7!(C)$) {} ;

\draw [line width = 1.5 pt,red] (c.center) -- (C.center) ;

\end{scope}

\begin{scope}[shift = {(-1.25,2.6)}]
\node (A) at (-1.25,1.25) {} ;
\node (B) at (1.25,1.25) {} ;
\node (C) at (0,0) {} ;
\node (D) at (0,-1) {} ;
\node (E) at (-1.25,-0.5) {};
\node (F) at (1.25,-0.5) {};
\node (G) at (-0.5,-0.5) {} ;

\draw (A.center) -- (C.center) ;
\draw (B.center) -- (C.center) ;
\draw (D.center) -- (C.center) ;

\draw[line width = 1.5 pt] (E.center) -- (F.center) ;

\node[below left] at (E) {$t^1_g$} ;

\end{scope}

\begin{scope}[shift = {(1.25,2.6)}]
\node (A) at (0,1.25) {} ;
\node (B) at (0,-1) {} ;
\node (C) at (-0.6,-0.5) {};
\node (D) at (0.6,-0.5) {};
\node (E) at (0,-0.5) {} ;

\draw (A.center) -- (B.center) ;

\draw[line width = 1.5 pt] (C.center) -- (D.center) ;

\node[below right] at (D) {$t^2_g$} ;

\end{scope}

\end{scope}

\end{tikzpicture}}

\newcommand{\exampleperturbationCTbreaksix}{\begin{tikzpicture}

\begin{scope}[yshift = -0.3cm]

\node (A) at (-1.25,1.25) {} ;
\node (B) at (0,1.25) {} ;
\node (C) at (1.25,1.25) {} ;
\node (D) at (-0.625,0.625) {} ;
\node (E) at (0,0) {} ;
\node (F) at (0,-1) {} ;

\node (G) at (-0.3,0.3) {} ;
\node (H) at (0.3,0.3) {} ;

\node (I) at (-1.25,0.3) {} ;
\node (J) at (1.25,0.3) {} ;

\node at (-1.5,-1) {$t_g$} ;

\draw (A.center) -- (D.center) ;
\draw (B.center) -- (D.center) ;
\draw (D.center) -- (G.center) ;
\draw (G.center) -- (E.center) ;
\draw (C.center) -- (H.center) ;
\draw (H.center) -- (E.center) ;
\draw (E.center) -- (F.center) ;

\node (h) at ($(E)!0.7!(H)$) {} ;
\node (e) at ($(H)!0.7!(E)$) {} ;

\draw [line width = 1.5 pt,green] (e.center) -- (E.center) ;
\draw [line width = 1.5 pt,green] (h.center) -- (H.center) ;

\node[green,xshift = 20 pt, yshift = 17 pt] at ($(H)!0.5!(E)$) { $\mathbb{Y}_{e,t_g}$} ;

\draw[line width = 1.5 pt] (I.center) -- (J.center) ;

\end{scope}

\begin{scope}[scale = 0.3,xshift = 12 cm]
\draw[->=latex] (-4,0) node{} -- (4,0) node{} node[midway,below = 20pt, align = center]{$\lim_{t^2} \textcolor{green}{\mathbb{Y}_{e,t_g}} = \textcolor{red}{\mathbb{X}_{e,t^2}}$} node[midway,below = 2pt, align = center]{$\lim_{t^2_g} \textcolor{green}{\mathbb{Y}_{e,t_g}} = \textcolor{green}{\mathbb{Y}_{e,t^2_g}}$} ;
\end{scope}

\begin{scope}[xshift = 7cm,yshift = -0.7 cm,scale = 0.5]

\begin{scope}
\node (A) at (-1.25,1.25) {} ;
\node (B) at (1.25,1.25) {} ;
\node (C) at (0,0) {} ;
\node (D) at (0,-1) {} ;

\draw (A.center) -- (C.center) ;
\draw (B.center) -- (C.center) ;
\draw (D.center) -- (C.center) ;

\node[below left] at (C) {$t^2$} ;

\node (c) at ($(B)!0.7!(C)$) {} ;

\draw [line width = 1.5 pt,red] (c.center) -- (C.center) ;

\end{scope}

\begin{scope}[shift = {(-1.25,2.6)}]
\node (A) at (-1.25,1.25) {} ;
\node (B) at (1.25,1.25) {} ;
\node (C) at (0,0) {} ;
\node (D) at (0,-1) {} ;
\node (E) at (-1.25,-0.5) {};
\node (F) at (1.25,-0.5) {};
\node (G) at (-0.5,-0.5) {} ;

\draw (A.center) -- (C.center) ;
\draw (B.center) -- (C.center) ;
\draw (D.center) -- (C.center) ;

\draw[line width = 1.5 pt] (E.center) -- (F.center) ;

\node[below left] at (E) {$t^1_g$} ;

\end{scope}

\begin{scope}[shift = {(1.25,2.6)}]
\node (A) at (0,1.25) {} ;
\node (B) at (0,-1) {} ;
\node (C) at (-0.6,-0.5) {};
\node (D) at (0.6,-0.5) {};
\node (E) at (0,-0.5) {} ;

\draw (A.center) -- (B.center) ;

\node (e) at ($(B)!0.3!(E)$) {} ;

\draw [line width = 1.5 pt,green] (e.center) -- (E.center) ;

\draw[line width = 1.5 pt] (C.center) -- (D.center) ;

\node[below right] at (D) {$t^2_g$} ;

\end{scope}

\end{scope}

\end{tikzpicture}}


\newcommand{\eqainfhomun}{
\begin{tikzpicture}[scale=0.2,baseline=(pt_base.base)]

\coordinate (pt_base) at (0,0) ;

\draw (-2,0.5) node[yshift = -8 pt]{\tiny $[0]$} ;

\draw[blue] (-2,1.25) -- (-0.625,0.5) ;
\draw[blue] (-1,1.25) -- (-0.3125,0.5) ;
\draw[blue] (1,1.25) -- (0.3125,0.5) ;
\draw[blue] (2,1.25) -- (0.625,0.5) ;

\draw[red] (-0.625,0.5) -- (0,0) ;
\draw[red] (-0.3125,0.5) -- (0,0) ;
\draw[red] (0.625,0.5) -- (0,0) ;
\draw[red] (0.3125,0.5) -- (0,0) ;
\draw[red] (0,-1.25) -- (0,0) ;

\draw[blue] [densely dotted] (-0.5,0.9) -- (0.5,0.9) ;

\draw  (-2,0.5) -- (2,0.5) ;

\end{tikzpicture}}

\newcommand{\eqainfhomdeux}{
\begin{tikzpicture}[scale=0.2,baseline=(pt_base.base)]

\coordinate (pt_base) at (0,0) ;

\draw (-2,0.5) node[yshift = -8 pt]{\tiny $[1]$} ;

\draw[blue] (-2,1.25) -- (-0.625,0.5) ;
\draw[blue] (-1,1.25) -- (-0.3125,0.5) ;
\draw[blue] (1,1.25) -- (0.3125,0.5) ;
\draw[blue] (2,1.25) -- (0.625,0.5) ;

\draw[red] (-0.625,0.5) -- (0,0) ;
\draw[red] (-0.3125,0.5) -- (0,0) ;
\draw[red] (0.625,0.5) -- (0,0) ;
\draw[red] (0.3125,0.5) -- (0,0) ;
\draw[red] (0,-1.25) -- (0,0) ;

\draw[blue] [densely dotted] (-0.5,0.9) -- (0.5,0.9) ;

\draw  (-2,0.5) -- (2,0.5) ;

\end{tikzpicture}}

\newcommand{\eqainfhomtrois}{
\begin{tikzpicture}[scale=0.2,baseline=(pt_base.base)]

\coordinate (pt_base) at (0,1.25) ;


\draw[blue] (-4,1.25) -- (-1.6,0.5) ;
\draw[blue] (-3,1.25) -- (-1.2,0.5) ;
\draw[blue] (4,1.25) -- (1.6,0.5) ;
\draw[blue] (3,1.25) -- (1.2,0.5) ;
\draw[blue] (0,1.25) -- (0,0.5) ;

\draw[red] (-1.6,0.5) -- (0,0) ;
\draw[red] (-1.2,0.5) -- (0,0) ;
\draw[red] (1.6,0.5) -- (0,0) ;
\draw[red] (1.2,0.5) -- (0,0) ;
\draw[red] (0,0.5) -- (0,0) ;
\draw[red] (0,-1.25) -- (0,0) ; 


\draw[blue] (-2,4) -- (0,2.75) ;
\draw[blue] (-1,4) -- (0,2.75) ;
\draw[blue] (1,4) -- (0,2.75) ;
\draw[blue] (2,4) -- (0,2.75) ;
\draw[blue] (0,1.5) -- (0,2.75) ;


\draw[blue] [densely dotted] (-0.5,3.65) -- (0.5,3.65) ;
\draw[blue] [densely dotted] (-0.5,0.9) -- (-1.5,0.9) ;
\draw[blue] [densely dotted] (0.5,0.9) -- (1.5,0.9) ;

\draw (-4,0.5) -- (4,0.5) ;

\draw (-4,0.5) node[yshift = - 8 pt] {\tiny $[0<1]$} ;

\end{tikzpicture}}


\newcommand{\eqainfhomquatre}{
\begin{tikzpicture}[scale=0.2,baseline=(pt_base.base)]

\coordinate (pt_base) at (0,1.25) ;


\begin{scope}

\draw[red] (-12,1.25) -- (0,0) ;
\draw[red] (-4,1.25) -- (0,0) ;
\draw[red] (12,1.25) -- (0,0) ;
\draw[red] (4,1.25) -- (0,0) ;
\draw[red] (0,-1.25) -- (0,0) ; 

\end{scope}


\begin{scope} [xshift = 8 cm]

\draw[blue] (2.5,4) -- (3.4,3.25) ;
\draw[blue] (3.5,4) -- (3.8,3.25) ;
\draw[blue] (4.5,4) -- (4.2,3.25) ;
\draw[blue] (5.5,4) -- (4.6,3.25) ;

\draw[red] (3.4,3.25) -- (4,2.75) ;
\draw[red] (3.8,3.25) -- (4,2.75) ;
\draw[red] (4.2,3.25) -- (4,2.75) ;
\draw[red] (4.6,3.25) -- (4,2.75) ;
\draw[red] (4,1.5) -- (4,2.75) ;

\draw (2.5,3.25) -- (5.5,3.25) ;

\draw (2.5,3.25) node[yshift = - 8pt,scale = 0.8] {\tiny $[1]$} ; 

\end{scope}


\begin{scope} [xshift = 0 cm]

\draw[blue] (2.5,4) -- (3.4,3.25) ;
\draw[blue] (3.5,4) -- (3.8,3.25) ;
\draw[blue] (4.5,4) -- (4.2,3.25) ;
\draw[blue] (5.5,4) -- (4.6,3.25) ;

\draw[red] (3.4,3.25) -- (4,2.75) ;
\draw[red] (3.8,3.25) -- (4,2.75) ;
\draw[red] (4.2,3.25) -- (4,2.75) ;
\draw[red] (4.6,3.25) -- (4,2.75) ;
\draw[red] (4,1.5) -- (4,2.75) ;

\draw (2.5,3.25) -- (5.5,3.25) ;

\draw (2.5,3.25) node[yshift = - 8pt,scale = 0.8] {\tiny $[1]$} ; 

\end{scope}


\begin{scope} [xshift = - 4 cm]

\draw[blue] (2.5,4) -- (3.4,3.25) ;
\draw[blue] (3.5,4) -- (3.8,3.25) ;
\draw[blue] (4.5,4) -- (4.2,3.25) ;
\draw[blue] (5.5,4) -- (4.6,3.25) ;

\draw[red] (3.4,3.25) -- (4,2.75) ;
\draw[red] (3.8,3.25) -- (4,2.75) ;
\draw[red] (4.2,3.25) -- (4,2.75) ;
\draw[red] (4.6,3.25) -- (4,2.75) ;
\draw[red] (4,1.5) -- (4,2.75) ;

\draw (2.5,3.25) -- (5.5,3.25) ;

\draw (2.5,3.25) node[yshift = - 8pt,xshift=-3pt,scale = 0.8] {\tiny $[0<1]$} ; 

\end{scope}


\begin{scope} [xshift = - 8 cm]

\draw[blue] (2.5,4) -- (3.4,3.25) ;
\draw[blue] (3.5,4) -- (3.8,3.25) ;
\draw[blue] (4.5,4) -- (4.2,3.25) ;
\draw[blue] (5.5,4) -- (4.6,3.25) ;

\draw[red] (3.4,3.25) -- (4,2.75) ;
\draw[red] (3.8,3.25) -- (4,2.75) ;
\draw[red] (4.2,3.25) -- (4,2.75) ;
\draw[red] (4.6,3.25) -- (4,2.75) ;
\draw[red] (4,1.5) -- (4,2.75) ;

\draw (2.5,3.25) -- (5.5,3.25) ;

\draw (2.5,3.25) node[yshift = - 8pt,scale = 0.8] {\tiny $[0]$} ; 

\end{scope}


\begin{scope} [xshift = - 16 cm]

\draw[blue] (2.5,4) -- (3.4,3.25) ;
\draw[blue] (3.5,4) -- (3.8,3.25) ;
\draw[blue] (4.5,4) -- (4.2,3.25) ;
\draw[blue] (5.5,4) -- (4.6,3.25) ;

\draw[red] (3.4,3.25) -- (4,2.75) ;
\draw[red] (3.8,3.25) -- (4,2.75) ;
\draw[red] (4.2,3.25) -- (4,2.75) ;
\draw[red] (4.6,3.25) -- (4,2.75) ;
\draw[red] (4,1.5) -- (4,2.75) ;

\draw (2.5,3.25) -- (5.5,3.25) ;

\draw (2.5,3.25) node[yshift = - 8pt,scale = 0.8] {\tiny $[0]$} ; 

\end{scope}


\draw [densely dotted,red] (-1,0.9) -- (1,0.9) ;
\draw [densely dotted,red] (-6,0.9) -- (-4,0.9) ;
\draw [densely dotted,red] (4,0.9) -- (6,0.9) ;
\draw [densely dotted] (-7,2.75) -- (-10,2.75) node[midway,yshift = 8pt] {\tiny $[0]$} ; 
\draw [densely dotted] (6,2.75) -- (9,2.75) node[midway,yshift = 8pt] {\tiny $[1]$} ; 

\end{tikzpicture}}

\newcommand{\ainfhomun}{
\begin{tikzpicture}[scale=0.2,baseline=(pt_base.base)]

\coordinate (pt_base) at (0,-0.5) ;

\draw (-2,0.5) node[yshift = -8 pt]{\tiny $[0]$} ;

\draw[blue] (-2,1.25) -- (-0.625,0.5) ;
\draw[blue] (-1,1.25) -- (-0.3125,0.5) ;
\draw[blue] (1,1.25) -- (0.3125,0.5) ;
\draw[blue] (2,1.25) -- (0.625,0.5) ;

\draw[red] (-0.625,0.5) -- (0,0) ;
\draw[red] (-0.3125,0.5) -- (0,0) ;
\draw[red] (0.625,0.5) -- (0,0) ;
\draw[red] (0.3125,0.5) -- (0,0) ;
\draw[red] (0,-1.25) -- (0,0) ;

\draw[blue] [densely dotted] (-0.5,0.9) -- (0.5,0.9) ;

\draw  (-2,0.5) -- (2,0.5) ;

\end{tikzpicture}}

\newcommand{\ainfhomdeux}{
\begin{tikzpicture}[scale=0.2,baseline=(pt_base.base)]

\coordinate (pt_base) at (0,-0.5) ;

\draw (-2,0.5) node[yshift = -8 pt,xshift = - 3 pt]{\tiny $[0<1]$} ;

\draw[blue] (-2,1.25) -- (-0.625,0.5) ;
\draw[blue] (-1,1.25) -- (-0.3125,0.5) ;
\draw[blue] (1,1.25) -- (0.3125,0.5) ;
\draw[blue] (2,1.25) -- (0.625,0.5) ;

\draw[red] (-0.625,0.5) -- (0,0) ;
\draw[red] (-0.3125,0.5) -- (0,0) ;
\draw[red] (0.625,0.5) -- (0,0) ;
\draw[red] (0.3125,0.5) -- (0,0) ;
\draw[red] (0,-1.25) -- (0,0) ;

\draw[blue] [densely dotted] (-0.5,0.9) -- (0.5,0.9) ;

\draw  (-2,0.5) -- (2,0.5) ;

\end{tikzpicture}}

\newcommand{\ainfhomdeuxprime}{
\begin{tikzpicture}[scale=0.15,baseline=(pt_base.base)]

\coordinate (pt_base) at (0,-0.5) ;

\draw (-2,0.5) node[yshift = -8 pt,xshift = - 3 pt]{\tiny $[0<1]$} ;

\draw[blue] (-2,1.25) -- (-0.625,0.5) ;
\draw[blue] (-1,1.25) -- (-0.3125,0.5) ;
\draw[blue] (1,1.25) -- (0.3125,0.5) ;
\draw[blue] (2,1.25) -- (0.625,0.5) ;

\draw[red] (-0.625,0.5) -- (0,0) ;
\draw[red] (-0.3125,0.5) -- (0,0) ;
\draw[red] (0.625,0.5) -- (0,0) ;
\draw[red] (0.3125,0.5) -- (0,0) ;
\draw[red] (0,-1.25) -- (0,0) ;

\draw[blue] [densely dotted] (-0.5,0.9) -- (0.5,0.9) ;

\draw  (-2,0.5) -- (2,0.5) ;

\end{tikzpicture}}

\newcommand{\ainfhomtrois}{
\begin{tikzpicture}[scale=0.2,baseline=(pt_base.base)]

\coordinate (pt_base) at (0,-0.5) ;

\draw (-2,0.5) node[yshift = -8 pt]{\tiny $[1]$} ;

\draw[blue] (-2,1.25) -- (-0.625,0.5) ;
\draw[blue] (-1,1.25) -- (-0.3125,0.5) ;
\draw[blue] (1,1.25) -- (0.3125,0.5) ;
\draw[blue] (2,1.25) -- (0.625,0.5) ;

\draw[red] (-0.625,0.5) -- (0,0) ;
\draw[red] (-0.3125,0.5) -- (0,0) ;
\draw[red] (0.625,0.5) -- (0,0) ;
\draw[red] (0.3125,0.5) -- (0,0) ;
\draw[red] (0,-1.25) -- (0,0) ;

\draw[blue] [densely dotted] (-0.5,0.9) -- (0.5,0.9) ;

\draw  (-2,0.5) -- (2,0.5) ;

\end{tikzpicture}}

\newcommand{\eqainfnmorphun}{
\begin{tikzpicture}[scale=0.25,baseline=(pt_base.base)]

\coordinate (pt_base) at (0,0) ;

\draw (-2,0.5) node[yshift = -5 pt,scale = 0.7]{\tiny $\partial_j^{sing} I$} ;

\draw[blue] (-2,1.25) -- (-0.625,0.5) ;
\draw[blue] (-1,1.25) -- (-0.3125,0.5) ;
\draw[blue] (1,1.25) -- (0.3125,0.5) ;
\draw[blue] (2,1.25) -- (0.625,0.5) ;

\draw[red] (-0.625,0.5) -- (0,0) ;
\draw[red] (-0.3125,0.5) -- (0,0) ;
\draw[red] (0.625,0.5) -- (0,0) ;
\draw[red] (0.3125,0.5) -- (0,0) ;
\draw[red] (0,-1.25) -- (0,0) ;

\draw[blue] [densely dotted] (-0.5,0.9) -- (0.5,0.9) ;

\draw  (-2,0.5) -- (2,0.5) ;

\end{tikzpicture}}

\newcommand{\eqainfnmorphdeux}{
\begin{tikzpicture}[scale=0.25,baseline=(pt_base.base)]

\coordinate (pt_base) at (0,1.25) ;


\draw[blue] (-4,1.25) -- (-1.6,0.5) ;
\draw[blue] (-3,1.25) -- (-1.2,0.5) ;
\draw[blue] (4,1.25) -- (1.6,0.5) ;
\draw[blue] (3,1.25) -- (1.2,0.5) ;
\draw[blue] (0,1.25) -- (0,0.5) ;

\draw[red] (-1.6,0.5) -- (0,0) ;
\draw[red] (-1.2,0.5) -- (0,0) ;
\draw[red] (1.6,0.5) -- (0,0) ;
\draw[red] (1.2,0.5) -- (0,0) ;
\draw[red] (0,0.5) -- (0,0) ;
\draw[red] (0,-1.25) -- (0,0) ; 


\draw[blue] (-2,4) -- (0,2.75) ;
\draw[blue] (-1,4) -- (0,2.75) ;
\draw[blue] (1,4) -- (0,2.75) ;
\draw[blue] (2,4) -- (0,2.75) ;
\draw[blue] (0,1.5) -- (0,2.75) ;


\draw[blue] [densely dotted] (-0.5,3.65) -- (0.5,3.65) ;
\draw[blue] [densely dotted] (-0.5,0.9) -- (-1.5,0.9) ;
\draw[blue] [densely dotted] (0.5,0.9) -- (1.5,0.9) ;

\draw (-4,0.5) -- (4,0.5) ;

\draw (-4,0.5) node[yshift = - 8 pt] {\tiny $I$} ;

\end{tikzpicture}}

\newcommand{\eqainfnmorphtrois}{
\begin{tikzpicture}[scale=0.25,baseline=(pt_base.base)]

\coordinate (pt_base) at (0,1.25) ;


\draw[red] (-4,1.25) -- (0,0) ;
\draw[red] (-2,1.25) -- (0,0) ;
\draw[red] (4,1.25) -- (0,0) ;
\draw[red] (2,1.25) -- (0,0) ;
\draw[red] (0,-1.25) -- (0,0) ; 


\draw[blue] (2.5,4) -- (3.4,3.25) ;
\draw[blue] (3.5,4) -- (3.8,3.25) ;
\draw[blue] (4.5,4) -- (4.2,3.25) ;
\draw[blue] (5.5,4) -- (4.6,3.25) ;

\draw[red] (3.4,3.25) -- (4,2.75) ;
\draw[red] (3.8,3.25) -- (4,2.75) ;
\draw[red] (4.2,3.25) -- (4,2.75) ;
\draw[red] (4.6,3.25) -- (4,2.75) ;
\draw[red] (4,1.5) -- (4,2.75) ;

\draw (2.5,3.25) -- (5.5,3.25) ;

\draw (2.5,3.25) node[scale = 0.7,yshift = - 8 pt] {\tiny $I_s$} ;


\draw[blue] (-2.5,4) -- (-3.4,3.25) ;
\draw[blue] (-3.5,4) -- (-3.8,3.25) ;
\draw[blue] (-4.5,4) -- (-4.2,3.25) ;
\draw[blue] (-5.5,4) -- (-4.6,3.25) ;

\draw[red] (-3.4,3.25) -- (-4,2.75) ;
\draw[red] (-3.8,3.25) -- (-4,2.75) ;
\draw[red] (-4.2,3.25) -- (-4,2.75) ;
\draw[red] (-4.6,3.25) -- (-4,2.75) ;
\draw[red] (-4,1.5) -- (-4,2.75) ;

\draw (-2.5,3.25) -- (-5.5,3.25) ;

\draw (-5.5,3.25) node[scale = 0.7,yshift = - 8 pt] {\tiny $I_1$} ;


\draw [densely dotted,red] (-1,0.9) -- (1,0.9) ;
\draw [densely dotted] (2,2.75) -- (-2,2.75) ; 
\draw [densely dotted,blue] (3.8,3.7) -- (4.2,3.7) ;
\draw [densely dotted,blue] (-3.8,3.7) -- (-4.2,3.7) ;

\end{tikzpicture}}

\newcommand{\eqainfnmorphquatre}{
\begin{tikzpicture}[scale=0.25,baseline=(pt_base.base)]

\coordinate (pt_base) at (0,1.25) ;


\draw[red] (-4,1.25) -- (0,0) ;
\draw[red] (-2,1.25) -- (0,0) ;
\draw[red] (4,1.25) -- (0,0) ;
\draw[red] (2,1.25) -- (0,0) ;
\draw[red] (0,-1.25) -- (0,0) ; 


\draw[blue] (2.5,4) -- (3.4,3.25) ;
\draw[blue] (3.5,4) -- (3.8,3.25) ;
\draw[blue] (4.5,4) -- (4.2,3.25) ;
\draw[blue] (5.5,4) -- (4.6,3.25) ;

\draw[red] (3.4,3.25) -- (4,2.75) ;
\draw[red] (3.8,3.25) -- (4,2.75) ;
\draw[red] (4.2,3.25) -- (4,2.75) ;
\draw[red] (4.6,3.25) -- (4,2.75) ;
\draw[red] (4,1.5) -- (4,2.75) ;

\draw (2.5,3.25) -- (5.5,3.25) ;

\draw (2.5,3.25) node[scale = 0.7,yshift = - 8 pt] {\tiny $I_1$} ;


\draw[blue] (-2.5,4) -- (-3.4,3.25) ;
\draw[blue] (-3.5,4) -- (-3.8,3.25) ;
\draw[blue] (-4.5,4) -- (-4.2,3.25) ;
\draw[blue] (-5.5,4) -- (-4.6,3.25) ;

\draw[red] (-3.4,3.25) -- (-4,2.75) ;
\draw[red] (-3.8,3.25) -- (-4,2.75) ;
\draw[red] (-4.2,3.25) -- (-4,2.75) ;
\draw[red] (-4.6,3.25) -- (-4,2.75) ;
\draw[red] (-4,1.5) -- (-4,2.75) ;

\draw (-2.5,3.25) -- (-5.5,3.25) ;

\draw (-5.5,3.25) node[scale = 0.7,yshift = - 8 pt] {\tiny $I_s$} ;


\draw [densely dotted,red] (-1,0.9) -- (1,0.9) ;
\draw [densely dotted] (2,2.75) -- (-2,2.75) ; 
\draw [densely dotted,blue] (3.8,3.7) -- (4.2,3.7) ;
\draw [densely dotted,blue] (-3.8,3.7) -- (-4.2,3.7) ;

\end{tikzpicture}}

\newcommand{\ainfnmorphun}{
\begin{tikzpicture}[scale=0.2,baseline=(pt_base.base)]

\coordinate (pt_base) at (0,-0.5) ;

\draw (-2,0.5) node[yshift = -5 pt,scale = 0.7]{\tiny $I$} ;

\draw[blue] (-2,1.25) -- (-0.625,0.5) ;
\draw[blue] (-1,1.25) -- (-0.3125,0.5) ;
\draw[blue] (1,1.25) -- (0.3125,0.5) ;
\draw[blue] (2,1.25) -- (0.625,0.5) ;

\draw[red] (-0.625,0.5) -- (0,0) ;
\draw[red] (-0.3125,0.5) -- (0,0) ;
\draw[red] (0.625,0.5) -- (0,0) ;
\draw[red] (0.3125,0.5) -- (0,0) ;
\draw[red] (0,-1.25) -- (0,0) ;

\draw[blue] [densely dotted] (-0.5,0.9) -- (0.5,0.9) ;

\draw  (-2,0.5) -- (2,0.5) ;

\end{tikzpicture}}

\newcommand{\ainfnmorphundeltan}{
\begin{tikzpicture}[scale=0.2,baseline=(pt_base.base)]

\coordinate (pt_base) at (0,-0.5) ;

\draw (-2,0.5) node[yshift = -5 pt,scale = 0.7]{\tiny $\Delta^n$} ;

\draw[blue] (-2,1.25) -- (-0.625,0.5) ;
\draw[blue] (-1,1.25) -- (-0.3125,0.5) ;
\draw[blue] (1,1.25) -- (0.3125,0.5) ;
\draw[blue] (2,1.25) -- (0.625,0.5) ;

\draw[red] (-0.625,0.5) -- (0,0) ;
\draw[red] (-0.3125,0.5) -- (0,0) ;
\draw[red] (0.625,0.5) -- (0,0) ;
\draw[red] (0.3125,0.5) -- (0,0) ;
\draw[red] (0,-1.25) -- (0,0) ;

\draw[blue] [densely dotted] (-0.5,0.9) -- (0.5,0.9) ;

\draw  (-2,0.5) -- (2,0.5) ;

\end{tikzpicture}}

\newcommand{\arbreopnmorph}[1]{
\begin{tikzpicture}[scale=#1,baseline=(pt_base.base)]

\coordinate (pt_base) at (0,-0.5) ;

\draw[blue] (-2,1.25) -- (-0.625,0.5) ;
\draw[blue] (-1,1.25) -- (-0.3125,0.5) ;
\draw[blue] (1,1.25) -- (0.3125,0.5) ;
\draw[blue] (2,1.25) -- (0.625,0.5) ;

\draw[red] (-0.625,0.5) -- (0,0) ;
\draw[red] (-0.3125,0.5) -- (0,0) ;
\draw[red] (0.625,0.5) -- (0,0) ;
\draw[red] (0.3125,0.5) -- (0,0) ;
\draw[red] (0,-1.25) -- (0,0) ;

\draw[blue] [densely dotted] (-0.5,0.9) -- (0.5,0.9) ;

\draw  (-2,0.5) -- (2,0.5) ;

\node[yshift = -4pt,scale=0.8] at (-2,0.5) {\tiny $I$} ;

\end{tikzpicture}}

\newcommand{\arbreopnmorphbordsing}[1]{
\begin{tikzpicture}[scale=#1,baseline=(pt_base.base)]

\coordinate (pt_base) at (0,-0.5) ;

\draw[blue] (-2,1.25) -- (-0.625,0.5) ;
\draw[blue] (-1,1.25) -- (-0.3125,0.5) ;
\draw[blue] (1,1.25) -- (0.3125,0.5) ;
\draw[blue] (2,1.25) -- (0.625,0.5) ;

\draw[red] (-0.625,0.5) -- (0,0) ;
\draw[red] (-0.3125,0.5) -- (0,0) ;
\draw[red] (0.625,0.5) -- (0,0) ;
\draw[red] (0.3125,0.5) -- (0,0) ;
\draw[red] (0,-1.25) -- (0,0) ;

\draw[blue] [densely dotted] (-0.5,0.9) -- (0.5,0.9) ;

\draw  (-2,0.5) -- (2,0.5) ;

\node[yshift = -4pt,scale=0.8] at (-2,0.5) {\tiny $\partial^{sing}_jI$} ;

\end{tikzpicture}}

\newcommand{\arbreopunmorphn}{
\begin{tikzpicture}[scale=0.15,baseline=(pt_base.base)]

\coordinate (pt_base) at (0,-0.5) ;

\draw[blue] (0,1.25) -- (0,0.5) ;
\draw[red] (0,-1.25) -- (0,0.5) ;

\draw (-1.25,0.5) -- (1.25,0.5) ;

\node[yshift = -4pt,scale=0.8] at (-1.25,0.5) {\tiny $I$} ;

\end{tikzpicture}
}


\newcommand{\arbreopdeuxmorphn}{
\begin{tikzpicture}[scale=0.15,baseline=(pt_base.base)]

\coordinate (pt_base) at (0,-0.5) ;

\draw[blue] (-1.25,1.25) -- (-0.5,0.5) ;
\draw[blue] (1.25,1.25) -- (0.5,0.5) ;

\draw[red] (-0.5,0.5) -- (0,0) ;
\draw[red] (0.5,0.5) -- (0,0) ;
\draw[red] (0,-1.25) -- (0,0) ;

\draw (-1.25,0.5) -- (1.25,0.5) ;

\node[yshift = -4pt,scale=0.8] at (-1.25,0.5) {\tiny $I$} ;

\end{tikzpicture}
}


\newcommand{\arbreoptroismorphn}{
\begin{tikzpicture}[scale=0.15,baseline=(pt_base.base)]

\coordinate (pt_base) at (0,-0.5) ;

\draw[blue] (0,1.25) -- (0,0.5) ;
\draw[blue] (-1.25,1.25) -- (-0.5,0.5) ;
\draw[blue] (1.25,1.25) -- (0.5,0.5) ;

\draw[red] (0,0.5) -- (0,0) ;
\draw[red] (-0.5,0.5) -- (0,0) ;
\draw[red] (0.5,0.5) -- (0,0) ;
\draw[red] (0,-1.25) -- (0,0) ;

\draw (-1.25,0.5) -- (1.25,0.5) ;

\node[yshift = -4pt,scale=0.8] at (-1.25,0.5) {\tiny $I$} ;

\end{tikzpicture}
}


\newcommand{\arbreopquatremorphn}{
\begin{tikzpicture}[scale=0.15,baseline=(pt_base.base)]

\coordinate (pt_base) at (0,-0.5) ;

\draw[blue] (-1.5,1.25) -- (-0.6,0.5) ;
\draw[blue] (1.5,1.25) -- (0.6,0.5) ;
\draw[blue] (-0.5,1.25) -- (-0.2,0.5) ;
\draw[blue] (0.5,1.25) -- (0.2,0.5) ;

\draw[red] (-0.6,0.5) -- (0,0) ;
\draw[red] (0.6,0.5) -- (0,0) ;
\draw[red] (0.2,0.5) -- (0,0) ;
\draw[red] (-0.2,0.5) -- (0,0) ;
\draw[red] (0,-1.25) -- (0,0) ;

\draw (-1.5,0.5) -- (1.5,0.5) ;

\node[yshift = -4pt,scale=0.8] at (-1.5,0.5) {\tiny $I$} ;

\end{tikzpicture}
}

\newcommand{\arbrebicoloreLn}{
\begin{tikzpicture}[scale=0.15,baseline=(pt_base.base)]

\coordinate (pt_base) at (0,-0.5) ;

\draw[blue] (-1.25,1.25) -- (-0.5,0.5) ;
\draw[blue] (1.25,1.25) -- (0.5,0.5) ;

\draw[red] (-0.5,0.5) -- (0,0) ;
\draw[red] (0.5,0.5) -- (0,0) ;
\draw[red] (0,-1.25) -- (0,0) ;

\draw (-1.25,0.5) -- (1.25,0.5) ;

\node[yshift = -4pt,scale=0.8] at (-1.25,0.5) {\tiny $I$} ;

\end{tikzpicture}
}

\newcommand{\arbrebicoloreMn}{
\begin{tikzpicture}[scale=0.15,baseline=(pt_base.base)]

\coordinate (pt_base) at (0,-0.5) ;

\draw[blue] (-1.25,1.25) -- (0,0) ;
\draw[blue] (1.25,1.25) -- (0,0) ;

\draw[blue] (0,-0.5) -- (0,0) ;
\draw[red] (0,-0.5) -- (0,-1.25) ;

\draw (-1.25,-0.5) -- (1.25,-0.5) ;

\node[yshift = -4pt,scale=0.8] at (-1.25,-0.5) {\tiny $I$} ;

\end{tikzpicture}
}

\newcommand{\arbrebicoloreNn}{
\begin{tikzpicture}[scale=0.15,baseline=(pt_base.base)]

\coordinate (pt_base) at (0,-0.5) ;

\draw[blue] (-1.25,1.25) -- (0,0) ;
\draw[blue] (1.25,1.25) -- (0,0) ;

\draw[red] (0,-1.25) -- (0,0) ;

\draw (-1.25,0) -- (1.25,0) ;

\node[yshift = -4pt,scale=0.8] at (-1.25,0) {\tiny $I$} ;

\end{tikzpicture}
}

\newcommand{\exarbrebicoloreA}{
\begin{tikzpicture}[scale = 0.15 ,baseline=(pt_base.base)]

\coordinate (pt_base) at (0,-0.5) ;
\node (A) at (-1.5,1) {} ;
\node (B) at (0,0) {} ;
\node (C) at (1.5,1) {} ;
\node (D) at (0,-1) {} ;
\node (E) at (0.25,2) {} ;
\node (F) at (2.75,2) {} ;
\node (G) at (-0.25,2) {} ;
\node (I) at (-2.75,2) {} ;
\node (O) at (0.9,0.6) {} ;
\node (P) at (-0.9,0.6) {} ;

\node (K) at (-2.75,0.6) {} ;
\node (L) at (2.75,0.6) {} ;

\node (M) at (3,0.6) {} ;
\node (N) at (3,0) {} ;

\draw[blue] (F.center) -- (C.center) ;
\draw[blue] (E.center) -- (C.center) ;
\draw[blue] (O.center) -- (C.center) ;
\draw[blue] (A.center) -- (I.center) ;
\draw[blue] (A.center) -- (G.center) ;
\draw[blue] (A.center) -- (P.center) ;
\draw[red] (B.center) -- (P.center) ;
\draw[red] (B.center) -- (O.center) ;
\draw[red] (B.center) -- (D.center) ;

\draw (K.center) -- (L.center) ;

\draw (-2,0.5) node[yshift = -5 pt,xshift = -6pt,scale = 0.7]{\tiny $[0<1<2]$} ;

\end{tikzpicture}
}
\newcommand{\exarbrebicoloreB}{
\begin{tikzpicture}[scale = 0.15 ,baseline=(pt_base.base)]

\coordinate (pt_base) at (0,-0.5) ;
\node (A) at (-1.5,1) {} ;
\node (B) at (0,0) {} ;
\node (C) at (1.5,1) {} ;
\node (D) at (0,-1) {} ;
\node (E) at (0.25,2) {} ;
\node (F) at (2.75,2) {} ;
\node (G) at (-0.25,2) {} ;
\node (I) at (-2.75,2) {} ;
\node (O) at (0.9,0.6) {} ;
\node (P) at (-0.9,0.6) {} ;

\node (K) at (-2.75,0.6) {} ;
\node (L) at (2.75,0.6) {} ;

\node (M) at (3,0.6) {} ;
\node (N) at (3,0) {} ;

\draw[blue] (F.center) -- (C.center) ;
\draw[blue] (E.center) -- (C.center) ;
\draw[blue] (O.center) -- (C.center) ;
\draw[blue] (A.center) -- (I.center) ;
\draw[blue] (A.center) -- (G.center) ;
\draw[blue] (A.center) -- (P.center) ;
\draw[red] (B.center) -- (P.center) ;
\draw[red] (B.center) -- (O.center) ;
\draw[red] (B.center) -- (D.center) ;

\draw (K.center) -- (L.center) ;

\draw (-2,0.5) node[yshift = -5 pt,xshift = -3pt,scale = 0.7]{\tiny $[1<2]$} ;

\end{tikzpicture}
}

\newcommand{\exarbrebicoloreC}{
\begin{tikzpicture}[scale = 0.15 ,baseline=(pt_base.base)]

\coordinate (pt_base) at (0,-0.5) ;
\node (A) at (-1.5,1) {} ;
\node (B) at (0,0) {} ;
\node (C) at (1.5,1) {} ;
\node (D) at (0,-1) {} ;
\node (E) at (0.25,2) {} ;
\node (F) at (2.75,2) {} ;
\node (G) at (-0.25,2) {} ;
\node (I) at (-2.75,2) {} ;
\node (O) at (0.9,0.6) {} ;
\node (P) at (-0.9,0.6) {} ;

\node (K) at (-2.75,0.6) {} ;
\node (L) at (2.75,0.6) {} ;

\node (M) at (3,0.6) {} ;
\node (N) at (3,0) {} ;

\draw[blue] (F.center) -- (C.center) ;
\draw[blue] (E.center) -- (C.center) ;
\draw[blue] (O.center) -- (C.center) ;
\draw[blue] (A.center) -- (I.center) ;
\draw[blue] (A.center) -- (G.center) ;
\draw[blue] (A.center) -- (P.center) ;
\draw[red] (B.center) -- (P.center) ;
\draw[red] (B.center) -- (O.center) ;
\draw[red] (B.center) -- (D.center) ;

\draw (K.center) -- (L.center) ;

\draw (-2,0.5) node[yshift = -5 pt,xshift = -3pt,scale = 0.7]{\tiny $[0<2]$} ;

\end{tikzpicture}
}

\newcommand{\exarbrebicoloreD}{
\begin{tikzpicture}[scale = 0.15 ,baseline=(pt_base.base)]

\coordinate (pt_base) at (0,-0.5) ;
\node (A) at (-1.5,1) {} ;
\node (B) at (0,0) {} ;
\node (C) at (1.5,1) {} ;
\node (D) at (0,-1) {} ;
\node (E) at (0.25,2) {} ;
\node (F) at (2.75,2) {} ;
\node (G) at (-0.25,2) {} ;
\node (I) at (-2.75,2) {} ;
\node (O) at (0.9,0.6) {} ;
\node (P) at (-0.9,0.6) {} ;

\node (K) at (-2.75,0.6) {} ;
\node (L) at (2.75,0.6) {} ;

\node (M) at (3,0.6) {} ;
\node (N) at (3,0) {} ;

\draw[blue] (F.center) -- (C.center) ;
\draw[blue] (E.center) -- (C.center) ;
\draw[blue] (O.center) -- (C.center) ;
\draw[blue] (A.center) -- (I.center) ;
\draw[blue] (A.center) -- (G.center) ;
\draw[blue] (A.center) -- (P.center) ;
\draw[red] (B.center) -- (P.center) ;
\draw[red] (B.center) -- (O.center) ;
\draw[red] (B.center) -- (D.center) ;

\draw (K.center) -- (L.center) ;

\draw (-2,0.5) node[yshift = -5 pt,xshift = -3pt,scale = 0.7]{\tiny $[0<1]$} ;

\end{tikzpicture}
}

\newcommand{\exarbrebicoloreE}{
\begin{tikzpicture}[scale = 0.15 ,baseline=(pt_base.base)]

\coordinate (pt_base) at (0,-0.5) ;
\node (A) at (-1.5,1) {} ;
\node (B) at (0,0) {} ;
\node (C) at (0.9,0.6) {} ;
\node (D) at (0,-1) {} ;
\node (E) at (0.25,2) {} ;
\node (F) at (2.75,2) {} ;
\node (G) at (-0.25,2) {} ;
\node (I) at (-2.75,2) {} ;
\node (O) at (0.9,0.6) {} ;
\node (P) at (-0.9,0.6) {} ;

\node (K) at (-2.75,0.6) {} ;
\node (L) at (2.75,0.6) {} ;

\node (M) at (3,0.6) {} ;
\node (N) at (3,0) {} ;

\draw[blue] (F.center) -- (C.center) ;
\draw[blue] (E.center) -- (C.center) ;
\draw[blue] (O.center) -- (C.center) ;
\draw[blue] (A.center) -- (I.center) ;
\draw[blue] (A.center) -- (G.center) ;
\draw[blue] (A.center) -- (P.center) ;
\draw[red] (B.center) -- (P.center) ;
\draw[red] (B.center) -- (O.center) ;
\draw[red] (B.center) -- (D.center) ;

\draw (K.center) -- (L.center) ;

\draw (-2,0.5) node[yshift = -5 pt,xshift = -6pt,scale = 0.7]{\tiny $[0<1<2]$} ;

\end{tikzpicture}
}

\newcommand{\exarbrebicoloreF}{
\begin{tikzpicture}[scale = 0.15 ,baseline=(pt_base.base)]

\coordinate (pt_base) at (0,-0.5) ;
\node (A) at (-0.9,0.6) {} ;
\node (B) at (0,0) {} ;
\node (C) at (1.5,1) {} ;
\node (D) at (0,-1) {} ;
\node (E) at (0.25,2) {} ;
\node (F) at (2.75,2) {} ;
\node (G) at (-0.25,2) {} ;
\node (I) at (-2.75,2) {} ;
\node (O) at (0.9,0.6) {} ;
\node (P) at (-0.9,0.6) {} ;

\node (K) at (-2.75,0.6) {} ;
\node (L) at (2.75,0.6) {} ;

\node (M) at (3,0.6) {} ;
\node (N) at (3,0) {} ;

\draw[blue] (F.center) -- (C.center) ;
\draw[blue] (E.center) -- (C.center) ;
\draw[blue] (O.center) -- (C.center) ;
\draw[blue] (A.center) -- (I.center) ;
\draw[blue] (A.center) -- (G.center) ;
\draw[blue] (A.center) -- (P.center) ;
\draw[red] (B.center) -- (P.center) ;
\draw[red] (B.center) -- (O.center) ;
\draw[red] (B.center) -- (D.center) ;

\draw (K.center) -- (L.center) ;

\draw (-2,0.5) node[yshift = -5 pt,xshift = -6pt,scale = 0.7]{\tiny $[0<1<2]$} ;

\end{tikzpicture}
}

\newcommand{\exarbrebicoloreG}{
\begin{tikzpicture}[scale = 0.15 ,baseline=(pt_base.base)]

\coordinate (pt_base) at (0,-0.5) ;
\node (A) at (-1.5,1) {} ;
\node (B) at (0,0) {} ;
\node (C) at (1.5,1) {} ;
\node (D) at (0,-1) {} ;
\node (E) at (0.25,2) {} ;
\node (F) at (2.75,2) {} ;
\node (G) at (-0.25,2) {} ;
\node (I) at (-2.75,2) {} ;

\node (K) at (-2.75,0) {} ;
\node (L) at (2.75,0) {} ;

\draw[blue] (F.center) -- (C.center) ;
\draw[blue] (E.center) -- (C.center) ;
\draw[blue] (B.center) -- (C.center) ;
\draw[blue] (A.center) -- (I.center) ;
\draw[blue] (A.center) -- (G.center) ;
\draw[blue] (A.center) -- (B.center) ;
\draw[red] (B.center) -- (D.center) ;

\draw (K.center) -- (L.center) ;

\draw (-2,-0.1) node[yshift = -5 pt,xshift = -6pt,scale = 0.7]{\tiny $[0<1<2]$} ;

\end{tikzpicture}
}

\newcommand{\exarbrebicoloreH}{
\begin{tikzpicture}[scale = 0.18,baseline=(pt_base.base)]

\coordinate (pt_base) at (0,-0.5) ;

\begin{scope}
\node (K) at (-0.625,0.625) {} ;
\node (L) at (0,0) {} ;
\draw[blue] (-0.625,0.625) -- (0,1.25) ;
\draw[blue] (-1.25,1.25) -- (-0.3,0.3) ;
\draw[blue] (1.25,1.25) -- (0.3,0.3) ;
\draw[red] (0,0) -- (-0.3,0.3) ;
\draw[red] (0,0) -- (0.3,0.3) ;
\draw[red] (0,-1.25) -- (0,0) ;

\draw (-1.25,0.3) -- (1.25,0.3) ;

\draw (-1.1,0.2) node[yshift = -5 pt,xshift = -9pt,scale = 0.7]{\tiny $[0<1<2]$} ;
\end{scope}

\begin{scope}[shift = {(1.25,2.1)},blue]
\draw (-0.625,0.625) -- (0,0) ;
\draw (0.625,0.625) -- (0,0) ;
\draw (0,-0.625) -- (0,0) ;
\end{scope}
\end{tikzpicture}
}

\newcommand{\exarbrebicoloreI}{
\begin{tikzpicture}[scale = 0.18,baseline=(pt_base.base)]

\coordinate (pt_base) at (0,-0.5) ;

\begin{scope}
\node (K) at (-0.625,0.625) {} ;
\node (L) at (0,0) {} ;
\draw[blue] (0.625,0.625) -- (0,1.25) ;
\draw[blue] (-1.25,1.25) -- (-0.3,0.3) ;
\draw[blue] (1.25,1.25) -- (0.3,0.3) ;
\draw[red] (0,0) -- (-0.3,0.3) ;
\draw[red] (0,0) -- (0.3,0.3) ;
\draw[red] (0,-1.25) -- (0,0) ;

\draw (-1.25,0.3) -- (1.25,0.3) ;

\draw (-1.1,0.2) node[yshift = -5 pt,xshift = -9pt,scale = 0.7]{\tiny $[0<1<2]$} ;
\end{scope}

\begin{scope}[shift = {(-1.25,2.1)},blue]
\draw (-0.625,0.625) -- (0,0) ;
\draw (0.625,0.625) -- (0,0) ;
\draw (0,-0.625) -- (0,0) ;
\end{scope}
\end{tikzpicture}
}

\newcommand{\exarbrebicoloreJ}{
\begin{tikzpicture}[scale = 0.1,baseline=(pt_base.base)]

\coordinate (pt_base) at (0,-0.5) ;

\begin{scope}[red]
\draw (-1.6,1.25) -- (0,0) ;
\draw (1.6,1.25) -- (0,0) ;
\draw (0,-1.25) -- (0,0) ;
\end{scope}

\begin{scope}[shift = {(-1.6,2.8)}]
\draw[blue] (-1.25,1.25) -- (0,0) ;
\draw[blue] (1.25,1.25) -- (0,0) ;
\draw[blue] (0,-0.625) -- (0,0) ;
\draw[red] (0,-0.625) -- (0,-1.25) ;
\draw (-1.25,-0.625) -- (1.25,-0.625) ;

\draw (-1.1,-0.5) node[yshift = -4 pt,xshift = -2pt,scale = 0.7]{\tiny $[0]$} ;
\end{scope}

\begin{scope}[shift = {(1.6,2.8)}]
\draw[blue] (-1.25,1.25) -- (0,0) ;
\draw[blue] (1.25,1.25) -- (0,0) ;
\draw[blue] (0,-0.625) -- (0,0) ;
\draw[red] (0,-0.625) -- (0,-1.25) ;
\draw (-1.25,-0.625) -- (1.25,-0.625) ;

\draw (1.3,-0.5) node[yshift = -4 pt,xshift = 11pt,scale = 0.7]{\tiny $[0<1<2]$} ;
\end{scope}
\end{tikzpicture}
}

\newcommand{\exarbrebicoloreK}{
\begin{tikzpicture}[scale = 0.1,baseline=(pt_base.base)]

\coordinate (pt_base) at (0,-0.5) ;

\begin{scope}[red]
\draw (-1.6,1.25) -- (0,0) ;
\draw (1.6,1.25) -- (0,0) ;
\draw (0,-1.25) -- (0,0) ;
\end{scope}

\begin{scope}[shift = {(-1.6,2.8)}]
\draw[blue] (-1.25,1.25) -- (0,0) ;
\draw[blue] (1.25,1.25) -- (0,0) ;
\draw[blue] (0,-0.625) -- (0,0) ;
\draw[red] (0,-0.625) -- (0,-1.25) ;
\draw (-1.25,-0.625) -- (1.25,-0.625) ;

\draw (-1.1,-0.5) node[yshift = -4 pt,xshift = -7pt,scale = 0.7]{\tiny $[0<1]$} ;
\end{scope}

\begin{scope}[shift = {(1.6,2.8)}]
\draw[blue] (-1.25,1.25) -- (0,0) ;
\draw[blue] (1.25,1.25) -- (0,0) ;
\draw[blue] (0,-0.625) -- (0,0) ;
\draw[red] (0,-0.625) -- (0,-1.25) ;
\draw (-1.25,-0.625) -- (1.25,-0.625) ;

\draw (1.3,-0.5) node[yshift = -4 pt,xshift = 7pt,scale = 0.7]{\tiny $[1<2]$} ;
\end{scope}
\end{tikzpicture}
}

\newcommand{\exarbrebicoloreL}{
\begin{tikzpicture}[scale = 0.1,baseline=(pt_base.base)]

\coordinate (pt_base) at (0,-0.5) ;

\begin{scope}[red]
\draw (-1.6,1.25) -- (0,0) ;
\draw (1.6,1.25) -- (0,0) ;
\draw (0,-1.25) -- (0,0) ;
\end{scope}

\begin{scope}[shift = {(-1.6,2.8)}]
\draw[blue] (-1.25,1.25) -- (0,0) ;
\draw[blue] (1.25,1.25) -- (0,0) ;
\draw[blue] (0,-0.625) -- (0,0) ;
\draw[red] (0,-0.625) -- (0,-1.25) ;
\draw (-1.25,-0.625) -- (1.25,-0.625) ;

\draw (-1.1,-0.5) node[yshift = -4 pt,xshift = -13pt,scale = 0.7]{\tiny $[0<1<2]$} ;
\end{scope}

\begin{scope}[shift = {(1.6,2.8)}]
\draw[blue] (-1.25,1.25) -- (0,0) ;
\draw[blue] (1.25,1.25) -- (0,0) ;
\draw[blue] (0,-0.625) -- (0,0) ;
\draw[red] (0,-0.625) -- (0,-1.25) ;
\draw (-1.25,-0.625) -- (1.25,-0.625) ;

\draw (1.3,-0.5) node[yshift = -4 pt,xshift = 2pt,scale = 0.7]{\tiny $[2]$} ;
\end{scope}
\end{tikzpicture}
}


\newcommand{\etiquettedeltaunjdeuxA}{
\begin{tikzpicture}[scale=0.15,baseline=(pt_base.base)]

\coordinate (pt_base) at (0,-0.5) ;

\draw[blue] (-1.25,1.25) -- (-0.5,0.5) ;
\draw[blue] (1.25,1.25) -- (0.5,0.5) ;

\draw[red] (-0.5,0.5) -- (0,0) ;
\draw[red] (0.5,0.5) -- (0,0) ;
\draw[red] (0,-1.25) -- (0,0) ;

\draw (-1.25,0.5) -- (1.25,0.5) ;

\node[yshift = -4pt,scale=0.4] at (-1.25,0.5) {\tiny $[0]$} ;

\end{tikzpicture}
}

\newcommand{\etiquettedeltaunjdeuxB}{
\begin{tikzpicture}[scale=0.15,baseline=(pt_base.base)]

\coordinate (pt_base) at (0,-0.5) ;

\draw[blue] (-1.25,1.25) -- (-0.5,0.5) ;
\draw[blue] (1.25,1.25) -- (0.5,0.5) ;

\draw[red] (-0.5,0.5) -- (0,0) ;
\draw[red] (0.5,0.5) -- (0,0) ;
\draw[red] (0,-1.25) -- (0,0) ;

\draw (-1.25,0.5) -- (1.25,0.5) ;

\node[yshift = -4pt,scale=0.4] at (-1.25,0.5) {\tiny $[1]$} ;

\end{tikzpicture}
}

\newcommand{\etiquettedeltaunjdeuxC}{
\begin{tikzpicture}[scale=0.15,baseline=(pt_base.base)]

\coordinate (pt_base) at (0,-0.5) ;

\draw[blue] (-1.25,1.25) -- (-0.5,0.5) ;
\draw[blue] (1.25,1.25) -- (0.5,0.5) ;

\draw[red] (-0.5,0.5) -- (0,0) ;
\draw[red] (0.5,0.5) -- (0,0) ;
\draw[red] (0,-1.25) -- (0,0) ;

\draw (-1.25,0.5) -- (1.25,0.5) ;

\node[yshift = -4pt,xshift=-1pt,scale=0.4] at (-1.25,0.5) {\tiny $[0<1]$} ;

\end{tikzpicture}
}

\newcommand{\etiquettedeltaunjdeuxE}{
\begin{tikzpicture}[scale=0.15,baseline=(pt_base.base)]

\begin{scope}[red]
\draw (-1.25,1.25) -- (0,0) ;
\draw (1.25,1.25) -- (0,0) ;
\draw (0,-1.25) -- (0,0) ;
\end{scope}

\begin{scope}[shift = {(-1.25,3)}]

\draw[blue] (0,1.25) -- (0,0.5) ;
\draw[red] (0,-1.25) -- (0,0.5) ;

\draw (-1.25,0.5) -- (0.75,0.5) ;

\node[yshift = -4pt,xshift=-1pt,scale=0.4] at (-1.25,0.5) {\tiny $[0]$} ;

\end{scope}

\begin{scope}[shift = {(1.25,3)}]

\draw[blue] (0,1.25) -- (0,0.5) ;
\draw[red] (0,-1.25) -- (0,0.5) ;

\draw (-0.75,0.5) -- (1.25,0.5) ;

\node[yshift = -4pt,xshift=1pt,scale=0.4] at (1.25,0.5) {\tiny $[0<1]$} ;

\end{scope}

\end{tikzpicture}
}

\newcommand{\etiquettedeltaunjdeuxF}{
\begin{tikzpicture}[scale=0.15,baseline=(pt_base.base)]

\begin{scope}[red]
\draw (-1.25,1.25) -- (0,0) ;
\draw (1.25,1.25) -- (0,0) ;
\draw (0,-1.25) -- (0,0) ;
\end{scope}

\begin{scope}[shift = {(-1.25,3)}]

\draw[blue] (0,1.25) -- (0,0.5) ;
\draw[red] (0,-1.25) -- (0,0.5) ;

\draw (-1.25,0.5) -- (0.75,0.5) ;

\node[yshift = -4pt,xshift=-1pt,scale=0.4] at (-1.25,0.5) {\tiny $[0<1]$} ;

\end{scope}

\begin{scope}[shift = {(1.25,3)}]

\draw[blue] (0,1.25) -- (0,0.5) ;
\draw[red] (0,-1.25) -- (0,0.5) ;

\draw (-0.75,0.5) -- (1.25,0.5) ;

\node[yshift = -4pt,xshift=1pt,scale=0.4] at (1.25,0.5) {\tiny $[1]$} ;

\end{scope}

\end{tikzpicture}
}

\newcommand{\etiquettedeltaunjdeuxD}{
\begin{tikzpicture}[scale=0.15,baseline=(pt_base.base)]

\begin{scope}
\draw[blue] (0,1.25) -- (0,0.5) ;
\draw[red] (0,-1.25) -- (0,0.5) ;

\draw (-1.25,0.5) -- (1.25,0.5) ;

\node[yshift = -4pt,xshift=-1pt,scale=0.4] at (-1.25,0.5) {\tiny $[0<1]$} ;
\end{scope}

\begin{scope}[blue, shift = {(0,3)}]
\draw (-1.25,1.25) -- (0,0) ;
\draw (1.25,1.25) -- (0,0) ;
\draw (0,-1.25) -- (0,0) ;
\end{scope}

\end{tikzpicture}
}

\newcommand{\etiquettedeltadeuxjdeuxA}{
\begin{tikzpicture}[scale=0.15,baseline=(pt_base.base)]

\coordinate (pt_base) at (0,-0.5) ;

\draw[blue] (-1.25,1.25) -- (-0.5,0.5) ;
\draw[blue] (1.25,1.25) -- (0.5,0.5) ;

\draw[red] (-0.5,0.5) -- (0,0) ;
\draw[red] (0.5,0.5) -- (0,0) ;
\draw[red] (0,-1.25) -- (0,0) ;

\draw (-1.25,0.5) -- (1.25,0.5) ;

\node[yshift = -4pt,scale=0.4] at (-1.25,0.5) {\tiny $[0<1]$} ;

\end{tikzpicture}
}

\newcommand{\etiquettedeltadeuxjdeuxB}{
\begin{tikzpicture}[scale=0.15,baseline=(pt_base.base)]

\coordinate (pt_base) at (0,-0.5) ;

\draw[blue] (-1.25,1.25) -- (-0.5,0.5) ;
\draw[blue] (1.25,1.25) -- (0.5,0.5) ;

\draw[red] (-0.5,0.5) -- (0,0) ;
\draw[red] (0.5,0.5) -- (0,0) ;
\draw[red] (0,-1.25) -- (0,0) ;

\draw (-1.25,0.5) -- (1.25,0.5) ;

\node[yshift = -4pt,scale=0.4] at (-1.25,0.5) {\tiny $[1<2]$} ;

\end{tikzpicture}
}

\newcommand{\etiquettedeltadeuxjdeuxF}{
\begin{tikzpicture}[scale=0.15,baseline=(pt_base.base)]

\coordinate (pt_base) at (0,-0.5) ;

\draw[blue] (-1.25,1.25) -- (-0.5,0.5) ;
\draw[blue] (1.25,1.25) -- (0.5,0.5) ;

\draw[red] (-0.5,0.5) -- (0,0) ;
\draw[red] (0.5,0.5) -- (0,0) ;
\draw[red] (0,-1.25) -- (0,0) ;

\draw (-1.25,0.5) -- (1.25,0.5) ;

\node[yshift = -4pt,scale=0.4] at (-1.25,0.5) {\tiny $[0<2]$} ;

\end{tikzpicture}
}

\newcommand{\etiquettedeltadeuxjdeuxO}{
\begin{tikzpicture}[scale=0.15,baseline=(pt_base.base)]

\coordinate (pt_base) at (0,-0.5) ;

\draw[blue] (-1.25,1.25) -- (-0.5,0.5) ;
\draw[blue] (1.25,1.25) -- (0.5,0.5) ;

\draw[red] (-0.5,0.5) -- (0,0) ;
\draw[red] (0.5,0.5) -- (0,0) ;
\draw[red] (0,-1.25) -- (0,0) ;

\draw (-1.25,0.5) -- (1.25,0.5) ;

\node[yshift = -4pt,xshift=-3pt,scale=0.4] at (-1.25,0.5) {\tiny $[0<1<2]$} ;

\end{tikzpicture}
}

\newcommand{\etiquettedeltadeuxjdeuxC}{
\begin{tikzpicture}[scale=0.15,baseline=(pt_base.base)]

\begin{scope}[red]
\draw (-1.25,1.25) -- (0,0) ;
\draw (1.25,1.25) -- (0,0) ;
\draw (0,-1.25) -- (0,0) ;
\end{scope}

\begin{scope}[shift = {(-1.25,3)}]

\draw[blue] (0,1.25) -- (0,0.5) ;
\draw[red] (0,-1.25) -- (0,0.5) ;

\draw (-1.25,0.5) -- (0.75,0.5) ;

\node[yshift = -4pt,xshift=-1pt,scale=0.4] at (-1.25,0.5) {\tiny $[0]$} ;

\end{scope}

\begin{scope}[shift = {(1.25,3)}]

\draw[blue] (0,1.25) -- (0,0.5) ;
\draw[red] (0,-1.25) -- (0,0.5) ;

\draw (-0.75,0.5) -- (1.25,0.5) ;

\node[yshift = -4pt,xshift=3pt,scale=0.4] at (1.25,0.5) {\tiny $[0<1<2]$} ;

\end{scope}

\end{tikzpicture}
}

\newcommand{\etiquettedeltadeuxjdeuxD}{
\begin{tikzpicture}[scale=0.15,baseline=(pt_base.base)]

\begin{scope}[red]
\draw (-1.25,1.25) -- (0,0) ;
\draw (1.25,1.25) -- (0,0) ;
\draw (0,-1.25) -- (0,0) ;
\end{scope}

\begin{scope}[shift = {(-1.25,3)}]

\draw[blue] (0,1.25) -- (0,0.5) ;
\draw[red] (0,-1.25) -- (0,0.5) ;

\draw (-1.25,0.5) -- (0.75,0.5) ;

\node[yshift = -4pt,xshift=-1pt,scale=0.4] at (-1.25,0.5) {\tiny $[0<1]$} ;

\end{scope}

\begin{scope}[shift = {(1.25,3)}]

\draw[blue] (0,1.25) -- (0,0.5) ;
\draw[red] (0,-1.25) -- (0,0.5) ;

\draw (-0.75,0.5) -- (1.25,0.5) ;

\node[yshift = -4pt,xshift=1pt,scale=0.4] at (1.25,0.5) {\tiny $[1<2]$} ;

\end{scope}

\end{tikzpicture}
}

\newcommand{\etiquettedeltadeuxjdeuxE}{
\begin{tikzpicture}[scale=0.15,baseline=(pt_base.base)]

\begin{scope}[red]
\draw (-1.25,1.25) -- (0,0) ;
\draw (1.25,1.25) -- (0,0) ;
\draw (0,-1.25) -- (0,0) ;
\end{scope}

\begin{scope}[shift = {(-1.25,3)}]

\draw[blue] (0,1.25) -- (0,0.5) ;
\draw[red] (0,-1.25) -- (0,0.5) ;

\draw (-1.25,0.5) -- (0.75,0.5) ;

\node[yshift = -4pt,xshift=-3pt,scale=0.4] at (-1.25,0.5) {\tiny $[0<1<2]$} ;

\end{scope}

\begin{scope}[shift = {(1.25,3)}]

\draw[blue] (0,1.25) -- (0,0.5) ;
\draw[red] (0,-1.25) -- (0,0.5) ;

\draw (-0.75,0.5) -- (1.25,0.5) ;

\node[yshift = -4pt,xshift=1pt,scale=0.4] at (1.25,0.5) {\tiny $[2]$} ;

\end{scope}

\end{tikzpicture}
}

\newcommand{\etiquettedeltadeuxjdeuxG}{
\begin{tikzpicture}[scale=0.15,baseline=(pt_base.base)]

\begin{scope}
\draw[blue] (0,1.25) -- (0,0.5) ;
\draw[red] (0,-1.25) -- (0,0.5) ;

\draw (-1.25,0.5) -- (1.25,0.5) ;

\node[yshift = -4pt,xshift=-3pt,scale=0.4] at (-1.25,0.5) {\tiny $[0<1<2]$} ;
\end{scope}

\begin{scope}[blue, shift = {(0,3)}]
\draw (-1.25,1.25) -- (0,0) ;
\draw (1.25,1.25) -- (0,0) ;
\draw (0,-1.25) -- (0,0) ;
\end{scope}

\end{tikzpicture}
}

\newcommand{\etiquettedeltaunjtroisB}{
\begin{tikzpicture}[scale=0.15,baseline=(pt_base.base)]

\begin{scope}
\draw[blue] (-1.25,1.25) -- (-0.5,0.5) ;
\draw[blue] (1.25,1.25) -- (0.5,0.5) ;

\draw[red] (-0.5,0.5) -- (0,0) ;
\draw[red] (0.5,0.5) -- (0,0) ;
\draw[red] (0,-1.25) -- (0,0) ;

\draw (-1.25,0.5) -- (1.25,0.5) ;

\node[yshift = -4pt,xshift=-1pt,scale=0.4] at (-1.25,0.5) {\tiny $[0<1]$} ;
\end{scope}

\begin{scope}[shift = {(1.25,3)},blue]
\draw (-1.25,1.25) -- (0,0) ;
\draw (1.25,1.25) -- (0,0) ;
\draw (0,-1.25) -- (0,0) ;
\end{scope}
\end{tikzpicture}
}

\newcommand{\etiquettedeltaunjtroisC}{
\begin{tikzpicture}[scale=0.15,baseline=(pt_base.base)]

\begin{scope}[shift = {(1.25,3)}]
\draw[blue] (-1.25,1.25) -- (-0.5,0.5) ;
\draw[blue] (1.25,1.25) -- (0.5,0.5) ;

\draw[red] (-0.5,0.5) -- (0,0) ;
\draw[red] (0.5,0.5) -- (0,0) ;
\draw[red] (0,-1.25) -- (0,0) ;

\draw (-1.25,0.5) -- (1.25,0.5) ;
\node[yshift = -4pt,xshift=1pt,scale=0.4] at (1.25,0.5) {\tiny $[0<1]$} ;
\end{scope}

\begin{scope}[red]
\draw (-1.25,1.25) -- (0,0) ;
\draw (1.25,1.25) -- (0,0) ;
\draw (0,-1.25) -- (0,0) ;
\end{scope}

\begin{scope}[shift = {(-1.25,3)}]
\draw[blue] (0,1.25) -- (0,0.5) ;
\draw[red] (0,-1.25) -- (0,0.5) ;

\draw (-1.25,0.5) -- (0.75,0.5) ;

\node[yshift = -4pt,xshift=-1pt,scale=0.4] at (-1.25,0.5) {\tiny $[0]$} ;
\end{scope}
\end{tikzpicture}
}

\newcommand{\etiquettedeltaunjtroisD}{
\begin{tikzpicture}[scale=0.15,baseline=(pt_base.base)]

\begin{scope}[shift = {(1.25,3)}]
\draw[blue] (-1.25,1.25) -- (-0.5,0.5) ;
\draw[blue] (1.25,1.25) -- (0.5,0.5) ;

\draw[red] (-0.5,0.5) -- (0,0) ;
\draw[red] (0.5,0.5) -- (0,0) ;
\draw[red] (0,-1.25) -- (0,0) ;

\draw (-1.25,0.5) -- (1.25,0.5) ;
\node[yshift = -4pt,xshift=1pt,scale=0.4] at (1.25,0.5) {\tiny $[1]$} ;
\end{scope}

\begin{scope}[red]
\draw (-1.25,1.25) -- (0,0) ;
\draw (1.25,1.25) -- (0,0) ;
\draw (0,-1.25) -- (0,0) ;
\end{scope}

\begin{scope}[shift = {(-1.25,3)}]
\draw[blue] (0,1.25) -- (0,0.5) ;
\draw[red] (0,-1.25) -- (0,0.5) ;

\draw (-1.25,0.5) -- (0.75,0.5) ;

\node[yshift = -4pt,xshift=-2pt,scale=0.4] at (-1.25,0.5) {\tiny $[0<1]$} ;
\end{scope}
\end{tikzpicture}
}

\newcommand{\etiquettedeltaunjtroisG}{
\begin{tikzpicture}[scale=0.15,baseline=(pt_base.base)]

\begin{scope}[shift = {(-1.25,3)}]
\draw[blue] (-1.25,1.25) -- (-0.5,0.5) ;
\draw[blue] (1.25,1.25) -- (0.5,0.5) ;

\draw[red] (-0.5,0.5) -- (0,0) ;
\draw[red] (0.5,0.5) -- (0,0) ;
\draw[red] (0,-1.25) -- (0,0) ;

\draw (-1.25,0.5) -- (1.25,0.5) ;
\node[yshift = -4pt,xshift=-2pt,scale=0.4] at (-1.25,0.5) {\tiny $[0<1]$} ;
\end{scope}

\begin{scope}[red]
\draw (-1.25,1.25) -- (0,0) ;
\draw (1.25,1.25) -- (0,0) ;
\draw (0,-1.25) -- (0,0) ;
\end{scope}

\begin{scope}[shift = {(1.25,3)}]
\draw[blue] (0,1.25) -- (0,0.5) ;
\draw[red] (0,-1.25) -- (0,0.5) ;

\draw (-0.75,0.5) -- (1.25,0.5) ;
\node[yshift = -4pt,xshift=-2pt,scale=0.4] at (1.25,0.5) {\tiny $[1]$} ;
\end{scope}
\end{tikzpicture}
}

\newcommand{\etiquettedeltaunjtroisH}{
\begin{tikzpicture}[scale=0.15,baseline=(pt_base.base)]

\begin{scope}[shift = {(-1.25,3)}]
\draw[blue] (-1.25,1.25) -- (-0.5,0.5) ;
\draw[blue] (1.25,1.25) -- (0.5,0.5) ;

\draw[red] (-0.5,0.5) -- (0,0) ;
\draw[red] (0.5,0.5) -- (0,0) ;
\draw[red] (0,-1.25) -- (0,0) ;

\draw (-1.25,0.5) -- (1.25,0.5) ;
\node[yshift = -4pt,scale=0.4] at (-1.25,0.5) {\tiny $[0]$} ;
\end{scope}

\begin{scope}[red]
\draw (-1.25,1.25) -- (0,0) ;
\draw (1.25,1.25) -- (0,0) ;
\draw (0,-1.25) -- (0,0) ;
\end{scope}

\begin{scope}[shift = {(1.25,3)}]
\draw[blue] (0,1.25) -- (0,0.5) ;
\draw[red] (0,-1.25) -- (0,0.5) ;

\draw (-0.75,0.5) -- (1.25,0.5) ;
\node[yshift = -4pt,scale=0.4] at (1.25,0.5) {\tiny $[0<1]$} ;
\end{scope}
\end{tikzpicture}
}

\newcommand{\etiquettedeltaunjtroisF}{
\begin{tikzpicture}[scale=0.15,baseline=(pt_base.base)]

\begin{scope}
\draw[blue] (-1.25,1.25) -- (-0.5,0.5) ;
\draw[blue] (1.25,1.25) -- (0.5,0.5) ;

\draw[red] (-0.5,0.5) -- (0,0) ;
\draw[red] (0.5,0.5) -- (0,0) ;
\draw[red] (0,-1.25) -- (0,0) ;

\draw (-1.25,0.5) -- (1.25,0.5) ;
\node[yshift = -4pt,xshift=-2pt,scale=0.4] at (-1.25,0.5) {\tiny $[0<1]$} ;
\end{scope}

\begin{scope}[shift = {(-1.25,3)},blue]
\draw (-1.25,1.25) -- (0,0) ;
\draw (1.25,1.25) -- (0,0) ;
\draw (0,-1.25) -- (0,0) ;
\end{scope}
\end{tikzpicture}
}

\newcommand{\etiquettedeltaunjtroisJ}{
\begin{tikzpicture}[scale=0.15,baseline=(pt_base.base)]

\begin{scope}[shift = {(-1.25,3)}]
\draw[blue] (0,1.25) -- (0,0.5) ;
\draw[red] (0,-1.25) -- (0,0.5) ;

\draw (-1.25,0.5) -- (0.5,0.5) ;
\node[yshift = -4pt,xshift=-2pt,scale=0.4] at (-1.25,0.5) {\tiny $[0<1]$} ;
\end{scope}

\begin{scope}[shift = {(0,3)}]
\draw[blue] (0,1.25) -- (0,0.5) ;
\draw[red] (0,-1.25) -- (0,0.5) ;

\draw (-0.5,0.5) -- (0.5,0.5) ;
\node[yshift = -4pt,scale=0.4] at (-0.5,0.5) {\tiny $[1]$} ;
\end{scope}

\begin{scope}[shift = {(1.25,3)}]
\draw[blue] (0,1.25) -- (0,0.5) ;
\draw[red] (0,-1.25) -- (0,0.5) ;

\draw (-0.5,0.5) -- (1.25,0.5) ;
\node[yshift = -4pt,xshift=-2pt,scale=0.4] at (1.25,0.5) {\tiny $[1]$} ;
\end{scope}

\begin{scope}[red]
\draw (0,1.25) -- (0,0) ;
\draw (-1.25,1.25) -- (0,0) ;
\draw (1.25,1.25) -- (0,0) ;
\draw (0,-1.25) -- (0,0) ;
\end{scope}

\end{tikzpicture}
}

\newcommand{\etiquettedeltaunjtroisK}{
\begin{tikzpicture}[scale=0.15,baseline=(pt_base.base)]

\begin{scope}[shift = {(-1.25,3)}]
\draw[blue] (0,1.25) -- (0,0.5) ;
\draw[red] (0,-1.25) -- (0,0.5) ;

\draw (-1.25,0.5) -- (0.5,0.5) ;
\node[yshift = -4pt,xshift=-2pt,scale=0.4] at (-1.25,0.5) {\tiny $[0]$} ;
\end{scope}

\begin{scope}[shift = {(0,3)}]
\draw[blue] (0,1.25) -- (0,0.5) ;
\draw[red] (0,-1.25) -- (0,0.5) ;

\draw (-0.5,0.5) -- (0.5,0.5) ;
\node[yshift = -4pt,scale=0.4] at (0,0.5) {\tiny $[0<1]$} ;
\end{scope}

\begin{scope}[shift = {(1.25,3)}]
\draw[blue] (0,1.25) -- (0,0.5) ;
\draw[red] (0,-1.25) -- (0,0.5) ;

\draw (-0.5,0.5) -- (1.25,0.5) ;
\node[yshift = -4pt,scale=0.4] at (1.25,0.5) {\tiny $[1]$} ;
\end{scope}

\begin{scope}[red]
\draw (0,1.25) -- (0,0) ;
\draw (-1.25,1.25) -- (0,0) ;
\draw (1.25,1.25) -- (0,0) ;
\draw (0,-1.25) -- (0,0) ;
\end{scope}

\end{tikzpicture}
}

\newcommand{\etiquettedeltaunjtroisL}{
\begin{tikzpicture}[scale=0.15,baseline=(pt_base.base)]

\begin{scope}[shift = {(-1.25,3)}]
\draw[blue] (0,1.25) -- (0,0.5) ;
\draw[red] (0,-1.25) -- (0,0.5) ;

\draw (-1.25,0.5) -- (0.5,0.5) ;
\node[yshift = -4pt,scale=0.4] at (-1.25,0.5) {\tiny $[0]$} ;
\end{scope}

\begin{scope}[shift = {(0,3)}]
\draw[blue] (0,1.25) -- (0,0.5) ;
\draw[red] (0,-1.25) -- (0,0.5) ;

\draw (-0.5,0.5) -- (0.5,0.5) ;
\node[yshift = -4pt,scale=0.4] at (-0.5,0.5) {\tiny $[0]$} ;
\end{scope}

\begin{scope}[shift = {(1.25,3)}]
\draw[blue] (0,1.25) -- (0,0.5) ;
\draw[red] (0,-1.25) -- (0,0.5) ;

\draw (-0.5,0.5) -- (1.25,0.5) ;
\node[yshift = -4pt,xshift=1pt,scale=0.4] at (1.25,0.5) {\tiny $[0<1]$} ;
\end{scope}

\begin{scope}[red]
\draw (0,1.25) -- (0,0) ;
\draw (-1.25,1.25) -- (0,0) ;
\draw (1.25,1.25) -- (0,0) ;
\draw (0,-1.25) -- (0,0) ;
\end{scope}

\end{tikzpicture}
}

\newcommand{\etiquettedeltaunjtroisE}{
\begin{tikzpicture}[scale=0.15,baseline=(pt_base.base)]

\begin{scope}
\draw[blue] (0,1.25) -- (0,0.5) ;
\draw[red] (0,-1.25) -- (0,0.5) ;

\draw (-1.25,0.5) -- (1.25,0.5) ;
\node[yshift = -4pt,xshift=-2pt,scale=0.4] at (-1.25,0.5) {\tiny $[0<1]$} ;
\end{scope}

\begin{scope}[blue, shift = {(0,3)}]
\draw (0,1.25) -- (0,0) ;
\draw (-1.25,1.25) -- (0,0) ;
\draw (1.25,1.25) -- (0,0) ;
\draw (0,-1.25) -- (0,0) ;
\end{scope}

\end{tikzpicture}
}

\newcommand{\etiquettedeltaunjtroisA}{
\begin{tikzpicture}[scale=0.15,baseline=(pt_base.base)]

\coordinate (pt_base) at (0,-0.5) ;

\draw[blue] (0,1.25) -- (0,0.5) ;
\draw[blue] (-1.25,1.25) -- (-0.5,0.5) ;
\draw[blue] (1.25,1.25) -- (0.5,0.5) ;

\draw[red] (0,0.5) -- (0,0) ;
\draw[red] (-0.5,0.5) -- (0,0) ;
\draw[red] (0.5,0.5) -- (0,0) ;
\draw[red] (0,-1.25) -- (0,0) ;

\draw (-1.25,0.5) -- (1.25,0.5) ;

\node[yshift = -4pt,scale=0.4] at (-1.25,0.5) {\tiny $[0]$} ;

\end{tikzpicture}
}

\newcommand{\etiquettedeltaunjtroisI}{
\begin{tikzpicture}[scale=0.15,baseline=(pt_base.base)]

\coordinate (pt_base) at (0,-0.5) ;

\draw[blue] (0,1.25) -- (0,0.5) ;
\draw[blue] (-1.25,1.25) -- (-0.5,0.5) ;
\draw[blue] (1.25,1.25) -- (0.5,0.5) ;

\draw[red] (0,0.5) -- (0,0) ;
\draw[red] (-0.5,0.5) -- (0,0) ;
\draw[red] (0.5,0.5) -- (0,0) ;
\draw[red] (0,-1.25) -- (0,0) ;

\draw (-1.25,0.5) -- (1.25,0.5) ;

\node[yshift = -4pt,scale=0.4] at (-1.25,0.5) {\tiny $[1]$} ;

\end{tikzpicture}
}

\newcommand{\etiquettedeltaunjtroisO}{
\begin{tikzpicture}[scale=0.15,baseline=(pt_base.base)]

\coordinate (pt_base) at (0,-0.5) ;

\draw[blue] (0,1.25) -- (0,0.5) ;
\draw[blue] (-1.25,1.25) -- (-0.5,0.5) ;
\draw[blue] (1.25,1.25) -- (0.5,0.5) ;

\draw[red] (0,0.5) -- (0,0) ;
\draw[red] (-0.5,0.5) -- (0,0) ;
\draw[red] (0.5,0.5) -- (0,0) ;
\draw[red] (0,-1.25) -- (0,0) ;

\draw (-1.25,0.5) -- (1.25,0.5) ;

\node[yshift = -4pt,xshift=-2pt,scale=0.4] at (-1.25,0.5) {\tiny $[0<1]$} ;

\end{tikzpicture}
}

\newcommand{\etiquettedeltaunjdeuxfinA}{
\begin{tikzpicture}[scale=0.15,baseline=(pt_base.base)]

\draw[blue] (-1.25,1.25) -- (-0.5,0.5) ;
\draw[blue] (1.25,1.25) -- (0.5,0.5) ;

\draw[red] (-0.5,0.5) -- (0,0) ;
\draw[red] (0.5,0.5) -- (0,0) ;
\draw[red] (0,-1.25) -- (0,0) ;

\draw (-1.25,0.5) -- (1.25,0.5) ;

\node[yshift = -4pt,scale=0.4] at (-1.25,0.5) {\tiny $[0]$} ;

\end{tikzpicture}
}

\newcommand{\etiquettedeltaunjdeuxfinB}{
\begin{tikzpicture}[scale=0.15,baseline=(pt_base.base)]

\draw[blue] (-1.25,1.25) -- (0,0) ;
\draw[blue] (1.25,1.25) -- (0,0) ;

\draw[blue] (0,-0.5) -- (0,0) ;
\draw[red] (0,-1.25) -- (0,-0.5) ;

\draw (-1.25,-0.5) -- (1.25,-0.5) ;

\node[yshift = 2pt,scale=0.4] at (-1.25,-0.5) {\tiny $[0]$} ;

\end{tikzpicture}
}

\newcommand{\etiquettedeltaunjdeuxfinC}{
\begin{tikzpicture}[scale=0.15,baseline=(pt_base.base)]

\draw[blue] (-1.25,1.25) -- (-0.5,0.5) ;
\draw[blue] (1.25,1.25) -- (0.5,0.5) ;

\draw[red] (-0.5,0.5) -- (0,0) ;
\draw[red] (0.5,0.5) -- (0,0) ;
\draw[red] (0,-1.25) -- (0,0) ;

\draw (-1.25,0.5) -- (1.25,0.5) ;

\node[yshift = -4pt,xshift=-3pt,scale=0.4] at (-1.25,0.5) {\tiny $[0<1]$} ;

\end{tikzpicture}
}

\newcommand{\etiquettedeltaunjdeuxfinD}{
\begin{tikzpicture}[scale=0.15,baseline=(pt_base.base)]

\draw[blue] (-1.25,1.25) -- (0,0) ;
\draw[blue] (1.25,1.25) -- (0,0) ;

\draw[blue] (0,-0.5) -- (0,0) ;
\draw[red] (0,-1.25) -- (0,-0.5) ;

\draw (-1.25,-0.5) -- (1.25,-0.5) ;

\node[yshift = 2pt,xshift=-3pt,scale=0.4] at (-1.25,-0.5) {\tiny $[0<1]$} ;

\end{tikzpicture}
}

\newcommand{\etiquettedeltaunjdeuxfinE}{
\begin{tikzpicture}[scale=0.15,baseline=(pt_base.base)]

\draw[blue] (-1.25,1.25) -- (-0.5,0.5) ;
\draw[blue] (1.25,1.25) -- (0.5,0.5) ;

\draw[red] (-0.5,0.5) -- (0,0) ;
\draw[red] (0.5,0.5) -- (0,0) ;
\draw[red] (0,-1.25) -- (0,0) ;

\draw (-1.25,0.5) -- (1.25,0.5) ;

\node[yshift = -4pt,scale=0.4] at (-1.25,0.5) {\tiny $[1]$} ;

\end{tikzpicture}
}

\newcommand{\etiquettedeltaunjdeuxfinF}{
\begin{tikzpicture}[scale=0.15,baseline=(pt_base.base)]

\draw[blue] (-1.25,1.25) -- (0,0) ;
\draw[blue] (1.25,1.25) -- (0,0) ;

\draw[blue] (0,-0.5) -- (0,0) ;
\draw[red] (0,-1.25) -- (0,-0.5) ;

\draw (-1.25,-0.5) -- (1.25,-0.5) ;

\node[yshift = 2pt,scale=0.4] at (-1.25,-0.5) {\tiny $[1]$} ;

\end{tikzpicture}
}

\newcommand{\etiquettedeltaunjdeuxfinJ}{
\begin{tikzpicture}[scale=0.15,baseline=(pt_base.base)]

\draw[blue] (-1.25,1.25) -- (-0,0) ;
\draw[blue] (1.25,1.25) -- (0,0) ;

\draw[red] (0,-1.25) -- (0,0) ;

\draw (-1.25,0) -- (1.25,0) ;

\node[yshift = -4pt,xshift=-3pt,scale=0.4] at (-1.25,0) {\tiny $[0<1]$} ;

\end{tikzpicture}
}

\newcommand{\etiquettedeltaunjtroisfinA}{
\begin{tikzpicture}[scale=0.15,baseline=(pt_base.base)]
\node (A) at (0,0) {} ;
\node (B) at (0,-1) {} ;
\node (C) at (1.25,1.25) {} ;
\node (D) at (-1.25,1.25) {} ;
\node (E) at (0,1.25) {} ;
\node (F) at (0.625,0.625) {} ;
\node (G) at (0,-0.3) {} ;
\node (H) at (-1.25,-0.3) {} ;
\node (I) at (1.25,-0.3) {}  ;
\node (J) at (1.5,-0.3) {} ;
\node (K) at (1.5,0) {} ;

\draw[blue] (D.center) -- (A.center) ;
\draw[blue] (A.center) -- (G.center) ;
\draw[blue] (A.center) -- (F.center) ;
\draw[blue] (F.center) -- (E.center) ;
\draw[blue] (C.center) -- (F.center) ;

\draw[red] (G.center) -- (B.center) ;

\draw (H.center) -- (I.center) ;

\node[yshift = -4pt,scale=0.4] at (-1.25,-0.3) {\tiny $[0]$} ;
\end{tikzpicture}}

\newcommand{\etiquettedeltaunjtroisfinE}{
\begin{tikzpicture}[scale=0.15,baseline=(pt_base.base)]
\node (A) at (0,0) {} ;
\node (B) at (0,-1) {} ;
\node (C) at (1.25,1.25) {} ;
\node (D) at (-1.25,1.25) {} ;
\node (E) at (0,1.25) {} ;
\node (F) at (0.625,0.625) {} ;
\node (G) at (0,-0.3) {} ;
\node (H) at (-1.25,-0.3) {} ;
\node (I) at (1.25,-0.3) {}  ;

\draw[blue] (D.center) -- (A.center) ;
\draw[blue] (A.center) -- (G.center) ;
\draw[blue] (A.center) -- (E.center) ;
\draw[blue] (C.center) -- (A.center) ;

\draw[red] (G.center) -- (B.center) ;

\draw (H.center) -- (I.center) ;

\node[yshift = -4pt,xshift=-3pt,scale=0.4] at (-1.25,-0.3) {\tiny $[0<1]$} ;
\end{tikzpicture}}

\newcommand{\etiquettedeltaunjtroisfinC}{
\begin{tikzpicture}[scale=0.15,baseline=(pt_base.base)]
\node (A) at (0,0) {} ;
\node (B) at (0,-1) {} ;
\node (C) at (1.25,1.25) {} ;
\node (D) at (-1.25,1.25) {} ;
\node (E) at (0,1.25) {} ;
\node (F) at (0.625,0.625) {} ;
\node (G) at (0,-0.3) {} ;
\node (H) at (-1.25,0) {} ;
\node (I) at (1.25,0) {}  ;

\draw[blue] (D.center) -- (A.center) ;
\draw[blue] (A.center) -- (G.center) ;
\draw[blue] (A.center) -- (F.center) ;
\draw[blue] (F.center) -- (E.center) ;
\draw[blue] (C.center) -- (F.center) ;

\draw[red] (A.center) -- (B.center) ;

\draw (H.center) -- (I.center) ;

\node[yshift = -4pt,xshift=-3pt,scale=0.4] at (-1.25,0) {\tiny $[0<1]$} ;
\end{tikzpicture}}

\newcommand{\etiquettedeltaunjtroisfinD}{
\begin{tikzpicture}[scale=0.15,baseline=(pt_base.base)]
\node (A) at (0,0) {} ;
\node (B) at (0,-1) {} ;
\node (C) at (1.25,1.25) {} ;
\node (D) at (-1.25,1.25) {} ;
\node (E) at (0,1.25) {} ;
\node (F) at (0.625,0.625) {} ;
\node (G) at (0,-0.3) {} ;
\node (H) at (-1.25,-0.3) {} ;
\node (I) at (1.25,-0.3) {}  ;
\node (J) at (1.5,-0.3) {} ;
\node (K) at (1.5,0) {} ;

\draw[blue] (D.center) -- (A.center) ;
\draw[blue] (A.center) -- (G.center) ;
\draw[blue] (A.center) -- (F.center) ;
\draw[blue] (F.center) -- (E.center) ;
\draw[blue] (C.center) -- (F.center) ;

\draw[red] (G.center) -- (B.center) ;

\draw (H.center) -- (I.center) ;

\node[yshift = -4pt,scale=0.4] at (-1.25,-0.3) {\tiny $[1]$} ;
\end{tikzpicture}}

\newcommand{\etiquettedeltaunjtroisfinB}{
\begin{tikzpicture}[scale=0.15,baseline=(pt_base.base)]

\begin{scope}
\draw[blue] (-1.25,1.25) -- (0,0) ;
\draw[blue] (1.25,1.25) -- (0,0) ;

\draw[blue] (0,-0.3) -- (0,0) ;
\draw[red] (0,-1.25) -- (0,-0.3) ;

\draw (-1.25,-0.3) -- (1.25,-0.3) ;

\node[yshift = -4pt,xshift=-3pt,scale=0.4] at (-1.25,-0.3) {\tiny $[0<1]$} ;
\end{scope}

\begin{scope}[shift = {(1.25,3)},blue]
\draw (-1.25,1.25) -- (0,0) ;
\draw (1.25,1.25) -- (0,0) ;
\draw (0,-1.25) -- (0,0) ;
\end{scope}
\end{tikzpicture}
}

\newcommand{\etiquettedeltaunjtroisfinF}{
\begin{tikzpicture}[scale=0.15,baseline=(pt_base.base)]

\begin{scope}
\draw[blue] (0,1.25) -- (0,0.5) ;
\draw[red] (0,-1.25) -- (0,0.5) ;

\draw (-1.25,0.5) -- (1.25,0.5) ;
\node[yshift = -4pt,xshift=-2pt,scale=0.4] at (-1.25,0.5) {\tiny $[0<1]$} ;
\end{scope}

\begin{scope}[blue,shift = {(0,3)}]
\draw (0,1.25) -- (0.625,0.625) ;
\draw (-1.25,1.25) -- (0,0) ;
\draw (1.25,1.25) -- (0,0) ;
\draw (0,-1.25) -- (0,0) ;
\end{scope}

\end{tikzpicture}
}

\newcommand{\midarrow}{
\begin{tikzpicture}

\draw (-2,1) -- (0,0) ;
\draw (-2,-1) -- (0,0) ;
\end{tikzpicture}}

\newcommand{\cross}{
\begin{tikzpicture}[very thick]
\draw (0.5,0.5) -- (-0.5,-0.5) ;
\draw (-0.5,0.5) -- (0.5,-0.5) ;
\end{tikzpicture}}

\newcommand{\arbredegradientbicoloreexempledelta}{
\begin{tikzpicture}

\node (A) at (-2,2) {};
\node (B) at (-0.5,2) {} ;
\node (C) at (0.5,2)  {};
\node (D) at (2,2) {} ;
\node (E) at (-1,1)  {};
\node (F) at (1,1) {} ;
\node (G) at (-0.5,0.5)  {};
\node (H) at (0.5,0.5) {};
\node (I) at (0,0) {} ;
\node (J) at (0,-1)  {};
\node (K) at (-2,0.5) {} ;
\node (L) at (2,0.5)  {};
\node (a) at ($(A)!0.7!(E)$) {} ;
\node (b) at ($(B)!0.7!(E)$)  {};
\node (c) at ($(C)!0.7!(F)$) {};
\node (d) at ($(D)!0.7!(F)$) {} ;
\node (e) at ($(G)!0.7!(E)$)  {};
\node (f) at ($(H)!0.7!(F)$) {} ;
\node (g) at ($(E)!0.7!(G)$)  {};
\node (gg) at ($(I)!0.7!(G)$) {} ;
\node (h) at ($(F)!0.7!(H)$)  {};
\node (hh) at ($(I)!0.7!(H)$) {};
\node (i) at ($(G)!0.7!(I)$) {} ;
\node (ii) at ($(H)!0.7!(I)$)  {};
\node (iii) at ($(J)!0.7!(I)$) {} ;

\draw (A.center) -- (a.center) node[midway,sloped,allow upside down,scale=0.05]{\midarrow} ;
\draw (B.center) -- (b.center) node[midway,sloped,allow upside down,scale=0.05]{\midarrow} ;
\draw (C.center) -- (c.center) node[midway,sloped,allow upside down,scale=0.05]{\midarrow} ;
\draw (D.center) -- (d.center) node[midway,sloped,allow upside down,scale=0.05]{\midarrow} ;

\draw [line width = 1.5 pt,green!40!gray] (a.center) -- (E.center) ;
\draw [line width = 1.5 pt,green!40!gray] (b.center) -- (E.center) ;
\draw [line width = 1.5 pt,green!40!gray] (c.center) -- (F.center) ;
\draw [line width = 1.5 pt,green!40!gray] (d.center) -- (F.center)  ;

\draw [line width = 1.5 pt,green!40!gray] (e.center) -- (E.center)  ;
\draw [line width = 1.5 pt,green!40!gray] (f.center) -- (F.center)  ;

\draw (e.center) -- (g.center) node[midway,sloped,allow upside down,scale=0.05]{\midarrow} ;
\draw (f.center) -- (h.center) node[midway,sloped,allow upside down,scale=0.05]{\midarrow} ;

\draw [line width = 1.5 pt,green!40!gray] (g.center) -- (G.center) ;
\draw [line width = 1.5 pt,green!40!gray] (h.center) -- (H.center)  ;

\draw [line width = 1.5 pt,green!40!gray] (hh.center) -- (H.center) ;
\draw [line width = 1.5 pt,green!40!gray] (gg.center) -- (G.center) ;

\draw (gg.center) -- (i.center) node[midway,sloped,allow upside down,scale=0.05]{\midarrow} ;
\draw (hh.center) -- (ii.center) node[midway,sloped,allow upside down,scale=0.05]{\midarrow} ;

\draw [line width = 1.5 pt,green!40!gray] (i.center) -- (I.center) ;
\draw [line width = 1.5 pt,green!40!gray] (ii.center) -- (I.center) ;
\draw [line width = 1.5 pt,green!40!gray] (iii.center) -- (I.center) ;

\draw (iii.center) -- (J.center) node[midway,sloped,allow upside down,scale=0.05]{\midarrow} ;

\draw (K) -- (L) ;

\node[above=2 pt] at (A) {\small $x_1$}  ; 
\node[above,scale = 0.1] at (A) {\cross}  ;

\node[above=2 pt] at (B) {\small $x_2$}  ; 
\node[above,scale = 0.1] at (B) {\cross}  ;

\node[above=2 pt] at (C) {\small $x_3$}  ; 
\node[above,scale = 0.1] at (C) {\cross}  ;

\node[above=2 pt] at (D) {\small $x_4$}  ; 
\node[above,scale = 0.1] at (D) {\cross}  ;

\node[below=2 pt] at (J) {\small $y$}  ; 
\node[below,scale = 0.1] at (J) {\cross}  ;

\end{tikzpicture}}

\newcommand{\exampleperturbationCTbreakundelta}{\begin{tikzpicture}

\begin{scope}[yshift = -0.3cm]

\node (A) at (-1.25,1.25) {} ;
\node (B) at (0,1.25) {} ;
\node (C) at (1.25,1.25) {} ;
\node (D) at (-0.625,0.625) {} ;
\node (E) at (0,0) {} ;
\node (F) at (0,-1) {} ;

\node (G) at (-0.3,0.3) {} ;
\node (H) at (0.3,0.3) {} ;

\node (I) at (-1.25,0.3) {} ;
\node (J) at (1.25,0.3) {} ;

\node at (-1.5,-1) {$t_c$} ;

\draw (A.center) -- (D.center) ;
\draw (B.center) -- (D.center) ;
\draw (D.center) -- (G.center) ;
\draw (G.center) -- (E.center) ;
\draw (C.center) -- (H.center) ;
\draw (H.center) -- (E.center) ;
\draw (E.center) -- (F.center) ;

\draw (I.center) -- (J.center) ;

\node (d) at ($(A)!0.7!(D)$) {} ;

\draw [line width = 1.5 pt,green!40!gray] (d.center) -- (D.center) ;

\node[green!40!gray,xshift = -20 pt, yshift = -10 pt] at ($(A)!0.5!(D)$) { $\mathbb{Y}_{\delta,t_c,e_c}$} ;

\end{scope}

\begin{scope}[scale = 0.3,xshift = 15 cm]
\draw[->=latex] (-4,0) node{} -- (4,0) node{} node[midway,below = 2pt,align = center]{$\lim \textcolor{green!40!gray}{\mathbb{Y}_{\delta,t_c,e_c}} = \textcolor{blue}{\mathbb{X}^f_{t^2,e_c}}$} ;
\end{scope}

\begin{scope}[xshift = 9cm,yshift = -0.7 cm,scale = 0.5]

\begin{scope}
\node (A) at (-1.25,1.25) {} ;
\node (B) at (1.25,1.25) {} ;
\node (C) at (0,0) {} ;
\node (D) at (0,-1) {} ;
\node (E) at (-1.25,0.5) {};
\node (F) at (1.25,0.5) {};

\draw (A.center) -- (C.center) ;
\draw (B.center) -- (C.center) ;
\draw (D.center) -- (C.center) ;

\draw (E.center) -- (F.center) ;

\node at (-1.5,-1) {$t^1_c$} ;
\end{scope}

\begin{scope}[shift = {(-1.25,2.6)}]
\node (A) at (-1.25,1.25) {} ;
\node (B) at (1.25,1.25) {} ;
\node (C) at (0,0) {} ;
\node (D) at (0,-1) {} ;

\node (c) at ($(A)!0.7!(C)$) {} ;

\draw (A.center) -- (C.center) ;
\draw (B.center) -- (C.center) ;
\draw (D.center) -- (C.center) ;

\node at (-1.5,-0.5) {$t^2$} ;

\draw [line width = 1.5 pt,blue] (c.center) -- (C.center) ;

\end{scope}

\end{scope}

\end{tikzpicture}}

\newcommand{\exampleperturbationCTbreakdeuxdelta}{\begin{tikzpicture}

\begin{scope}[yshift = -0.3cm]

\node (A) at (-1.25,1.25) {} ;
\node (B) at (0,1.25) {} ;
\node (C) at (1.25,1.25) {} ;
\node (D) at (-0.625,0.625) {} ;
\node (E) at (0,0) {} ;
\node (F) at (0,-1) {} ;

\node (G) at (-0.3,0.3) {} ;
\node (H) at (0.3,0.3) {} ;

\node (I) at (-1.25,0.3) {} ;
\node (J) at (1.25,0.3) {} ;

\node at (-1.5,-1) {$t_c$} ;

\draw (A.center) -- (D.center) ;
\draw (B.center) -- (D.center) ;
\draw (D.center) -- (G.center) ;
\draw (G.center) -- (E.center) ;
\draw (C.center) -- (H.center) ;
\draw (H.center) -- (E.center) ;
\draw (E.center) -- (F.center) ;

\node (g) at ($(E)!0.7!(G)$) {} ;
\node (e) at ($(G)!0.7!(E)$) {} ;

\draw [line width = 1.5 pt,green!40!gray] (g.center) -- (G.center) ;
\draw [line width = 1.5 pt,green!40!gray] (e.center) -- (E.center) ;

\node[green!40!gray,xshift = -20 pt, yshift = -17 pt] at ($(E)!0.5!(G)$) { $\mathbb{Y}_{\delta,t_c,e_c}$} ;

\draw (I.center) -- (J.center) ;

\end{scope}

\begin{scope}[scale = 0.3,xshift = 15 cm]
\draw[->=latex] (-4,0) node{} -- (4,0) node{} node[midway,below = 2pt, align = center]{$\lim \textcolor{green!40!gray}{\mathbb{Y}_{\delta,t_c,e_c}} = \textcolor{green!40!gray}{\mathbb{Y}_{\delta,t^1_c,e_c}}$} ;
\end{scope}

\begin{scope}[xshift = 9cm,yshift = -0.7 cm,scale = 0.5]

\begin{scope}
\node (A) at (-1.25,1.25) {} ;
\node (B) at (1.25,1.25) {} ;
\node (C) at (0,0) {} ;
\node (D) at (0,-1) {} ;
\node (E) at (-1.25,0.5) {};
\node (F) at (1.25,0.5) {};
\node (G) at (-0.5,0.5) {} ;

\draw (A.center) -- (C.center) ;
\draw (B.center) -- (C.center) ;
\draw (D.center) -- (C.center) ;

\node (g) at ($(C)!0.7!(G)$) {} ;
\node (c) at ($(G)!0.7!(C)$) {} ;

\draw [line width = 1.5 pt,green!40!gray] (c.center) -- (C.center) ;
\draw [line width = 1.5 pt,green!40!gray] (g.center) -- (G.center) ;

\draw (E.center) -- (F.center) ;

\node at (-1.5,-1) {$t^1_c$} ;
\end{scope}

\begin{scope}[shift = {(-1.25,2.6)}]
\node (A) at (-1.25,1.25) {} ;
\node (B) at (1.25,1.25) {} ;
\node (C) at (0,0) {} ;
\node (D) at (0,-1) {} ;

\draw (A.center) -- (C.center) ;
\draw (B.center) -- (C.center) ;
\draw (D.center) -- (C.center) ;

\node at (-1.5,-0.5) {$t^2$} ;

\end{scope}

\end{scope}

\end{tikzpicture}}

\newcommand{\exampleperturbationCTbreaktroisdelta}{\begin{tikzpicture}

\begin{scope}[yshift = -0.3cm]

\node (A) at (-1.25,1.25) {} ;
\node (B) at (0,1.25) {} ;
\node (C) at (1.25,1.25) {} ;
\node (D) at (-0.625,0.625) {} ;
\node (E) at (0,0) {} ;
\node (F) at (0,-1) {} ;

\node (G) at (-0.3,0.3) {} ;
\node (H) at (0.3,0.3) {} ;

\node (I) at (-1.25,0.3) {} ;
\node (J) at (1.25,0.3) {} ;

\node at (-1.5,-1) {$t_c$} ;

\draw (A.center) -- (D.center) ;
\draw (B.center) -- (D.center) ;
\draw (D.center) -- (G.center) ;
\draw (G.center) -- (E.center) ;
\draw (C.center) -- (H.center) ;
\draw (H.center) -- (E.center) ;
\draw (E.center) -- (F.center) ;

\node (g) at ($(D)!0.7!(G)$) {} ;
\node (d) at ($(G)!0.7!(D)$) {} ;

\draw [line width = 1.5 pt,green!40!gray] (g.center) -- (G.center) ;
\draw [line width = 1.5 pt,green!40!gray] (d.center) -- (D.center) ;

\node[green!40!gray,xshift = -20 pt, yshift = -17 pt] at ($(D)!0.5!(G)$) { $\mathbb{Y}_{\delta,t_c,e_c}$} ;

\draw (I.center) -- (J.center) ;

\end{scope}

\begin{scope}[scale = 0.3,xshift = 15 cm]
\draw[->=latex] (-4,0) node{} -- (4,0) node{} node[midway,below = 20pt, align = center]{$\lim _{t^1_c}\textcolor{green!40!gray}{\mathbb{Y}_{\delta,t_c,e_c}} = \textcolor{green!40!gray}{\mathbb{Y}_{\delta,t^1_c,e_c}}$} node{} node[midway,below = 2pt, align = center]{$\lim _{t^2}\textcolor{green!40!gray}{\mathbb{Y}_{\delta,t_c,e_c}} = \textcolor{blue}{\mathbb{X}^f_{t^2,e_c}}$} ;
\end{scope}

\begin{scope}[xshift = 9cm,yshift = -0.7 cm,scale = 0.5]

\begin{scope}
\node (A) at (-1.25,1.25) {} ;
\node (B) at (1.25,1.25) {} ;
\node (C) at (0,0) {} ;
\node (D) at (0,-1) {} ;
\node (E) at (-1.25,0.5) {};
\node (F) at (1.25,0.5) {};
\node (G) at (-0.5,0.5) {} ;

\draw (A.center) -- (C.center) ;
\draw (B.center) -- (C.center) ;
\draw (D.center) -- (C.center) ;

\node (g) at ($(A)!0.7!(G)$) {} ;

\draw [line width = 1.5 pt,green!40!gray] (g.center) -- (G.center) ;

\draw (E.center) -- (F.center) ;

\node at (-1.5,-1) {$t^1_c$} ;
\end{scope}

\begin{scope}[shift = {(-1.25,2.6)}]
\node (A) at (-1.25,1.25) {} ;
\node (B) at (1.25,1.25) {} ;
\node (C) at (0,0) {} ;
\node (D) at (0,-1) {} ;

\draw (A.center) -- (C.center) ;
\draw (B.center) -- (C.center) ;
\draw (D.center) -- (C.center) ;

\node (c) at ($(D)!0.7!(C)$) {} ;

\draw [line width = 1.5 pt,blue] (c.center) -- (C.center) ;

\node at (-1.5,-0.5) {$t^2$} ;

\end{scope}

\end{scope}

\end{tikzpicture}}

\newcommand{\exampleperturbationCTbreakquatredelta}{\begin{tikzpicture}

\begin{scope}[yshift = -0.3cm]

\node (A) at (-1.25,1.25) {} ;
\node (B) at (0,1.25) {} ;
\node (C) at (1.25,1.25) {} ;
\node (D) at (-0.625,0.625) {} ;
\node (E) at (0,0) {} ;
\node (F) at (0,-1) {} ;

\node (G) at (-0.3,0.3) {} ;
\node (H) at (0.3,0.3) {} ;

\node (I) at (-1.25,0.3) {} ;
\node (J) at (1.25,0.3) {} ;

\node at (-1.5,-1) {$t_c$} ;

\draw (A.center) -- (D.center) ;
\draw (B.center) -- (D.center) ;
\draw (D.center) -- (G.center) ;
\draw (G.center) -- (E.center) ;
\draw (C.center) -- (H.center) ;
\draw (H.center) -- (E.center) ;
\draw (E.center) -- (F.center) ;

\node (d) at ($(A)!0.7!(D)$) {} ;

\draw [line width = 1.5 pt,green!40!gray] (d.center) -- (D.center) ;

\node[green!40!gray,xshift = -20 pt, yshift = -10 pt] at ($(D)!0.5!(A)$) { $\mathbb{Y}_{\delta,t_c,e_c}$} ;

\draw (I.center) -- (J.center) ;

\end{scope}

\begin{scope}[scale = 0.3,xshift = 15 cm]
\draw[->=latex] (-4,0) node{} -- (4,0) node{} node[midway,below = 2pt, align = center]{$\lim \textcolor{green!40!gray}{\mathbb{Y}_{\delta,t_c,e_c}} = \textcolor{green!40!gray}{\mathbb{Y}_{\mathrm{pr}_1(\delta),t^1_c,e_c}}$} ;
\end{scope}

\begin{scope}[xshift = 9cm,yshift = -0.7 cm,scale = 0.5]

\begin{scope}
\node (A) at (-1.25,1.25) {} ;
\node (B) at (1.25,1.25) {} ;
\node (C) at (0,0) {} ;
\node (D) at (0,-1) {} ;

\draw (A.center) -- (C.center) ;
\draw (B.center) -- (C.center) ;
\draw (D.center) -- (C.center) ;

\node[below left] at (C) {$t^0$} ;

\end{scope}

\begin{scope}[shift = {(-1.25,2.6)}]
\node (A) at (-1.25,1.25) {} ;
\node (B) at (1.25,1.25) {} ;
\node (C) at (0,0) {} ;
\node (D) at (0,-1) {} ;
\node (E) at (-1.25,-0.5) {};
\node (F) at (1.25,-0.5) {};
\node (G) at (-0.5,-0.5) {} ;

\draw (A.center) -- (C.center) ;
\draw (B.center) -- (C.center) ;
\draw (D.center) -- (C.center) ;

\node (c) at ($(A)!0.7!(C)$) {} ;

\draw [line width = 1.5 pt,green!40!gray] (c.center) -- (C.center) ;

\draw (E.center) -- (F.center) ;

\node[below left] at (E) {$t^1_c$} ;

\end{scope}

\begin{scope}[shift = {(1.25,2.6)}]
\node (A) at (0,1.25) {} ;
\node (B) at (0,-1) {} ;
\node (C) at (-0.6,-0.5) {};
\node (D) at (0.6,-0.5) {};
\node (E) at (0,-0.5) {} ;

\draw (A.center) -- (B.center) ;

\draw (C.center) -- (D.center) ;

\node[below right] at (D) {$t^2_c$} ;

\end{scope}

\end{scope}

\end{tikzpicture}}

\newcommand{\exampleperturbationCTbreakcinqdelta}{\begin{tikzpicture}

\begin{scope}[yshift = -0.3cm]

\node (A) at (-1.25,1.25) {} ;
\node (B) at (0,1.25) {} ;
\node (C) at (1.25,1.25) {} ;
\node (D) at (-0.625,0.625) {} ;
\node (E) at (0,0) {} ;
\node (F) at (0,-1) {} ;

\node (G) at (-0.3,0.3) {} ;
\node (H) at (0.3,0.3) {} ;

\node (I) at (-1.25,0.3) {} ;
\node (J) at (1.25,0.3) {} ;

\node at (-1.5,-1) {$t_c$} ;

\draw (A.center) -- (D.center) ;
\draw (B.center) -- (D.center) ;
\draw (D.center) -- (G.center) ;
\draw (G.center) -- (E.center) ;
\draw (C.center) -- (H.center) ;
\draw (H.center) -- (E.center) ;
\draw (E.center) -- (F.center) ;

\node (e) at ($(F)!0.7!(E)$) {} ;

\draw [line width = 1.5 pt,green!40!gray] (e.center) -- (E.center) ;

\node[green!40!gray,xshift = -20 pt] at ($(F)!0.5!(E)$) { $\mathbb{Y}_{\delta,t_c,e_c}$} ;

\draw (I.center) -- (J.center) ;

\end{scope}

\begin{scope}[scale = 0.3,xshift = 15 cm]
\draw[->=latex] (-4,0) node{} -- (4,0) node{} node[midway,below = 2pt, align = center]{$\lim \textcolor{green!40!gray}{\mathbb{Y}_{\delta,t_c,e_c}} = \textcolor{red}{\mathbb{X}^g_{t^0,e_c}}$} ;
\end{scope}

\begin{scope}[xshift = 9cm,yshift = -0.7 cm,scale = 0.5]

\begin{scope}
\node (A) at (-1.25,1.25) {} ;
\node (B) at (1.25,1.25) {} ;
\node (C) at (0,0) {} ;
\node (D) at (0,-1) {} ;

\draw (A.center) -- (C.center) ;
\draw (B.center) -- (C.center) ;
\draw (D.center) -- (C.center) ;

\node[below left] at (C) {$t^0$} ;

\node (c) at ($(D)!0.7!(C)$) {} ;

\draw [line width = 1.5 pt,red] (c.center) -- (C.center) ;

\end{scope}

\begin{scope}[shift = {(-1.25,2.6)}]
\node (A) at (-1.25,1.25) {} ;
\node (B) at (1.25,1.25) {} ;
\node (C) at (0,0) {} ;
\node (D) at (0,-1) {} ;
\node (E) at (-1.25,-0.5) {};
\node (F) at (1.25,-0.5) {};
\node (G) at (-0.5,-0.5) {} ;

\draw (A.center) -- (C.center) ;
\draw (B.center) -- (C.center) ;
\draw (D.center) -- (C.center) ;

\draw (E.center) -- (F.center) ;

\node[below left] at (E) {$t^1_c$} ;

\end{scope}

\begin{scope}[shift = {(1.25,2.6)}]
\node (A) at (0,1.25) {} ;
\node (B) at (0,-1) {} ;
\node (C) at (-0.6,-0.5) {};
\node (D) at (0.6,-0.5) {};
\node (E) at (0,-0.5) {} ;

\draw (A.center) -- (B.center) ;

\draw (C.center) -- (D.center) ;

\node[below right] at (D) {$t^2_c$} ;

\end{scope}

\end{scope}

\end{tikzpicture}}

\newcommand{\exampleperturbationCTbreaksixdelta}{\begin{tikzpicture}

\begin{scope}[yshift = -0.3cm]

\node (A) at (-1.25,1.25) {} ;
\node (B) at (0,1.25) {} ;
\node (C) at (1.25,1.25) {} ;
\node (D) at (-0.625,0.625) {} ;
\node (E) at (0,0) {} ;
\node (F) at (0,-1) {} ;

\node (G) at (-0.3,0.3) {} ;
\node (H) at (0.3,0.3) {} ;

\node (I) at (-1.25,0.3) {} ;
\node (J) at (1.25,0.3) {} ;

\node at (-1.5,-1) {$t_c$} ;

\draw (A.center) -- (D.center) ;
\draw (B.center) -- (D.center) ;
\draw (D.center) -- (G.center) ;
\draw (G.center) -- (E.center) ;
\draw (C.center) -- (H.center) ;
\draw (H.center) -- (E.center) ;
\draw (E.center) -- (F.center) ;

\node (h) at ($(E)!0.7!(H)$) {} ;
\node (e) at ($(H)!0.7!(E)$) {} ;

\draw [line width = 1.5 pt,green!40!gray] (e.center) -- (E.center) ;
\draw [line width = 1.5 pt,green!40!gray] (h.center) -- (H.center) ;

\node[green!40!gray,xshift = 20 pt, yshift = -10 pt] at ($(H)!0.5!(E)$) { $\mathbb{Y}_{\delta,t_c,e_c}$} ;

\draw (I.center) -- (J.center) ;

\end{scope}

\begin{scope}[scale = 0.3,xshift = 15 cm]
\draw[->=latex] (-4,0) node{} -- (4,0) node{} node[midway,below = 20pt, align = center]{$\lim_{t^2} \textcolor{green!40!gray}{\mathbb{Y}_{\delta,t_c,e_c}} = \textcolor{red}{\mathbb{X}^g_{t^0,e_c}}$} node[midway,below = 2pt, align = center]{$\lim_{t^2_c} \textcolor{green!40!gray}{\mathbb{Y}_{\delta,t_c,e_c}} = \textcolor{green!40!gray}{\mathbb{Y}_{\mathrm{pr}_2(\delta),t^2_c,e_c}}$} ;
\end{scope}

\begin{scope}[xshift = 9cm,yshift = -0.7 cm,scale = 0.5]

\begin{scope}
\node (A) at (-1.25,1.25) {} ;
\node (B) at (1.25,1.25) {} ;
\node (C) at (0,0) {} ;
\node (D) at (0,-1) {} ;

\draw (A.center) -- (C.center) ;
\draw (B.center) -- (C.center) ;
\draw (D.center) -- (C.center) ;

\node[below left] at (C) {$t^0$} ;

\node (c) at ($(B)!0.7!(C)$) {} ;

\draw [line width = 1.5 pt,red] (c.center) -- (C.center) ;

\end{scope}

\begin{scope}[shift = {(-1.25,2.6)}]
\node (A) at (-1.25,1.25) {} ;
\node (B) at (1.25,1.25) {} ;
\node (C) at (0,0) {} ;
\node (D) at (0,-1) {} ;
\node (E) at (-1.25,-0.5) {};
\node (F) at (1.25,-0.5) {};
\node (G) at (-0.5,-0.5) {} ;

\draw (A.center) -- (C.center) ;
\draw (B.center) -- (C.center) ;
\draw (D.center) -- (C.center) ;

\draw (E.center) -- (F.center) ;

\node[below left] at (E) {$t^1_c$} ;

\end{scope}

\begin{scope}[shift = {(1.25,2.6)}]
\node (A) at (0,1.25) {} ;
\node (B) at (0,-1) {} ;
\node (C) at (-0.6,-0.5) {};
\node (D) at (0.6,-0.5) {};
\node (E) at (0,-0.5) {} ;

\draw (A.center) -- (B.center) ;

\node (e) at ($(B)!0.3!(E)$) {} ;

\draw [line width = 1.5 pt,green!40!gray] (e.center) -- (E.center) ;

\draw (C.center) -- (D.center) ;

\node[below right] at (D) {$t^2_c$} ;

\end{scope}

\end{scope}

\end{tikzpicture}}

\newcommand{\arbreopmorphformule}{
\begin{tikzpicture}[scale=0.2,baseline=(pt_base.base)]

\coordinate (pt_base) at (0,0) ;

\draw (-2,1.25) node[above]{\tiny 1} ;
\draw (-1,1.25) node[above]{\tiny 2} ;
\draw (1,1.25) node[above]{};
\draw (2,1.25) node[above]{\tiny $m$} ;

\draw[blue] (-2,1.25) -- (-0.625,0.5) ;
\draw[blue] (-1,1.25) -- (-0.3125,0.5) ;
\draw[blue] (1,1.25) -- (0.3125,0.5) ;
\draw[blue] (2,1.25) -- (0.625,0.5) ;

\draw[red] (-0.625,0.5) -- (0,0) ;
\draw[red] (-0.3125,0.5) -- (0,0) ;
\draw[red] (0.625,0.5) -- (0,0) ;
\draw[red] (0.3125,0.5) -- (0,0) ;
\draw[red] (0,-1.25) -- (0,0) ;

\draw[blue] [densely dotted] (-0.5,0.9) -- (0.5,0.9) ;

\draw  (-2,0.5) -- (2,0.5) ;

\end{tikzpicture}}

\newcommand{\exampleCTmItg}{
\begin{tikzpicture}[scale=0.25,baseline=(pt_base.base)]

\coordinate (pt_base) at (0,1.25) ;


\draw[red] (-4,1.25) -- (0,0) ;
\draw[red] (-2,1.25) -- (0,0) ;
\draw[red] (4,1.25) -- (0,0) ;
\draw[red] (2,1.25) -- (0,0) ;
\draw[red] (0,-1.25) -- (0,0) ; 


\draw[blue] (2.5,4) -- (3.4,3.25) ;
\draw[blue] (3.5,4) -- (3.8,3.25) ;
\draw[blue] (4.5,4) -- (4.2,3.25) ;
\draw[blue] (5.5,4) -- (4.6,3.25) ;

\draw[red] (3.4,3.25) -- (4,2.75) ;
\draw[red] (3.8,3.25) -- (4,2.75) ;
\draw[red] (4.2,3.25) -- (4,2.75) ;
\draw[red] (4.6,3.25) -- (4,2.75) ;
\draw[red] (4,1.5) -- (4,2.75) ;

\draw (2.5,3.25) -- (5.5,3.25) ;

\draw (2.5,3.25) node[scale = 0.7,yshift = - 8 pt] {$I_b$} ;


\draw[blue] (-2.5,4) -- (-3.4,3.25) ;
\draw[blue] (-3.5,4) -- (-3.8,3.25) ;
\draw[blue] (-4.5,4) -- (-4.2,3.25) ;
\draw[blue] (-5.5,4) -- (-4.6,3.25) ;

\draw[red] (-3.4,3.25) -- (-4,2.75) ;
\draw[red] (-3.8,3.25) -- (-4,2.75) ;
\draw[red] (-4.2,3.25) -- (-4,2.75) ;
\draw[red] (-4.6,3.25) -- (-4,2.75) ;
\draw[red] (-4,1.5) -- (-4,2.75) ;

\draw (-2.5,3.25) -- (-5.5,3.25) ;

\draw (-5.5,3.25) node[scale = 0.7,yshift = - 8 pt] {$I_1$} ;


\draw [densely dotted,red] (-1,0.9) -- (1,0.9) ;
\draw [densely dotted] (2,2.75) -- (-2,2.75) ; 
\draw [densely dotted,blue] (3.8,3.7) -- (4.2,3.7) ;
\draw [densely dotted,blue] (-3.8,3.7) -- (-4.2,3.7) ;

\end{tikzpicture}}

\newcommand{\orderingarbrebicoloresignesA}{
\begin{tikzpicture}[scale = 0.3 ,baseline=(pt_base.base)]

\coordinate (pt_base) at (0,-0.5) ;
\node (A) at (-1.5,1) {} ;
\node (B) at (0,0) {} ;
\node (C) at (1.5,1) {} ;
\node (D) at (0,-1) {} ;
\node (E) at (0.25,2) {} ;
\node (F) at (2.75,2) {} ;
\node (G) at (-0.25,2) {} ;
\node (I) at (-2.75,2) {} ;
\node (O) at (0.9,0.6) {} ;
\node (P) at (-0.9,0.6) {} ;

\node (K) at (-2.75,0.6) {} ;
\node (L) at (2.75,0.6) {} ;

\node at (-1.5,0) {\tiny $e_1$} ;
\node at (1.5,0) {\tiny $e_2$} ;

\node (M) at (3,0.6) {} ;
\node (N) at (3,0) {} ;

\draw[blue] (F.center) -- (C.center) ;
\draw[blue] (E.center) -- (C.center) ;
\draw[blue] (O.center) -- (C.center) ;
\draw[blue] (A.center) -- (I.center) ;
\draw[blue] (A.center) -- (G.center) ;
\draw[blue] (A.center) -- (P.center) ;
\draw[red] (B.center) -- (P.center) ;
\draw[red] (B.center) -- (O.center) ;
\draw[red] (B.center) -- (D.center) ;

\draw (K.center) -- (L.center) ;

\end{tikzpicture}
}

\newcommand{\arbrebijaugeexempleun}{
\begin{tikzpicture}
\node (A) at (0,0) {} ;
\node (B) at (0,-1) {} ;
\node (C) at (1.25,1.25) {} ;
\node (D) at (-1.25,1.25) {} ;
\node (E) at (0,1.25) {} ;
\node (F) at (-0.625,0.625) {} ;
\node (G) at (-0.3,0.3) {} ;
\node (H) at (0.3,0.3) {} ;
\node (I) at (-0.9,0.9) {} ;
\node (J) at (-0.35,0.9) {} ;
\node (K) at (0.9,0.9) {} ;

\node (L) at (-1.25,0.3) {} ;
\node (M) at (1.25,0.3) {} ;

\node (N) at (1.5,0.3) {}  ;
\node (O) at (1.5,0) {} ;

\node (P) at (-1.25,0.9) {} ;
\node (Q) at (1.25,0.9) {} ;

\node (R) at (2,0.9) {}  ;
\node (S) at (2,0) {} ;

\draw[blue] (D.center) -- (I.center) ;
\draw[blue] (E.center) -- (J.center) ;
\draw[blue] (C.center) -- (K.center) ;
\draw[orange] (I.center) -- (F.center) ;
\draw[orange] (J.center) -- (F.center) ;
\draw[orange] (F.center) -- (G.center) ;
\draw[orange] (K.center) -- (H.center) ;
\draw[red] (G.center) -- (A.center) ;
\draw[red] (H.center) -- (A.center) ;
\draw[red] (A.center) -- (B.center) ;

\node[below left = -3 pt] at ($(A)!0.4!(F)$) {\tiny $l$} ;

\draw (L.center) -- (M.center) ;
\draw (P.center) -- (Q.center) ;

\draw[densely dotted,>=stealth,<->] (N.center) -- (O.center) ;
\node[right] at ($(N)!0.5!(O)$) {\tiny $\lambda_1$} ;

\draw[densely dotted,>=stealth,<->] (R.center) -- (S.center) ;
\node[right] at ($(R)!0.5!(S)$) {\tiny $\lambda_2$} ;

\node (T) at (-1.5,0.9) {} ;
\node (U) at (-1.5,0.3) {} ;
\draw[densely dotted,>=stealth,<->] (T.center) -- (U.center) ;
\node[left] at ($(T)!0.5!(U)$) {\tiny $\delta$} ;

\end{tikzpicture}
}

\newcommand{\arbrebijaugeexempledeux}{
\begin{tikzpicture}
\node (A) at (0,0) {} ;
\node (B) at (0,-1) {} ;
\node (C) at (1.25,1.25) {} ;
\node (D) at (-1.25,1.25) {} ;
\node (E) at (0,1.25) {} ;
\node (F) at (-0.625,0.625) {} ;
\node (G) at (-0.3,0.3) {} ;
\node (H) at (0.3,0.3) {} ;
\node (I) at (-0.55,0.55) {} ;
\node (J) at (0.55,0.55) {} ;

\node (L) at (-1.25,0.3) {} ;
\node (M) at (1.25,0.3) {} ;

\node (N) at (1.5,0.3) {}  ;
\node (O) at (1.5,0) {} ;

\node (P) at (-1.25,0.55) {} ;
\node (Q) at (1.25,0.55) {} ;

\node (R) at (2,0.55) {}  ;
\node (S) at (2,0) {} ;

\draw[blue] (D.center) -- (F.center) ;
\draw[blue] (E.center) -- (F.center) ;
\draw[blue] (F.center) -- (I.center) ;
\draw[blue] (C.center) -- (J.center) ;
\draw[orange] (J.center) -- (H.center) ;
\draw[orange] (I.center) -- (G.center) ;
\draw[red] (G.center) -- (A.center) ;
\draw[red] (H.center) -- (A.center) ;
\draw[red] (A.center) -- (B.center) ;

\node[below left = -3 pt] at ($(A)!0.4!(F)$) {\tiny $l$} ;

\draw (L.center) -- (M.center) ;
\draw (P.center) -- (Q.center) ;

\draw[densely dotted,>=stealth,<->] (N.center) -- (O.center) ;
\node[right] at ($(N)!0.5!(O)$) {\tiny $\lambda_1$} ;

\draw[densely dotted,>=stealth,<->] (R.center) -- (S.center) ;
\node[right] at ($(R)!0.5!(S)$) {\tiny $\lambda_2$} ;

\node (T) at (-1.5,0.55) {} ;
\node (U) at (-1.5,0.3) {} ;
\draw[densely dotted,>=stealth,<->] (T.center) -- (U.center) ;
\node[left] at ($(T)!0.5!(U)$) {\tiny $\delta$} ;

\end{tikzpicture}
}

\newcommand{\arbrebijaugebordun}{
\begin{tikzpicture}[baseline=(pt_base.base)]

\coordinate (pt_base) at (0,-0.5) ;
\node at (0,-1.25) {(i)} ;

\begin{scope}[scale=0.5]

\node (A) at (-1.5,1) {} ;
\node (B) at (0,0) {} ;
\node (C) at (1.5,1) {} ;
\node (D) at (0,-1) {} ;
\node (E) at (0.25,2) {} ;
\node (F) at (2.75,2) {} ;
\node (G) at (-0.25,2) {} ;
\node (I) at (-2.75,2) {} ;
\node (O) at (0.9,0.6) {} ;
\node (P) at (-0.9,0.6) {} ;

\node (K) at (-2.75,0.6) {} ;
\node (L) at (2.75,0.6) {} ;

\node (M) at (3,0.6) {} ;
\node (N) at (3,0) {} ;

\draw[blue] (F.center) -- (C.center) ;
\draw[blue] (E.center) -- (C.center) ;
\draw[blue] (O.center) -- (C.center) ;
\draw[blue] (A.center) -- (I.center) ;
\draw[blue] (A.center) -- (G.center) ;
\draw[blue] (A.center) -- (P.center) ;
\draw[red] (B.center) -- (P.center) ;
\draw[red] (B.center) -- (O.center) ;
\draw[red] (B.center) -- (D.center) ;

\draw (K.center) -- (L.center) ;

\end{scope}

\end{tikzpicture}
}

\newcommand{\arbrebijaugeborddeux}{
\begin{tikzpicture}[baseline=(pt_base.base)]

\coordinate (pt_base) at (0,-0.5) ;
\node at (0,-1.25) {(ii)} ;

\begin{scope}[scale=0.7]

\begin{scope}
\node (A) at (0,0) {} ;
\node (B) at (0,-1) {} ;
\node (C) at (1.25,1.25) {} ;
\node (D) at (-1.25,1.25) {} ;
\node (E) at (0,1.25) {} ;
\node (F) at (-0.625,0.625) {} ;
\node (G) at (-0.3,0.3) {} ;
\node (H) at (0.3,0.3) {} ;
\node (I) at (-0.9,0.9) {} ;
\node (J) at (-0.35,0.9) {} ;
\node (K) at (0.9,0.9) {} ;

\node (L) at (-1.25,0.3) {} ;
\node (M) at (1.25,0.3) {} ;

\node (N) at (1.5,0.3) {}  ;
\node (O) at (1.5,0) {} ;

\node (P) at (-1.25,0.9) {} ;
\node (Q) at (1.25,0.9) {} ;

\node (R) at (2,0.9) {}  ;
\node (S) at (2,0) {} ;

\draw[blue] (D.center) -- (I.center) ;
\draw[blue] (E.center) -- (J.center) ;
\draw[blue] (C.center) -- (K.center) ;
\draw[orange] (I.center) -- (F.center) ;
\draw[orange] (J.center) -- (F.center) ;
\draw[orange] (F.center) -- (G.center) ;
\draw[orange] (K.center) -- (H.center) ;
\draw[red] (G.center) -- (A.center) ;
\draw[red] (H.center) -- (A.center) ;
\draw[red] (A.center) -- (B.center) ;

\draw (L.center) -- (M.center) ;
\draw (P.center) -- (Q.center) ;

\end{scope}

\begin{scope}[shift = {(1.25,2.1)},blue]
\draw (-0.625,0.625) -- (0,0) ;
\draw (0.625,0.625) -- (0,0) ;
\draw (0,-0.625) -- (0,0) ;
\end{scope}

\end{scope}
\end{tikzpicture}
}

\newcommand{\arbrebijaugebordquatre}{
\begin{tikzpicture}[baseline=(pt_base.base)]

\coordinate (pt_base) at (0,-0.5) ;
\node at (0,-1.25) {(iv)} ;

\begin{scope}[scale=0.4]

\begin{scope}
\draw[orange] (-1.6,1.25) -- (-0.64,0.5) ;
\draw[orange] (1.6,1.25) -- (0.64,0.5) ;
\draw[red] (-0.64,0.5) -- (0,0) ;
\draw[red] (0.64,0.5) -- (0,0) ;
\draw[red] (0,-1.25) -- (0,0) ;

\draw (-1.6,0.5) -- (1.6,0.5) ;
\end{scope}

\begin{scope}[shift = {(-1.6,2.8)}]
\draw[blue] (-1.25,1.25) -- (0,0) ;
\draw[blue] (1.25,1.25) -- (0,0) ;
\draw[blue] (0,-0.625) -- (0,0) ;
\draw[orange] (0,-0.625) -- (0,-1.25) ;
\draw (-1.25,-0.625) -- (1.25,-0.625) ;
\end{scope}

\begin{scope}[shift = {(1.6,2.8)}]
\draw[blue] (-1.25,1.25) -- (0,0) ;
\draw[blue] (1.25,1.25) -- (0,0) ;
\draw[blue] (0,-0.625) -- (0,0) ;
\draw[orange] (0,-0.625) -- (0,-1.25) ;
\draw (-1.25,-0.625) -- (1.25,-0.625) ;
\end{scope}

\end{scope}
\end{tikzpicture}
}

\newcommand{\arbrebijaugebordtrois}{
\begin{tikzpicture}[baseline=(pt_base.base)]

\coordinate (pt_base) at (0,-0.5) ;
\node at (0,-1.25) {(iii)} ;

\begin{scope}[scale=0.4]

\begin{scope}
\draw[red] (-1.6,1.25) -- (0,0) ;
\draw[red] (1.6,1.25) -- (0,0) ;
\draw[red] (0,-1.25) -- (0,0) ;

\node at (2,0) {\tiny $\delta_1 = \delta_2$} ;
\end{scope}

\begin{scope}[shift = {(-1.6,2.8)}]
\draw[blue] (-1.25,1.25) -- (-0.625,0.625) ;
\draw[blue] (1.25,1.25) -- (0.625,0.625) ;
\draw[orange] (-0.625,0.625) -- (0,0) ;
\draw[orange] (0.625,0.625) -- (0,0) ;
\draw[orange] (0,-0.625) -- (0,0) ;
\draw[red] (0,-0.625) -- (0,-1.25) ;
\draw (-1.25,-0.625) -- (1.25,-0.625) ;
\draw (-1.25,0.625) -- (1.25,0.625) ;

\node (R) at (-1.5,0.625) {}  ;
\node (S) at (-1.5,-0.625) {} ;
\draw[densely dotted,>=stealth,<->] (R.center) -- (S.center) ;
\node[left] at ($(R)!0.5!(S)$) {\tiny $\delta_1$} ;
\end{scope}

\begin{scope}[shift = {(1.6,2.8)}]
\draw[blue] (-1.25,1.25) -- (-0.625,0.625) ;
\draw[blue] (1.25,1.25) -- (0.625,0.625) ;
\draw[orange] (-0.625,0.625) -- (0,0) ;
\draw[orange] (0.625,0.625) -- (0,0) ;
\draw[orange] (0,-0.625) -- (0,0) ;
\draw[red] (0,-0.625) -- (0,-1.25) ;
\draw (-1.25,-0.625) -- (1.25,-0.625) ;
\draw (-1.25,0.625) -- (1.25,0.625) ;

\node (R) at (1.5,0.625) {}  ;
\node (S) at (1.5,-0.625) {} ;
\draw[densely dotted,>=stealth,<->] (R.center) -- (S.center) ;
\node[right] at ($(R)!0.5!(S)$) {\tiny $\delta_2$} ;
\end{scope}

\end{scope}
\end{tikzpicture}
}

\newcommand{\pointbullet}{
\begin{tikzpicture}[very thick]
\draw[fill,black] (0,0) circle (0.5) ;
\end{tikzpicture}}

\newcommand{\disquebordlagrangienassoc}{
\begin{tikzpicture}[scale = 0.5]

\draw (0,0) circle (3) ; 

\draw (360/20 : 3) node[scale = 0.1]{\pointbullet};
\draw (3*360/20 : 3) node[scale = 0.1]{\pointbullet};
\draw (9*360/20 : 3) node[scale = 0.1]{\pointbullet};
\draw (7*360/20 : 3) node[scale = 0.1]{\pointbullet};
\draw (270 : 3) node[scale = 0.1]{\pointbullet};

\draw (360/20 : 3) node[right,scale=0.7]{$x_n$};
\draw (3*360/20 : 3) node[above=3pt,scale=0.7]{$x_{n-1}$};
\draw (9*360/20 : 3) node[left,scale=0.7]{$x_1$};
\draw (7*360/20 : 3) node[above=3pt,scale=0.7]{$x_2$};
\draw (270 : 3) node[below=3pt,scale=0.7]{$y$};

\draw (2*360/20 : 3) node[above right,scale=0.8]{$L_{n-1}$};
\draw (8*360/20 : 3) node[above left,scale=0.8]{$L_1$};
\draw (11.5*360/20 : 3) node[below left,scale=0.8]{$L_0$};
\draw (18.5*360/20 : 3) node[below right,scale=0.8]{$L_n$};

\draw[densely dotted] (0,3.3) arc (90:108:3.3) ;
\draw[densely dotted] (0,3.3) arc (90:72:3.3) ;

\node at (0,0) {$M$} ;

\end{tikzpicture}
}

\newcommand{\quilteddiskexample}{
\begin{tikzpicture}[scale = 0.5]

\draw[fill,blue] (0,0) circle (3) ;
\draw (0,0) circle (3) ;

\draw (45 : 3) node[scale=0.1]{\pointbullet};
\draw (90 : 3) node[scale=0.1]{\pointbullet};
\draw (135 : 3) node[scale=0.1]{\pointbullet};

\draw (45 : 3) node[above=3pt,scale=0.7]{$z_3$};
\draw (90 : 3) node[above=3pt,scale=0.7]{$z_2$};
\draw (135 : 3) node[above=3pt,scale=0.7]{$z_1$};
\draw (270 : 3) node[below=3pt,scale=0.7]{$z_0$};

\draw[fill,red] (0,-1) circle (2) ;
\draw (0,-1) circle (2) ;

\draw (30 : 2) + (0,-1) node[xshift=-0.4cm,scale=0.8]{$C$};

\draw (270 : 3) node[scale = 0.1]{\pointbullet};

\end{tikzpicture}
}

\newcommand{\disquebordlagrangienmultipl}{
\begin{tikzpicture}[scale = 0.5]

\draw[fill,blue] (0,0) circle (3) ;
\draw (0,0) circle (3) ;

\draw (360/20 : 3) node[scale = 0.1]{\pointbullet};
\draw (3*360/20 : 3) node[scale = 0.1]{\pointbullet};
\draw (9*360/20 : 3) node[scale = 0.1]{\pointbullet};
\draw (7*360/20 : 3) node[scale = 0.1]{\pointbullet};

\draw (360/20 : 3) node[right,scale=0.7]{$x_n$};
\draw (3*360/20 : 3) node[above=3pt,scale=0.7]{$x_{n-1}$};
\draw (9*360/20 : 3) node[left,scale=0.7]{$x_1$};
\draw (7*360/20 : 3) node[above=3pt,scale=0.7]{$x_2$};
\draw (270 : 3) node[below=3pt,scale=0.7]{$y$};

\draw (2*360/20 : 3) node[above right,scale=0.8]{$L_{n-1}$};
\draw (8*360/20 : 3) node[above left,scale=0.8]{$L_1$};
\draw (11.5*360/20 : 3) node[below left,scale=0.8]{$L_0$};
\draw (18.5*360/20 : 3) node[below right,scale=0.8]{$L_n$};

\draw[densely dotted] (0,3.3) arc (90:108:3.3) ;
\draw[densely dotted] (0,3.3) arc (90:72:3.3) ;

\node at (0,2) {$M_0$} ;

\draw[fill,red] (0,-1) circle (2) ;
\draw (0,-1) circle (2) ;

\node at (0,-1) {$M_1$} ;

\draw (30 : 2) + (0,-1) node[xshift=-0.22cm,scale=0.8]{$\mathcal{L}_{01}$};

\draw (270 : 3) node[scale = 0.1]{\pointbullet};

\end{tikzpicture}
}

\newcommand{\biquilteddiskexample}{
\begin{tikzpicture}[scale = 0.5]

\draw[fill,blue] (0,0) circle (3) ;
\draw (0,0) circle (3) ;

\draw (45 : 3) node[scale=0.1]{\pointbullet};
\draw (90 : 3) node[scale=0.1]{\pointbullet};
\draw (135 : 3) node[scale=0.1]{\pointbullet};

\draw (45 : 3) node[above=3pt,scale=0.7]{$z_3$};
\draw (90 : 3) node[above=3pt,scale=0.7]{$z_2$};
\draw (135 : 3) node[above=3pt,scale=0.7]{$z_1$};
\draw (270 : 3) node[below=3pt,scale=0.7]{$z_0$};

\draw[fill,orange] (0,-0.75) circle (2.25) ;
\draw (0,-0.75) circle (2.25) ;

\draw (30 : 2.25) + (0,-0.75) node[xshift=-0.3cm,scale=0.8]{$C_1$};

\draw[fill,red] (0,-1.5) circle (1.5) ;
\draw (0,-1.5) circle (1.5) ;

\draw (30 : 1.5) + (0,-1.5) node[xshift=-0.3cm,scale=0.8]{$C_2$};

\draw (270 : 3) node[scale = 0.1]{\pointbullet};

\end{tikzpicture}
}

\newcommand{\disquebordlagrangienhomot}{
\begin{tikzpicture}[scale = 0.5]

\draw[fill,blue] (0,0) circle (3) ;
\draw (0,0) circle (3) ;

\draw (360/20 : 3) node[scale = 0.1]{\pointbullet};
\draw (3*360/20 : 3) node[scale = 0.1]{\pointbullet};
\draw (9*360/20 : 3) node[scale = 0.1]{\pointbullet};
\draw (7*360/20 : 3) node[scale = 0.1]{\pointbullet};
\draw (270 : 3) node[scale = 0.1]{\pointbullet};

\draw (360/20 : 3) node[right,scale=0.7]{$x_n$};
\draw (3*360/20 : 3) node[above=3pt,scale=0.7]{$x_{n-1}$};
\draw (9*360/20 : 3) node[left,scale=0.7]{$x_1$};
\draw (7*360/20 : 3) node[above=3pt,scale=0.7]{$x_2$};
\draw (270 : 3) node[below=3pt,scale=0.7]{$y$};

\draw (2*360/20 : 3) node[above right,scale=0.8]{$L_{n-1}$};
\draw (8*360/20 : 3) node[above left,scale=0.8]{$L_1$};
\draw (11.5*360/20 : 3) node[below left,scale=0.8]{$L_0$};
\draw (18.5*360/20 : 3) node[below right,scale=0.8]{$L_n$};

\draw[densely dotted] (0,3.3) arc (90:108:3.3) ;
\draw[densely dotted] (0,3.3) arc (90:72:3.3) ;

\node at (0,2.25) {$M_0$} ;

\draw[fill,orange] (0,-0.75) circle (2.25) ;
\draw (0,-0.75) circle (2.25) ;

\node at (0,0.75) {$M_1$} ;

\draw (30 : 2.25) + (0,-0.75) node[xshift=-0.22cm,scale=0.8]{$\mathcal{L}_{01}$};

\draw[fill,red] (0,-1.5) circle (1.5) ;
\draw (0,-1.5) circle (1.5) ;

\node at (0,-1.5) {$M_2$} ;

\draw (30 : 1.5) + (0,-1.5) node[xshift=-0.26cm,scale=0.8]{$\mathcal{L}_{12}$};

\draw (270 : 3) node[scale = 0.1]{\pointbullet};

\end{tikzpicture}
}

\newcommand{\exemplenecklaceun}{
\begin{tikzpicture}[scale=1.2]
\draw (0,0) node[below,scale=0.5]{0} --  (1,0) node[below,scale=0.5]{1} -- (1,1) node[above,scale=0.5]{2} ;

\node at (1.5 , 0.5) {$\subset$} ;

\begin{scope}[xshift = 2 cm]
\draw[fill,opacity=0.1] (0,0)  --  (1,0)  -- (1,1) -- cycle ;
\draw (0,0) node[below,scale=0.5]{0} --  (1,0) node[below,scale=0.5]{1} -- (1,1) node[above,scale=0.5]{2} -- cycle ;
\end{scope}

\end{tikzpicture} }

\newcommand{\exemplenecklacedeux}{
\begin{tikzpicture}[scale=1.2]

\begin{scope}[xshift = 2 cm , x= {(-0.2cm,-0.353cm)}, z={(0cm,1cm)}, y={(1cm,0cm)}]

\draw[densely dotted] (1,1,0) -- (0,0,0) ;

\draw[fill,opacity = 0.05] (0,0,0) -- (1,0,0) -- (1,1,1) -- cycle ;
\draw[fill,opacity = 0.1] (1,0,0) -- (1,1,0) -- (1,1,1) -- cycle ;

\draw (0,0,0) node[left = 2 pt,scale=0.5]{0} --  (1,0,0) node[below,scale=0.5]{1} ;
\draw (1,0,0) -- (1,1,0) node[below,scale=0.5]{2} -- cycle ;
\draw (1,1,0) -- (1,1,1) node[above,scale=0.5]{3} ;

\draw (1,1,1) -- (1,0,0) ;
\draw (1,1,1) -- (0,0,0) ;

\end{scope}

\begin{scope}[x= {(-0.2cm,-0.353cm)}, z={(0cm,1cm)}, y={(1cm,0cm)}]
\draw (0,0,0) node[left = 2 pt,scale=0.5]{0} --  (1,0,0) node[below,scale=0.5]{1} ;
\draw (1,0,0) -- (1,1,0) node[below,scale=0.5]{2} ;
\draw (1,1,0) -- (1,1,1) node[above,scale=0.5]{3} ;

\node at (1 , 1.5 , 0.5) {$\subset$} ;
\end{scope}

\end{tikzpicture} }

\newcommand{\exemplenecklacetrois}{
\begin{tikzpicture}[scale=1.2]

\begin{scope}[xshift = 2 cm , x= {(-0.2cm,-0.353cm)}, z={(0cm,1cm)}, y={(1cm,0cm)}]

\draw[densely dotted] (1,1,0) -- (0,0,0) ;

\draw[fill,opacity = 0.05] (0,0,0) -- (1,0,0) -- (1,1,1) -- cycle ;
\draw[fill,opacity = 0.1] (1,0,0) -- (1,1,0) -- (1,1,1) -- cycle ;

\draw (0,0,0) node[left = 2 pt,scale=0.5]{0} --  (1,0,0) node[below,scale=0.5]{1} ;
\draw (1,0,0) -- (1,1,0) node[below,scale=0.5]{2} -- cycle ;
\draw (1,1,0) -- (1,1,1) node[above,scale=0.5]{3} ;

\draw (1,1,1) -- (1,0,0) ;
\draw (1,1,1) -- (0,0,0) ;

\end{scope}

\begin{scope}[x= {(-0.2cm,-0.353cm)}, z={(0cm,1cm)}, y={(1cm,0cm)}]
\draw (0,0,0) node[left = 2 pt,scale=0.5]{0} --  (1,0,0) node[below,scale=0.5]{1} ;
\draw (1,0,0) -- (1,1,0) node[below,scale=0.5]{2} ;
\draw (1,1,0) -- (1,1,1) node[above,scale=0.5]{3} ;
\draw[fill, opacity = 0.05] (1,1,0) -- (0,0,0) -- (1,0,0) -- cycle ;
\draw (1,1,0) -- (0,0,0) ;

\node at (1 , 1.5 , 0.5) {$\subset$} ;
\end{scope}

\end{tikzpicture} }

\newcommand{\exemplefillingruleun}{
\begin{tikzpicture}
\draw (0,0)  --  (1,0)  -- (1,1)  ;

\node at (1.5 , 0.5) {$\subsetneq$} ;

\begin{scope}[xshift = 2 cm]
\draw[fill,opacity=0.1] (0,0)  --  (1,0)  -- (1,1) -- cycle ;
\draw (0,0)  --  (1,0)  -- (1,1)  -- cycle ;
\end{scope}

\end{tikzpicture} }

\newcommand{\exemplefillingruledeux}{
\begin{tikzpicture}

\begin{scope}[xshift = 2 cm , x= {(-0.2cm,-0.353cm)}, z={(0cm,1cm)}, y={(1cm,0cm)}]

\draw (0,0,0)  --  (1,0,0)  ;
\draw (1,0,0) -- (1,1,0)  ;
\draw (1,1,0) -- (1,1,1)  ;
\draw[fill, red ,  opacity = 0.2] (1,1,0) -- (0,0,0) -- (1,0,0) -- cycle ;
\draw (1,1,0) -- (0,0,0) ;

\draw[fill, blue ,  opacity = 0.2] (0,0,0) -- (1,1,0) -- (1,1,1) -- cycle ;
\draw (0,0,0) -- (1,1,0) -- (1,1,1) -- cycle ;

\node at (1 , 1.5 , 0.5) {$\subsetneq$} ;

\end{scope}

\begin{scope}[xshift = 4 cm , x= {(-0.2cm,-0.353cm)}, z={(0cm,1cm)}, y={(1cm,0cm)}]

\draw[fill, blue, opacity = 0.2] (0,0,0) -- (1,1,0) -- (1,1,1) -- cycle ;
\draw[densely dotted] (0,0,0) -- (0.15,0.15,0) ;
\draw (0,0,0) -- (1,1,1) -- (1,1,0) -- (0.15,0.15,0);

\draw (0,0,0)  --  (1,0,0)  ;
\draw (1,0,0) -- (1,1,0)  ;
\draw (1,1,0) -- (1,1,1)  ;
\draw[fill, red ,  opacity = 0.2] (1,1,0) -- (0,0,0) -- (1,0,0) -- cycle ;

\draw[fill, yellow, opacity = 0.2] (0,0,0) -- (1,0,0) -- (1,1,1) -- cycle ;
\draw (0,0,0) -- (1,0,0) -- (1,1,1) -- cycle ;

\node at (1 , 1.5 , 0.5) {$\subsetneq$} ;

\end{scope}

\begin{scope}[xshift = 6 cm , x= {(-0.2cm,-0.353cm)}, z={(0cm,1cm)}, y={(1cm,0cm)}]

\draw[fill, red, opacity = 0.2] (0,0,0) -- (1,1,0) -- (1,1,1) -- cycle ;
\draw[densely dotted] (0,0,0) -- (1,1,0) ;
\draw (0,0,0) -- (1,1,1) -- (1,1,0);

\draw (0,0,0)  --  (1,0,0)  ;
\draw (1,0,0) -- (1,1,0)  ;
\draw (1,1,0) -- (1,1,1)  ;
\draw[fill, blue ,  opacity = 0.2] (1,1,0) -- (0,0,0) -- (1,0,0) -- cycle ;

\draw[fill, yellow, opacity = 0.2] (0,0,0) -- (1,0,0) -- (1,1,1) -- cycle ;
\draw (0,0,0) -- (1,0,0) -- (1,1,1) -- cycle ;

\draw[fill, green, opacity = 0.2] (1,0,0) -- (1,1,0) -- (1,1,1) -- cycle ;
\draw (1,0,0) -- (1,1,0) -- (1,1,1) -- cycle ;

\end{scope}

\begin{scope}[x= {(-0.2cm,-0.353cm)}, z={(0cm,1cm)}, y={(1cm,0cm)}]
\draw (0,0,0)  --  (1,0,0)  ;
\draw (1,0,0) -- (1,1,0)  ;
\draw (1,1,0) -- (1,1,1)  ;
\draw[fill, red ,  opacity = 0.2] (1,1,0) -- (0,0,0) -- (1,0,0) -- cycle ;
\draw (1,1,0) -- (0,0,0) ;

\node at (1 , 1.5 , 0.5) {$\subsetneq$} ;
\end{scope}

\end{tikzpicture} }

\newcommand{\fillingruleequivalentun}{
\begin{tikzpicture}[baseline=(pt_base.base)]

\coordinate (pt_base) at (0,0) ;

\begin{scope}[xshift = 2 cm , x= {(-0.2cm,-0.353cm)}, z={(0cm,1cm)}, y={(1cm,0cm)}]

\draw (0,0,0)  --  (1,0,0)  ;
\draw (1,0,0) -- (1,1,0)  ;
\draw (1,1,0) -- (1,1,1)  ;
\draw[fill, red ,  opacity = 0.2] (1,1,0) -- (0,0,0) -- (1,0,0) -- cycle ;
\draw (1,1,0) -- (0,0,0) ;

\draw[fill, blue ,  opacity = 0.2] (0,0,0) -- (1,1,0) -- (1,1,1) -- cycle ;
\draw (0,0,0) -- (1,1,0) -- (1,1,1) -- cycle ;

\node at (1 , 1.5 , 0.5) {$\subsetneq$} ;

\end{scope}

\begin{scope}[xshift = 4 cm , x= {(-0.2cm,-0.353cm)}, z={(0cm,1cm)}, y={(1cm,0cm)}]

\draw[fill, blue, opacity = 0.2] (0,0,0) -- (1,1,0) -- (1,1,1) -- cycle ;
\draw[densely dotted] (1,1,0) -- (0.15,0.15,0) ;
\draw (0.15,0.15,0) -- (0,0,0) -- (1,1,1);

\draw (0,0,0)  --  (1,0,0)  ;
\draw (1,0,0) -- (1,1,0)  ;
\draw (1,1,0) -- (1,1,1)  ;
\draw[fill, red ,  opacity = 0.2] (1,1,0) -- (0,0,0) -- (1,0,0) -- cycle ;

\draw[fill, green, opacity = 0.2] (1,0,0) -- (1,1,0) -- (1,1,1) -- cycle ;
\draw (1,0,0) -- (1,1,0) -- (1,1,1) -- cycle ;

\node at (1 , 1.5 , 0.5) {$\subsetneq$} ;

\end{scope}

\begin{scope}[xshift = 6 cm , x= {(-0.2cm,-0.353cm)}, z={(0cm,1cm)}, y={(1cm,0cm)}]

\draw[fill, red, opacity = 0.2] (0,0,0) -- (1,1,0) -- (1,1,1) -- cycle ;
\draw[densely dotted] (0,0,0) -- (1,1,0) ;
\draw (0,0,0) -- (1,1,1) -- (1,1,0);

\draw (0,0,0)  --  (1,0,0)  ;
\draw (1,0,0) -- (1,1,0)  ;
\draw (1,1,0) -- (1,1,1)  ;
\draw[fill, blue ,  opacity = 0.2] (1,1,0) -- (0,0,0) -- (1,0,0) -- cycle ;

\draw[fill, yellow, opacity = 0.2] (0,0,0) -- (1,0,0) -- (1,1,1) -- cycle ;
\draw (0,0,0) -- (1,0,0) -- (1,1,1) -- cycle ;

\draw[fill, green, opacity = 0.2] (1,0,0) -- (1,1,0) -- (1,1,1) -- cycle ;
\draw (1,0,0) -- (1,1,0) -- (1,1,1) -- cycle ;

\end{scope}

\begin{scope}[x= {(-0.2cm,-0.353cm)}, z={(0cm,1cm)}, y={(1cm,0cm)}]
\draw (0,0,0)  --  (1,0,0)  ;
\draw (1,0,0) -- (1,1,0)  ;
\draw (1,1,0) -- (1,1,1)  ;
\draw[fill, red ,  opacity = 0.2] (1,1,0) -- (0,0,0) -- (1,0,0) -- cycle ;
\draw (1,1,0) -- (0,0,0) ;

\node at (1 , 1.5 , 0.5) {$\subsetneq$} ;
\end{scope}

\end{tikzpicture} }

\newcommand{\fillingruleequivalentdeux}{
\begin{tikzpicture}[baseline=(pt_base.base)]

\coordinate (pt_base) at (0,0) ;

\begin{scope}[xshift = 2 cm , x= {(-0.2cm,-0.353cm)}, z={(0cm,1cm)}, y={(1cm,0cm)}]

\draw (0,0,0)  --  (1,0,0)  ;
\draw (1,0,0) -- (1,1,0)  ;
\draw[fill, red ,  opacity = 0.2] (1,1,0) -- (0,0,0) -- (1,0,0) -- cycle ;

\draw[densely dotted] (1,1,0) -- (0.15,0.15,0) ;
\draw (0.15,0.15,0) -- (0,0,0) ;

\draw[fill, green, opacity = 0.2] (1,0,0) -- (1,1,0) -- (1,1,1) -- cycle ;
\draw (1,0,0) -- (1,1,0) -- (1,1,1) -- cycle ;

\node at (1 , 1.5 , 0.5) {$\subsetneq$} ;

\end{scope}

\begin{scope}[xshift = 4 cm , x= {(-0.2cm,-0.353cm)}, z={(0cm,1cm)}, y={(1cm,0cm)}]

\draw[fill, blue, opacity = 0.2] (0,0,0) -- (1,1,0) -- (1,1,1) -- cycle ;
\draw[densely dotted] (1,1,0) -- (0.15,0.15,0) ;
\draw (0.15,0.15,0) -- (0,0,0) -- (1,1,1);

\draw (0,0,0)  --  (1,0,0)  ;
\draw (1,0,0) -- (1,1,0)  ;
\draw (1,1,0) -- (1,1,1)  ;
\draw[fill, red ,  opacity = 0.2] (1,1,0) -- (0,0,0) -- (1,0,0) -- cycle ;

\draw[fill, green, opacity = 0.2] (1,0,0) -- (1,1,0) -- (1,1,1) -- cycle ;
\draw (1,0,0) -- (1,1,0) -- (1,1,1) -- cycle ;

\node at (1 , 1.5 , 0.5) {$\subsetneq$} ;

\end{scope}

\begin{scope}[xshift = 6 cm , x= {(-0.2cm,-0.353cm)}, z={(0cm,1cm)}, y={(1cm,0cm)}]

\draw[fill, red, opacity = 0.2] (0,0,0) -- (1,1,0) -- (1,1,1) -- cycle ;
\draw[densely dotted] (0,0,0) -- (1,1,0) ;
\draw (0,0,0) -- (1,1,1) -- (1,1,0);

\draw (0,0,0)  --  (1,0,0)  ;
\draw (1,0,0) -- (1,1,0)  ;
\draw (1,1,0) -- (1,1,1)  ;
\draw[fill, blue ,  opacity = 0.2] (1,1,0) -- (0,0,0) -- (1,0,0) -- cycle ;

\draw[fill, yellow, opacity = 0.2] (0,0,0) -- (1,0,0) -- (1,1,1) -- cycle ;
\draw (0,0,0) -- (1,0,0) -- (1,1,1) -- cycle ;

\draw[fill, green, opacity = 0.2] (1,0,0) -- (1,1,0) -- (1,1,1) -- cycle ;
\draw (1,0,0) -- (1,1,0) -- (1,1,1) -- cycle ;

\end{scope}

\begin{scope}[x= {(-0.2cm,-0.353cm)}, z={(0cm,1cm)}, y={(1cm,0cm)}]
\draw (0,0,0)  --  (1,0,0)  ;
\draw (1,0,0) -- (1,1,0)  ;
\draw (1,1,0) -- (1,1,1)  ;
\draw[fill, red ,  opacity = 0.2] (1,1,0) -- (0,0,0) -- (1,0,0) -- cycle ;
\draw (1,1,0) -- (0,0,0) ;

\node at (1 , 1.5 , 0.5) {$\subsetneq$} ;
\end{scope}

\end{tikzpicture} }

\newcommand{\fillingruleequivalenttrois}{
\begin{tikzpicture}[baseline=(pt_base.base)]

\coordinate (pt_base) at (0,0) ;

\begin{scope}[xshift = 2 cm , x= {(-0.2cm,-0.353cm)}, z={(0cm,1cm)}, y={(1cm,0cm)}]

\draw (0,0,0)  --  (1,0,0)  ;
\draw (1,0,0) -- (1,1,0)  ;
\draw (1,1,0) -- (1,1,1)  ;
\draw[fill, red ,  opacity = 0.2] (1,1,0) -- (0,0,0) -- (1,0,0) -- cycle ;
\draw (1,1,0) -- (0,0,0) ;

\draw[fill, blue ,  opacity = 0.2] (0,0,0) -- (1,1,0) -- (1,1,1) -- cycle ;
\draw (0,0,0) -- (1,1,0) -- (1,1,1) -- cycle ;

\node at (1 , 1.5 , 0.5) {$\subsetneq$} ;

\end{scope}

\begin{scope}[xshift = 4 cm , x= {(-0.2cm,-0.353cm)}, z={(0cm,1cm)}, y={(1cm,0cm)}]

\draw[fill, blue, opacity = 0.2] (0,0,0) -- (1,1,0) -- (1,1,1) -- cycle ;
\draw[densely dotted] (0,0,0) -- (0.15,0.15,0) ;
\draw (0,0,0) -- (1,1,1) -- (1,1,0) -- (0.15,0.15,0);

\draw (0,0,0)  --  (1,0,0)  ;
\draw (1,0,0) -- (1,1,0)  ;
\draw (1,1,0) -- (1,1,1)  ;
\draw[fill, red ,  opacity = 0.2] (1,1,0) -- (0,0,0) -- (1,0,0) -- cycle ;

\draw[fill, yellow, opacity = 0.2] (0,0,0) -- (1,0,0) -- (1,1,1) -- cycle ;
\draw (0,0,0) -- (1,0,0) -- (1,1,1) -- cycle ;

\node at (1 , 1.5 , 0.5) {$\subsetneq$} ;

\end{scope}

\begin{scope}[xshift = 6 cm , x= {(-0.2cm,-0.353cm)}, z={(0cm,1cm)}, y={(1cm,0cm)}]

\draw[fill, red, opacity = 0.2] (0,0,0) -- (1,1,0) -- (1,1,1) -- cycle ;
\draw[densely dotted] (0,0,0) -- (1,1,0) ;
\draw (0,0,0) -- (1,1,1) -- (1,1,0);

\draw (0,0,0)  --  (1,0,0)  ;
\draw (1,0,0) -- (1,1,0)  ;
\draw (1,1,0) -- (1,1,1)  ;
\draw[fill, blue ,  opacity = 0.2] (1,1,0) -- (0,0,0) -- (1,0,0) -- cycle ;

\draw[fill, yellow, opacity = 0.2] (0,0,0) -- (1,0,0) -- (1,1,1) -- cycle ;
\draw (0,0,0) -- (1,0,0) -- (1,1,1) -- cycle ;

\draw[fill, green, opacity = 0.2] (1,0,0) -- (1,1,0) -- (1,1,1) -- cycle ;
\draw (1,0,0) -- (1,1,0) -- (1,1,1) -- cycle ;

\end{scope}

\begin{scope}[x= {(-0.2cm,-0.353cm)}, z={(0cm,1cm)}, y={(1cm,0cm)}]
\draw (0,0,0)  --  (1,0,0)  ;
\draw (1,0,0) -- (1,1,0)  ;
\draw (1,1,0) -- (1,1,1)  ;
\draw[fill, red ,  opacity = 0.2] (1,1,0) -- (0,0,0) -- (1,0,0) -- cycle ;
\draw (1,1,0) -- (0,0,0) ;

\node at (1 , 1.5 , 0.5) {$\subsetneq$} ;
\end{scope}

\end{tikzpicture} }

\newcommand{\exampleunderlyingbroken}{
\begin{tikzpicture}[scale=0.2,baseline=(pt_base.base)]

\begin{scope}

\begin{scope}[shift = {(1.25,3)}]
\draw[blue] (-1.25,1.25) -- (-0.5,0.5) ;
\draw[blue] (1.25,1.25) -- (0.5,0.5) ;

\draw[red] (-0.5,0.5) -- (0,0) ;
\draw[red] (0.5,0.5) -- (0,0) ;
\draw[red] (0,-1.25) -- (0,0) ;

\draw (-1.25,0.5) -- (1.25,0.5) ;
\end{scope}

\begin{scope}[red]
\draw (-1.25,1.25) -- (0,0) ;
\draw (1.25,1.25) -- (0,0) ;
\draw (0,-1.25) -- (0,0) ;
\end{scope}

\begin{scope}[shift = {(-1.25,3)}]
\draw[blue] (0,1.25) -- (0,0.5) ;
\draw[red] (0,-1.25) -- (0,0.5) ;

\draw (-1.25,0.5) -- (0.75,0.5) ;
\end{scope}

\end{scope}

\draw[>=stealth,->] (3,1) -- (5,1) ; 

\begin{scope}[shift = {(8,0)}]

\coordinate (pt_base) at (0,0) ;

\begin{scope}
\draw (-1.25,1.25) -- (0,0) ;
\draw (1.25,1.25) -- (0,0) ;
\draw (0,-1.25) -- (0,0) ;
\end{scope}

\begin{scope}[shift = {(1.25,3)}]
\draw (-1.25,1.25) -- (0,0) ;
\draw (1.25,1.25) -- (0,0) ;
\draw (0,-1.25) -- (0,0) ;
\end{scope}

\end{scope}

\end{tikzpicture}}

\newcommand{\transformationgaugevertexA}{
\begin{tikzpicture}[scale=0.4,baseline=(pt_base.base)]

\begin{scope}[shift = {(8,0)}]

\coordinate (pt_base) at (0,0) ;

\draw[blue] (-2,1.25) -- (0,0) ;
\draw[blue] (-1,1.25) -- (0,0) ;
\draw[blue] (1,1.25) -- (0,0) ;
\draw[blue] (2,1.25) -- (0,0) ;
\draw[red] (0,-0.5) -- (0,0) ;

\draw[red] (0,-1.25) -- (0,-0.5) ;

\node[scale = 0.1] at (0,0) {\pointbullet} ;

\draw[blue] [densely dotted] (-0.5,0.9) -- (0.5,0.9) ;

\draw  (-1.5,0) -- (1.5,0) ;

\end{scope}

\draw[>=stealth,->] (3,0) -- (5,0) ; 

\begin{scope}

\coordinate (pt_base) at (0,0) ;

\draw[blue] (-2,1.25) -- (-0.625,0.5) ;
\draw[blue] (-1,1.25) -- (-0.3125,0.5) ;
\draw[blue] (1,1.25) -- (0.3125,0.5) ;
\draw[blue] (2,1.25) -- (0.625,0.5) ;

\draw[red] (-0.625,0.5) -- (0,0) ;
\draw[red] (-0.3125,0.5) -- (0,0) ;
\draw[red] (0.625,0.5) -- (0,0) ;
\draw[red] (0.3125,0.5) -- (0,0) ;
\draw[red] (0,-1.25) -- (0,0) ;

\node[scale = 0.1] at (0,0) {\pointbullet} ;

\draw[blue] [densely dotted] (-0.5,0.9) -- (0.5,0.9) ;

\draw  (-2,0.5) -- (2,0.5) ;

\end{scope}

\end{tikzpicture}}

\newcommand{\transformationgaugevertexB}{
\begin{tikzpicture}[scale=0.4,baseline=(pt_base.base)]

\begin{scope}[shift = {(8,0)}]

\coordinate (pt_base) at (0,0) ;

\draw[blue] (-2,1.25) -- (0,0) ;
\draw[blue] (-1,1.25) -- (0,0) ;
\draw[blue] (1,1.25) -- (0,0) ;
\draw[blue] (2,1.25) -- (0,0) ;
\draw[red] (0,-0.5) -- (0,0) ;

\draw[red] (0,-1.25) -- (0,-0.5) ;

\node[scale = 0.1] at (0,0) {\pointbullet} ;

\draw[blue] [densely dotted] (-0.5,0.9) -- (0.5,0.9) ;

\draw  (-1.5,0) -- (1.5,0) ;

\end{scope}

\draw[>=stealth,->] (3,0) -- (5,0) ; 

\begin{scope}

\coordinate (pt_base) at (0,0) ;

\draw[blue] (-2,1.25) -- (0,0) ;
\draw[blue] (-1,1.25) -- (0,0) ;
\draw[blue] (1,1.25) -- (0,0) ;
\draw[blue] (2,1.25) -- (0,0) ;
\draw[blue] (0,-0.5) -- (0,0) ;

\draw[red] (0,-1.25) -- (0,-0.5) ;

\node[scale = 0.1] at (0,0) {\pointbullet} ;

\draw[blue] [densely dotted] (-0.5,0.9) -- (0.5,0.9) ;

\draw  (-1,-0.5) -- (1,-0.5) ;

\end{scope}

\end{tikzpicture}}

\newcommand{\ordreunbrokengaugedtreeintersectionsA}{
\begin{tikzpicture}[scale=0.4,baseline=(pt_base.base)]
\begin{scope}

\coordinate (pt_base) at (0,0) ;

\draw[blue] (-2,1.25) -- (0,0) ;
\draw[blue] (-1,1.25) -- (0,0) ;
\draw[blue] (1,1.25) -- (0,0) ;
\draw[blue] (2,1.25) -- (0,0) ;
\draw[red] (0,-0.5) -- (0,0) ;

\draw[red] (0,-1.25) -- (0,-0.5) ;

\node[scale = 0.1] at (0,0) {\pointbullet} ;
\node[below right] at (0,0) {\small $v_1$} ;

\draw[blue] [densely dotted] (-0.5,0.9) -- (0.5,0.9) ;

\draw  (-1.5,0) -- (1.5,0) ;

\end{scope}

\draw[densely dotted,black] (3,0) -- (7,0) ; 

\begin{scope}[shift = {(10,0)}]

\coordinate (pt_base) at (0,0) ;

\draw[blue] (-2,1.25) -- (0,0) ;
\draw[blue] (-1,1.25) -- (0,0) ;
\draw[blue] (1,1.25) -- (0,0) ;
\draw[blue] (2,1.25) -- (0,0) ;
\draw[red] (0,-0.5) -- (0,0) ;

\draw[red] (0,-1.25) -- (0,-0.5) ;

\node[scale = 0.1] at (0,0) {\pointbullet} ;
\node[below right] at (0,0) {\small $v_j$} ;

\draw[blue] [densely dotted] (-0.5,0.9) -- (0.5,0.9) ;

\draw  (-1.5,0) -- (1.5,0) ;

\end{scope}

\end{tikzpicture}}

\newcommand{\triangprodunun}{
\begin{tikzpicture}[scale=3]
\node[scale=0.5] (A) at (0,0) {$\left( \begin{array}{c}
         0 \\ 0
  \end{array} \right)$} ;
\node[scale=0.5] (B) at (1,0) {$\left( \begin{array}{c}
         0 \\ 1
  \end{array} \right)$} ;
\node[scale=0.5] (C) at (1,1) {$\left( \begin{array}{c}
         1 \\ 1
  \end{array} \right)$} ;
\node[scale=0.5] (D) at (0,1) {$\left( \begin{array}{c}
         1 \\ 0
  \end{array} \right)$} ;
\node[scale=0.5] (E) at (0.5,0.5) {$\left( \begin{array}{cc}
         0&1 \\ 0 & 1
  \end{array} \right)$} ;
\node[scale=0.5] at (0.25,0.75) {$\left( \begin{array}{ccc}
         0&1&1 \\ 0&0&1
  \end{array} \right)$} ;
\node[scale=0.5] at (0.75,0.25) {$\left( \begin{array}{ccc}
         0&0&1 \\ 0&1&1
  \end{array} \right)$} ;

\draw (A) -- node[midway,below,scale=0.5] {$\left( \begin{array}{cc}
         0&0 \\ 0 & 1
  \end{array} \right)$} (B)  ;
  
\draw (B) -- node[midway,right,scale=0.5] {$\left( \begin{array}{cc}
         0&1 \\ 1 & 1
  \end{array} \right)$} (C)  ;  
  
\draw (C) -- node[midway,above,scale=0.5] {$\left( \begin{array}{cc}
         1&1 \\ 0 & 1
  \end{array} \right)$} (D)  ; 
  
\draw (D) -- node[midway,left,scale=0.5] {$\left( \begin{array}{cc}
         0&1 \\ 0 & 0
  \end{array} \right)$} (A)  ; 
\draw (A) -- (E) -- (C) ; 
\end{tikzpicture}}

\newcommand{\triangprodundeux}{
\begin{tikzpicture}[scale=2.5,x= {(-0.2cm,-0.353cm)}, z={(0cm,1cm)}, y={(1cm,0cm)}]

\draw (0,1,0) -- (1,1,0) -- (0,0,0) ;
\draw[densely dotted] (0,0,0) -- (0,1,0) ;

\draw (0,0,1) -- (0,1,1) -- (1,1,1) -- cycle ;

\draw (0,0,0) -- (0,0,1) ;
\draw (0,1,0) -- (0,1,1) ;
\draw (1,1,0) -- (1,1,1) ;

\draw[fill,red,opacity = 0.1] (0,0,0) -- (0,1,1) -- (1,1,1) -- cycle ;
\draw[fill,blue,opacity = 0.2] (0,1,0) -- (1,1,1) -- (0,0,0) -- cycle ;

\draw[densely dotted] (0,0,0) -- (0,1,1) ;
\draw (1,1,1) -- (0,0,0) ;
\draw (0,1,0) -- (1,1,1) ;
\end{tikzpicture}}

\maketitle

\begin{abstract}
This paper introduces the notion of $n$-morphisms between two \Ainf -algebras, such that 0-morphisms correspond to standard \Ainf -morphisms and 1-morphisms correspond to \Ainf -homotopies between \Ainf -morphisms. The set of higher morphisms between two \Ainf -algebras then defines a simplicial set which has the property of being a Kan complex, whose simplicial homotopy groups can be explicitly computed. The operadic structure of $n-\Ainf$-morphisms is also encoded by new families of polytopes, which we call the $n$-multiplihedra and which generalize the standard multiplihedra. These are constructed from the standard simplices and multiplihedra by lifting the Alexander-Whitney map to the level of simplices. Rich combinatorics arise in this context, as conveniently described in terms of overlapping partitions. Shifting from the \Ainf\ to the \ombas\ framework, we define the analogous notion of $n$-morphisms between \ombas -algebras, which are again encoded by the $n$-multiplihedra, endowed with a refined cell decomposition by stable gauged ribbon tree type. We then realize this higher algebra of \Ainf\ and \ombas -algebras in Morse theory. Given two Morse functions $f$ and $g$, we construct $n-\ombas$-morphisms between their respective Morse cochain complexes endowed with their \ombas -algebra structures, by counting perturbed Morse gradient trees associated to an admissible simplex of perturbation data. We moreover show that the simplicial set consisting of higher morphisms defined by a count of perturbed Morse gradient trees is a contractible Kan complex. 
\end{abstract}

\vspace{45pt}

\begin{figure}[h]
\includegraphics[scale=0.115]{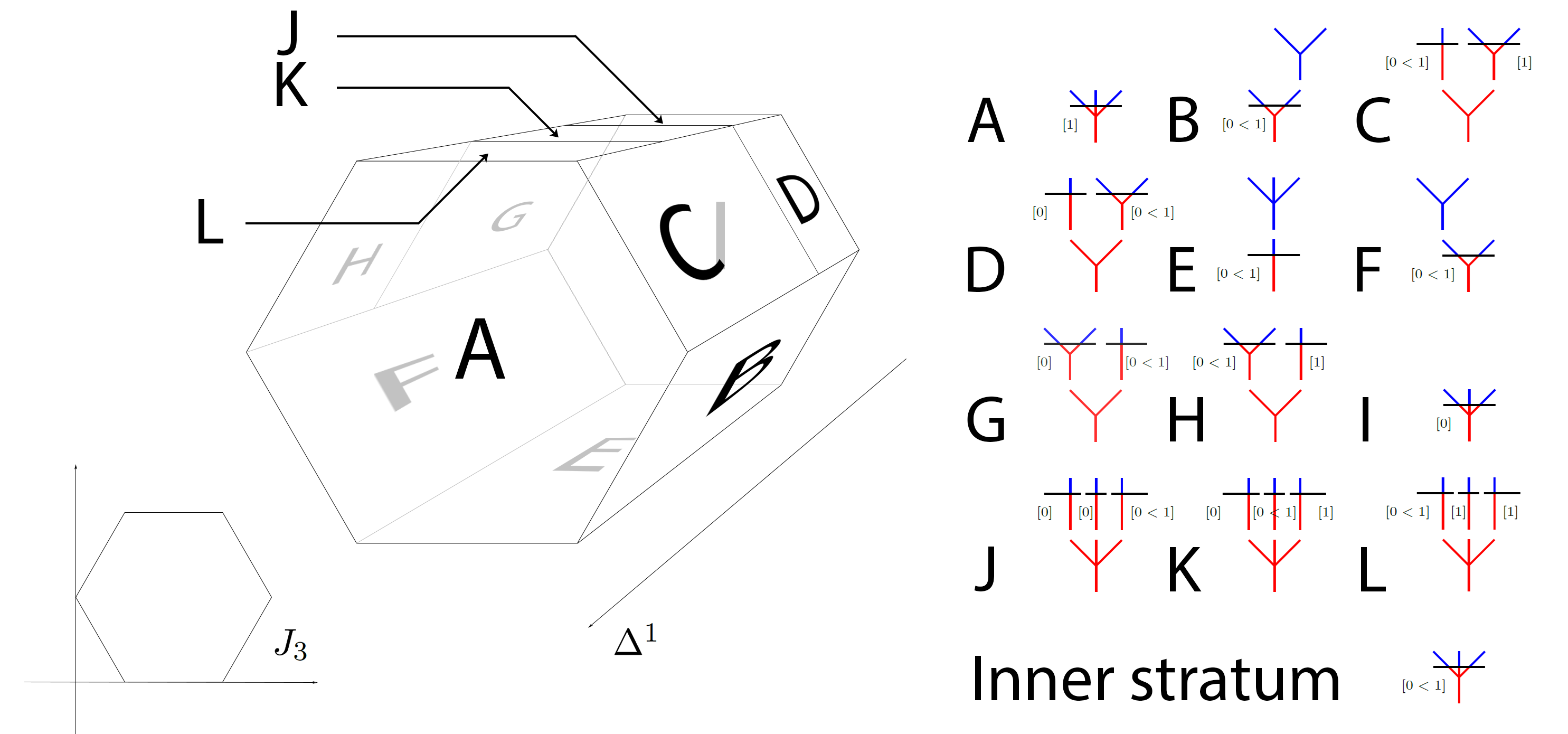}
\caption*{\textit{The $1$-multiplihedron $\Delta^1 \times J_3$ ...}}
\end{figure}
\newpage

\setcounter{tocdepth}{1}

\tableofcontents

\vspace{45pt}

\begin{figure}[h]
\includegraphics[scale=0.115]{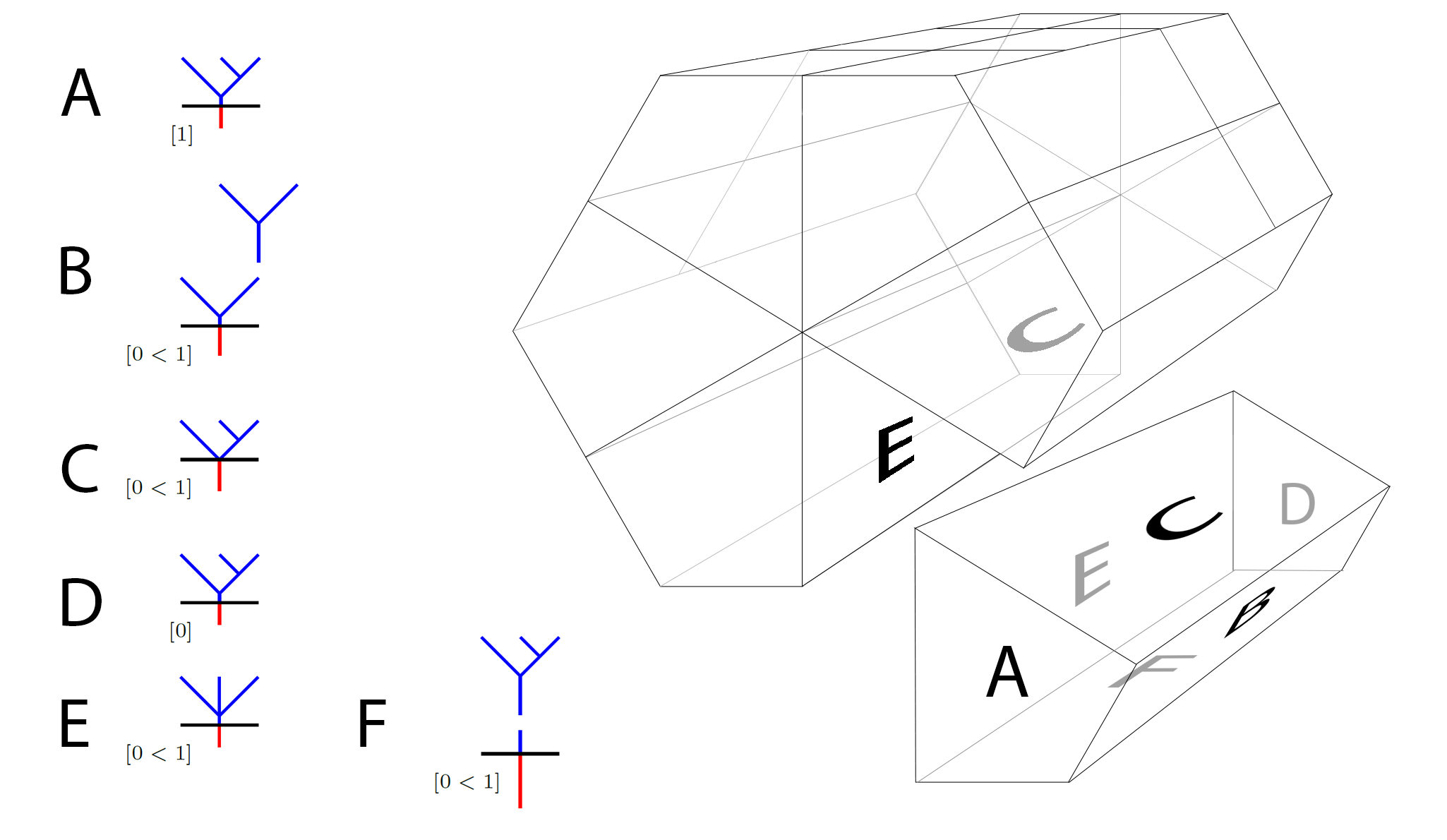}
\caption*{\textit{... and its $\Omega B As$-cell decomposition}}
\end{figure}

\newpage

\begin{leftbar}
\part*{Introduction}
\end{leftbar}

\paragraph{\textit{\textbf{Summary and results of article I}}}

This article is the direct sequel to~\cite{mazuir-I}. We thus begin by summarizing our first article, after which we outline the main results and constructions carried out in the present paper.

The structure of strong homotopy associative algebra, or equivalently \Ainf -algebra, was introduced in the seminal paper of Stasheff~\cite{stasheff-homotopy}. It provides an operadic model for the notion of differential graded algebra whose product is associative up to homotopy. It is defined as the datum of a set of operations $\{ m_m : A^{\otimes m} \rightarrow A \}_{m \geqslant 2}$ of degree $2-m$ on a dg-\Z -module $(A,\partial)$, which satisfy the sequence of equations
\[ \left[ \partial , m_m \right] = \sum_{\substack{i_1+i_2+i_3=m \\ 2 \leqslant i_2 \leqslant m-1}} \pm m_{i_1+1+i_3} (\ide^{\otimes i_1} \otimes m_{i_2} \otimes \ide^{\otimes i_3} ) . \]
The first two equations respectively ensure that $m_2$ is compatible with $\partial$ and that it is associative up to the homotopy $m_3$. This algebraic structure is encoded by an operad in dg-\Z -modules, called the operad \Ainf . As shown in~\cite{masuda-diagonal-assoc}, this operad stems in fact from an operad in the category of polytopes, whose arity $m$ space of operations is defined to be the $(m-2)$-dimensional associahedron~$K_m$.

Similarly, the notion of \Ainf -morphism between two \Ainf -algebras $A$ and $B$ offers an operadic model for the notion of morphism of strong homotopy associative algebras which preserves the product up to homotopy. It is defined as the datum of a set of operations $\{ f_m : A^{\otimes m} \rightarrow B \}_{m \geqslant 1}$ of degree $1-m$ which satisfy the sequence of equations 
\[ \left[ \partial , f_m \right] = \sum_{\substack{i_1+i_2+i_3=m \\ i_2 \geqslant 2}} \pm f_{i_1+1+i_3} (\ide^{\otimes i_1} \otimes m_{i_2} \otimes \ide^{\otimes i_3}) + \sum_{\substack{i_1 + \cdots + i_s = m \\ s \geqslant 2 }} \pm m_s ( f_{i_1} \otimes \cdots \otimes f_{i_s}) \ . \]
The first two equations show this time that $f_1$ commutes with the differentials and that it preserves the product up to the homotopy $f_2$. From the point of view of operadic algebra, \Ainf -morphisms are encoded by an operadic bimodule in dg-\Z -modules : the operadic bimodule \infmor . It occurs from an operadic bimodule in polytopes, whose arity $m$ space of operations is the $(m-1)$-dimensional multiplihedron $J_m$ as shown in~\cite{masuda-diagonal-multipl}.

\Ainf -algebras and \Ainf -morphisms between them provide a satisfactory framework for homotopy theory. The most famous instance of this statement is the homotopy transfer theorem : given $(A,\partial_A)$ and $(H,\partial_H)$ two cochain complexes and a homotopy retract diagram
\begin{center} 
\begin{tikzcd}
\arrow[loop left,"h"](A,d_A) \arrow[r, shift left, "p"] & \arrow[l,shift left,"i"] (H,d_H) \ ,
\end{tikzcd}
\end{center}
if $(A,\partial_A)$ is endowed with an \Ainf -algebra structure, then $H$ can be made into an \Ainf -algebra such that $i$ and $p$ extend to \Ainf -morphisms. See also~\cite{vallette-homotopy}~and~\cite{lefevre-hasegawa} for an extensive study on the homotopy theory of \Ainf -algebras.

The associahedra and multiplihedra, respectively encoding the operad \Ainf\ and the operadic bimodule \infmor , can in fact be both realized as moduli spaces of metric trees. The associahedron $K_m$ is isomorphic as a CW-complex to the compactified moduli space of stable metric ribbon trees $\overline{\mathcal{T}}_m$ as first pointed out in~\cite{boardman-vogt}. The multiplihedron $J_m$ is isomorphic as a CW-complex to the compactified moduli space of stable gauged metric ribbon trees $\overline{\mathcal{CT}}_m$ as shown in~\cite{forcey-multipl}~and~\cite{mau-woodward}. These moduli spaces come in fact with refined cell decompositions, called their \ombas -cell decompositions~: the cell decomposition by stable ribbon tree type for $\overline{\mathcal{T}}_m$, and the cell decomposition by stable gauged ribbon tree type for $\overline{\mathcal{CT}}_m$. These refined decompositions provide another operadic model for strong homotopy associative algebras with morphisms preserving the product up to homotopy between them : the standard operad \ombas\ and the operadic bimodule \ombasm\ introduced in~\cite{mazuir-I}. We show moreover in~\cite{mazuir-I} that one can naturally shift from the \ombas\ to the \Ainf\ framework via a geometric morphism of operads $\Ainf \rightarrow \ombas$ and a geometric morphism of operadic bimodules $\infmor \rightarrow \ombasm$. 

Consider now a Morse function $f$ on a closed oriented Riemannian manifold $M$ together with a Morse-Smale metric. Following~\cite{hutchings-floer}, the Morse cochain complex $C^*(f)$ is a homotopy retract of the singular cochain complex $C^*_{sing}(M)$ which is a dg-algebra with respect to the standard cup product. The dg-algebra structure on $C^*_{sing}(M)$ can thus be transferred to an \Ainf -algebra structure on $C^*(f)$ using the homotopy transfer theorem. We show in~\cite{mazuir-I} that one can in fact directly define an \ombas -algebra structure on the Morse cochains $C^*(f)$ by realizing the moduli spaces of stable metric ribbon trees $\mathcal{T}_m$ in Morse theory. Given a choice of perturbation data $\{ \mathbb{X}_{m} \}_{m \geqslant 2}$ on the moduli spaces $\mathcal{T}_m$ as introduced by Abouzaid in~\cite{abouzaid-plumbings} and further studied by Mescher in~\cite{mescher-morse}, we define the moduli spaces of perturbed Morse gradient trees modeled on a stable ribbon tree type $t$ and connecting the critical points $x_1 , \dots , x_m \in \mathrm{Crit}(f)$ to the critical point $y \in  \mathrm{Crit}(f)$, denoted $\mathcal{T}_t^{\mathbb{X}} (y ; x_1, \dots , x_m)$.
We prove in~\cite{mazuir-I} that under generic assumptions on the choice of perturbation data, these moduli spaces are in fact orientable manifolds of finite dimension. If they have dimension 1, they can moreover be compactified to 1-dimensional manifolds with boundary, whose boundary is modeled on the top dimensional strata in the boudary of the compactified moduli space $\overline{\mathcal{T}}_m$. The \ombas -algebra structure on the Morse cochains $C^*(f)$ is finally defined by counting the points of the 0-dimensional moduli spaces $\mathcal{T}_t^{\mathbb{X}} (y ; x_1, \dots , x_m)$. The induced geometric \Ainf -algebra structure on $C^*(f)$ is then quasi-isomorphic to the \Ainf -algebra structure on $C^*(f)$ given by the homotopy transfer theorem.

Consider now two Morse functions $f$ and $g$ on $M$ together with generic choices of perturbation data $\mathbb{X}_f$ and $\mathbb{X}_g$. Endow the Morse cochains $C^*(f)$ and $C^*(g)$ with their associated \ombas -algebra structures. We prove in~\cite{mazuir-I} that one can adapt the construction of the previous paragraph, to define an \ombas -morphism from the \ombas -algebra $C^*(f)$ to the \ombas -algebra $C^*(g)$. We count this time 0-dimensional moduli spaces of perturbed Morse stable gauged trees modeled on a stable gauged ribbon tree type $t_g$ and connecting the critical points $x_1 , \dots , x_m \in \mathrm{Crit}(f)$ to the critical point $y \in  \mathrm{Crit}(g)$, denoted $\mathcal{CT}_{t_g}^{\mathbb{Y}} (y ; x_1, \dots , x_m)$,
after making a generic choice of perturbation data $\mathbb{Y}$ on the moduli spaces $\mathcal{CT}_m$.

\paragraph{\textit{\textbf{Motivational question}}} Let $\mathbb{Y}$ and $\mathbb{Y}'$ be two admissible choices of perturbations data on the moduli spaces $\mathcal{CT}_m$. Writing $\mu^{\mathbb{Y}}$ resp. $\mu^{\mathbb{Y}'}$ for the \ombas -morphisms they define, the question which motivates this paper is to know whether $\mu^{\mathbb{Y}}$ and $\mu^{\mathbb{Y}'}$ are always homotopic or not
\[ \begin{tikzcd}[row sep=large, column sep = large]
C^*(f) \arrow[bend left=20]{r}[above]{\mu^{\mathbb{Y}}}[name=U,below,pos=0.5]{}
\arrow[bend right=20]{r}[below]{\mu^{\mathbb{Y}'}}[name=D,pos=0.5]{}
& C^*(g)
\arrow[Rightarrow, from=U, to=D]
\end{tikzcd} \ . \]
In particular, one needs to determine what is the correct notion of a homotopy between two \ombas -morphisms.

\paragraph{\textit{\textbf{Outline of the present paper and main results}}}

The first step towards answering this problem is carried out on the algebraic side in part~\ref{p:algebra}, where we define the notion of $n$-morphisms between \Ainf -algebras and $n$-morphisms between \ombas -algebras. In section~\ref{alg:s:n-ainf-morph}, we recall at first the suspended bar construction point of view on \Ainf -algebras and the definition of an \Ainf -homotopy between \Ainf -morphisms from~\cite{lefevre-hasegawa}. After introducing the cosimplicial dg-coalgebra $\pmb{\Delta}^n$ together with the language of overlapping partitions, we can finally define a $n$-morphism between two \Ainf -algebras $A$ and~$B$ :

\theoremstyle{definition}
\newtheorem*{alg:def:n-morph-ainf-susp}{Definition~\ref{alg:def:n-morph-ainf-susp}}
\begin{alg:def:n-morph-ainf-susp}
Let $A$ and $B$ be two \Ainf -algebras. A \emph{$n$-morphism} from $A$ to $B$ is defined to be a morphism of dg-coalgebras 
\[ F : \pmb{\Delta}^n \otimes \overline{T}(sA) \longrightarrow \overline{T}(sB) \ , \]
where $\overline{T}(sA)$ denotes the suspended bar construction of $A$ (see subsection~\ref{alg:ss:recoll-def}).
\end{alg:def:n-morph-ainf-susp}

\noindent Using the universal property of the bar construction, this definition is equivalent to the following one in terms of operations :

\theoremstyle{definition}
\newtheorem*{alg:def:n-morph-ainf}{Definition~\ref{alg:def:n-morph-ainf}}
\begin{alg:def:n-morph-ainf}
Let $A$ and $B$ be two \Ainf -algebras. A \emph{$n$-morphism} from $A$ to $B$ is defined to be a collection of maps $f_{I}^{(m)} : A^{\otimes m} \longrightarrow B$ of degree $1 - m - \mathrm{dim} (I)$ for $I \subset \Delta^n$ and $m \geqslant 1$, that satisfy 
\[ \resizebox{\hsize}{!}{$\displaystyle{
\left[ \partial , f^{(m)}_I \right] =  \sum_{j=0}^{\mathrm{dim}(I)} (-1)^j f^{(m)}_{\partial_jI} + (-1)^{|I|} \sum_{\substack{i_1+i_2+i_3=m \\ i_2 \geqslant 2}} \pm f^{(i_1+1+i_3)}_I (\ide^{\otimes i_1} \otimes m_{i_2} \otimes \ide^{\otimes i_3}) + \sum_{\substack{i_1 + \cdots + i_s = m \\ I_1 \cup \cdots \cup I_s = I \\ s \geqslant 2 }} \pm m_s ( f^{(i_1)}_{I_1} \otimes \cdots \otimes f^{(i_s)}_{I_s}) \ .  }$} \]
\end{alg:def:n-morph-ainf}

\noindent We show in Proposition~\ref{alg:prop:equivalent-def} that the datum of a $n$-morphism is also equivalent to the datum of a morphism of \Ainf -algebras $A \rightarrow \pmb{\Delta}_n \otimes B$, where $\pmb{\Delta}_n$ is the dg-algebra dual to the dg-coalgebra $\pmb{\Delta}^n$. While the operad \Ainf\ stems from the associahedra $K_m$ and the operadic bimodule \infmor\ stems from the multiplihedra $J_m$, we introduce in section~\ref{alg:s:polytopes} a family of polytopes encoding the \Ainf -equations for $n$-morphisms~: \emph{the $n$-multiplihedra $n-J_m$}. In this regard, we begin by introducing a lift of the Alexander-Whitney coproduct $\mathrm{AW}$ at the level of the polytopes $\Delta^n$, following~\cite{masuda-diagonal-assoc}. The map $\mathrm{AW}^{\circ s} := (\ide^{\times (s-1)} \times \mathrm{AW}) \circ \cdots \circ (\ide \times \mathrm{AW}) \circ \mathrm{AW}$ then induces a refined polytopal subdivision of $\Delta^n$, whose top dimensional cells can be labeled by all overlapping $(s+1)$-partitions of $\Delta^n$. 
After introducing the maps $\mathrm{AW}_{\pmb{a}}$, which generalize the maps $\mathrm{AW}^{\circ s}$ and still induce the previous subdivisions on the simplices $\Delta^n$, we construct a refined polytopal subdivision of the polytopes $\Delta^n \times J_m$ :

\theoremstyle{definition}
\newtheorem*{alg:def:n-multipl}{Definition~\ref{alg:def:n-multipl}}
\begin{alg:def:n-multipl}
The polytopes $\Delta^n \times J_m$ endowed with the polytopal subdivisions induced by the maps $\mathrm{AW}_{\pmb{a}}$ will be called the \emph{$n$-multiplihedra} and denoted $n-J_m$.
\end{alg:def:n-multipl}

\noindent The boundaries of the $n$-multiplihedra $n-J_m$ yield the $n-\Ainf$-equations : 

\theoremstyle{theorem}
\newtheorem*{alg:prop:n-multipl}{Proposition~\ref{alg:prop:n-multipl}}
\begin{alg:prop:n-multipl}
The boundary of the top dimensional cell $[n-J_m]$ of the $n$-multiplihedron $n-J_m$ is given by 
\[ \resizebox{\hsize}{!}{$\displaystyle{\partial^{sing} [n-J_m] \cup \bigcup_{\substack{h+k=m+1 \\ 1 \leqslant i \leqslant k \\ h \geqslant 2}} [n-J_k] \times_i [K_h] \cup \bigcup_{\substack{ i_1 + \dots + i_s = m \\ I_1 \cup \dots \cup I_s = \Delta^n \\ s \geqslant 2}} [K_s] \times [\mathrm{dim}(I_1) - J_{i_1}] \times \cdots \times [\mathrm{dim}(I_s) - J_{i_s}] \ ,}$} \]
where $I_1 \cup \dots \cup I_s = \Delta^n$ is an overlapping partition of $\Delta^n$. In other words, the $n$-multiplihedra encode the \Ainf -equations for $n$-morphisms.
\end{alg:prop:n-multipl}

\noindent We then show in section~\ref{alg:s:n-ombas-morph} that these constructions can be transported from the \Ainf\ to the \ombas\ realm. We define $n$-morphisms between \ombas -algebras as follows~:

\theoremstyle{definition}
\newtheorem*{alg:def:op-bimod-ombas-n}{Definition~\ref{alg:def:op-bimod-ombas-n}}
\begin{alg:def:op-bimod-ombas-n} 
\emph{$n-\Omega B As$-morphisms} are the higher morphisms between $\Omega B As$-algebras encoded by the quasi-free operadic bimodule generated by all pairs (face $I \subset \Delta^n$ , two-colored stable ribbon tree), 
\[ n-\Omega B As - \mathrm{Morph} := \mathcal{F}^{\Omega B As, \Omega B As}( \arbreopunmorphn , \arbrebicoloreLn , \arbrebicoloreMn , \arbrebicoloreNn , \cdots , (I , SCRT_n) ,\cdots ; I \subset \Delta^n) \ .  \]
An operation $t_{I,g} := (I,t_g)$, whose underlying stable ribbon tree $t$ has $e(t)$ inner edges, and such that its gauge crosses $j$ vertices of $t$, is defined to have degree $| t_{I,g} | := j - 1 - e(t) - \mathrm{dim}(I) = |I| + |t_g|$.
The differential of $t_{I,g}$ is given by the rule prescribed by the top dimensional strata in the boundary of $\overline{\mathcal{CT}}_m(t_g)$ combined with the algebraic combinatorics of overlapping partitions, added to the simplicial differential of $I$, i.e. 
\[ \partial t_{I,g} = t_{\partial^{sing}I,g} + \pm (\partial^{\overline{\mathcal{CT}}_m} t_g)_{I} \ . \]
\end{alg:def:op-bimod-ombas-n}

\noindent We show that the $n-\ombas$-equations are also encoded by the $n$-multiplihedra, endowed this time with a refined cell decomposition taking the \ombas -decomposition of the multiplihedra $J_m$ into account.
What's more, a $n$-morphism between \ombas -algebras naturally yields a $n$-morphism between \Ainf -algebras :

\theoremstyle{theorem}
\newtheorem*{alg:prop:morph-op-bimod}{Proposition~\ref{alg:prop:morph-op-bimod}}
\begin{alg:prop:morph-op-bimod} 
There exists a morphism of $(\Ainf , \Ainf )$-operadic bimodules
$\infmorn \rightarrow n-\Omega B As - \mathrm{Morph}$.
\end{alg:prop:morph-op-bimod}

\noindent Using the same tools as in~\cite{mazuir-I}, we finally unravel all sign conventions in section~\ref{alg:s:signs-n-morph}.

In part~\ref{p:simplicial}, we study the simplicial set $\mathrm{HOM}_{\mathsf{\Ainf -Alg}}(A,B)_{\bullet}$ of higher morphisms from $A$ to $B$, whose $n$-simplices are the $n$-morphisms from $A$ to $B$. We recall at first basic results on $\infty$-categories and Kan complexes, which are simplicial sets having the left-lifting property with respect to the inner horn inclusions resp. to all horn inclusions $\mathsf{\Lambda}^k_n \subset \Delta^n$. We also introduce the convenient setting of cosimplicial resolutions in model categories, following~\cite{hirschhorn}. We can then prove the following theorem in section~\ref{alg:s:hom-are-inf-gr}~:

\theoremstyle{theorem}
\newtheorem*{alg:th:infinity-gr}{Theorem~\ref{alg:th:infinity-gr}}
\begin{alg:th:infinity-gr}
For $A$ and $B$ two \Ainf -algebras, the simplicial set $\mathrm{HOM}_{\Ainf}(A,B)_\bullet$ is a Kan complex.
\end{alg:th:infinity-gr}

\noindent This Kan complex is in particular an algebraic $\infty$-category as explained in Proposition~\ref{alg:prop:algebraic-infty}. Fix now $F: A \rightarrow B$ an \Ainf -morphism, i.e. a point of the simplicial set $\mathrm{HOM}_{\Ainf}(A,B)_\bullet$. We proceed to compute the simplicial homotopy groups with basepoint $F$ of this Kan complex in subsection~\ref{alg:ss:homotopy-groups-HOM}~: 

\theoremstyle{theorem}
\newtheorem*{alg:th:simpl-hom-groups}{Theorem~\ref{alg:th:simpl-hom-groups}}
\begin{alg:th:simpl-hom-groups}
\begin{enumerate}[label=(\roman*)]
\item For $n \geqslant 1$, the set $\pi_n \left( \mathrm{HOM}_{\mathsf{\Ainf -Alg}}(A,B)_{\bullet} , F \right)$ consists of the equivalence classes of collections of degree $-n$ maps $F_{\Delta^n}^{(m)} : (sA)^{\otimes m} \rightarrow sB$ satisfying the following equations 
\begin{align*} 
&(-1)^{n} \sum_{i_1+i_2+i_3=m} F^{(i_1+1+i_3)}_{\Delta^n} \left( \ide^{\otimes i_1} \otimes b_{i_2} \otimes \ide^{\otimes i_3} \right) \\
= &\sum_{\substack{i_1 + \cdots + i_s  + l \\ + j_1 + \cdots + j_t = m}} b_{s+1+t} \left( F^{(i_1)} \otimes \cdots \otimes F^{(i_s)} \otimes F_{\Delta^n}^{(l)} \otimes F^{(j_1)} \otimes \cdots \otimes F^{(j_t)} \right) \ , 
\end{align*}
where two such collections of maps $( F_{\Delta^n}^{(m)} )^{m \geqslant 1}$ and $( G_{\Delta^n}^{(m)} )^{m \geqslant 1}$ are equivalent if and only if there exists a collection of degree $-(n+1)$ maps $H^{(m)} : (sA)^{\otimes m} \rightarrow sB$ such that
\begin{align*}
&G_{\Delta^n}^{(m)} - F_{\Delta^n}^{(m)} + (-1)^{n+1} \sum_{i_1+i_2+i_3=m} H^{(i_1+1+i_3)} (\ide^{\otimes i_1} \otimes b_{i_2} \otimes \ide^{\otimes i_3}) \\
= &\sum_{\substack{i_1 + \cdots + i_s  + l \\ + j_1 + \cdots + j_t = m}} b_{s+1+t} ( F^{(i_1)} \otimes \cdots \otimes F^{(i_s)} \otimes H^{(l)} \otimes F^{(j_1)} \otimes \cdots \otimes F^{(j_t)} ) \ .
\end{align*}
\item If $n = 1$, given two such collection of maps $( F_{\Delta^1}^{(m)} )^{m \geqslant 1}$ and $( G_{\Delta^1}^{(m)} )^{m \geqslant 1}$, the composition law on $\pi_1 \left( \mathrm{HOM}_{\mathsf{\Ainf -Alg}}(A,B)_{\bullet} , F \right)$ is given by the formula
\[ \resizebox{\hsize}{!}{\begin{math} \begin{aligned}
G_{\Delta^1}^{(m)} + F_{\Delta^1}^{(m)} - \sum_{\substack{i_1 + \cdots + i_s  + l_1 \\ + j_1 + \cdots + j_t + l_2 \\ + k_1 + \cdots + k_u = m}} b_{s+t+u+2} ( F^{(i_1)} \otimes \cdots \otimes F^{(i_s)} \otimes F_{\Delta^1}^{(l_1)} \otimes F^{(j_1)} \otimes \cdots \otimes F^{(j_t)}  \otimes G_{\Delta^1}^{(l_2)} \otimes F^{(k_1)} \otimes \cdots \otimes F^{(k_u)}) \ .
\end{aligned}
\end{math}} \]
\item If $n \geqslant 2$, given two such collection of maps $( F_{\Delta^n}^{(m)} )^{m \geqslant 1}$ and $( G_{\Delta^n}^{(m)} )^{m \geqslant 1}$, the composition law on $\pi_n \left( \mathrm{HOM}_{\mathsf{\Ainf -Alg}}(A,B)_{\bullet} , F \right)$ is given by the formula
\[ G_{\Delta^n}^{(m)} + F_{\Delta^n}^{(m)} \ . \]
\end{enumerate}
\end{alg:th:simpl-hom-groups}

\noindent In section~\ref{alg:s:a-inf-ainf-a-b}, we begin by generalizing the notion of a $n$-morphism between \Ainf -algebras to that of a $n$-functor between \Ainf -categories. We define the simplicial set $\mathrm{HOM}_{\mathsf{\Ainf -Cat}}(\mathcal{A},\mathcal{B})_{\bullet}$ of higher functors between two \Ainf -categories, which we expect to also be a Kan complex. We then recall the definition of the \Ainf -category of \Ainf -functors $\mathrm{Func}_{\mathcal{A},\mathcal{B}}$ of \cite{fukaya-floer}, as well as the simplicial nerve functor $N_{\Ainf}$ of \cite{faonte-simplicial}. These constructions yield a new simplicial set $N_{\Ainf}(\mathrm{Func}_{\mathcal{A},\mathcal{B}})$ which has the property of being an $\infty$-category. 
Although the simplicial sets $\mathrm{HOM}_{\mathsf{\Ainf -Cat}}(\mathcal{A},\mathcal{B})_{\bullet}$ and $N_{\Ainf}(\mathrm{Func}_{\mathcal{A},\mathcal{B}})$ bear many similarities, they actually differ fundamentally : while the simplices of $\mathrm{HOM}_{\mathsf{\Ainf -Cat}}(\mathcal{A},\mathcal{B})_{\bullet}$ correspond to higher homotopies between \Ainf -functors, the simplices of $N_{\Ainf}(\mathrm{Func}_{\mathcal{A},\mathcal{B}})$ correspond to higher natural transformations between \Ainf -functors $\mathcal{A} \rightarrow \mathcal{B}$. Heuristically, the simplicial set $N_{\Ainf}(\mathrm{Func}_{\mathcal{A},\mathcal{B}})$ has thereby no reason to be a Kan complex, as homotopies are reversible whether functors are not. Nevertheless, the Kan complex $\mathrm{HOM}_{\mathsf{\Ainf -Cat}}(\mathcal{A},\mathcal{B})_{\bullet}$ and the \Ainf -category $\mathrm{Func}_{\mathcal{A},\mathcal{B}}$ each define a notion of homotopy between \Ainf -functors, that we compare when the \Ainf -category $\mathcal{B}$ is unital by recalling a proposition of \cite{fukaya-unobstructed}. In section~\ref{fd:s:simplic-enrich}, we finally explore two approaches to lift the composition of \Ainf -morphisms to a composition between $n-\Ainf$-morphisms. We fall however short of defining a natural simplicial enrichment of the category $\mathsf{\Ainf -Alg}$. We also discuss the results of Faonte, Lyubashenko, Fukaya and Bottman concerning a statement of a similar nature involving the \Ainf -categories $\mathrm{Func}_{\mathcal{A},\mathcal{B}}$.

In part~\ref{p:geo} we illustrate how $n$-morphisms naturally arise in geometry, here in the context of Morse theory, solving our motivational question at the same time. In section~\ref{geo:s:n-morph-morse} we detail the construction of $n$-morphisms between \ombas -algebras in Morse theory. Given two Morse functions $f$ and $g$ on a closed oriented manifold $M$, endow their Morse cochains with their \ombas -algebra structure coming from a choice of perturbation data on the moduli spaces $\mathcal{T}_m$. A $n$-morphism between $C^*(f)$ and $C^*(g)$ can be constructed by adapting the techniques of~\cite{abouzaid-plumbings} and~\cite{mescher-morse} that we used in~\cite{mazuir-I} for moduli spaces of perturbed Morse gradient trees. We define to this extent the notion of $n$-simplices of perturbation data $\mathbb{Y}_{\Delta^n}$ :

\theoremstyle{definition}
\newtheorem*{geo:def:n-simpl-perturb-data}{Definition~\ref{geo:def:n-simpl-perturb-data}}
\begin{geo:def:n-simpl-perturb-data}
A \emph{$n$-simplex of perturbation data} for a gauged metric stable ribbon tree $T_g$ is defined to be a choice of perturbation data $\mathbb{Y}_{\delta,T_g}$ for $T_g$ for every $\delta \in \mathring{\Delta}^n$.
\end{geo:def:n-simpl-perturb-data}

\noindent Given a smooth $n$-simplex of perturbation data $\mathbb{Y}_{\Delta^n,t_g}$ on the moduli space $\mathcal{CT}_m(t_g)$, we introduce the following moduli spaces of perturbed Morse gradient trees :

\theoremstyle{definition}
\newtheorem*{geo:def:moduli-space-n-simplex}{Definition~\ref{geo:def:moduli-space-n-simplex}}
\begin{geo:def:moduli-space-n-simplex}
Let $y \in \mathrm{Crit}(g)$ and $x_1,\dots ,x_m\in \mathrm{Crit}(f)$, we define the moduli spaces
\[ \mathcal{CT}_{\Delta^n,t_g}^{\mathbb{Y}_{\Delta^n,t_g}}(y ; x_1,\dots,x_m) := \bigcup_{\delta \in \mathring{\Delta}^n} \mathcal{CT}_{t_g}^{\mathbb{Y}_{\delta,t_g}}(y;x_1,\dots,x_m) \ . \]
\end{geo:def:moduli-space-n-simplex}
 
\noindent As in~\cite{mazuir-I}, these moduli spaces are orientable manifolds under some generic transversality assumptions on the perturbation data~:

\theoremstyle{theorem}
\newtheorem*{geo:th:exist-compact-mod-space}{Theorems~\ref{geo:th:existence}~and~\ref{geo:th:compactification}}
\begin{geo:th:exist-compact-mod-space}
Under some generic assumptions on the choice of perturbation data $(\mathbb{Y}_{I,m})^{m \geqslant 1}_{I \subset \Delta^n}$, the moduli spaces $\mathcal{CT}_{I,t_g}(y;x_1,\dots,x_m)$ are orientable manifolds. If they have dimension 0 they are moreover compact. If they have dimension 1 they can be compactified to 1-dimensional manifolds with boundary, whose boundary is modeled on the boundary of the $n$-multiplihedron $n-J_m$ endowed with its $n-\ombas$ -cell decomposition.
\end{geo:th:exist-compact-mod-space}

\noindent Perturbation data $(\mathbb{Y}_{I,m})^{m \geqslant 1}_{I \subset \Delta^n}$ satisfying the generic assumptions under which Theorems~\ref{geo:th:existence}~and~\ref{geo:th:compactification} hold will be called admissible. Given admissible choices of perturbation data $\mathbb{X}^f$ and $\mathbb{X}^g$, we construct a $n-\ombas$-morphism between the \ombas -algebras $C^*(f)$ and $C^*(g)$ by counting 0-dimensional moduli spaces of Morse gradient trees~: 

\theoremstyle{theorem}
\newtheorem*{geo:th:n-ombas-Morse}{Theorem~\ref{geo:th:n-ombas-Morse}}
\begin{geo:th:n-ombas-Morse}
Let $(\mathbb{Y}_{I,m})^{m \geqslant 1}_{I \subset \Delta^n}$ be an admissible choice of perturbation data.
For every $m$ and $t_g \in SCRT_m$, and every $I \subset \Delta^n$ we define the operation $\mu_{I,t_g}$ as
\begin{align*}
\mu_{I,t_g} : C^*(f) \otimes \cdots \otimes C^*(f) &\longrightarrow C^*(g) \\
x_1 \otimes \cdots \otimes x_m &\longmapsto \sum_{|y|= \sum_{i=1}^m|x_i| + |t_{I,g}|} \# \mathcal{CT}_{I,t_g}^{\mathbb{Y}_{I,t_g}}(y ; x_1,\cdots,x_m) \cdot y \ .
\end{align*} 
This set of operations then defines a $n - \ombas$-morphism $(C^*(f),m_t^{\mathbb{X}^f}) \rightarrow (C^*(g),m_t^{\mathbb{X}^g})$.
\end{geo:th:n-ombas-Morse}

\noindent This $n$-morphism is in fact a twisted $n$-morphism as defined in~\cite{mazuir-I}. 
We subsequently prove a filling theorem for simplicial complexes of perturbation data :

\theoremstyle{theorem}
\newtheorem*{geo:th:filler}{Theorem~\ref{geo:th:filler}}
\begin{geo:th:filler}
For every admissible choice of perturbation data $\mathbb{Y}_{S}$ parametrized by a simplicial subcomplex $S \subset \Delta^n$, there exists an admissible $n$-simplex of perturbation data $\mathbb{Y}_{\Delta^n}$ extending $\mathbb{Y}_{S}$.
\end{geo:th:filler}

\noindent Defining $\mathrm{HOM}^{geom}_{\ombas}(C^*(f),C^*(g))_\bullet$ to be the simplicial subset of $\mathrm{HOM}_{\ombas}(C^*(f),C^*(g))_\bullet$ consisting of higher morphisms defined by a count of perturbed Morse gradient trees, we prove that Theorem~\ref{geo:th:filler} implies the following theorem~:

\theoremstyle{theorem}
\newtheorem*{geo:th:contractible}{Theorem~\ref{geo:th:contractible}}
\begin{geo:th:contractible}
The simplicial set $\mathrm{HOM}^{geom}_{\ombas}(C^*(f),C^*(g))_\bullet$ is a Kan complex which is contractible.
\end{geo:th:contractible}

\noindent This solves in particular the motivational question to this paper.
It is quite clear that given two compact symplectic manifolds $M$ and $N$, one should be able to construct $n$-functors between their Fukaya categories $\mathrm{Fuk}(M)$ and $\mathrm{Fuk}(N)$ by counting pseudo-holomorphic quilted disks with Lagrangian correspondence seam condition, as suggested by the construction of geometric \Ainf -functors between Fukaya categories in \cite{mau-wehrheim-woodward}.

All transversality arguments and sign computations are performed in section~\ref{geo:s:trans-or-signs}~: they are mere adaptations of the analogous constructions in~\cite{mazuir-I}. We finally recall the second question stated at the end of~\cite{mazuir-I} in section~\ref{geo:s:towards}, which is going to be tackled in an upcoming article.

\paragraph{\textit{\textbf{Acknowledgements}}} My first thanks go to my advisor Alexandru Oancea, for his continuous help and support through the settling of this series of papers. I also express my gratitude to Bruno Vallette for his constant reachability and his suggestions and ideas on the algebra underlying this work. I specially thank Jean-Michel Fischer and Guillaume Laplante-Anfossi who repeatedly took the time to offer explanations on higher algebra and $\infty$-categories. I finally adress my thanks to Florian Bertuol, Thomas Massoni, Amiel Peiffer-Smadja, Victor Roca Lucio, Geoffroy Horel, Brice Le Grignou, Nate Bottman and the members of the Roberta seminar for useful discussions.

\newpage

\begin{leftbar}
\part{Higher morphisms between \Ainf\ and \ombas -algebras} \label{p:algebra}
\end{leftbar}

\section{$n-\Ainf$-morphisms} \label{alg:s:n-ainf-morph}

This section is dedicated to the study of the \emph{higher algebra of \Ainf -algebras}. Our starting point is the study of homotopy theory in the category of \Ainf -algebras. Putting it simply, considering two \Ainf -morphisms $F,G$ between \Ainf -algebras, we would like to determine which notion would give a satisfactory meaning to the sentence "$F$ and $G$ are homotopic". This question is solved in section~\ref{alg:ss:ainf-hom} following~\cite{lefevre-hasegawa}, where we define the notion of an \emph{\Ainf -homotopy}. 

Studying higher algebra of \Ainf -algebras means that we will be concerned with the higher homotopy theory of \Ainf -algebras. Typically, the questions arising are the following ones. Homotopies being defined, what is now a good notion of a homotopy between homotopies ? And of a homotopy between two homotopies between homotopies ? And so on. Higher algebra is a general term standing for all problems that involve defining coherent sets of \emph{higher homotopies} (also called \emph{$n$-morphisms}) when starting from a basic homotopy setting.

The sections following the definition of \Ainf -homotopies will then be concerned with defining a good notion of $n$-morphisms between \Ainf -algebras, i.e. such that \Ainf -morphisms correspond to 0-morphisms and \Ainf -homotopies to 1-morphisms. This will be done using the viewpoint of section~\ref{alg:ss:recoll-def}, which defines the category of \Ainf -algebras as a full subcategory of the category of dg-coalgebras. Sections~\ref{alg:ss:some-def}~and~\ref{alg:ss:n-morph-ainf} consist in a pedestrian approach to the construction of these $n$-morphisms, and section~\ref{alg:ss:resume-n-ainf} sums it all up. In section~\ref{alg:ss:equivalent-definition} we moreover introduce an equivalent definition of $n$-morphisms, that we will need in section~\ref{fd:ss:open-question-plus-result} of part~\ref{p:simplicial}. We postpone all sign computations to section~\ref{alg:ss:signs-n-ainf-morph}.

\subsection{Recollections and definitions} \label{alg:ss:recoll-def}

Let $A$ be a graded \Z -module. We introduce its suspension $sA$ defined as the graded \Z -module $(sA)^i := A^{i + 1}$. In other words, $|sa|=|a|-1$. This is merely a notation that gives a convenient way to handle certain degrees. Note for instance that a degree $2-n$ map $A^{\otimes n} \rightarrow A$ is simply a degree $+1$ map $(sA)^{\otimes n} \rightarrow sA$.

Our main category of interest will be the category whose objects are \Ainf -algebras and whose morphisms are \Ainf -morphisms. It will be written as $\mathsf{\Ainf -Alg}$. Recall that a structure of \Ainf -algebra on a dg-\Z -module $A$ can equivalently be defined as a collection of operations $m_n : A^{\otimes n} \rightarrow A$ satisfying the \Ainf -equations, or as a codifferential $D_A$ on its shifted bar construction $\overline{T}(sA)$. Similarly, an \Ainf -morphism is equivalently defined as a collection of operations $f_n : A^{\otimes n} \rightarrow B$ satisfying the \Ainf -equations, or as a morphism of dg-coalgebras $(\overline{T}(sA),D_A) \rightarrow (\overline{T}(sB),D_B)$. We refer to the first article of this series~\cite{mazuir-I} for a detailed discussion on these results. 

As a consequence, the shifted bar construction functor identifies the category $\mathsf{\Ainf -Alg}$ with a full subcategory of the category of dg-coalgebras $\mathsf{dg-Cog}$, that is
\[ \mathsf{\Ainf -Alg} \subset \mathsf{dg-Cog} \ . \]
This basic idea is the key to our first construction of $n$-morphisms in this section. We will perform some natural constructions in the category $\mathsf{dg-Cog}$, and then specialize them to the category $\mathsf{\Ainf -Alg}$ using the above inclusion. As before, these natural constructions will then admit an interpretation in terms of operations $A^{\otimes n} \rightarrow B$, using the universal property of the bar construction.

\subsection{\Ainf -homotopies} \label{alg:ss:ainf-hom}

The material presented in this section is taken from the thesis of Lefèvre-Hasegawa~\cite{lefevre-hasegawa}.

\subsubsection{Homotopies between morphisms of dg-coalgebras} \label{alg:sss:hom-dg-cog}

\begin{definition}[\cite{lefevre-hasegawa}]
Let $C$ and $C'$ be two dg-coalgebras. Let $F$ and $G$ be morphisms $C \rightarrow C'$ of dg-coalgebras. A \emph{$(F,G)$-coderivation} is defined to be a map $H : C \rightarrow C'$ such that 
\[ \Delta_{C'} H = (F \otimes H + H \otimes G) \Delta_{C} \ . \]
The morphisms $F$ and $G$ are then said to be \emph{homotopic} if there exists a $(F,G)$-coderivation $H$ of degree -1 such that 
\[ [ \partial , H ] = G - F \ . \]
\end{definition}

Introduce the dg-coalgebra
\[ \pmb{\Delta}^1 := \Z  [ 0 ] \oplus \Z  [ 1 ] \oplus \Z  [0<1]  \ . \]
Its differential is the singular differential $\partial^{sing}$
\begin{align*}
\partial^{sing} ( [0<1] ) = [1] - [0] && \partial^{sing} ( [ 0 ] ) = 0 && \partial^{sing} ( [ 1 ] ) = 0 \ ,
\end{align*}
its coproduct is the Alexander-Whitney coproduct
\begin{align*}
\Delta_{\pmb{\Delta}^1} ( [0<1] ) =  [0] \otimes [0<1] + [0<1] \otimes [1] && \Delta_{\pmb{\Delta}^1} ( [ 0 ] ) = [ 0 ] \otimes [ 0 ] && \Delta_{\pmb{\Delta}^1} ( [ 1 ] ) = [ 1 ] \otimes [ 1 ] \ ,
\end{align*}
the elements $[0]$ and $[1]$ have degree $0$, and the element $[0<1]$ has degree $-1$. We refer to subsection~\ref{alg:sss:cosimpl-dg-cog} for a broader interpretation of $\pmb{\Delta}^1$.

\begin{proposition}[\cite{lefevre-hasegawa}]
There is a one-to-one correspondence between $(F,G)$-coderivations and morphisms of dg-coalgebras $\pmb{\Delta}^1 \otimes C \longrightarrow C'$.
\end{proposition}

\begin{proof}
One checks indeed that :
\begin{enumerate}[label=(\roman*)]
\item $F$ and $G$ are the restrictions to the summands $\Z  [ 0 ] \otimes C$ and $\Z  [ 1 ]\otimes C$, $H$ is the restriction to the summand $\Z  [ 0 < 1 ] \otimes C$ ;
\item the coderivation relation is given by the compatibility with the coproduct ;
\item the homotopy relation is given by the compatibility with the differential.
\end{enumerate} 
\end{proof}

\subsubsection{\Ainf -homotopies} \label{alg:sss:ainf-hom}

Using the inclusion $\mathsf{\Ainf -Alg} \subset \mathsf{dg-Cog}$, this yields a notion of homotopy between two \Ainf -morphisms, which we call a \emph{\Ainf -homotopy} :

\begin{definition}[\cite{lefevre-hasegawa}]
Let $(\overline{T}(sA),D_A)$ and $(\overline{T}(sB),D_B)$ be two \Ainf -algebras. Given two \Ainf -morphisms $F,G : (\overline{T}(sA),D_A) \rightarrow (\overline{T}(sB),D_B)$, an \emph{\Ainf -homotopy from $F$ to $G$} is defined to be a morphism of dg-coalgebras
\[ H : \pmb{\Delta}^1 \otimes \overline{T}(sA) \longrightarrow \overline{T}(sB) \ , \]
whose restriction to the $[0]$ summand is $F$ and whose restriction to the $[1]$ summand is $G$.
\end{definition}

An alternative and equivalent definition ensues then as follows (see subsection~\ref{alg:sss:n-morph-ainf} for a more general proof of the equivalence between the two definitions) :

\begin{definition}[\cite{lefevre-hasegawa}]
An \emph{\Ainf -homotopy} between two \Ainf -morphisms $(f_n)_{n \geqslant 1}$ and $(g_n)_{n \geqslant 1}$ of $\Ainf$-algebras $A$ and $B$ is defined to be a collection of maps
\[ h_n : A^{\otimes n} \longrightarrow B \ , \]
 of degree $-n$, which satisfy the equations
\begin{align*}
[ \partial , h_n ] =  &g_n - f_n + \sum_{\substack{i_1+i_2+i_3=m \\ i_2 \geqslant 2}} \pm h_{i_1 + 1 + i_3} (\ide^{\otimes i_1} \otimes m_{i_2} \otimes \ide^{\otimes i_3}) \\
&+ \sum_{\substack{i_1 + \cdots + i_s  + l \\ + j_1 + \cdots + j_t = n \\ s + 1 + t \geqslant 2}} \pm m_{s+1+t} ( f_{i_1} \otimes \cdots \otimes f_{i_s} \otimes h_l \otimes g_{j_1} \otimes \cdots \otimes g_{j_t} ) \ .
\end{align*}
\end{definition}

The signs will be made explicit in section~\ref{alg:ss:signs-n-ainf-morph}. Using the same symbolic formalism as in~\cite{mazuir-I}, this can be represented as
\[ [ \partial , \ainfhomdeuxprime ] = \eqainfhomdeux - \eqainfhomun + \sum \pm \eqainfhomtrois + \sum \pm \eqainfhomquatre \ , \]
where we denote \ainfhomun , \ainfhomdeux\ and \ainfhomtrois\ respectively for the $f_n$, the $h_n$ and the $g_n$.

\subsubsection{On this notion of homotopy} \label{alg:sss:on-this-homotopy}

The relation \emph{being \Ainf -homotopic} on the class of \Ainf -morphisms is in fact an equivalence relation. It is moreover stable under composition. These results cannot be proven using naive tools, and are obtained through considerations of model categories. We refer to Lefèvre-Hasegawa~\cite{lefevre-hasegawa} for the reader interested in the proof of these two results.

\subsection{Some definitions} \label{alg:ss:some-def}

\subsubsection{The cosimplicial dg-coalgebra $\pmb{\Delta}^n$} \label{alg:sss:cosimpl-dg-cog}

\begin{definition}
Define $\pmb{\Delta}^n$ to be the graded \Z -module generated by the faces of the standard $n$-simplex $\Delta^n$,
\[ \pmb{\Delta}^n = \bigoplus_{0 \leqslant i_0 < \dots < i_k \leqslant n} \Z [i_0 < \dots < i_k ] \ , \]
where the grading is $| I | := - \mathrm{dim}(I)$ for $I$ a face of $\Delta^n$.
We endow this graded \Z -module with a dg-coalgebra structure, whose differential is the simplicial differential
\[ \partial_{\pmb{\Delta}^n} ([i_0 < \dots < i_k ]) := \sum_{j=0}^k (-1)^j [i_0 < \dots < \widehat{i_j} < \dots < i_k] \ , \]
and whose coproduct is the Alexander-Whitney coproduct
\[ \Delta_{\pmb{\Delta}^n} ([i_0 < \dots < i_k ]) := \sum_{j=0}^k [i_0 < \dots < i_j] \otimes [i_j < \dots < i_k] \ . \]
\end{definition}

These dg-coalgebras are to be seen as the realizations of the simplices $\Delta^n$ in the world of dg-coalgebras.
The collection of dg-coalgebras $\pmb{\Delta}^\bullet := \{ \pmb{\Delta}^n \}_{n \geqslant 0}$ is then naturally a cosimplicial dg-coalgebra. The coface map
\[ \delta_i : \pmb{\Delta}^{n-1} \longrightarrow \pmb{\Delta}^{n} \ , 0 \leqslant i \leqslant n \ , \]
is obtained by seeing the simplex $\Delta^{n-1}$ as the $i$-th face of the simplex $\Delta^n$. The codegeneracy map 
\[ \sigma_i : \pmb{\Delta}^{n+1} \longrightarrow \pmb{\Delta}^{n} \ , 0 \leqslant i \leqslant n \ , \]
is defined as 
\begin{align*}
[ j_0 < \cdots < j_r < \hat{i} < j_{r+1} < \cdots < j_s ] &\longmapsto [ j_0 < \cdots < j_r < j_{r+1} -1 < \cdots < j_s -1 ] \ , \\
[ j_0 < \cdots < j_r < \widehat{i+1} < j_{r+1} < \cdots < j_s ] &\longmapsto [ j_0 < \cdots < j_r < j_{r+1} -1 < \cdots < j_s -1 ] \ , \\
[ j_0 < \cdots < j_s ] &\longmapsto 0 \ \ \ \ \ \text{if} \ [ i < i +1] \subset [ j_0 < \cdots < j_s ] \ .
\end{align*}
In other words, the face $[ 0 < \cdots < \hat{i} < \cdots < n+1]$ and its subfaces are identified with $\Delta^n$ and its subfaces. The same goes for $[ 0 < \cdots < \widehat{i+1} < \cdots < n+1]$ and its subfaces. All faces of $\Delta^{n+1}$ that contain $[ i < i +1]$ are taken to 0.

Heuristically, the coface and codegeneracy maps are obtained by applying the functor
\[ C_{-*}^{sing} : \mathsf{Spaces} \longrightarrow \mathsf{dg-Cog} \]
to the cosimplicial space $\Delta^n$, and then quotienting out each $C_{-*}^{sing}(\Delta^n)$ by the subcomplex generated by all degenerate singular simplices. For instance, the codegeneracy map $\sigma_i : \Delta^{n+1} \rightarrow \Delta^n$ is obtained by contracting the edge $[ i < i+1]$ of $\Delta^{n+1}$, which yields the above codegeneracy map $\sigma_i : \pmb{\Delta}^{n+1} \rightarrow \pmb{\Delta}^{n}$. We refer to~\cite{goerss-simplicial} for more details on the matter.

\subsubsection{Overlapping partitions} \label{alg:sss:overlapping-part}

\begin{definition}[\cite{mcclure-smith}] \label{alg:def:def-mcclure-smith}
Let $I$ be a face of $\Delta^n$. An \emph{overlapping partition} of $I$ is defined to be a sequence of faces $(I_l)_{1 \leqslant \ell \leqslant s}$ of $I$ such that
\begin{enumerate}[label=(\roman*)]
\item the union of this sequence of faces is $I$, i.e. $\cup_{1 \leqslant \ell \leqslant s} I_l = I$ ;
\item for all $1 \leqslant \ell < s$, $\mathrm{max} ( I_{\ell} ) = \mathrm{min} ( I_{\ell+1} )$.
\end{enumerate}
\end{definition}

These two requirements then imply in particular that $\mathrm{min} ( I_{1} ) = \mathrm{min} ( I )$ and $\mathrm{max} ( I_{s} )=\mathrm{max} ( I )$. If the overlapping partition has $s$ components $I_\ell$, we will refer to it as an \emph{overlapping $s$-partition}.
These sequences of faces are those which naturally arise when applying several times the Alexander-Whitney coproduct to a face $I$. For instance, the Alexander-Whitney coproduct corresponds to the sum of all overlapping 2-partitions of $I$. Iterating $n$ times the Alexander-Whitney coproduct, we get the sum of all overlapping $(n+1)$-partitions of $I$. An overlapping 6-partition for $[ 0 < 1 <2]$ is for instance
\[
[ 0 < 1 < 2] = [0] \cup [0] \cup [0 < 1] \cup [1] \cup [1<2] \cup [2]
\ . \]

\subsection{$n$-morphisms between \Ainf -algebras} \label{alg:ss:n-morph-ainf}

We now want to define a notion of \emph{higher homotopies}, or \emph{$n$-morphisms}, between \Ainf -algebras, such that 0-morphisms are \Ainf -morphisms and 1-morphisms are \Ainf -homotopies. Since \Ainf -morphisms correspond to the set
\[ \mathrm{Hom}_{\mathsf{dg-Cog}}(\overline{T}(sA),\overline{T}(sB)) \]
and \Ainf -homotopies correspond to the set 
\[ \mathrm{Hom}_{\mathsf{dg-Cog}}(\pmb{\Delta}^1 \otimes \overline{T}(sA),\overline{T}(sB)) \ , \]
a natural candidate for the set of $n$-morphisms is
\[ \mathrm{HOM}_{\mathsf{\Ainf -Alg}}(A,B)_{n} := \mathrm{Hom}_{\mathsf{dg-Cog}}(\pmb{\Delta}^n \otimes \overline{T}(sA) ,\overline{T}(sB)) \ . \]

\subsubsection{$n$-morphisms between dg -coalgebras} \label{alg:sss:n-morph-dg-cog}

We begin by making explicit the $n$-simplices of the $\mathrm{HOM}$-simplicial sets 
\[ \mathrm{HOM}_{\mathsf{dg-Cog}}(C,C')_{n} := \mathrm{Hom}_{\mathsf{dg-Cog}}(\pmb{\Delta}^n \otimes C ,C') \ . \]
Take a morphism of dg-coalgebras 
\[ f  :  \pmb{\Delta}^n \otimes C \longrightarrow C' \ . \]
Write $f_{[i_0 < \dots < i_k]} : C \rightarrow C'$ for its restriction to the~$\Z [i_0 < \dots < i_k] \otimes C$ summand. Then the property that $f$ is a morphism of dg-\Z -modules is equivalent to the system of equations
\begin{equation} \label{alg:eq:diff}
[ \partial , f_{[i_0 < \dots < i_k]} ] = \sum_{j=0}^k (-1)^j f_{[i_0 < \dots < \widehat{i_j} < \dots < i_k]} \ , 
\end{equation} 
while the property that $f$ is a morphism of coalgebras is equivalent to the system of equations
\begin{equation} \label{alg:eq:coprod}
\Delta_{C'} f_{[i_0 < \dots < i_k]} = \sum_{j=0}^k (f_{[i_0 < \dots < i_j]} \otimes f_{[i_j < \dots < i_k]}) \Delta_C \ . 
\end{equation} 

These two sets of equations of morphisms hence characterize the $n$-simplices of the HOM-simplicial sets $\mathrm{HOM}_{\mathsf{dg-Cog}}(C,C')_{\bullet}$, i.e. the $n$-morphisms between the dg-coalgebras $C$ and $C'$.

\subsubsection{$n$-morphisms between \Ainf -algebras} \label{alg:sss:n-morph-ainf}

We now use the previous characterization of $n$-morphisms between dg-coalgebras to obtain a simpler definition for $n$-morphisms between two \Ainf -algebras :
\begin{definition} \label{alg:def:n-morph-ainf-susp}
Let $A$ and $B$ be two \Ainf -algebras. A \emph{$n$-morphism} from $A$ to $B$ is defined to be a morphism of dg-coalgebras 
\[ F : \pmb{\Delta}^n \otimes \overline{T}(sA) \longrightarrow \overline{T}(sB) \ . \]
\end{definition}

We will write $b_n$ for the degree $+1$ maps associated to the \Ainf -operations $m_n$, which define the codifferentials on $\overline{T}(sA)$ and $\overline{T}(sB)$.
The property of being a morphism of coalgebras is equivalent to the property of satisfying equations~\ref{alg:eq:coprod}. Using the universal property of the bar construction, this is equivalent to saying that the $n$-morphism is given by a collection of maps of degree $|I|$, 
\[ F_{I}^{(m)} : (sA)^{\otimes m} \longrightarrow sB \ , \]
where $I$ is a face of $\Delta^n$ and $m \geqslant 1$.
The restriction of the map $F_{I} : \overline{T}(sA) \rightarrow \overline{T}(sB)$ to $(sA)^{\otimes m}$ is then given by
\[
F_{I}^{(m)} + \sum_{\substack{i_1 + i_2 = m \\ I_1 \cup I_2 = I}} F_{I_1}^{(i_1)} \otimes F_{I_2}^{(i_2)} + \dots + \sum_{\substack{i_1 + \cdots + i_s = m \\ I_1 \cup \cdots \cup I_s = I}} F_{I_1}^{(i_1)} \otimes \cdots \otimes F_{I_s}^{(i_s)} + \dots + \sum_{I_1 \cup \cdots \cup I_m = I} F_{I_1}^{(1)} \otimes \cdots \otimes F_{I_m}^{(1)} \ ,
\]
where $I_1 \cup \cdots \cup I_s = I$ stands for an overlapping partition of $I$.
Corestricting to $B^{\otimes s}$ yields the morphism
\[ \sum_{\substack{i_1 + \cdots + i_s = m \\ I_1 \cup \cdots \cup I_s = I}} F_{I_1}^{(i_1)} \otimes \cdots \otimes F_{I_s}^{(i_s)} : (sA)^{\otimes m} \longrightarrow (sB)^{\otimes s} \ . \]

The property of being compatible with the differentials is equivalent to the property of satisfying equations~\ref{alg:eq:diff}. This is itself equivalent to the fact that the collection of morphisms $F_{I}^{(m)}$ satisfies the following family of equations involving morphisms $(sA)^{\otimes m} \rightarrow sB$,
\[ \resizebox{\hsize}{!}{$ \displaystyle{\sum_{j=0}^{\mathrm{dim}(I)} (-1)^j F^{(m)}_{\partial_j I} + (-1)^{|I|} \sum_{i_1+i_2+i_3=m} F^{(i_1+1+i_3)}_I (\ide^{\otimes i_1} \otimes b_{i_2} \otimes \ide^{\otimes i_3}) \\
= \sum_{\substack{i_1 + \cdots + i_s = m \\ I_1 \cup \cdots \cup I_s = I}} b_s ( F^{(i_1)}_{I_1} \otimes \cdots \otimes F^{(i_s)}_{I_s}) \ .}$} \]

We unwind the signs obtained by changing the $b_n$ into the $m_n$ and the degree $|I|$ maps $ F_{I}^{(m)} : (sA)^{\otimes m} \longrightarrow sB$ into degree $1 - m + |I|$ maps $ f_{I}^{(m)} : A^{\otimes m} \longrightarrow B $ in subsection~\ref{alg:sss:choice-convention-ainf}. The final equations read as 
\begin{align*} \label{alg:eq:n-ainf-morph}
\left[ \partial , f^{(m)}_I \right] =  \sum_{j=0}^{\mathrm{dim}(I)} (-1)^j f^{(m)}_{\partial_jI} &+ (-1)^{|I|} \sum_{\substack{i_1+i_2+i_3=m \\ i_2 \geqslant 2}} \pm f^{(i_1+1+i_3)}_I (\ide^{\otimes i_1} \otimes m_{i_2} \otimes \ide^{\otimes i_3}) \tag{$\star$} \\ &+ \sum_{\substack{i_1 + \cdots + i_s = m \\ I_1 \cup \cdots \cup I_s = I \\ s \geqslant 2 }} \pm m_s ( f^{(i_1)}_{I_1} \otimes \cdots \otimes f^{(i_s)}_{I_s}) \ ,  \\
\end{align*}
or equivalently and more visually,
\[ [ \partial , \ainfnmorphun ] = \sum_{j=0}^{\mathrm{dim}(I)} (-1)^j \eqainfnmorphun + \sum_{I_1 \cup \cdots \cup I_s = I} \pm \eqainfnmorphtrois
 + \sum \pm \eqainfnmorphdeux \ . \]
 
\begin{definition} \label{alg:def:n-morph-ainf}
Let $A$ and $B$ be two \Ainf -algebras. A \emph{$n$-morphism} from $A$ to $B$ is defined to be a collection of maps $f_{I}^{(m)} : A^{\otimes m} \longrightarrow B$ of degree $1 - m + |I|$ for $I \subset \Delta^n$ and $m \geqslant 1$, that satisfy equations~\ref{alg:eq:n-ainf-morph}.
\end{definition}

\subsection{An equivalent definition for $n$-morphisms} \label{alg:ss:equivalent-definition}

We show in this section that given $A$ and $B$ two \Ainf -algebras, the datum of a $n$-morphism from $A$ to $B$ is equivalent to the datum of a morphism of \Ainf -algebras $A \rightarrow \pmb{\Delta}_n \otimes B$. 

Consider first $C$ a dg-algebra and $B$ an \Ainf -algebra. Then the tensor product $C \otimes B$ can be naturally endowed with an \Ainf -algebra structure by defining $m_n : (A \otimes B)^{\otimes n} \rightarrow A \otimes B$ as
\[ m_n := ((m_2^A)^{\circ n-1} \otimes m_n^B ) \circ \tau_n \ , \]
where $\tau_n$ denotes the map rearranging an element $a_1 b_1 \dots a_n b_n$ of $(A \otimes B)^{\otimes n}$ into an element $a_1 \dots a_n b_1 \dots b_n$ of $A^{\otimes n} \otimes B^{\otimes n}$ and $(m_2^A)^{\circ n-1} : A^{\otimes n} \rightarrow A$ denotes the $(n-1)$-th iterate of the multiplication $m_2^A$ on $A$.
We moreover define $\pmb{\Delta}_n$ to be the simplicial dg-algebra $\pmb{\Delta}_n := \mathrm{Hom}(\pmb{\Delta}^n,\Z)$ dual to the cosimplicial dg-coalgebra $\pmb{\Delta}^n$. Its underlying graded module is in particular 
\[ \pmb{\Delta}_n = \bigoplus_{0 \leqslant i_0 < \dots < i_k \leqslant n} \Z [i_0 < \dots < i_k ] \]
where the grading is $| I |_{\pmb{\Delta}_n} := \mathrm{dim}(I)$ for $I$ a face of $\Delta^n$. It is endowed with the standard cup product.

An \Ainf -morphism $F : A \rightarrow \pmb{\Delta}_n \otimes B$ then corresponds to a collection of degree $1-m$ maps
$F^{(m)} : A^{\otimes m} \rightarrow \pmb{\Delta}_n \otimes B$ which can be rewritten as a collection of degree $1-m-\mathrm{dim}(I)$ maps $f_I^{(m)} : A^{\otimes m} \rightarrow B$ such that 
\[ F^{(m)} = \bigoplus_{I \subset \Delta^n} I \otimes f_I^{(m)} \ . \]
We denote $\pi_I : \pmb{\Delta}_n \rightarrow \Z \cdot I$ the projection from $\pmb{\Delta}_n$ to its summand labeled by $I$. Then, the \Ainf -equations for the \Ainf -morphism $F : A \rightarrow \pmb{\Delta}_n \otimes B$ read as 
\[ \left[ \partial , F^{(m)} \right] = \sum_{\substack{i_1+i_2+i_3=m \\ i_2 \geqslant 2}} \pm F^{(i_1+1+i_3)} (\ide^{\otimes i_1} \otimes m_{i_2}^A \otimes \ide^{\otimes i_3}) + \sum_{\substack{i_1 + \cdots + i_s = m \\ s \geqslant 2 }} \pm m_s^{\pmb{\Delta}_n \otimes B} ( F^{(i_1)} \otimes \cdots \otimes F^{(i_s)}) \ , \]
and their images under the map $\pi_I$ yield exactly the \Ainf -equations~\ref{alg:eq:n-ainf-morph} for the collection of morphisms $f_I^{(m)}$.

\begin{proposition} \label{alg:prop:equivalent-def}
Let $A$ and $B$ be two \Ainf -algebras. A \emph{$n$-morphism} from $A$ to $B$ can be equivalently defined as an \Ainf -morphism $ A \rightarrow \pmb{\Delta}_n \otimes B$.
\end{proposition}

We will only need this equivalent definition of $n$-morphisms in subsection~\ref{fd:ss:open-question-plus-result} of part \ref{p:simplicial}, and will stick to the definition $\pmb{\Delta}^n \otimes \overline{T}(sA) \rightarrow \overline{T}(sB)$ and to the definition in terms of operations in the rest of this paper. We moreover point out that the natural sign convention for $n$-morphisms arising from this new definition differs slightly from the one arising from the two previous definitions, as we explain in subsection~\ref{alg:sss:sign-conv-equivalent}.

\subsection{Résumé} \label{alg:ss:resume-n-ainf}

Given $A$ and $B$ two \Ainf -algebras, we define a \emph{$n$-morphism} between $A$ and $B$ to be an element of the simplicial set
\[ \mathrm{HOM}_{\mathsf{\Ainf -Alg}}(A,B)_{n} := \mathrm{Hom}_{\mathsf{dg-Cog}}( \pmb{\Delta}^n \otimes \overline{T}(sA),\overline{T}(sB)) \ , \]
or equivalently a collection of operations $ \ainfnmorphun : A^{\otimes m} \rightarrow B$ of degree $1-m-\mathrm{dim}(I)$ for all faces $I$ of $\Delta^n$ and all $m \geqslant 1$, satisfying the \Ainf -equations
\[ [ \partial , \ainfnmorphun ] = \sum_{j=0}^{\mathrm{dim}(I)} (-1)^j \eqainfnmorphun + \sum_{I_1 \cup \cdots \cup I_s = I} \pm \eqainfnmorphtrois
 + \sum \pm \eqainfnmorphdeux \ , \]
where we refer to subsection~\ref{alg:sss:choice-convention-ainf} for signs.

\section{The $n$-multiplihedra} \label{alg:s:polytopes}

Recall from~\cite{mazuir-I} that, in the language of operadic algebra, \Ainf -algebras are governed by the operad \Ainf , and \Ainf -morphisms are governed by the $(\Ainf , \Ainf)$-operadic bimodule \infmor . These two operadic objects actually stem from collections of polytopes. Under the functor $C_{-*}^{cell}$ the associahedra $\{ K_m \}$ realise the operad \Ainf , while the multiplihedra $ \{ J_m \}$ form a $(\{ K_m \},\{ K_m \})$-operadic bimodule realising \infmor .

The first section shows that the operadic bimodule formalism for \Ainf -morphisms can be generalised to the setting of $n-\Ainf$ -morphisms : for each $n \geqslant 0$ there exists an $(\Ainf , \Ainf)$-operadic bimodule \infmorn , which encodes $n$-morphisms between \Ainf -algebras. In fact, they fit into a cosimplicial operadic bimodule $\{ \infmorn \}_{n \geqslant 0}$.
Reproducing the previous progression, we would like to realise the combinatorics of $n$-morphisms at the level of polytopes. The first step in this direction is performed in section~\ref{alg:ss:cell-decompo-AW} : we explain how to lift the Alexander-Whitney coproduct to the level of the standard simplices $\Delta^n$ and study the rich combinatorics that arise in this problem. Section~\ref{alg:ss:poly-n-Jm} subsequently introduces the \emph{$n$-multiplihedra} $n-J_m$, which are the polytopes $\Delta^n \times J_m$ endowed with a refined polytopal subdivision. These polytopes do not form a $(\{ K_m \},\{ K_m \})$-operadic bimodule, but they suffice to recover all the combinatorics of $n$-morphisms.

\subsection{The cosimplicial $(\Ainf , \Ainf )$-operadic bimodule encoding higher morphisms} \label{alg:ss:cosimpl-op-bimod}

\subsubsection{The $(\Ainf , \Ainf )$-operadic bimodules \infmorn} \label{alg:sss:informn}

The $(\Ainf , \Ainf)$-operadic bimodule encoding \Ainf -morphisms is the quasi-free $(\Ainf , \Ainf)$-operadic bimodule generated in arity $n$ by one operation \arbreopmorph{0.15}\ of degree $1-n$,
\[ \infmor = \mathcal{F}^{\Ainf , \Ainf}(\arbreopunmorph , \arbreopdeuxmorph , \arbreoptroismorph , \arbreopquatremorph , \cdots ) \ .  \]
Representing the generating operations of the operad \Ainf\ acting on the right in blue \arbreopbleu{0.15} and the ones of the operad \Ainf\ acting on the left in red \arbreoprouge{0.15}, its differential is defined by
\[ \partial ( \arbreopmorphformule ) = \sum_{\substack{h+k = m+1 \\ 1 \leqslant i \leqslant k \\ h \geqslant 2 }} \pm \eqainfmorphun + \sum_{\substack{i_1 + \cdots + i_s = m \\ s \geqslant 2 }} \pm \eqainfmorphdeux . \] 

\begin{definition}
The $(\Ainf , \Ainf)$-operadic bimodule encoding $n-\Ainf$-morphisms is the quasi-free $(\Ainf , \Ainf)$-operadic bimodule generated in arity $m$ by the operations $f^{(m)}_I$ of degree $1-m+|I|$, for all faces $I$ of $\Delta^n$, and whose differential is defined by
\[ \resizebox{\hsize}{!}{$ \displaystyle{\partial (f^{(m)}_I) = \sum_{j=0}^{\mathrm{dim}I} (-1)^j f^{(m)}_{\partial^{sing}_jI} + \sum_{\substack{i_1+i_2+i_3=m \\ i_2 \geqslant 2}} \pm f^{(i_1+1+i_3)}_I (\ide^{\otimes i_1} \otimes m_{i_2} \otimes \ide^{\otimes i_3}) + \sum_{\substack{i_1 + \cdots + i_s = m \\ I_1 \cup \cdots I_s = I \\ s \geqslant 2 }} \pm m_s ( f^{(i_1)}_{I_1} \otimes \cdots \otimes f^{(i_s)}_{I_s}) \ .}$} \]
\end{definition}

Representing the operations $f^{(m)}_I$ as \ainfnmorphun , this can be rewritten as 
\[ \infmorn = \mathcal{F}^{\Ainf , \Ainf}(\arbreopunmorphn , \arbreopdeuxmorphn , \arbreoptroismorphn , \arbreopquatremorphn , \cdots ; I \subset \Delta^n ) \ .  \]
where
\[ \partial ( \ainfnmorphun ) = \sum_{j=0}^{\mathrm{dim}I} (-1)^j \eqainfnmorphun + \sum_{I_1 \cup \cdots \cup I_s = I} \pm \eqainfnmorphtrois
 + \sum \pm \eqainfnmorphdeux \ . \]

The collection of $(\Ainf ,\Ainf)$-operadic bimodules $\{ \infmorn \}_{n \geqslant 0}$ forms a cosimplicial $(\Ainf ,\Ainf)$-operadic bimodule whose coface and codegeneracy maps are built out of those of section~\ref{alg:ss:some-def}. Given two \Ainf -algebras $\Ainf \rightarrow \mathrm{Hom}(A)$ and $\Ainf \rightarrow \mathrm{Hom}(B)$, the set of $n$-morphisms is then simply given by
\[ \mathrm{HOM}_{\mathsf{\Ainf -Alg}} (A,B)_n = \mathrm{Hom}_{\mathsf{(\Ainf , \Ainf )-op.
bimod.}} (\infmorn , \mathrm{Hom}(A,B) ) \ . \]

\subsubsection{The two-colored operadic viewpoint} \label{alg:sss:two-col-op-viewpoint}

Recall that \Ainf -algebras and \Ainf -morphisms between them are naturally encoded by the quasi-free two-colored operad
\[ A_\infty^2 := \mathcal{F} (\arbreopdeuxcol{red} , \arbreoptroiscol{red} , \arbreopquatrecol{red}, \cdots, \arbreopdeuxcol{blue} , \arbreoptroiscol{blue} , \arbreopquatrecol{blue} , \cdots, \arbreopunmorph , \arbreopdeuxmorph , \arbreoptroismorph , \arbreopquatremorph , \cdots ) \ , \]
with differential given by the \Ainf -algebra relations on the one-colored operations, and the \Ainf -morphism relations on the two-colored operations.

Similarly, \Ainf -algebras and $n-\Ainf$-morphisms between them are naturally encoded by the quasi-free two-colored operad
\[ n-A_\infty^2 := \mathcal{F} (\arbreopdeuxcol{red} , \arbreoptroiscol{red} , \arbreopquatrecol{red}, \cdots, \arbreopdeuxcol{blue} , \arbreoptroiscol{blue} , \arbreopquatrecol{blue} , \cdots, (\arbreopunmorphn , \arbreopdeuxmorphn , \arbreoptroismorphn , \arbreopquatremorphn , \cdots ; I \subset \Delta^n ) ) \ , \]
with differential given by the \Ainf -algebra relations on the one-colored operations, and the $n-\Ainf$-morphism relations on the two-colored operations. The collection of two-colored operads $\{ n-A_\infty^2 \}_{n \geqslant 0}$ constitutes again a cosimplicial two-colored operad. 

\subsection{Polytopal subdivisions on $\Delta^n$ induced by the Alexander-Whitney coproduct} \label{alg:ss:cell-decompo-AW}

One way of interpreting the Alexander-Whitney coproduct 
\[ \Delta_{\pmb{\Delta}^n} : \pmb{\Delta}^n \longrightarrow \pmb{\Delta}^n \otimes \pmb{\Delta}^n \  \]
is to say that it is a diagonal on the dg-\Z -module $\pmb{\Delta}^n$. The following natural question then arises. \emph{Does there exist a diagonal (i.e. a polytopal map that is homotopic to the usual diagonal - the usual diagonal map failing to be polytopal in general) on the standard $n$-simplex $\Delta^n$,  
\[ \mathrm{AW} : \Delta^n \longrightarrow \Delta^n \times \Delta^n \ , \]
such that its image under the functor $C^{cell}_{-*}$ is $\mathrm{AW}_{-*} = \Delta_{\pmb{\Delta}^n}$ ?}

\noindent The answer to this question is positive, and contains rich combinatorics that we now lay out.

\subsubsection{The map $\mathrm{AW}$} \label{alg:sss:map-AW}

We recall in this section the construction of a diagonal on the standard simplices explained in~\cite{masuda-diagonal-assoc} (example 1 of section 2.3.). 

\begin{definition}[\cite{masuda-diagonal-assoc}]
Consider the realizations of the standard $n$-simplices
\[ \Delta^n := \mathrm{conv} \{ (1,\dots,1,0,\dots,0) \in \R^n \} = \{ (z_1,\dots,z_n) \in \R^n | 1 \geqslant z_1 \geqslant \cdots \geqslant z_n \geqslant 0 \} \ . \]
We define the map $\mathrm{AW}$ by the formula 
\[ \mathrm{AW}(z_1,\cdots,z_n) = ((2z_1-1,\dots,2z_i-1,0,\dots,0),(1,\cdots,1,2z_{i+1},\dots,2z_n)) \ , \]
for $1 \geqslant z_1 \geqslant \cdots \geqslant z_i \geqslant 1/2 \geqslant z_{i+1} \geqslant \cdots \geqslant z_n \geqslant 0$.
\end{definition}

In particular, the map $\mathrm{AW}$ comes with a refined polytopal subdivision of $\Delta^n$, whose $n+1$ top dimensional strata are given by the subsets
\[ \{ (z_1,\dots,z_n) \in \R^n | 1 > z_1 > \cdots > z_i > 1/2 > z_{i+1} > \cdots > z_n > 0 \} \subset \Delta^n \ , \]
and whose $i$-codimensional strata are simply obtained by replacing $i$ symbols "$>$" by a symbol "$=$" in the previous sequence of inequalities.
This refined subdivision is represented on the figures~\ref{alg:fig:delta-subdivision-AW},~\ref{alg:fig:values-AW-delta-un}~and~\ref{alg:fig:values-AW-delta-deux}, together with the value of $\mathrm{AW}$ on each stratum of the subdivision.

\begin{figure}[h] 
    \centering
    \begin{subfigure}{0.45\textwidth}
    \centering
       \includegraphics[scale=0.25]{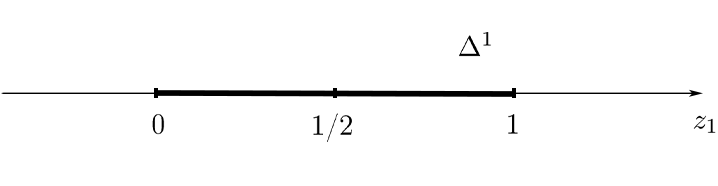}
    \end{subfigure} ~
    \begin{subfigure}{0.45\textwidth}
    \centering
        \includegraphics[scale=0.2]{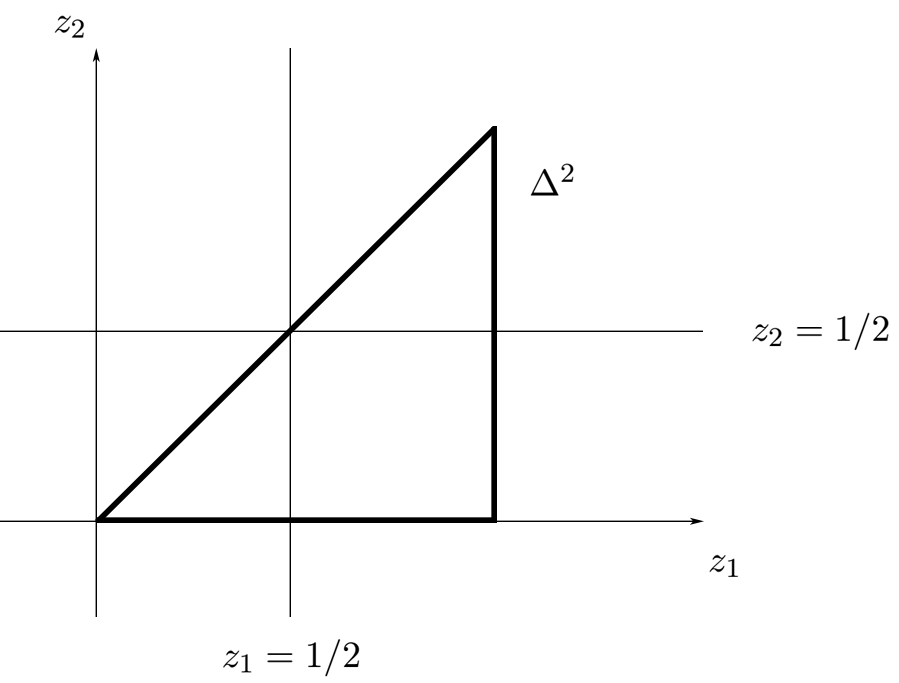}
    \end{subfigure} 
    \caption{The AW-subdivision of $\Delta^1$ and $\Delta^2$} \label{alg:fig:delta-subdivision-AW}
\end{figure}

\begin{figure}[h]
\[ \raisebox{1pt}{\includegraphics[scale=0.3]{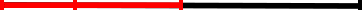}} \longmapsto \raisebox{1pt}{\includegraphics[scale=0.3]{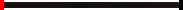}} \times \raisebox{1pt}{\includegraphics[scale=0.3]{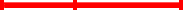}} \ , \]
~
\[ \raisebox{1pt}{\includegraphics[scale=0.3]{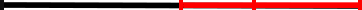}} \longmapsto \raisebox{1pt}{\includegraphics[scale=0.3]{delta-un-valeur-AW-3}} \times \raisebox{1pt}{\includegraphics[scale=0.3]{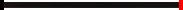}}  \ . \]
\caption{Values of AW on $\Delta^1$ : the stratum to which AW is applied is colored in red} \label{alg:fig:values-AW-delta-un}
\end{figure}

\begin{figure}[h]
\[ \raisebox{-27pt}{\includegraphics[scale=0.15]{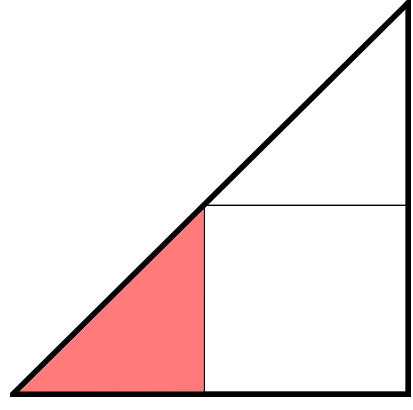}} \ \ \ \longmapsto \raisebox{-27pt}{\includegraphics[scale=0.15]{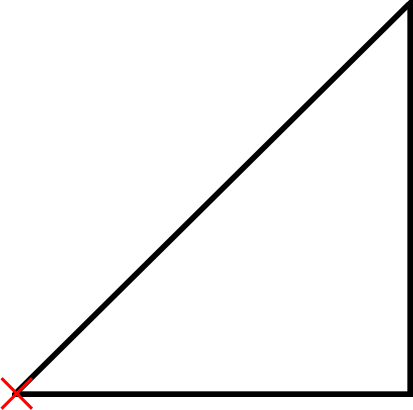}} \ \ \times \raisebox{-27pt}{\includegraphics[scale=0.15]{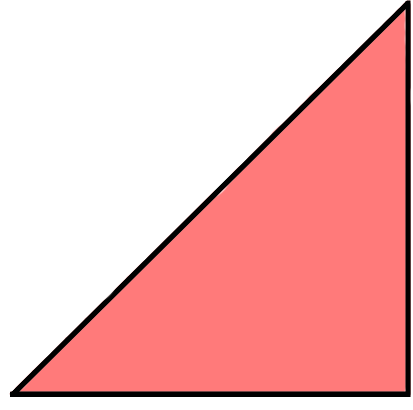}}  \ , \]
~
\[ \raisebox{-27pt}{\includegraphics[scale=0.15]{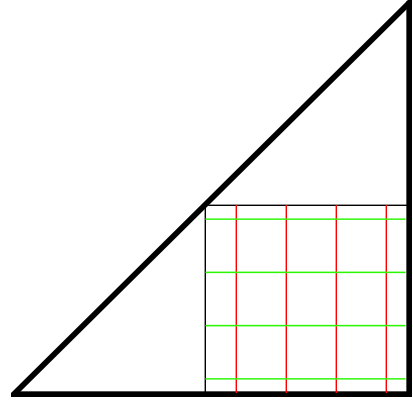}} \ \ \ \longmapsto \raisebox{-27pt}{\includegraphics[scale=0.15]{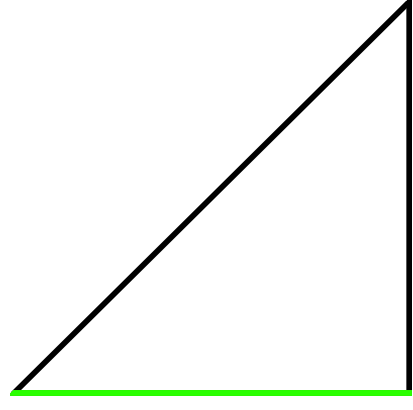}} \ \ \times \raisebox{-27pt}{\includegraphics[scale=0.15]{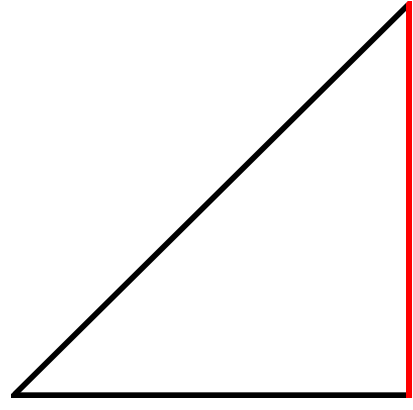}}  \ , \]
~
\[ \raisebox{-27pt}{\includegraphics[scale=0.15]{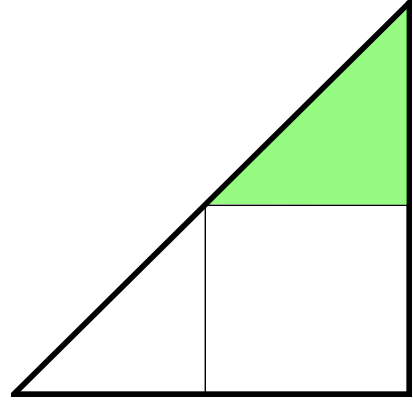}} \ \ \ \longmapsto \raisebox{-27pt}{\includegraphics[scale=0.15]{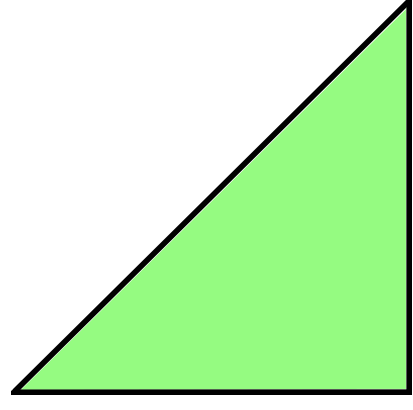}} \ \ \times \raisebox{-27pt}{\includegraphics[scale=0.15]{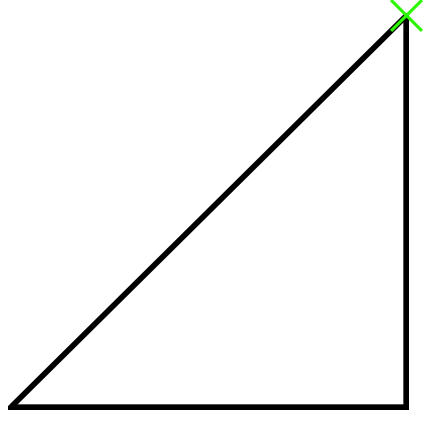}}  \ . \]
\caption{Values of AW on $\Delta^2$}\label{alg:fig:values-AW-delta-deux}
\end{figure}

\subsubsection{The polytopal map $\mathrm{AW}$ is not coassociative} \label{alg:sss:AW-not-coass}

The Alexander-Whitney coproduct $\Delta_{\pmb{\Delta}^n}$ on the dg-level is coassociative. However, the diagonal map $\mathrm{AW}$ is not ! This can be checked for the 1-simplex $\Delta^1$ :
\begin{align*} (\mathrm{AW} \times \ide) \circ \mathrm{AW} (2/5) &= \mathrm{AW} \times \ide (0,4/5) = (0,0,4/5) \\
(\ide \times \mathrm{AW}) \circ \mathrm{AW} (2/5) &= \ide \times \mathrm{AW} (0,4/5) = (0,3/5,1) \ . 
\end{align*}

\begin{proposition} \label{alg:prop:pas-coassoc}
The polytopal map $\mathrm{AW}$ is not coassociative.
\end{proposition}

\noindent The polytopal subdivisions that the polytopal maps 
\begin{align*}
(\mathrm{AW} \times \ide) \circ \mathrm{AW} : \Delta^n \longrightarrow \Delta^n \times \Delta^n \times \Delta^n , \\
(\ide \times \mathrm{AW}) \circ \mathrm{AW} : \Delta^n \longrightarrow \Delta^n \times \Delta^n \times \Delta^n \ 
\end{align*} 
induce on $\Delta^n$ are also different. See an instance on figure~\ref{alg:fig:AW-s-subdivision}.

\begin{figure}[h]
    \centering
    \begin{subfigure}{0.3\textwidth}
    \centering
       \includegraphics[scale=0.2]{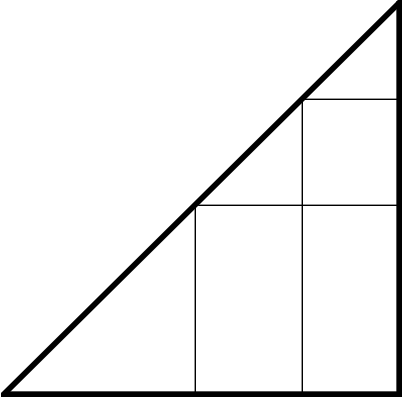}
    \end{subfigure} ~
    \begin{subfigure}{0.3\textwidth}
    \centering
        \includegraphics[scale=0.2]{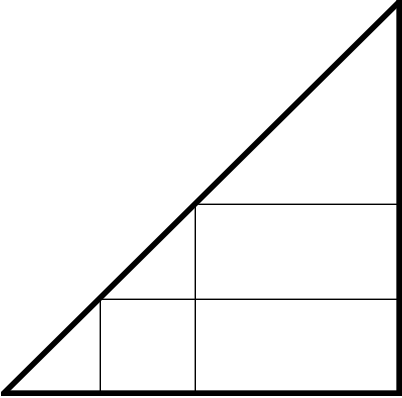}
    \end{subfigure} 
    \caption{The ($\mathrm{AW} \times \ide) \circ \mathrm{AW}$-subdivision and the $(\ide \times \mathrm{AW}) \circ \mathrm{AW}$-subdivision of $\Delta^2$} \label{alg:fig:AW-s-subdivision}
\end{figure}

\subsubsection{$i$-overlapping $s$-partitions} \label{alg:sss:i-over-s-part}

We defined in subsection~\ref{alg:sss:overlapping-part} the notion of an overlapping $s$-partition of a face $I$ of $\Delta^n$. We refine it now :

\begin{definition} 
An \emph{$i$-overlapping $s$-partition of $I$} is a sequence of faces $(I_\ell)_{1 \leqslant \ell \leqslant s}$ of $I$ such that
\begin{enumerate}[label=(\roman*)]
\item the union of this sequence of faces is $I$, i.e. $\cup_{1 \leqslant \ell \leqslant s} I_\ell = I$ ;
\item there are exactly $i$ integers $\ell$ such that $1 \leqslant \ell < s$ and $\mathrm{max} ( I_{\ell} ) = \mathrm{min} ( I_{\ell+1} )$.
\end{enumerate}
\end{definition}

\noindent An overlapping $s$-partition as defined in definition~\ref{alg:def:def-mcclure-smith} is then simply a $(s-1)$-overlapping $s$-partition. A 1-overlapping 3-partition for $[ 0 < 1 < 2]$ is for instance
\[
[ 0 < 1 < 2] = [0] \cup [0] \cup [1 < 2]
\ . \]

\subsubsection{Polytopal subdivisions of $\Delta^n$ induced by iterations of $\mathrm{AW}$} \label{alg:sss:poly-subd-delta-n}

\begin{definition}
Define the \emph{$s$-th right iterate} of the map AW as
\[ \mathrm{AW}^{\circ s} := (\ide^{\times (s-1)} \times \mathrm{AW}) \circ \cdots \circ (\ide \times \mathrm{AW}) \circ \mathrm{AW} : \Delta^n \longrightarrow (\Delta^n)^{\times s+1} \ . \]
\end{definition}

For each $s \geqslant 1$, the map $\mathrm{AW}^{\circ s}$ induces a refined polytopal subdivision of $\Delta^n$. These subdivisions will be called the \emph{$\mathrm{AW}^{\circ s}$-subdivisions} of $\Delta^n$. They can be described rather simply. While the $\mathrm{AW}$-subdivision is obtained by dividing $\Delta^n$ into pieces using all hyperplanes $z_i = 1/2$ for $1 \leqslant i \leqslant n$, the $\mathrm{AW}^{\circ s}$-subdivision can be constructed as follows :
\begin{proposition}
The $\mathrm{AW}^{\circ s}$-subdivision of $\Delta^n$ is the subdivision obtained by dividing $\Delta^n$ using all hyperplanes $z_i = (1/2)^k$, for $1 \leqslant i \leqslant n$ and $1 \leqslant k \leqslant s$. 
\end{proposition}

\begin{figure}
\includegraphics[scale=0.25]{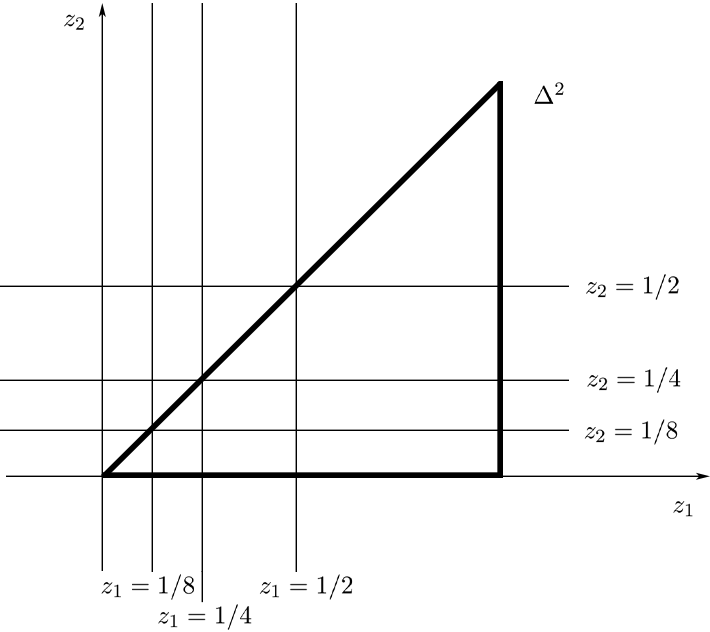}
\caption{The first three subdivisions of $\Delta^2$} \label{alg:fig:first-three-subd-delta-deux}
\end{figure}

\noindent The first three subdivisions of $\Delta^2$ are represented in figure~\ref{alg:fig:first-three-subd-delta-deux}. Note that a different choice for $\mathrm{AW}^{\circ s}$, for instance $\mathrm{AW}^{\circ 2} = (\mathrm{AW} \times \ide) \circ \mathrm{AW}$, would have yielded a different subdivision of $\Delta^n$. Choices have to be made, because $\mathrm{AW}$ is not coassociative.

The $n$-dimensional cells of $\Delta^n$ endowed with its $\mathrm{AW}^{\circ s}$-subdivision are then defined by inequalities
\[  \cdots \geqslant z_{i_k} \geqslant (1/2)^k \geqslant z_{i_k+1} \geqslant \cdots \]
for $1 \leqslant k \leqslant s$. We write $C^{i_1,\dots,i_s}$ for such a cell. 
An explicit formula for the map $\mathrm{AW}^{\circ s} : \Delta^n \rightarrow (\Delta^n)^{\times s + 1}$ can then be computed as follows. Its projection on the $k$-th factor $\Delta^n$ of $(\Delta^n)^{\times s+1}$ restricted to $C^{i_1,\dots,i_s} \subset \Delta^n$ is
\begin{align*}
(z_1 , \dots , z_n) &\longmapsto ( 1 , \dots , 1 , 2^k z_{i_{k-1}+1} - 1 , \dots , 2^k z_{i_k} - 1 , 0 , \dots , 0 ) \ \text{for $1 \leqslant k \leqslant s$,} \\
(z_1 , \dots , z_n) &\longmapsto ( 1 , \dots , 1 , 2^s z_{i_s+1} , \dots , 2^s z_n ) \ \text{for $k=s+1$.}
\end{align*}
This explicit formula for the map $\mathrm{AW}^{\circ s}$ implies the following proposition :
  
\begin{proposition} \label{alg:prop:cell-ci1ik}
The map $\mathrm{AW}^{\circ s}$ sends the cell $C^{i_1,\dots,i_s} \subset \Delta^n$ homeomorphically to the face
\[ [0 < \cdots < i_1 ] \times [i_1 < \cdots < i_2] \times \cdots \times [ i_s < \cdots < n] \subset (\Delta^n)^{\times s+1} \ . \]  
\end{proposition}

Hence not only does the map $\mathrm{AW}^{\circ s}$ determine a subdivision of the simplex $\Delta^n$ but it also determines a labeling of its strata. They are labeled by the term of $(\pmb{\Delta}^n)^{\otimes s+1}$ which they determine after taking the image of $\mathrm{AW}^{\circ s}$ under the functor $C_{-*}^{cell}$. Proposition~\ref{alg:prop:cell-ci1ik} implies that the top-dimensional strata defined by the inequalities
\[  \cdots > z_{i_k} > (1/2)^k > z_{i_k+1} > \cdots \]
are labeled by
\[ [0 < \cdots < i_1 ] \otimes [i_1 < \cdots < i_2] \otimes \cdots \otimes [ i_s < \cdots < n] \ . \]

\begin{proposition} \label{alg:prop:descri-subdivi-delta}
\begin{enumerate}[label=(\roman*)]
\item The codimension $i$ strata of the $\mathrm{AW}^{\circ s}$-subdivision of $\Delta^n$ lying in the interior of $\Delta^n$ are in one-to-one correspondence with the $(s-i)$-overlapping $(s+1)$-partitions of $\Delta^n$. More generally, given a face $I \subset \Delta^n$, the strata of the $\mathrm{AW}^{\circ s}$-subdivision of $\Delta^n$ which are lying in the interior of $I$ and have codimension $i$ w.r.t. the dimension of $I$ are in one-to-one correspondence with the $(s-i)$-overlapping $(s+1)$-partitions of $I$.
\item Consider a codimension $i$ stratum of the $\mathrm{AW}^{\circ s}$-subdivision of $\Delta^n$ lying in the interior of $\Delta^n$. This stratum is defined by $s-i$ inequalities of the form
\[  \cdots > z_{i_k} > (1/2)^k > z_{i_k+1} > \cdots \ ,\]
and $i$ equalities of the form 
\[ \cdots > z_{i_k} = (1/2)^k > z_{i_k+1} > \cdots \ . \]
The labeling of this stratum can then be obtained under the following simple transformation rules~:
\begin{align*}
\cdots > z_{i_k} > (1/2)^k > z_{i_k+1} > \cdots &\longmapsto \cdots < i_k] \otimes [i_k < \cdots \ , \\
\cdots > z_{i_k} = (1/2)^k > z_{i_k+1} > \cdots &\longmapsto \cdots < i_k-1] \otimes [i_k < \cdots \ .
\end{align*}
\end{enumerate}
\end{proposition}

\noindent This recipe easily carries over to the case of strata lying in the boundary of $\Delta^n$. The $\mathrm{AW}$ and $\mathrm{AW}^{\circ 2}$ subdivisions of $\Delta^2$ are represented in figure~\ref{alg:fig:AW-AW-deux-delta-deux}.

\begin{figure}[h]
    \centering
    \begin{subfigure}{0.45\textwidth}
    \centering
       \includegraphics[scale=0.2]{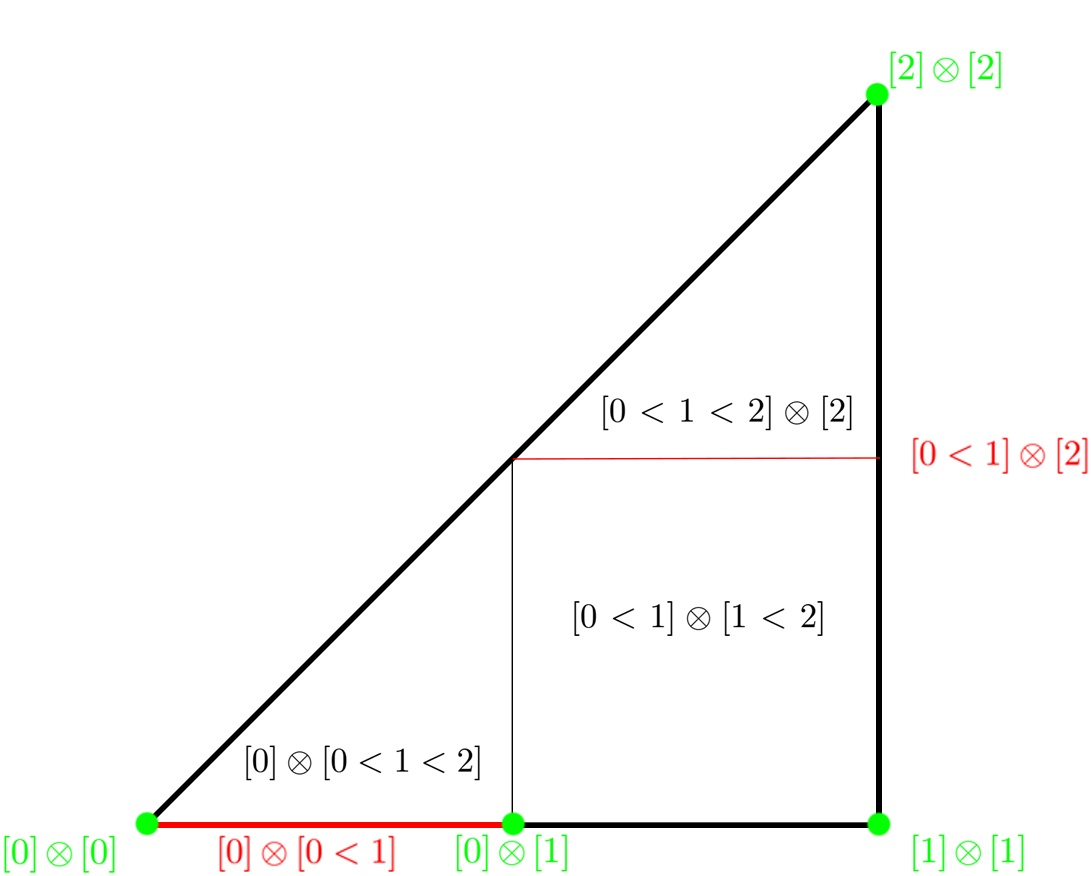}
    \end{subfigure} ~
    \begin{subfigure}{0.45\textwidth}
    \centering
        \includegraphics[scale=0.2]{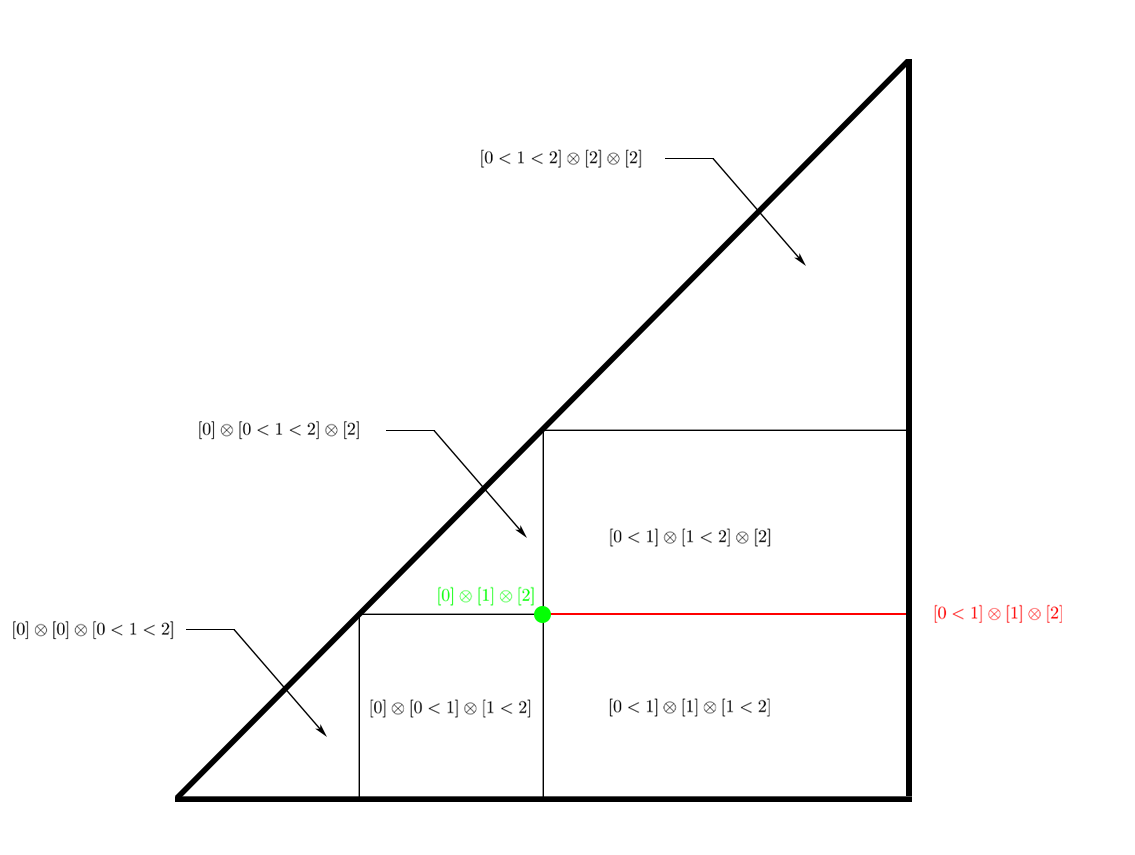}
    \end{subfigure} 
    \caption{The $\mathrm{AW}$ and $\mathrm{AW}^{\circ 2}$ subdivisions of $\Delta^2$} \label{alg:fig:AW-AW-deux-delta-deux}
\end{figure}

\subsubsection{The $\mathrm{AW}_{\pmb{a}}$-subdivisions of $\Delta^n$}

Let now $\pmb{a}$ be a sequence of real numbers $1 > a_1 > \cdots > a_{s} > 0$, where we denote $|\pmb{a}| := s$ the \emph{length of $\pmb{a}$}. We call such a sequence a \emph{dividing sequence}. We define \emph{the $\mathrm{AW}_{\pmb{a}}$-subdivision of $\Delta^n$} to be the subdivision obtained after dividing $\Delta^n$ by all hyperplanes $z_i = a_k$, for $1 \leqslant i \leqslant n$ and $1 \leqslant k \leqslant |\pmb{a}|$. We denote $\Delta^n_{\pmb{a}}$ for $\Delta^n$ endowed with its $\mathrm{AW}_{\pmb{a}}$-subdivision. The cells $C^{i_1,\dots,i_s}_{\pmb{a}}$ of $\Delta^n_{\pmb{a}}$ are again defined by the inequalities
\[  \cdots \geqslant z_{i_k} \geqslant a_k \geqslant z_{i_k+1} \geqslant \cdots \ ,\]
for $1 \leqslant k \leqslant |\pmb{a}|$. We define moreover the map $\mathrm{AW}_{\pmb{a}} : \Delta^n \rightarrow (\Delta^n)^{\times |\pmb{a}| + 1}$ as follows. Its projection on the $k$-th factor $\Delta^n$ of $(\Delta^n)^{\times |\pmb{a}|+1}$ restricted to the cell $C^{i_1,\dots,i_s}_{\pmb{a}} \subset \Delta^n_{\pmb{a}}$ is defined by the formula 
\begin{align*}
(z_1 , \dots , z_n) &\longmapsto ( 1 , \dots , 1 , \frac{(z_{i_{k-1}} - a_k)}{a_{k-1}-a_k} , \dots , \frac{(z_{i_{k}} - a_k)}{a_{k-1}-a_k} , 0 , \dots , 0 ) \ \text{for $1 \leqslant k \leqslant |\pmb{a}|$,}  \\
(z_1 , \dots , z_n) &\longmapsto ( 1 , \dots , 1 , z_{i_{|\pmb{a}|}+1}/a_{|\pmb{a}|} , \dots , z_n/a_{|\pmb{a}|} ) \ \text{for $k=|\pmb{a}|+1$,} 
\end{align*}
where we have set $a_0:= 1$.
We check in particular that for $\pmb{a} = 1/2 > \cdots > (1/2)^s$ we have $\mathrm{AW}_{\pmb{a}} := \mathrm{AW}^{\circ s}$. The maps $\mathrm{AW}_{\pmb{a}}$ are to be understood as generalizations of the maps $ \mathrm{AW}^{\circ s}$, that still realize the $|\pmb{a}|$-th iterate of the Alexander-Whitney coproduct under the functor $C_{-*}^{cell}$. In particular, the analogous statements of Propositions~\ref{alg:prop:pas-coassoc}, \ref{alg:prop:cell-ci1ik} and \ref{alg:prop:descri-subdivi-delta} still hold for the maps $\mathrm{AW}_{\pmb{a}}$.

We can now state a coassociativity-like property that the maps $\mathrm{AW}_{\pmb{a}}$ satisfy, which did not hold when only using the map $\mathrm{AW}$ as proven in Proposition~\ref{alg:prop:pas-coassoc}. For two dividing sequences $\pmb{a}$ and $\pmb{b}$, we write $\pmb{a} > \pmb{b}$ if $a_{|{\pmb{a}}|} > b_1$, and we then denote $\pmb{a} \cdot \pmb{b}$ the concatenation $a_1 > \cdots > a_{|{\pmb{a}}|} > b_1 > \cdots > b_{|{\pmb{b}}|}$.
\begin{proposition} \label{alg:prop:coassoc-a-b-c}
Let $\pmb{a}$, $\pmb{b}$ and $\pmb{c}$ be three dividing sequences such that $\pmb{a}>\pmb{b}>\pmb{c}$. Then, 
\[ \mathrm{AW}_{\pmb{a} \cdot \pmb{b} \cdot \pmb{c}} = ( \ide^{\times |\pmb{a}|} \times \mathrm{AW}_{{\pmb{b}'}} \times \ide^{\times |\pmb{c}|} ) \circ \mathrm{AW}_{{\pmb{a}} \cdot {\pmb{c}}} \ , \]
where $\pmb{b}'$ is the dividing sequence $1 > (b_1 - c_1)/(a_{\pmb{a}} - c_1) > \cdots > (b_{\pmb{b}} - c_1)/(a_{\pmb{a}} - c_1) > 0 $ which is obtained from $\pmb{b}$ by shifting by $c_1$ and then rescaling by $1/(a_{\pmb{a}} - c_1)$.
\end{proposition}
\noindent This proposition will be used in subsection \ref{geo:sss:below-break-bound} of part \ref{p:geo}. We illustrate it on the simplex $\Delta^2$ in figure~\ref{alg:fig:subdivisions-coassoci}, where $\pmb{a} := 6/7,5/7$, $\pmb{b}:= 4/7,3/7$ and $\pmb{c}:= 2/7,1/7$, which implies that $\pmb{b}':= 2/3,1/3$. On the left is represented the $\mathrm{AW}_{\pmb{a} \cdot \pmb{b} \cdot \pmb{c}}$-subdivision of $\Delta^2$, in the middle its  $\mathrm{AW}_{{\pmb{a}} \cdot {\pmb{c}}}$-subdivision and on the right the subdivision induced by the map $( \ide^{\times |\pmb{a}|} \times \mathrm{AW}_{{\pmb{b}'}} \times \ide^{\times |\pmb{c}|} ) \circ \mathrm{AW}_{{\pmb{a}} \cdot {\pmb{c}}}$, where the red lines represent the subdivision induced by $\mathrm{AW}_{{\pmb{b}'}}$. The left and right subdivisions then coincide.
\begin{figure}[h] 
    \centering
    \begin{subfigure}{0.3\textwidth}
    \centering
       \includegraphics[scale=0.2]{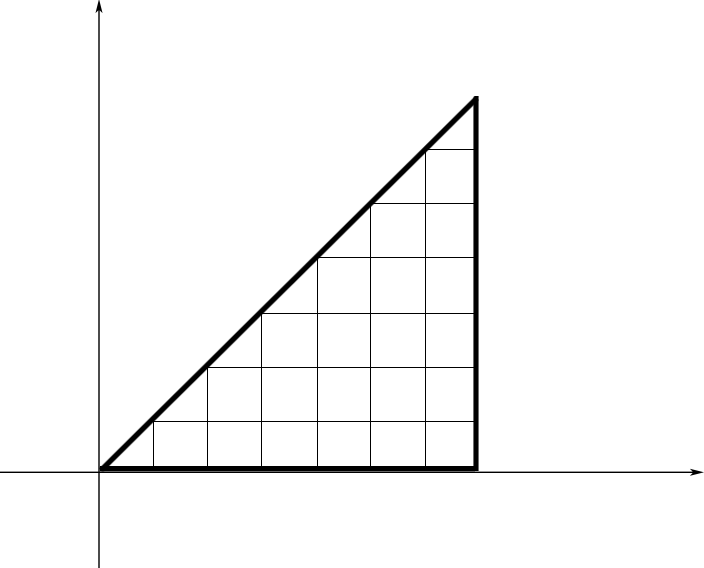}
    \end{subfigure} ~
    \begin{subfigure}{0.3\textwidth}
    \centering
        \includegraphics[scale=0.2]{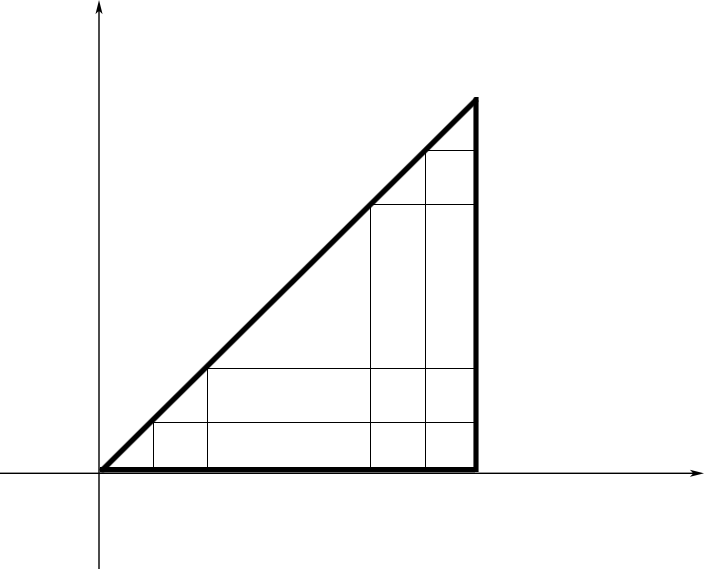}
    \end{subfigure} ~
    \begin{subfigure}{0.3\textwidth}
    \centering
        \includegraphics[scale=0.2]{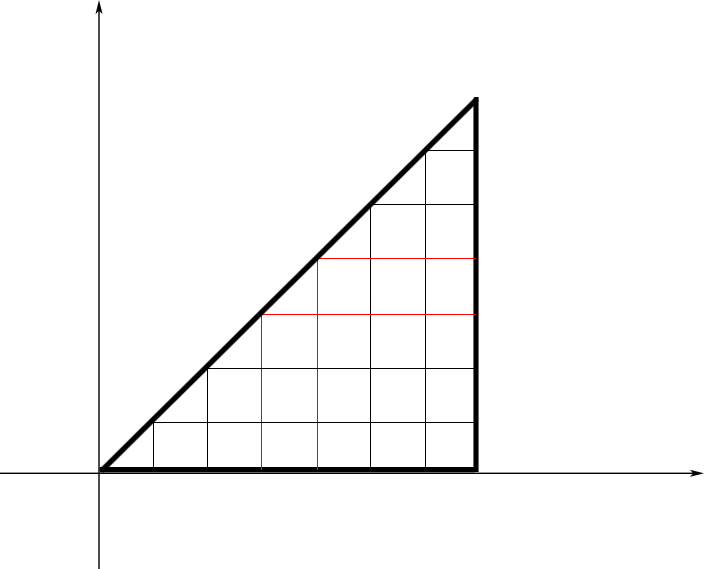}
    \end{subfigure}
    \caption{} \label{alg:fig:subdivisions-coassoci}
\end{figure}

\subsection{The $n$-multiplihedra $n-J_m$} \label{alg:ss:poly-n-Jm}

\subsubsection{The multiplihedra} \label{alg:sss:multiplihedra}

The polytopes encoding \Ainf -morphisms between \Ainf -algebras are the multiplihedra $J_m, \ m\geqslant 1$ : they form a collection $\{ J_m \}_{m \geqslant 1}$ which is a $( \{ K_m \} , \{ K_m \})$-operadic bimodule whose image under the functor $C_{-*}^{cell}$ is the $(\Ainf , \Ainf )$-operadic bimodule \infmor . The faces of codimension $i$ of $J_m$ are labeled by all possible broken two-colored trees obtained by blowing-up $i$ times the two-colored $m$-corolla. See for instance~\cite{mazuir-I} for pictures of the multiplihedra $J_1$, $J_2$ and $J_3$.
The multiplihedra $J_m$ can moreover be realized as the compactifications of moduli spaces of stable two-colored metric ribbon trees $\overline{\mathcal{CT}}_m$, where each $\mathcal{CT}_m$ is seen as the unique $(m-1)$-dimensional stratum of $\overline{\mathcal{CT}}_m$. 

\subsubsection{The $n$-multiplihedra $n-J_m$} \label{alg:sss:n-multiplihedra}

Consider the polytope $\Delta^n \times J_m$ for $n \geqslant 0$ and $m \geqslant 1$. It is the most natural candidate for a polytope encoding $n$-morphisms between \Ainf -algebras. However, it does not fulfill that property as it is. Indeed, its faces correspond to the data of a face of $\Delta^n$, that is of some $I \subset \Delta^n$, and of a face of $J_m$, that is of a broken two-colored tree obtained by blowing-up several times the two-colored $m$-corolla. This labeling is too coarse, as it does not contain the following trees, that appear in the \Ainf -equations for $n$-morphisms
\[
 \eqainfnmorphtrois \ .
\]

We resolve this issue by constructing a refined polytopal subdivision of $\Delta^n \times J_m$. Consider a face $F$ of $J_m$ labeled by a broken two-colored tree $t_{br,c}$ such that exactly $s$ unbroken two-colored trees $t_c^i$ for $r=1,\dots,s$ appear in $t_{br,c}$. We see the trees $t_c^r$ as ordered from left to right in $t_{br,c}$, write $i_r$ for the number of incoming edges of $t_c^{br}$ located above $t_c^r$ in $t_{br,c}$, and recall that $t_{br,c}$ has arity $m$. We have in particular that $i_1 + \cdots + i_s = m$. Define the dividing sequence $\pmb{a}_{t_{br,c}}$ of length $s-1$ as
\[ \frac{i_1 + \cdots + i_{s-1}}{m} > \frac{i_1 + \cdots + i_{s-2}}{m} > \cdots > \frac{i_1}{m} \ . \]
We then refine the polytopal subdivision of $\Delta^n \times F$ into $\Delta^n_{\pmb{a}_{t_{br,c}}} \times F$, where $\Delta^n_{\pmb{a}_{t_{br,c}}}$ denotes $\Delta^n$ endowed with its $\mathrm{AW}_{\pmb{a}_{t_{br,c}}}$-subdivision. This refinement process is moreover consistent : for two faces $F' \subset F$, the subdivision on $\Delta^n$ defined by the face $F'$ is a refinement of the subdivision on $\Delta^n$ defined by the face $F$.

\begin{definition} \label{alg:def:n-multipl}
The \emph{$n$-multiplihedra} are defined to be the polytopes $\Delta^n \times J_m$ endowed with the previous polytopal subdivision. We denote them $n-J_m$.
\end{definition}

\noindent See some examples in figures~\ref{alg:fig:delta-un-j-deux},~\ref{alg:fig:delta-deux-j-deux}~and~\ref{alg:fig:delta-un-j-trois}. We illustrate definition \ref{alg:def:n-multipl} with the construction of the $2$-multiplihedron $\Delta^2 \times J_2$ depicted on figure~\ref{alg:fig:delta-deux-j-deux}. The polytope $\Delta^2$ has one 2-dimensional face labeled by $[0<1<2]$ and three 1-dimensional faces labeled by $[0<1]$, $[1<2]$ and $[0<2]$. The polytope $J_2$ has one 1-dimensional face labeled by \arbreopdeuxmorph\ and has two 0-dimensional faces labeled by \arbreopdeuxunmorphbis\ and \arbreopdeuxdeuxmorphbis ~. Consider now the product polytope $\Delta^2 \times J_2$. Its has one unique 3-dimensional face labeled by $[0<1<2] \times \arbreopdeuxmorph$ and five 2-dimensional faces. The faces $[0<1] \times \arbreopdeuxmorph$, $[1<2] \times \arbreopdeuxmorph$, $[0<2] \times \arbreopdeuxmorph$ and $[0<1<2] \times \arbreopdeuxunmorphbis$ that are left unchanged under the construction of the previous paragraph, as they each feature only 1 unbroken two-colored tree. They respectively correspond to the faces A, B, F and G on figure~\ref{alg:fig:delta-deux-j-deux}. The fifth face is the face $[0<1<2] \times \arbreopdeuxdeuxmorphbis$. It features 2 unbroken two-colored trees : we thus have to refine the polytopal subdivision of $\Delta^2 \times \arbreopdeuxdeuxmorphbis$ into $\Delta^2_\mathrm{AW} \times \arbreopdeuxdeuxmorphbis$. This refinement produces the strata $([0] \otimes [0<1<2]) \times \arbreopdeuxdeuxmorphbis$, $([0<1] \otimes [1<2]) \times \arbreopdeuxdeuxmorphbis$ and $([0<1<2] \otimes [2]) \times \arbreopdeuxdeuxmorphbis$, which respectively correspond to the labels C, D and E on figure~\ref{alg:fig:delta-deux-j-deux}. This concludes the construction of the $2$-multiplihedron $\Delta^2 \times J_2$.

\begin{figure}[h]
\includegraphics[scale=0.03]{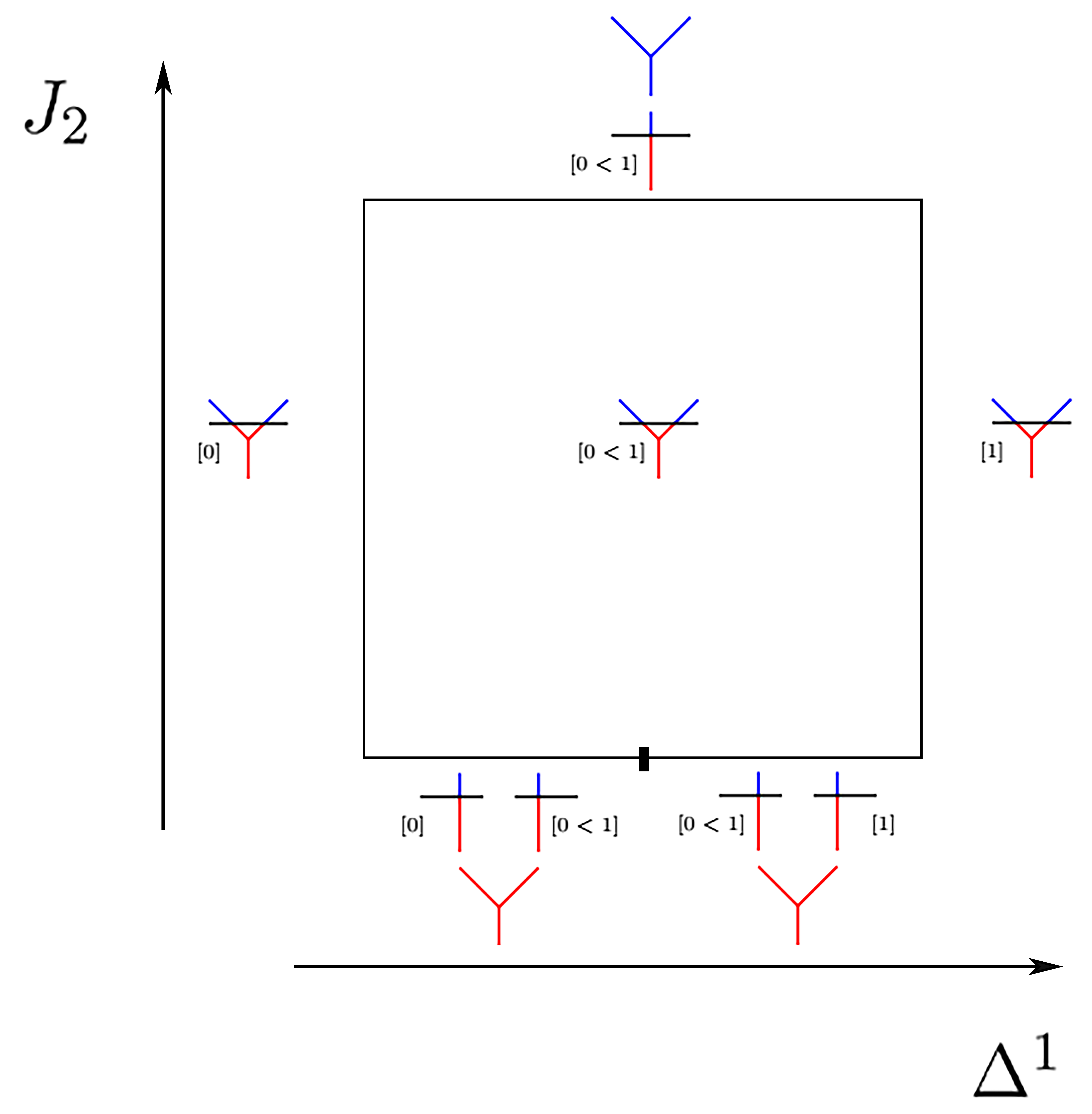}
\caption{The $1$-multiplihedron $\Delta^1 \times J_2$} \label{alg:fig:delta-un-j-deux}
\end{figure} 

\begin{figure}[h]
\includegraphics[scale=0.1]{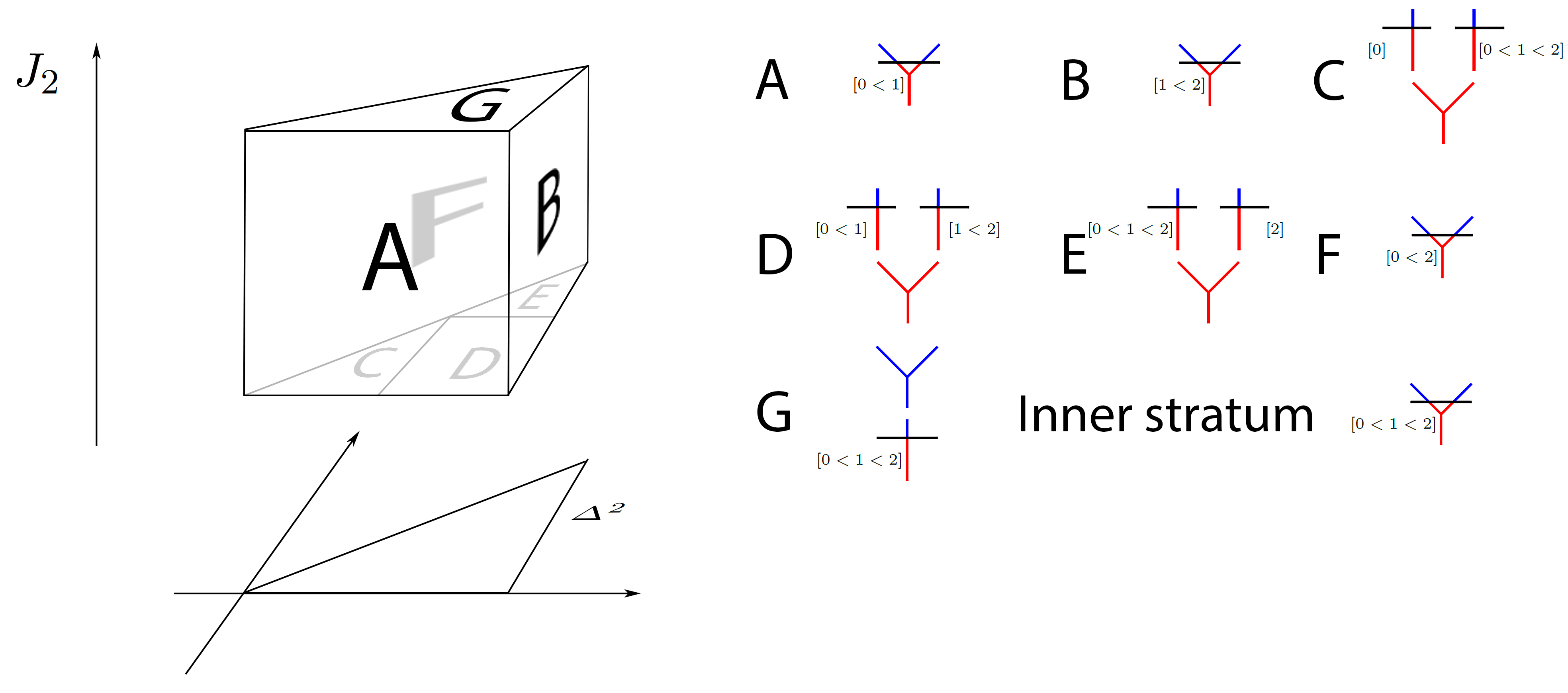}
\caption{The $2$-multiplihedron $\Delta^2 \times J_2$} \label{alg:fig:delta-deux-j-deux}
\end{figure}

\begin{figure}[h]
\includegraphics[scale=0.13]{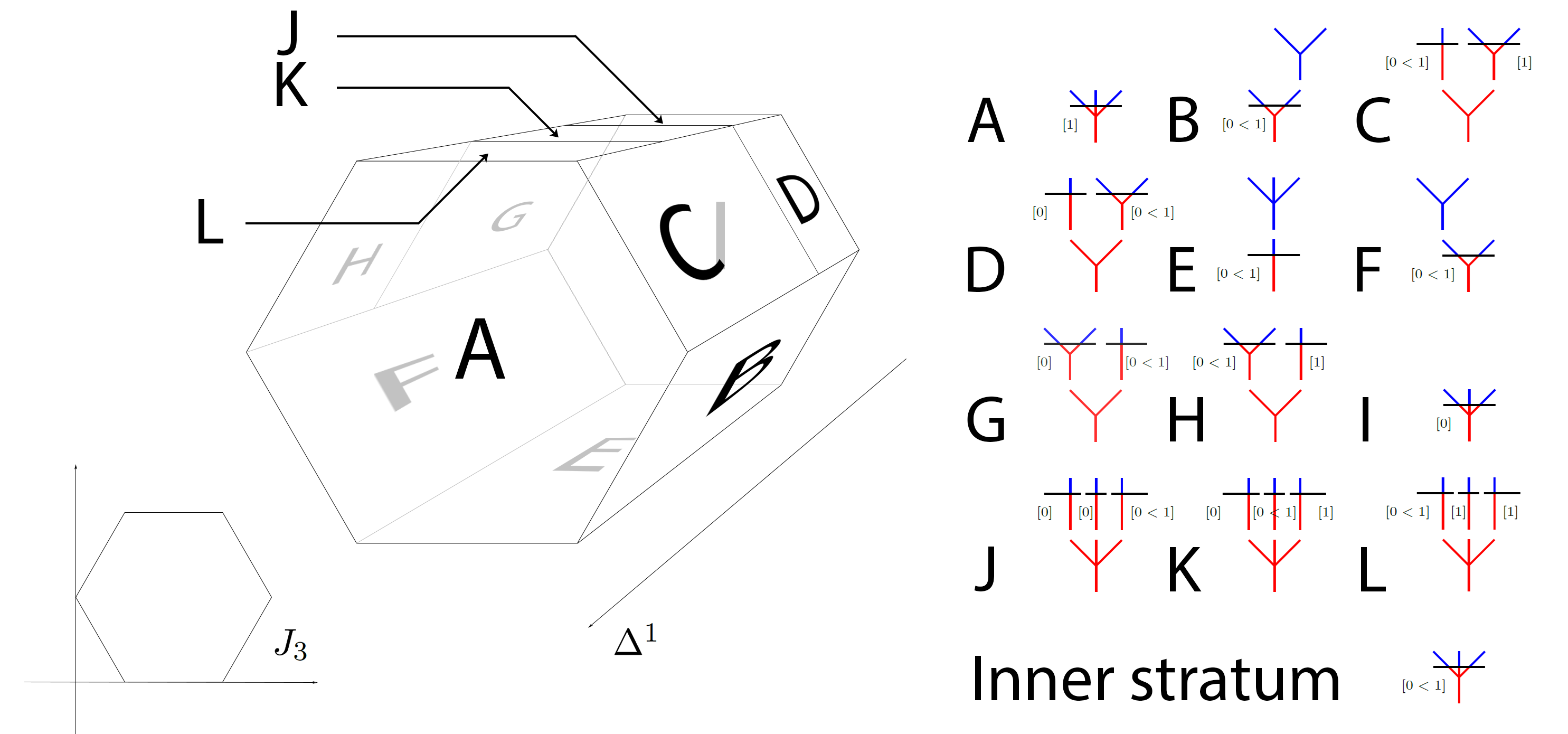}
\caption{The $1$-multiplihedron $\Delta^1 \times J_3$} \label{alg:fig:delta-un-j-trois}
\end{figure}

\subsubsection{The $n$-multiplihedra encode $n-\Ainf$-morphisms} \label{alg:sss:n-multiplihedra-encode}

Now in which sense do these polytopes encode $n-\Ainf$-morphisms ? Note first that the collection $\{ n -J_m \}_{m \geqslant 1}$ is not a $( \{ K_m \} , \{ K_m \})$-operadic bimodule ! Indeed, a $( \{ K_m \} , \{ K_m \})$-operadic bimodule structure would for instance make appear a stratum labeled by
\[
 \eqainfnmorphquatre \ ,
\]
where $I_1 \cup \dots \cup I_s = \Delta^n$ is an overlapping partition of $\Delta^n$. This stratum does not appear in the polytopal subdivision of $n -J_m$.
Hence these polytopes do not recover the $(\Ainf , \Ainf )$-operadic bimodule \infmorn .

However, the polytopal subdivision of $n-J_m$ still contains enough combinatorics to recover a $n$-morphism. This polytope has a unique $(n + m-1)$-dimensional cell $[n-J_m]$, which is labeled by \ainfnmorphundeltan . By construction :

\begin{proposition} \label{alg:prop:n-multipl}
The boundary of the cell $[n-J_m]$ is given by 
\[ \resizebox{\hsize}{!}{$\displaystyle{\partial^{sing} [n-J_m] \cup \bigcup_{\substack{h+k=m+1 \\ 1 \leqslant i \leqslant k \\ h \geqslant 2}} [n-J_k] \times_i [K_h] \cup \bigcup_{\substack{ i_1 + \dots + i_s = m \\ I_1 \cup \dots \cup I_s = \Delta^n \\ s \geqslant 2}} [K_s] \times [\mathrm{dim}(I_1) - J_{i_1}] \times \cdots \times [\mathrm{dim}(I_s) - J_{i_s}] \ ,}$} \]
where $I_1 \cup \dots \cup I_s = \Delta^n$ is an overlapping partition of $\Delta^n$.
\end{proposition}

Details on the orientation of the top dimensional strata in this boundary are worked out in section~\ref{alg:ss:signs-poly-n-Jm}.
Note moreover that the collection $\{ n -J_m \}_{n \geqslant 0}$ is a cosimplicial polytope. This implies that the image of each cell $[\mathrm{dim}(I)-J_m]$ under the functor $C^{cell}_{-*}$ yields an element whose boundary is exactly given by the \Ainf -equations for $n$-morphisms. It is in that sense that the $n-J_m$ encode $n$-morphisms.
The previous boundary formula also implies that the $n-J_m$ will constitute a good parametrizing space for constructing moduli spaces in symplectic topology, whose count should give rise to $n$-morphisms between Floer complexes.

\section{$n-\Omega B As$-morphisms} \label{alg:s:n-ombas-morph}

The multiplihedra $J_m$ can be realized by compactifying the moduli spaces of stable two-colored metric ribbon trees $\overline{\mathcal{CT}}_m$ and come with two cell decompositions. The first one consists in considering each $\mathcal{CT}_m$ as a $(m-1)$-dimensional stratum and encodes the operadic bimodule \infmor . The second one is obtained by considering the stratification of the moduli spaces $\mathcal{CT}_m$ by two-colored stable ribbon tree types, and encodes the operadic bimodule $\Omega B As - \mathrm{Morph}$. The $\ombas$-cell decomposition is moreover a refinement of the \Ainf -cell decomposition. As a consequence, there exists a morphism of operadic bimodules $\infmor \rightarrow \Omega B As - \mathrm{Morph}$, as shown in~\cite{mazuir-I}.
It is hence sufficient to construct an $\Omega B As$-morphism between \ombas -algebras to then naturally get an \Ainf -morphism between \Ainf -algebras.

We define in this section \emph{$n-\Omega B As$-morphisms} between \ombas -algebras. Building on the viewpoint of the previous paragraph, we then explain how, by refining the cell decomposition of the polytope $n-J_m$, we get a new cell decomposition encoding $n-\Omega B As$-morphisms. This construction yields in particular a morphism of operadic bimodules $\infmorn \rightarrow n-\Omega B As - \mathrm{Morph}$. All sign computations are moreover postponed to section~\ref{alg:ss:signs-n-ombas}.

\subsection{$n-\Omega B As$-morphisms} \label{alg:ss:n-ombas-morph}

\subsubsection{Recollections on $\Omega B As$-morphisms} \label{alg:sss:recoll-ombas-morph}

$\Omega B As$-morphisms are the morphisms between $\Omega B As$-algebras encoded by the quasi-free operadic bimodule generated by all two-colored stable ribbon trees
\[ \Omega B As - \mathrm{Morph} := \mathcal{F}^{\Omega B As, \Omega B As}( \arbreopunmorph , \arbrebicoloreL , \arbrebicoloreM , \arbrebicoloreN , \cdots , SCRT_n,\cdots) \ .  \]
A two-colored stable ribbon tree $t_g$ whose underlying stable ribbon tree $t$ has $e(t)$ inner edges, and such that its gauge crosses $j$ vertices of $t$, has degree $|t_g| :=  j-1-e(t)$.

The differential of a two-colored stable ribbon tree $t_g$ is given by the signed sum of all two-colored stable ribbon trees obtained from $t_g$ under the rule prescribed by the top dimensional strata in the boundary of $\overline{\mathcal{CT}}_n(t_g)$. : the gauge moves to cross exactly one additional vertex of the underlying stable ribbon tree (gauge-vertex) ; an internal edge located above the gauge or intersecting it breaks or, when the gauge is below the root, the outgoing edge breaks between the gauge and the root (above-break) ; edges (internal or incoming) that are possibly intersecting the gauge, break below it, such that there is exactly one edge breaking in each non-self crossing path from an incoming edge to the root (below-break) ; an internal edge that does not intersect the gauge collapses (int-collapse). 

\subsubsection{$n-\Omega B As$-morphisms} \label{alg:sss:n-ombas-morph}

\begin{definition} \label{alg:def:op-bimod-ombas-n}
\emph{$n-\Omega B As$-morphisms} are the higher morphisms between $\Omega B As$-algebras encoded by the quasi-free operadic bimodule generated by all pairs (face $I \subset \Delta^n$ , two-colored stable ribbon tree), 
\[ n-\Omega B As - \mathrm{Morph} := \mathcal{F}^{\Omega B As, \Omega B As}( \arbreopunmorphn , \arbrebicoloreLn , \arbrebicoloreMn , \arbrebicoloreNn , \cdots , (I , SCRT_n) ,\cdots ; I \subset \Delta^n) \ .  \]
An operation $t_{I,g} := (I,t_g)$ is defined to have degree $| t_{I,g} | := |I| + |t_g|$.
The differential of $t_{I,g}$ is given by the rule prescribed by the top dimensional strata in the boundary of $\overline{\mathcal{CT}}_m(t_g)$ combined with the algebraic combinatorics of overlapping partitions, added to the simplicial differential of $I$, i.e. 
\[ \partial t_{I,g} = t_{\partial^{sing}I,g} + \pm (\partial^{\overline{\mathcal{CT}}_m} t_g)_{I} \ . \]
\end{definition}

We refer to section~\ref{alg:ss:signs-n-ombas} for a more complete definition and sign conventions. The sign computations are in particular more involved, as we did not describe an ad hoc construction analogous to the shifted bar construction as in the \Ainf\ case. We also point out that the symbol \arbrebicoloreLn\ used here is the same as the one used for the arity 2 generating operation of \infmorn . It will however be clear from the context what \arbrebicoloreLn\ stands for in the rest of this paper.
We moreover compute the differential in the following instance
\begin{align*}
 | \exarbrebicoloreA | = &-5 \ , \\
 \partial (\exarbrebicoloreA) = &\pm \exarbrebicoloreB \pm \exarbrebicoloreC \pm \exarbrebicoloreD \\ 
 &\pm \exarbrebicoloreE \pm \exarbrebicoloreF \pm \exarbrebicoloreG \pm \exarbrebicoloreH \pm \exarbrebicoloreI \\
 & \pm \exarbrebicoloreJ \pm \exarbrebicoloreK \pm \exarbrebicoloreL \ .
\end{align*}

\subsubsection{From $n-\Omega B As$-morphisms to $n-\Ainf$-morphisms} \label{alg:sss:n-ombas-to-n-ainf}

A $n-\Omega B As$-morphism between two $\Omega B As$-algebras naturally yields a $n-\Ainf$-morphism between the induced \Ainf -algebras :
\begin{proposition} \label{alg:prop:morph-op-bimod}
There exists a morphism of $(\Ainf , \Ainf )$-operadic bimodules
$\infmorn \rightarrow n-\Omega B As - \mathrm{Morph}$
given on the generating operations of \infmorn\ by
\[ f_{I,m} \longmapsto \sum_{t_g \in CBRT_m} \pm f_{I,t_g} \ , \]
where $CBRT_m$ denotes the set of two-colored binary ribbon trees of arity $m$.
\end{proposition}

This proposition is proven in subsection~\ref{alg:sss:morph-op-bimod-infmorn}.
Note that the collection of operadic bimodules $\{ n-\Omega B As - \mathrm{Morph} \}_{n \geqslant 0}$ is once again a cosimplicial operadic bimodule, where the cofaces and codegeneracies are as in subsection~\ref{alg:sss:cosimpl-dg-cog}. This sequence of morphisms of operadic bimodules defines then in fact a morphism of cosimplicial operadic bimodules
\[ \{ \infmorn \}_{n \geqslant 0} \longrightarrow \{ n-\Omega B As - \mathrm{Morph} \}_{n \geqslant 0} \ . \]

\subsection{The $n$-multiplihedra encode $n-\Omega B As$-morphisms} \label{alg:ss:poly-encode-n-ombas}

\subsubsection{The $n-\Omega B As$-cell decomposition of $\Delta^n \times \overline{\mathcal{CT}}_m$} \label{alg:sss:n-ombas-subd-delta-n}

The polytopes encoding $n-\Ainf$-morphisms have been defined to be the polytopes $\Delta^n\times J_m$ endowed with a refined polytopal subdivision induced by the maps $\mathrm{AW}_{\pmb{a}}$. These refined subdivisions incorporate the combinatorics of $i$-overlapping $s$-partitions in the boundary of the polytopes $\Delta^n\times J_m$.  Consider now the multiplihedra $J_m = \overline{\mathcal{CT}}_m$ endowed with its \ombas -cell decomposition, i.e. its cell decomposition by broken stable two-colored ribbon tree type. We can define a refined cell decomposition on the product CW-complex $\Delta^n \times \overline{\mathcal{CT}}_m$ following the construction of subsection \ref{alg:sss:n-multiplihedra}. Each stratum $\mathcal{CT}_m(t_{br,c})$ of the moduli space $\overline{\mathcal{CT}}_m$ determines again a dividing sequence $\pmb{a}_{t_{br,c}}$ obtained from the unbroken two-colored trees of the two-colored tree $t_{br,c}$ labeling it. We then refine the cell decomposition of $\Delta^n \times \mathcal{CT}_m(t_{br,c})$ into $\Delta^n_{\pmb{a}_{t_{br,c}}} \times \mathcal{CT}_m(t_{br,c})$. This refinement process can again be done consistently in order to obtain a refined cell decomposition of $\Delta^n \times \overline{\mathcal{CT}}_m$. 

\begin{definition}
We define the \emph{$n-\Omega B As$-cell decomposition} of the $n$-multiplihedron $\Delta^n \times \overline{\mathcal{CT}}_m$ to be the cell decomposition described in the previous paragraph.
\end{definition}

\noindent See some examples in figures~\ref{alg:fig:delta-un-j-deux-fin}~and~\ref{alg:fig:delta-un-j-trois-fin}. By construction, the $n-\Omega B As$-cell decomposition of $\Delta^n \times \overline{\mathcal{CT}}_m$ is moreover a refinement of the $n-\Ainf$-cell decomposition of $\Delta^n \times \overline{\mathcal{CT}}_m$.

\begin{figure}[h]
\includegraphics[scale=0.03]{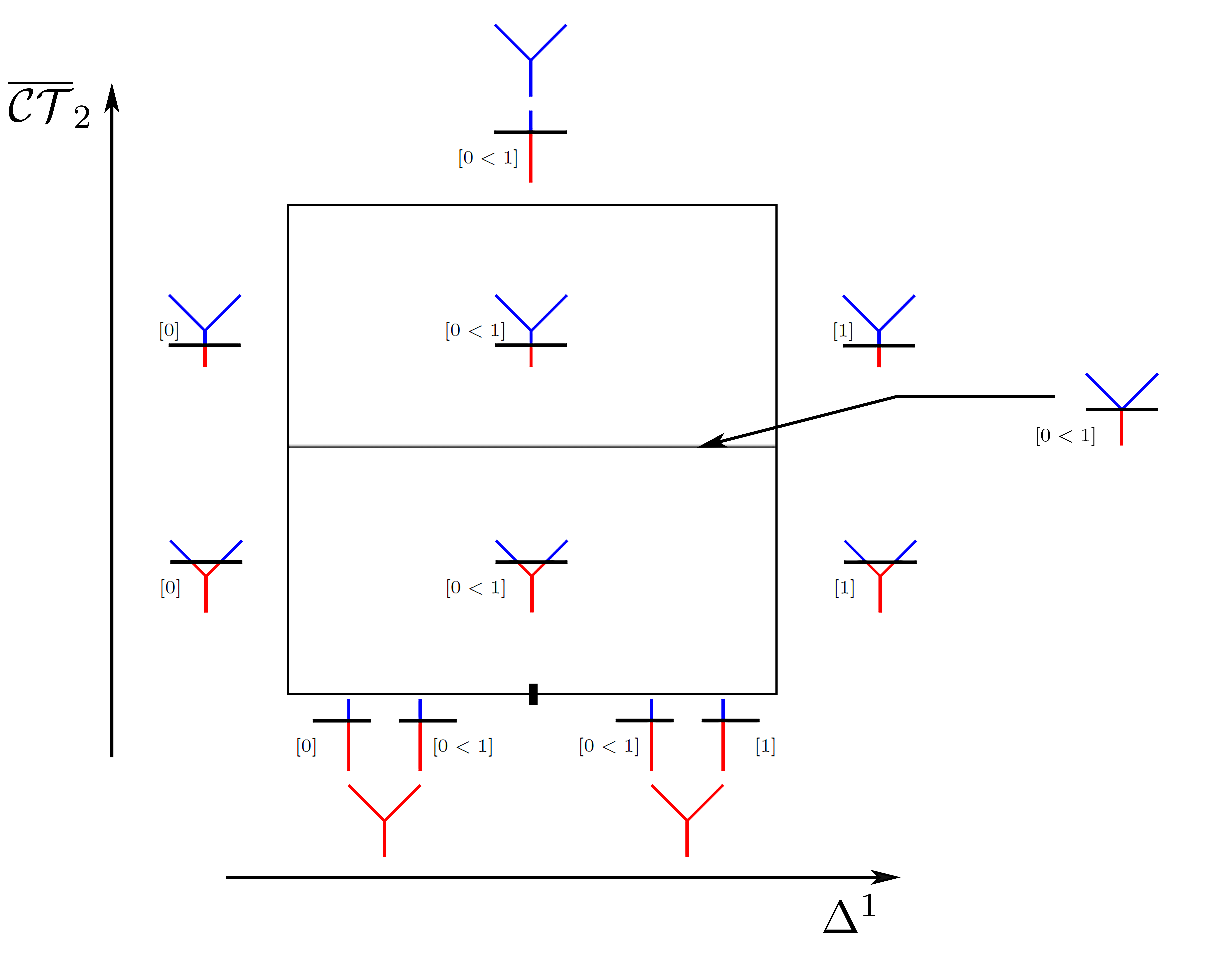}
\caption{The $1-\Omega B As$-cell decomposition of $\Delta^1 \times \overline{\mathcal{CT}}_2$}
\label{alg:fig:delta-un-j-deux-fin}
\end{figure} 

\begin{figure}[h]
\includegraphics[scale=0.15]{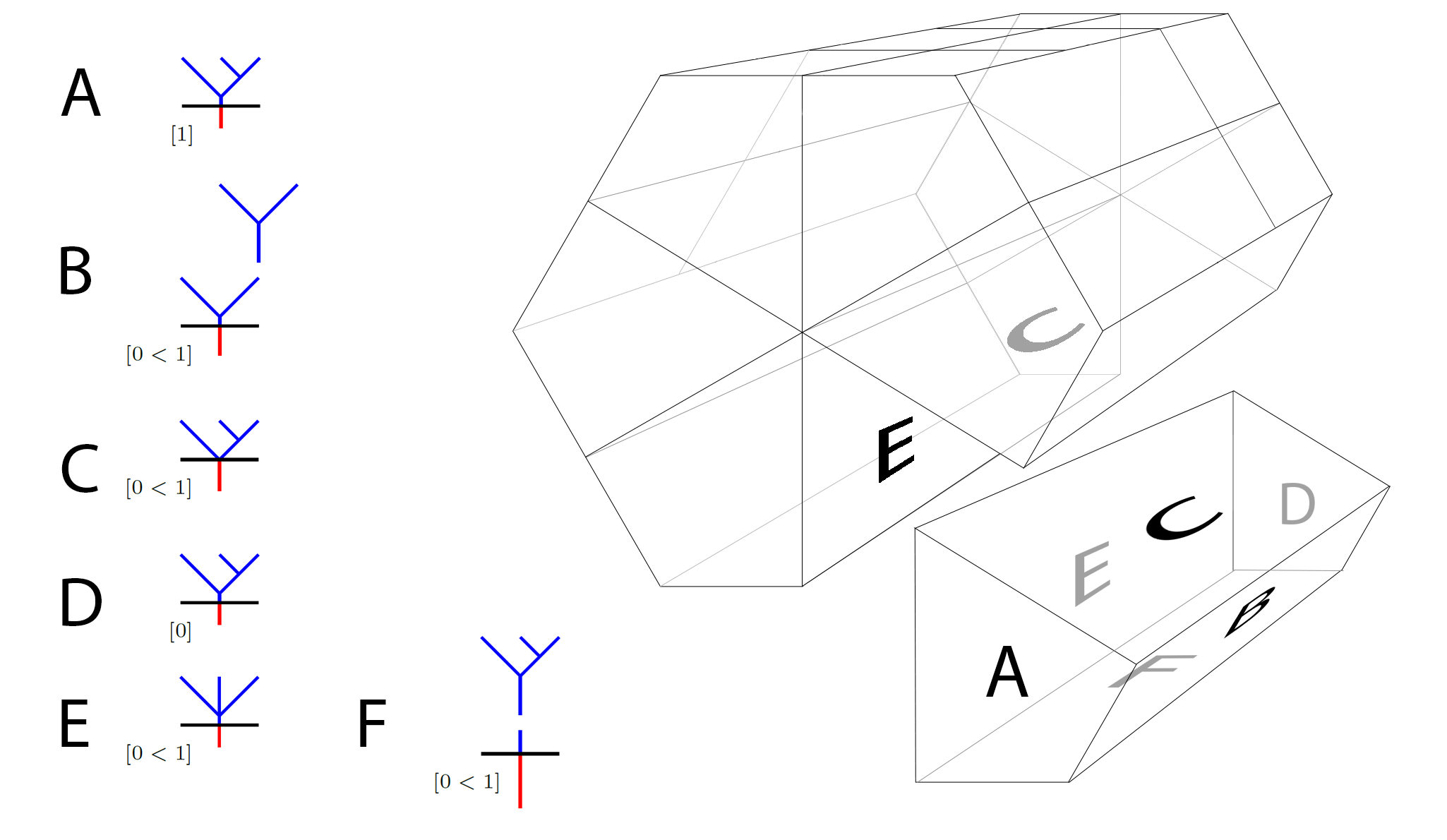}
\caption{The $1-\Omega B As$-cell decomposition of $\Delta^1 \times \overline{\mathcal{CT}}_3$} \label{alg:fig:delta-un-j-trois-fin}
\end{figure}

\subsubsection{These CW-complexes encode $n-\Omega B As$-morphisms} \label{alg:sss:these-poly-encode-ombas}

Consider the associahedra $K_m = \overline{\mathcal{T}}_m$ endowed with their \ombas -cell decompositions. We endow moreover the spaces $\Delta^n \times \overline{\mathcal{CT}}_m$ with their $n-\Omega B As$-cell decompositions. As in the \Ainf\ case, the collection of CW-complexes $\{ \Delta^n \times \overline{\mathcal{CT}}_m \}_{m \geqslant 1}$ is not a $( \{ \overline{\mathcal{T}}_m \} , \{ \overline{\mathcal{T}}_m \})$-operadic bimodule. Carrying over the details of subsection~\ref{alg:sss:n-multiplihedra-encode}, it contains however enough combinatorics to recover a $n-\Omega B As$-morphism. What's more, the collection $\{ \Delta^n \times \overline{\mathcal{CT}}_m \}_{n \geqslant 0}$ is again a cosimplicial CW-complex.

\subsection{Résumé} \label{alg:ss:resume-n-ombas}

The higher homotopies or $n$-morphisms extending the notion of \Ainf -morphisms and \Ainf -homotopies between \Ainf -algebras are defined to be the morphisms of dg-coalgebras
\[ \pmb{\Delta}^n \otimes T(sA) \longrightarrow T(sB) \ . \]
From an operadic viewpoint, they are naturally encoded by the operadic bimodule,
\[ \infmorn = \mathcal{F}^{\Ainf , \Ainf}(\arbreopunmorphn , \arbreopdeuxmorphn , \arbreoptroismorphn , \arbreopquatremorphn , \cdots ; I \subset \Delta^n ) \ .  \]
where the differential is defined as 
\[ [ \partial , \ainfnmorphun ] = \sum_{j=0}^{\mathrm{dim}(I)} (-1)^j \eqainfnmorphun + \sum_{I_1 \cup \cdots \cup I_s = I} \pm \eqainfnmorphtrois
 + \sum \pm \eqainfnmorphdeux \ . \]
The combinatorics of this differential are encoded by new families of polytopes called the $n$-multipli\-hedra, which are the data of the polytopes $\Delta^n \times J_m$ together with a polytopal subdivision induced by the maps $\mathrm{AW}_{\pmb{a}}$. They will constitute a good parametrizing space for constructing moduli spaces in symplectic topology, whose count should recover a $n$-morphism between Floer complexes.

On the other side, the natural $n$-morphisms extending the notion of $\Omega B As$-morphisms are defined by adapting the operadic viewpoint on $n-\Ainf$-morphisms. They are naturally encoded by the operadic bimodule,
\[ n-\Omega B As - \mathrm{Morph} = \mathcal{F}^{\Omega B As, \Omega B As}( \arbreopunmorphn , \arbrebicoloreLn , \arbrebicoloreMn , \arbrebicoloreNn , \cdots ,( I , SCRT_m ) ,\cdots ; I \subset \Delta^n) \ ,  \]
where the differential is again defined as a signed sum prescribed by a rule on two-colored trees combinatorics combined with the algebraic combinatorics of overlapping partitions, added to the simplicial differential. This differential is encoded in the data of the polytopes $\Delta^n \times J_m$ endowed with a refined cell decomposition induced by two-colored stable ribbon tree types and the maps $\mathrm{AW}_{\pmb{a}}$. It is moreover sufficient to construct a $n-\Omega B As$-morphism between $\Omega B As$-algebras in order to recover a $n-\Ainf$-morphism between the induced \Ainf -algebras, thanks to the morphism of operadic bimodules
\[ \infmorn \longrightarrow n-\Omega B As - \mathrm{Morph} \ . \]
We show in part~\ref{p:geo} that the previous CW-complexes constitute a good parametrizing space for moduli spaces in Morse theory, whose count will recover a $n-\ombas$-morphism between Morse cochain complexes.

\section{Signs for $n$-morphisms} \label{alg:s:signs-n-morph}

We now work out all the signs left uncomputed in the previous sections of this part. These computations will be done resorting to the basic conventions on signs and orientations that we were already using in~\cite{mazuir-I}, and that we briefly recall in the first section. In the next two sections, we display and explain the two natural sign conventions for $n - \Ainf$-morphisms ensuing from the bar construction viewpoint, and then show that one of these conventions is in fact contained in the polytopes $n - J_m$. We finally give a complete definition of the operadic bimodule $n-\Omega B As - \mathrm{Morph}$ and build the morphism of operadic bimodules $\infmorn \rightarrow n-\Omega B As - \mathrm{Morph}$ of Proposition~\ref{alg:prop:morph-op-bimod}.

\subsection{Conventions for signs and orientations} \label{alg:ss:conv-signs-or}

\subsubsection{Koszul sign rule} \label{alg:sss:Koszul-sign}

The formulae in this section will be written using the Koszul sign rule. We will moreover work exclusively with cohomological conventions.

Given $A$ and $B$ two dg \Z -modules, the differential on $A \otimes B$ is defined as 
\[ \partial_{A \otimes B} (a \otimes b) = \partial_A a \otimes b + (-1)^{|a|} a \otimes \partial_B b \ . \]
Given $A$ and $B$ two dg \Z -modules, we consider the graded \Z -module $\mathrm{Hom}(A,B)$ whose degree $r$ component is given by all maps $A \rightarrow B$ of degree $r$. We endow it with the differential 
\[ \partial_{\mathrm{Hom}(A,B)}(f) := \partial_B \circ f - (-1)^{|f|} f \circ \partial_A =: [\partial , f] \ . \]
Given $f : A \rightarrow A'$ and $g : B \rightarrow B'$ two graded maps between dg-\Z -modules, we set
\[ (f \otimes g) (a \otimes b) = (-1)^{|g| |a|} f(a) \otimes g(b) \ . \]
Finally, given $f : A \rightarrow A'$, $f' : A' \rightarrow A''$, $g : B \rightarrow B'$ and $g' : B' \rightarrow B''$, we define
\[ (f' \otimes g') \circ (f \otimes g) = (-1)^{|g'||f|}(f' \circ f) \otimes (g' \circ g) \ . \]
We check in particular that with this sign rule, the differential on a tensor product $A_1 \otimes \cdots \otimes A_n$ is given by
\[ \partial_{A_1 \otimes \cdots \otimes A_n} = \sum_{i=1}^n \ide_{A_1} \otimes \cdots \otimes \partial_{A_i} \otimes \cdots \otimes \ide_{A_n} \ . \]

\subsubsection{Tensor product of dg-coalgebras} \label{alg:sss:tensor-product-dg-cog}

Given $A$ and $B$ two dg \Z -modules, define the twist map $\tau : A \otimes B \rightarrow B \otimes A$,
\[ \tau (a \otimes b) = (-1)^{|a| |b|} b \otimes a \ . \]
Suppose now that $A$ and $B$ are dg-coalgebras, with respective coproducts $\Delta_A$ and $\Delta_B$. The tensor product $A \otimes B$ can then be endowed with a structure of dg-coalgebra whose coproduct is defined as 
\[ \Delta_{A \otimes B} := A \otimes B \underset{\Delta_A \otimes \Delta_B}{\longrightarrow} A \otimes A \otimes B \otimes B \underset{\ide_A \otimes \tau \otimes \ide_B}{\longrightarrow} (A \otimes B) \otimes (A \otimes B) \ , \]
and whose differential is the product differential 
\[ \partial_{A \otimes B} = \partial_A \otimes \ide_B + \ide_A \otimes \partial_B \ . \]

\subsubsection{Orientation of the boundary of a manifold with boundary} \label{alg:sss:or-boundary}

Let $(M,\partial M)$ be an oriented $n$-manifold with boundary. We choose to orient its boundary $\partial M$ as follows : given $x \in \partial M$, a basis $e_1,\dots,e_{n-1}$ of $T_x(\partial M)$, and an outward pointing vector $\nu \in T_xM$, the basis $e_1,\dots,e_{n-1}$ is positively oriented if and only if the basis $\nu,e_1,\dots,e_{n-1}$ is a positively oriented basis of $T_xM$. 

Under this convention, given two manifolds with boundary $K$ and $L$, the boundary of the product manifold $K \times L$ is then
\[ \partial (K \times L ) = \partial K \times L \cup (-1)^{\mathrm{dim} (K)} K \times \partial L \ , \]
where the $(-1)^{\mathrm{dim} (K)}$ sign means that the product orientation of $K \times \partial L$ differs from its orientation as the boundary of $K \times L$ by a $(-1)^{\mathrm{dim} (K)}$ sign.
This convention also recovers the classical singular and cubical differentials as detailed in~\cite{mazuir-I}~: 
\[ 
\partial \Delta^n = \bigcup_{i=0}^n (-1)^i \Delta^{n-1}_i \ \text{ and } \
\partial I^n = \bigcup_{i=1}^n (-1)^i ( I^{n-1}_{i,0} \cup - I^{n-1}_{i,1} ) \ . \]

\subsection{Signs for $n-\Ainf$-morphisms} \label{alg:ss:signs-n-ainf-morph}

We now work out the signs in the \Ainf -equations for $n-\Ainf$-morphisms, thus completing definition~\ref{alg:def:n-morph-ainf}. More precisely, we will unwind two sign conventions using the bar construction viewpoint. The impatient reader can straightaway jump to subsection~\ref{alg:sss:choice-convention-ainf} where the signs used in the rest of this paper are made explicit.

\subsubsection{Recollections on the bar construction and \Ainf -algebras} \label{alg:sss:recoll-bar-constr}

Let $A$ be a dg-\Z -module.  Define the suspension and desuspension maps 
\begin{align*}
s : A &\longrightarrow sA & w : sA &\rightarrow A \\
a &\longmapsto sa & sa &\longmapsto a  \ ,
\end{align*} 
which are respectively of degree $-1$ and $+1$. We verify that with the Koszul sign rule, 
\[ w^{\otimes m} \circ s^{\otimes m}  = (-1)^{\binom{m}{2}} \ide_{A^{\otimes m}} \ . \]
Then, note for instance that a degree $2-m$ map $m_m : A^{\otimes m} \rightarrow A$ yields a degree $+1$ map $b_m := s m_m w^{\otimes m} : (sA)^{\otimes m} \rightarrow sA$.

To the set of operations $b_m$ one can associate a unique coderivation $D$ on $\overline{T}(sA)$. We proved in~\cite{mazuir-I} using this viewpoint that the equation $D^2 = 0$ yields two sign conventions for the \Ainf -equations
\begin{align*} 
\left[ m_1 , m_m \right] &=  - \sum_{\substack{i_1+i_2+i_3=m \\ 2 \leqslant i_2 \leqslant n-1}} (-1)^{i_1i_2 + i_3} m_{i_1+1+i_3} (\ide^{\otimes i_1} \otimes m_{i_2} \otimes \ide^{\otimes i_3} ) \ , \tag{A}  \\
\left[ m_1 , m_m \right] &= - \sum_{\substack{i_1+i_2+i_3=m \\ 2 \leqslant i_2 \leqslant n-1}} (-1)^{i_1 + i_2i_3} m_{i_1+1+i_3} (\ide^{\otimes i_1} \otimes m_{i_2} \otimes \ide^{\otimes i_3} ) \ , \tag{B} 
\end{align*}
and that these conventions are related by a $(-1)^{\binom{m}{2}}$ twist applied to the operation $m_m$, which comes from the formula $w^{\otimes m} \circ s^{\otimes m}  = (-1)^{\binom{m}{2}} \ide_{A^{\otimes m}}$. 

We will adopt the exact same approach to work out two sign conventions for $n - \Ainf$-morphisms in the following subsection : first by writing \Ainf -equations without signs using the viewpoint of a morphism between bar constructions $F : \pmb{\Delta}^n \otimes \overline{T}(sA) \rightarrow \overline{T}(sB)$, and secondly by unfolding the signs coming from the suspension and desuspension maps.

\subsubsection{The two conventions coming from the bar construction} \label{alg:sss:two-conv-bar}

The two conventions for the \Ainf -equations for $n-\Ainf$-morphisms are
\begin{align*} 
\left[ m_1 , f^{(m)}_I \right] =  \sum_{j=0}^{\mathrm{dim}(I)} (-1)^j f^{(m)}_{\partial_jI} &+ (-1)^{|I|} \sum_{\substack{i_1+i_2+i_3=m \\ i_2 \geqslant 2}} (-1)^{i_1i_2 + i_3} f^{(i_1+1+i_3)}_I (\ide^{\otimes i_1} \otimes m_{i_2} \otimes \ide^{\otimes i_3}) \tag{A} \\ &- \sum_{\substack{i_1 + \cdots + i_s = m \\ I_1 \cup \cdots \cup I_s = I \\ s \geqslant 2 }} (-1)^{\epsilon_A} m_s ( f^{(i_1)}_{I_1} \otimes \cdots \otimes f^{(i_s)}_{I_s}) \ ,  \\
\left[ m_1 , f^{(m)}_I \right] = \sum_{j=0}^{\mathrm{dim}(I)} (-1)^j f^{(m)}_{\partial_jI} &+ (-1)^{|I|} \sum_{\substack{i_1+i_2+i_3=m \\ i_2 \geqslant 2}} (-1)^{i_1 + i_2i_3} f^{(i_1+1+i_3)}_I (\ide^{\otimes i_1} \otimes m_{i_2} \otimes \ide^{\otimes i_3})  \tag{B} \\ &- \sum_{\substack{i_1 + \cdots + i_s = m \\ I_1 \cup \cdots \cup I_s = I \\ s \geqslant 2 }} (-1)^{\epsilon_B} m_s ( f^{(i_1)}_{I_1} \otimes \cdots \otimes f^{(i_s)}_{I_s}) \ ,  
\end{align*}
which can we rewritten as
\begin{align*}
&\sum_{j=0}^{\mathrm{dim}(I)} (-1)^j f^{(m)}_{\partial_j I} + (-1)^{|I|} \sum_{i_1+i_2+i_3=m} (-1)^{i_1i_2 + i_3} f^{(i_1+1+i_3)}_I (\ide^{\otimes i_1} \otimes m_{i_2} \otimes \ide^{\otimes i_3}) \tag{A} \\ 
= &\sum_{ \substack{ i_1 + \cdots + i_s = m \\ I_1 \cup \cdots \cup I_s = I } } (-1)^{\epsilon_A} m_s ( f^{(i_1)}_{I_1} \otimes \cdots \otimes f^{(i_s)}_{I_s}) \ , \\
&\sum_{j=0}^{\mathrm{dim}(I)} (-1)^j f^{(m)}_{\partial_j I} + (-1)^{|I|} \sum_{i_1+i_2+i_3=m} (-1)^{i_1 + i_2i_3} f^{(i_1+1+i_3)}_I (\ide^{\otimes i_1} \otimes m_{i_2} \otimes \ide^{\otimes i_3}) \tag{B} \\ 
= &\sum_{ \substack{ i_1 + \cdots + i_s = m \\ I_1 \cup \cdots \cup I_s = I } } (-1)^{\epsilon_B} m_s ( f^{(i_1)}_{I_1} \otimes \cdots \otimes f^{(i_s)}_{I_s}) \ , \\
\end{align*} 
where 
\begin{align*}
\epsilon_A &= \sum_{j=1}^s (s-j)|I_j| + \sum_{j = 1}^s i_j \left( \sum_{k=j+1}^s (1-i_k - |I_k|) \right) \ , \\
\epsilon_B &= \sum_{j=1}^s \left(i_j \sum_{k=j+1}^s|I_k| \right) + \sum_{j=1}^s (s-j) (1-i_j - |I_j|) \ .
\end{align*}
These two sign conventions are equivalent : given a sequence of operations $m_m$ and $f^{(m)}_I$ satisfying equations (A), we check that the operations $m_m^\prime := (-1)^{\binom{m}{2}} m_m$ and 
$f_I^{\prime (m) } := (-1)^{\binom{m}{2}} f_I^{(m)}$ satisfy equations (B).

Consider now two dg-\Z -modules $A$ and $B$, together with a collection of degree $2-m$ maps $m_m : A^{\otimes m } \rightarrow A$ and $m_m : B^{\otimes m } \rightarrow B$ (we use the same notation for sake of readability), and a collection of degree $1-m+|I|$ maps $f^{(m)}_I : A^{\otimes m } \rightarrow B$. We associate to the maps $m_m$ the degree $+1$ maps $b_m := sm_m w^{\otimes m}$, and also associate to the maps $f^{(m)}_I$ the degree $|I|$ maps $F^{(m)}_I := sf^{(m)}_Iw^{\otimes m} : (sA)^{\otimes m} \rightarrow sB$. We denote $D_A$ and $D_B$ the unique coderivations coming from the maps $b_m$ acting respectively on $\overline{T}(sA)$ and $\overline{T}(sB)$, and $F : \pmb{\Delta}^n \otimes \overline{T}(sA) \rightarrow \overline{T}(sB)$ the unique morphism of coalgebras associated to the maps $F^{(m)}_I$. The equation $$F(\partial_{sing} \otimes \ide_{\overline{T}(sA)} + \ide_{\pmb{\Delta}^n} \otimes D_A)=D_BF$$ is then equivalent to the equations
\[ \sum_{j=0}^{\mathrm{dim}(I)} (-1)^j F^{(m)}_{\partial_j I} + (-1)^{|I|} \sum_{i_1+i_2+i_3=m} F^{(i_1+1+i_3)}_I (\ide^{\otimes i_1} \otimes b_{i_2} \otimes \ide^{\otimes i_3})
= \sum_{\substack{i_1 + \cdots + i_s = m \\ I_1 \cup \cdots \cup I_s = I}} b_s ( F^{(i_1)}_{I_1} \otimes \cdots \otimes F^{(i_s)}_{I_s}) \ . \] 
There are now two ways to unravel the signs from these equations, which will lead to conventions (A) and (B).

The first way consists in simply replacing the $b_m$ and the $F^{(m)}_I$ by their definition. It yields sign conventions (A). The left-hand side transforms as 
\begin{align*}
&\sum_{j=0}^{\mathrm{dim}(I)} (-1)^j F^{(m)}_{\partial_j I} + (-1)^{|I|} \sum_{i_1+i_2+i_3=m} F^{(i_1+1+i_3)}_I (\ide^{\otimes i_1} \otimes b_{i_2} \otimes \ide^{\otimes i_3})  \\ 
= &\sum_{j=0}^{\mathrm{dim}(I)} (-1)^j sf^{(m)}_{\partial_j I}w^{\otimes m} + (-1)^{|I|} \sum_{i_1+i_2+i_3=m} sf^{(i_1+1+i_3)}_I w^{\otimes i_1 + 1 + i_3} (\ide^{\otimes i_1} \otimes sm_{i_2}w^{\otimes i_2} \otimes \ide^{\otimes i_3}) \\
= &\sum_{j=0}^{\mathrm{dim}(I)} (-1)^j sf^{(m)}_{\partial_j I}w^{\otimes m} + (-1)^{|I|} \sum_{i_1+i_2+i_3=m} (-1)^{i_3} sf^{(i_1+1+i_3)}_I  (w^{\otimes i_1} \otimes wsm_{i_2}w^{\otimes i_2} \otimes w^{\otimes i_3}) \\
= &\sum_{j=0}^{\mathrm{dim}(I)} (-1)^j sf^{(m)}_{\partial_j I}w^{\otimes m} + (-1)^{|I|} \sum_{i_1+i_2+i_3=m} (-1)^{i_1 i _2 + i_3} sf^{(i_1+1+i_3)}_I  (\ide^{\otimes i_1} \otimes m_{i_2} \otimes \ide^{\otimes i_3}) (w^{\otimes i_1} \otimes w^{\otimes i_2} \otimes w^{\otimes i_3}) \\
= &s \left( \sum_{j=0}^{\mathrm{dim}(I)} (-1)^j f^{(m)}_{\partial_j I} + (-1)^{|I|} \sum_{i_1+i_2+i_3=m} (-1)^{i_1 i _2 + i_3} f^{(i_1+1+i_3)}_I (\ide^{\otimes i_1} \otimes m_{i_2} \otimes \ide^{\otimes i_3}) \right) w^{\otimes m} \ ,
\end{align*}
while the right-hand side transforms as  
\begin{align*}
&\sum_{\substack{i_1 + \cdots + i_s = m \\ I_1 \cup \cdots \cup I_s = I}} b_s ( F^{(i_1)}_{I_1} \otimes \cdots \otimes F^{(i_s)}_{I_s})  \\ 
= &\sum_{\substack{i_1 + \cdots + i_s = m \\ I_1 \cup \cdots \cup I_s = I}} sm_sw^{\otimes s} ( sf^{(i_1)}_{I_1}w^{\otimes i_1} \otimes \cdots \otimes sf^{(i_s)}_{I_s}w^{\otimes i_s}) \\
= &\sum_{\substack{i_1 + \cdots + i_s = m \\ I_1 \cup \cdots \cup I_s = I}} (-1)^{\sum_{j=1}^s (s-j)|I_j|}sm_s ( wsf^{(i_1)}_{I_1}w^{\otimes i_1} \otimes \cdots \otimes wsf^{(i_s)}_{I_s}w^{\otimes i_s}) \\
= &s \left( \sum_{\substack{i_1 + \cdots + i_s = m \\ I_1 \cup \cdots \cup I_s = I}} (-1)^{\epsilon_A}m_s ( f^{(i_1)}_{I_1} \otimes \cdots \otimes f^{(i_s)}_{I_s}) \right) w^{\otimes m} \ ,
\end{align*}
where $\epsilon_A = \sum_{j=1}^s (s-j)|I_j| + \sum_{j = 1}^s i_j \left( \sum_{k=j+1}^s (1-i_k - |I_k|) \right)$.

The second way consists in first composing and post-composing by $w$ and $s^{\otimes m}$ and then replacing the $b_m$ and $F^{(m)}_I$ by their definition. It yields the (B) sign conventions. We will denote $m_m^\prime:=(-1)^{\binom{m}{2}}m_m$ and $f^{\prime (m) }_I:=(-1)^{\binom{m}{2}}f^{(m)}_I$. The left-hand side then transforms as
\begin{align*}
&\sum_{j=0}^{\mathrm{dim}(I)} (-1)^j wF^{(m)}_{\partial_j I}s^{\otimes m} + (-1)^{|I|} \sum_{i_1+i_2+i_3=m} wF^{(i_1+1+i_3)}_I (\ide^{\otimes i_1} \otimes b_{i_2} \otimes \ide^{\otimes i_3})s^{\otimes m} \\ 
= &\sum_{j=0}^{\mathrm{dim}(I)} (-1)^j f^{ \prime(m) }_{\partial_j I} + (-1)^{|I|} \sum_{i_1+i_2+i_3=m} (-1)^{i_1} f^{(i_1+1+i_3)}_I w^{\otimes i_1 + 1 + i_3} (s^{\otimes i_1} \otimes sm_{i_2}w^{\otimes i_2}s^{\otimes i_2} \otimes s^{\otimes i_3}) \\
= &\sum_{j=0}^{\mathrm{dim}(I)} (-1)^j f^{ \prime(m) }_{\partial_j I} + (-1)^{|I|} \sum_{i_1+i_2+i_3=m} (-1)^{i_1+i_2 i_3} f^{(i_1+1+i_3)}_I w^{\otimes i_1 + 1 +i_3} s^{\otimes i_1 + 1 + i_3}  (\ide^{\otimes i_1} \otimes m_{i_2}^\prime \otimes \ide^{\otimes i_3}) \\
= &\sum_{j=0}^{\mathrm{dim}(I)} (-1)^j f^{\prime (m) }_{\partial_j I} + (-1)^{|I|} \sum_{i_1+i_2+i_3=m} (-1)^{i_1 + i_2 i_3} f^{\prime (i_1+1+i_3)}_I  (\ide^{\otimes i_1} \otimes m_{i_2}^\prime \otimes \ide^{\otimes i_3}) \ ,
\end{align*}
while the right-hand side transforms as 
\begin{align*}
&\sum_{\substack{i_1 + \cdots + i_s = m \\ I_1 \cup \cdots \cup I_s = I}} wb_s ( F^{(i_1)}_{I_1} \otimes \cdots \otimes F^{(i_s)}_{I_s})s^{\otimes m}  \\ 
= &\sum_{\substack{i_1 + \cdots + i_s = m \\ I_1 \cup \cdots \cup I_s = I}} m_sw^{\otimes s} ( sf^{(i_1)}_{I_1}w^{\otimes i_1} \otimes \cdots \otimes sf^{(i_s)}_{I_s}w^{\otimes i_s})s^{\otimes m} \\
= &\sum_{\substack{i_1 + \cdots + i_s = m \\ I_1 \cup \cdots \cup I_s = I}} (-1)^{\sum_{j=1}^s \left(i_j \sum_{k=j+1}^s|I_k| \right)}m_s w^{\otimes s} ( sf^{(i_1)}_{I_1}w^{\otimes i_1}s^{\otimes i_1} \otimes \cdots \otimes sf^{(i_s)}_{I_s}w^{\otimes i_s} s^{\otimes i_s}) \\
= &\sum_{\substack{i_1 + \cdots + i_s = m \\ I_1 \cup \cdots \cup I_s = I}} (-1)^{\epsilon_B}m_s w^{\otimes s} s^{\otimes s} ( f^{\prime (i_1)}_{I_1} \otimes \cdots \otimes f^{ \prime (i_s)}_{I_s}) \\
= &\sum_{\substack{i_1 + \cdots + i_s = m \\ I_1 \cup \cdots \cup I_s = I}} (-1)^{\epsilon_B}m_s^\prime ( f^{\prime (i_1)}_{I_1} \otimes \cdots \otimes f^{ \prime (i_s)}_{I_s}) \ ,
\end{align*}
where $\epsilon_B = \sum_{j=1}^s \left(i_j \sum_{k=j+1}^s|I_k| \right) + \sum_{j=1}^s (s-j) (1-i_j - |I_j|)$.

\subsubsection{Choice of convention in this paper} \label{alg:sss:choice-convention-ainf}

We will work in the rest of this paper with the set of conventions (B). The operations $m_m$ of an \Ainf -algebra will satisfy equations
\[ \left[ \partial , m_m \right] = - \sum_{\substack{i_1+i_2+i_3=m \\ 2 \leqslant i_2 \leqslant n-1}} (-1)^{i_1 + i_2i_3} m_{i_1+1+i_3} (\ide^{\otimes i_1} \otimes m_{i_2} \otimes \ide^{\otimes i_3} ) \ , \]
and a  $n-\Ainf$-morphism between two \Ainf -algebras will satisfy equations
\begin{align*}
\left[ \partial , f^{(m)}_I \right] = \sum_{j=0}^{\mathrm{dim}(I)} (-1)^j f^{(m)}_{\partial_jI} &+ (-1)^{|I|} \sum_{\substack{i_1+i_2+i_3=n \\ i_2 \geqslant 2}} (-1)^{i_1 + i_2i_3} f^{(i_1+1+i_3)}_I (\ide^{\otimes i_1} \otimes m_{i_2} \otimes \ide^{\otimes i_3})  \\ &- \sum_{\substack{i_1 + \cdots + i_s = m \\ I_1 \cup \cdots \cup I_s = I \\ s \geqslant 2 }} (-1)^{\epsilon_B} m_s ( f^{(i_1)}_{I_1} \otimes \cdots \otimes f^{(i_s)}_{I_s}) \ ,
\end{align*}
where $\epsilon_B = \sum_{j=1}^s \left(i_j \sum_{k=j+1}^s|I_k| \right) + \sum_{j=1}^s (s-j) (1-i_j - |I_j|)$.

In~\cite{mazuir-I} we had chosen conventions (B) for \Ainf -algebras and \Ainf -morphisms because they were the ones naturally arising in the realizations of the associahedra and the multiplihedra à la Loday. We prove a similar result in the following section : these sign conventions are contained in the polytopes $n-J_m = \Delta^n \times J_m$ where $J_m$ is a Forcey-Loday realization of the multiplihedron. 

\subsubsection{The sign conventions coming from Proposition~\ref{alg:prop:equivalent-def}} \label{alg:sss:sign-conv-equivalent}

We proved in Proposition~\ref{alg:prop:equivalent-def} that the datum of a $n$-morphism from $A$ to $B$ is equivalent to the datum of an \Ainf -morphism $A \rightarrow \pmb{\Delta}_n \otimes B$. In fact, the two sign conventions arising from this equivalent definition differ slightly from the two conventions (A) and (B) for $n$-morphisms computed from the bar construction formulation.

Indeed, we can check that if we work with convention (A) (resp. (B)) for \Ainf -morphisms (not higher morphisms !) and if we write as in subsection~\ref{alg:ss:equivalent-definition} the \Ainf -morphism $F : A \rightarrow \pmb{\Delta}_n \otimes B$ as
\[ F^{(m)} = \bigoplus_{I \subset \Delta^n} I \otimes f_I^{(m)} \ , \]
then the signs for the \Ainf -equations for $F$ read exactly as the signs for the \Ainf -equations for $n$-morphisms computed in the previous subsection, apart from the simplicial differential terms which read this time as 
\[ \sum_{j=0}^{\mathrm{dim}(I)} (-1)^{j+|I|+1} f^{(m)}_{\partial_jI} \ . \]

\subsection{Signs and the polytopes $n-J_m$} \label{alg:ss:signs-poly-n-Jm}

\subsubsection{Loday associahedra and Forcey-Loday multiplihedra} \label{alg:sss:forcey-loday}

In~\cite{mazuir-I} we introduced explicit polytopal realizations of the associahedra and the multiplihedra : the weighted Loday realizations $K_\omega$ of the associahedra from~\cite{masuda-diagonal-assoc} and the weighted Forcey-Loday realizations $J_\omega$ of the multiplihedra from~\cite{masuda-diagonal-multipl}. We then proved using basic considerations on affine geometry that, under the convention of section~\ref{alg:ss:conv-signs-or}, their boundaries were  equal to
\begin{align*}
\partial K_\omega &= - \bigcup_{\substack{i_1+i_2+i_3=n \\ 2 \leqslant i_2 \leqslant n-1}} (-1)^{i_1 + i_2 i_3} K_{\overline{\omega}} \times K_{\widetilde{\omega}} \ , \\
\partial J_\omega &= \bigcup_{\substack{i_1+i_2+i_3=n \\ i_2 \geqslant 2}} (-1)^{i_1 + i_2 i_3} J_{\overline{\omega}} \times K_{\widetilde{\omega}} \cup - \bigcup_{\substack{i_1 + \cdots + i_s = m \\ s \geqslant 2}} (-1)^{\varepsilon_B} K_{\overline{\omega}} \times J_{\widetilde{\omega}_1} \times \cdots \times J_{\widetilde{\omega}_s}
\end{align*}
where the weights $\overline{\omega}$, $\widetilde{\omega}$ and $\widetilde{\omega}_t$ are derived from the weights $\omega$, and $$\varepsilon_B = \sum_{j=1}^s (s-j) (1-i_j) \ .$$ 
In particular, these polytopes contain sign conventions (B) for \Ainf -algebras and \Ainf -morphisms.

\subsubsection{The boundary of $n-J_m$} \label{alg:sss:codim-one-boundary-n-Jm}

Consider now a $n$-multiplihedron $\Delta^n \times J_\omega$, where $J_\omega$ is a Forcey-Loday realization of the multiplihedron $J_m$. Forgetting for now about its refined polytopal subdivision, its boundary reads as 
\[  \partial (\Delta^n \times J_\omega) = \partial \Delta^n \times J_\omega \cup (-1)^n \Delta^n \times \partial J_\omega \ . \]
Recall moreover that given any dividing sequence $\pmb{a}$ of length $s$, each top dimensional cell in the $\mathrm{AW}_{\pmb{a}}$-polytopal subdivision of $\Delta^n$ labeled by an overlapping partition $I_1 \cup \cdots \cup I_{s+1} = \Delta^n$ is in fact isomorphic to the product $I_1 \times \cdots \times I_{s+1}$. We write this as
\[ \Delta^n_{\pmb{a}} = \bigcup_{I_1 \cup \cdots \cup I_s = \Delta^n} I_1 \times \cdots \times I_s \ . \]

\begin{proposition}
The n-multiplihedra $\Delta^n \times J_\omega$ endowed with their $n-\Ainf$-polytopal subdivision contain sign conventions (B) for $n-\Ainf$-morphisms. 
\end{proposition}
\begin{proof}
The first component of the boundary of $\Delta^n \times J_\omega$ is given by 
\[ \partial \Delta^n \times J_\omega = \bigcup_{i=0}^n (-1)^i \Delta^{n-1}_i \times J_\omega \ . \]
The second, by the first part of the boundary of $ \partial J_\omega$,
\[ (-1)^n \bigcup_{\substack{i_1+i_2+i_3=m \\ i_2 \geqslant 2}} (-1)^{i_1 + i_2 i_3} (\Delta^n \times J_{\overline{\omega}} ) \times K_{\widetilde{\omega}} \ . \]
The third and last component transforms as follows :
\begin{align*}
&(-1)^n \Delta^n \times (-1) \bigcup_{\substack{i_1 + \cdots + i_s = m \\ s \geqslant 2}} (-1)^{\varepsilon_B} K_{\overline{\omega}} \times J_{\widetilde{\omega}_1} \times \cdots \times J_{\widetilde{\omega}_s} \\
= & (-1)^{n+1} \bigcup_{\substack{i_1 + \cdots + i_s = m \\ s \geqslant 2}} (-1)^{\varepsilon_B} \Delta^n \times K_{\overline{\omega}} \times J_{\widetilde{\omega}_1} \times \cdots \times J_{\widetilde{\omega}_s} \\
= &(-1)^{n+1}\bigcup_{\substack{i_1 + \cdots + i_s = m \\ s \geqslant 2}} (-1)^{\varepsilon_B} \bigcup_{I_1 \cup \cdots \cup I_s = \Delta^n} I_1 \times \cdots \times I_s \times K_{\overline{\omega}} \times J_{\widetilde{\omega}_1} \times \cdots \times J_{\widetilde{\omega}_s} \\
= &(-1)^{n+1}\bigcup_{\substack{i_1 + \cdots + i_s = m \\ s \geqslant 2}} \bigcup_{I_1 \cup \cdots \cup I_s = \Delta^n} (-1)^{\varepsilon_B + s \sum_{j=1}^s|I_j|}  K_{\overline{\omega}} \times I_1 \times \cdots \times I_s \times  J_{\widetilde{\omega}_1} \times \cdots \times J_{\widetilde{\omega}_s} \\
= &- (-1)^{n}\bigcup_{\substack{ i_1 + \cdots + i_s = m \\ I_1 \cup \cdots \cup I_s = \Delta^n \\ s \geqslant 2}} (-1)^{\varepsilon_B + s n + \sum_{j=1}^s (i_j - 1) \left( \sum_{k=j+1}^s | I_k| \right)}  K_{\overline{\omega}} \times (I_1 \times J_{\widetilde{\omega}_1}) \times \cdots \times ( I_s \times J_{\widetilde{\omega}_s} ) \ .
\end{align*}
We then check that $\epsilon_B = n + \varepsilon_B + s n + \sum_{j=1}^s (i_j - 1) \left( \sum_{k=j+1}^s | I_k| \right)$ modulo 2. Hence, the polytopes $n-J_m$ contain indeed sign conventions (B) for $n-\Ainf$-morphisms.
\end{proof}

\subsection{The operadic bimodule $n-\Omega B As - \mathrm{Morph}$} \label{alg:ss:signs-n-ombas}

In~\cite{mazuir-I}, we computed the signs for $\Omega B As$-morphisms as follows. Endowing the compactified moduli spaces $\overline{\mathcal{CT}}_m$ with their \ombas -cell decompositions, we define the operadic bimodule $\Omega B As - \mathrm{Morph}$ to be the realization under the functor $C_{-*}^{cell}$ of the operadic bimodule $\{ \overline{\mathcal{CT}}_m \}_{m \geqslant 1}$. The signs in the differential are then computed as the signs arising in the top dimensional strata in the boundary of the moduli spaces $\overline{\mathcal{CT}}_m(t_g)$. The signs for the action-composition maps are the signs ensuing from the image under the functor $C_{-*}^{cell}$ of the action-composition maps for the moduli spaces $\overline{\mathcal{CT}}_m(t_g)$.

The goal of this section is to completely state definition~\ref{alg:def:op-bimod-ombas-n}, with explicit signs and formulae. We have however seen in subsection~\ref{alg:sss:these-poly-encode-ombas} that there is no operadic bimodule in compactified moduli spaces whose image under the functor $C_{-*}^{cell}$ could realize the operadic bimodule $n - \Omega B As - \mathrm{Morph}$. We will still compute the signs for the action-composition maps by introducing some suitable spaces of metric trees, which do not define an operadic bimodule but will however carry enough structure for our computations. The differential will simply be defined by reading the signs arising in the top dimensional strata of the boundary of the CW-complex $\Delta^n \times \overline{\mathcal{CT}}_m$ endowed with its $n - \ombas$-cell decomposition.

\subsubsection{Notation} \label{alg:sss:recoll-signs-n-ombas}

As in~\cite{mazuir-I}, we choose to use the formalism of orientations on trees to define the operadic bimodule $n - \Omega B As - \mathrm{Morph}$. Recall that this formalism originates from~\cite{markl-assoc}.

\begin{definition}
Given a broken stable ribbon tree $t_{br}$, an \emph{ordering} of $t_{br}$ is defined to be an ordering of its $i$ finite internal edges $e_{1} , \dots , e_{i}$. Two orderings are said to be \emph{equivalent} if one passes from one ordering to the other by an even permutation. An \emph{orientation} of $t_{br}$ is then defined to be an equivalence class of orderings, and written $\omega := e_{1} \wedge \cdots \wedge e_{i}$. Each tree $t_{br}$ has exactly two orientations. Given an orientation $\omega$ of $t_{br}$ we will write $-\omega$ for the second orientation on $t_{br}$, called its \emph{opposite orientation}.
\end{definition}

In this section, we write $t_{br,g}$ for a broken gauged stable ribbon tree, and $t_g$ for an unbroken gauged stable ribbon tree.

\begin{definition}
We set \arbreopunmorph\ to be the unique stable gauged tree of arity 1 and call it the \emph{trivial gauged tree}. We define the \emph{underlying broken stable ribbon tree} $t_{br}$ of a $t_{br,g}$ to be the ribbon tree obtained by first deleting all the \arbreopunmorph\ in $t_{br,g}$, and then forgetting all the remaining gauges of $t_{br,g}$. We will moreover refer to a gauge in $t_{br,g}$ which is associated to a non-trivial gauged tree, as a \emph{non-trivial gauge} of $t_{br,g}$. An \emph{orientation} on a broken gauged stable ribbon tree $t_{br,g}$ is then defined to be an orientation $\omega$ on $t_{br}$.
\end{definition}

\begin{figure}[h]
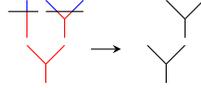
 
\exampleunderlyingbroken
\caption*{An instance of association $t_{br,g} \mapsto t_{br}$}
\end{figure} \label{alg:fig:underlying-broken}

\begin{definition}
Consider a gauged tree $t_{br,g}$ which has $b$ gauges, trivial or not. A list $\mathbb{I} := (I_1 , \dots , I_b)$ of faces $I_a \subset \Delta^n$ will be called a \emph{$\Delta^n$-labeling of $t_{br,g}$}. The tree $t_{br,g}$ endowed with its labeling will be written 
$(\mathbb{I} , t_{br,g} )$.
\end{definition}
\noindent We think of $(\mathbb{I} , t_{br,g} )$ as depicted in the figure below, where trees are represented as corollae for the sake of readability.
\[ \exampleCTmItg  \]

\subsubsection{Definition of the spaces of operations} \label{alg:sss:def-space-op-n-ombas}

\begin{definition}[Spaces of operations]
Consider the \Z -module freely generated by the pairs $(\mathbb{I} ,t_{br,g},\omega)$, where $\omega$ is an orientation on $t_{br,g}$ and $\mathbb{I}$ is a $\Delta^n$-labeling of $t_{br,g}$. We define the arity $m$ space of operations $n-\Omega B As - \mathrm{Morph}(m)_*$ to be the quotient of this \Z -module under the relation
\[ (\mathbb{I} , t_{br,g} , - \omega) = - (\mathbb{I},t_{br,g} , \omega) \ . \]
Introducing the notation $| \mathbb{I} | := \sum_{a=1}^b | I_a |$, a pair $(\mathbb{I}, t_{br,g} , \omega)$ is then defined to have degree $$|(\mathbb{I}, t_{br,g} , \omega)| :=  | \mathbb{I} | + |t_{br,g}| \ .$$
\end{definition}

\subsubsection{The oriented spaces $\mathcal{CT}_m(\mathbb{I},t_{br,g},\omega)$} \label{alg:sss:spaces-I-CTm}

Consider a $\Delta^n$-labeled gauged tree $(\mathbb{I} , t_{br,g} )$, together with a choice of orientation $\omega$ on $t_{br,g}$. We define the spaces
\[ \mathcal{CT}_m(\mathbb{I},t_{br,g},\omega ) := I_1 \times \cdots \times I_b \times \mathcal{CT}_m(t_{br,g},\omega ) \ . \]
An element of $\mathcal{CT}_m(\mathbb{I},t_{br,g},\omega )$ is thus of the form
\[ ( \delta_1 , \dots , \delta_b , \lambda_1, \dots , \lambda_g , l_{e_1} , \dots , l_{e(t_{br})} ) \in I_1 \times \cdots \times I_b \times ] - \infty , + \infty [^{g} \times ] 0 , +\infty[^{e(t_{br})}  \ , \]
where the $\lambda_i$ are the non-trivial gauges of $t_{br,g}$ ordered from left to right, and the $l_{e_i}$ are the lengths of the finite internal edges of $t_{br}$ ordered according to $\omega$. These spaces are then simply oriented by taking the product orientation of their factors.

\subsubsection{Definition of the action-compositions maps}  \label{alg:sss:def-action-comp-maps-n-ombas}

We may now introduce the "action-composition" maps on the spaces $\mathcal{CT}_m(\mathbb{I},t_{br,g})$, that we will use to define the signs of the action-composition maps for $n - \Omega B As - \mathrm{Morph}$. 
Define the maps
\begin{align*}
O_i : \ &\mathcal{CT}(\mathbb{I},t_{br,g},\omega) \times \mathcal{T}(t_{br}',\omega ') = \mathbb{I} \times \mathcal{CT}(t_{br,g},\omega) \times \mathcal{T}(t_{br}',\omega ') \\
&\longrightarrow \mathbb{I} \times \mathcal{CT}(t_{br,g} \circ_i t_{br}' , \omega \wedge \omega ') = \mathcal{CT}(\mathbb{I},t_{br,g} \circ_i t_{br}' , \omega \wedge \omega ')
\end{align*}
where $\mathbb{I}$ stands for the product $I_1 \times \cdots \times I_b$, and the arrow corresponds to the action-composition map 
\[ \mathcal{CT}(t_{br,g},\omega) \times \mathcal{T}(t_{br}',\omega ') \longrightarrow \mathcal{CT}(t_{br,g} \circ_i t_{br}' , \omega \wedge \omega ') \ , \]
of the operadic bimodule $\{ \mathcal{CT}_m \}_{m \geqslant 1}$.
Define also the maps
\begin{align*}
M : \ &\mathcal{T}(t_{br},\omega) \times \mathcal{CT}(\mathbb{I}_1,t_{br,g}^1,\omega_1) \times \cdots \times \mathcal{CT}(\mathbb{I}_s,t_{br,g}^s,\omega_s) \\
&\longrightarrow \mathbb{I}_1 \times \cdots \times \mathbb{I}_s \times \mathcal{T}(t_{br},\omega) \times \mathcal{CT}(t_{br,g}^1,\omega_1) \times \cdots \times \mathcal{CT}(t_{br,g}^s,\omega_s) \\
 &\longrightarrow \mathcal{CT}(\mathbb{I}_1 \cup \cdots \cup \mathbb{I}_s, \mu (t_{br} , t_{br,g}^1 \dots , t_{br,g}^s) , \omega \wedge \omega_1 \wedge \cdots \wedge \omega_s )
\end{align*}
where the second arrow corresponds to the action-composition map
\[ \mathcal{T}(t_{br},\omega) \times \mathcal{CT}(t_{br,g}^1,\omega_1) \times \cdots \times \mathcal{CT} (t_{br,g}^s,\omega_s) \longrightarrow \mathcal{CT} ( \mu (t_{br} , t_{br,g}^1 \dots , t_{br,g}^s) , \omega \wedge \omega_1 \wedge \cdots \wedge \omega_s ) \ .\]
The maps $O_i$ have sign $+1$. The maps $M$ have sign $(-1)^\dagger$, where $\dagger$ is defined as follows. Writing $g_i$ for the number of non-trivial gauges and $j_i$ for the number of gauge-vertex intersections of $t_{br,g}^i$, $i= 1 , \dots , s$, and setting $t_{br}^0 := t_{br}$ and $g_0 = j_0 = \mathrm{dim}(\mathbb{I}_0) = 0$, 
\[ \dagger := \sum_{i=1}^s |\mathbb{I}_i| \left( |t_{br}| + \sum_{l=1}^{i-1} |t_{br,g}^l | \right) + \sum_{i=1}^s g_i  \left( |t_{br}| + \sum_{l=1}^{i-1} |t_{br}^l | \right)  + \sum_{i=1}^s j_i \left( |t_{br}| + \sum_{l=1}^{i-1} |t_{br,g}^l| \right) \ . \]

\begin{definition}[Action-composition maps]
The action of the operad $\Omega B As$ on $ n - \Omega B As -\mathrm{Morph}$ is defined as
\[ \resizebox{\hsize}{!}{\begin{math} \begin{aligned}
(\mathbb{I} , t_{br,g},\omega) \circ_i (t_{br}',\omega ') &= ( \mathbb{I} , t_{br,g} \circ_i t_{br}' , \omega \wedge \omega ') \ , \\
\mu ((t_{br},\omega),( \mathbb{I}_1 , t_{br,g}^1,\omega_1),\dots,(\mathbb{I}_s , t_{br,g}^s,\omega_s)) &= (-1)^{\dagger} (\mathbb{I}_1 \cup \cdots \cup \mathbb{I}_s , \mu (t_{br} , t_{br,g}^1 \dots , t_{br,g}^s) , \omega \wedge \omega_1 \wedge \cdots \wedge \omega_s ) \ .
\end{aligned}
\end{math}} \]
\end{definition}

Using for instance the maps $O_i$ and $M$, and remembering the Koszul sign rules, we can check that these action-composition maps satisfy indeed all the associativity conditions for an operadic bimodule.
What's more, choosing a distinguished orientation for every gauged stable ribbon tree $t_g \in SCRT$, this definition of the operadic bimodule $n-\Omega B As -\mathrm{Morph}$ amounts to defining it as the free operadic bimodule in graded \Z -modules
\[ n-\Omega B As - \mathrm{Morph} = \mathcal{F}^{\Omega B As, \Omega B As}( \arbreopunmorphn , \arbrebicoloreLn , \arbrebicoloreMn , \arbrebicoloreNn , \cdots ,( I , SCRT_m ) ,\cdots ; I \subset \Delta^n) \ .  \]
It remains to define a differential on the generating operations $(I , t_g , \omega )$ to recover definition~\ref{alg:def:op-bimod-ombas-n}. 

\subsubsection{The boundary of the compactified moduli spaces $\overline{\mathcal{CT}}_m(t_g)$} \label{alg:sss:recoll-codim-1}

Before defining the differential on the operadic bimodule $n-\Omega B As - \mathrm{Morph}$, we recall the signs for the top dimensional strata in the boundary of the compactified moduli spaces $\overline{\mathcal{CT}}_m(t_g)$ that were computed in section I.5.2 in~\cite{mazuir-I}. 

We fix for the rest of this subsection a gauged stable ribbon tree $t_g$ whose gauge intersects $j$ of its vertices. We also choose an orientation $e_1 \wedge \cdots \wedge e_i$ on $t_g$ and order the $j$ gauge-vertex intersections from left to right 
\[ \ordreunbrokengaugedtreeintersectionsA \ . \]

The (int-collapse) boundary corresponds to the collapsing of an internal edge that does not intersect the gauge of the tree $t$. Suppose that it is the $p$-th edge $e_p$ of $t$ which collapses. Write moreover $(t/e_p)_g$ for the resulting gauged tree and $\omega_p := e_1 \wedge \cdots \wedge \widehat{e_p} \wedge \cdots \wedge e_i$ for the induced orientation on the edges of $t/e_p$. The boundary component $\mathcal{CT}_m((t/e_p)_{g}, \omega_p )$ bears a sign
\begin{align*}
(-1)^{p+1+j} \tag{\textit{int-collapse}}
\end{align*}
in the boundary of $\overline{\mathcal{CT}}_m(t_{g}, \omega )$.

The (gauge-vertex) boundary corresponds to the gauge crossing exactly one additional vertex of~$t$. We suppose that this intersection takes place between the $k$-th and $(k+1)$-th intersections of $t_g$ and write $t_g^{0}$ for the resulting gauged tree.
If the crossing results from a move
\[ \transformationgaugevertexA \ , \]
the boundary component $\mathcal{CT}_m(t_g^{0}, \omega )$ has sign
\begin{align*}
(-1)^{j+k} \tag{\textit{gauge-vertex A}}
\end{align*}
in the boundary of $\overline{\mathcal{CT}}_m(t_{g}, \omega )$.
If the crossing results from a move
\[ \transformationgaugevertexB \ , \]
the boundary component $\mathcal{CT}_m(t_g^{0}, \omega )$ has sign
\begin{align*}
(-1)^{j+k+1} \tag{\textit{gauge-vertex B}}
\end{align*}
in the boundary of $\overline{\mathcal{CT}}_m(t_{g}, \omega )$.

The (above-break) boundary corresponds either to the breaking of an internal edge of $t$, that is located above the gauge or intersects the gauge, or, when the gauge is below the root, to the outgoing edge breaking between the gauge and the root. Denote $e_0$ the outgoing edge of $t$. Suppose that it is the $p$-th edge $e_p$ of $t$ which breaks and write moreover $(t_p)_g$ for the resulting broken gauged tree.
The boundary component $\mathcal{CT}_m((t_p)_{g}, \omega_p )$ bears a sign
\begin{align*}
(-1)^{p+j} \tag{\textit{above-break}}
\end{align*}
in the boundary of $\overline{\mathcal{CT}}_m(t_{g}, \omega )$.

The (below-break) boundary corresponds to the breaking of edges of $t$ that are located below the gauge or intersect it, such that there is exactly one edge breaking in each non-self crossing path from an incoming edge to the root.  Write $(t_{br})_g$ for the resulting broken gauged tree. We order from left to right the $s$ non-trivial unbroken gauged trees $t_g^1 , \dots , t_g^s$ of $(t_{br})_g$ and denote $e_{j_1} , \dots , e_{j_s}$ the internal edges of $t$ whose breaking produces the trees $t_g^1 , \dots , t_g^s$. Beware that we do not necessarily have that $j_1 < \cdots < j_s$. To this extent, we denote $\varepsilon ( j_1 , \dots , j_s  ; \omega) $ the sign obtained after modifying $\omega$ by moving $e_{j_k}$ to the $k$-th spot in $\omega$. We write $\omega_{br}$ for the induced orientation on $(t_{br})_g$, which is obtained by deleting the edges $e_{j_k}$ in $\omega$. The boundary component $\mathcal{CT}_m((t_{br})_{g}, \omega_{br} )$ has sign
\begin{align*}
(-1)^{\varepsilon ( j_1 , \dots , j_s ; \omega) + 1 + j} \tag{\textit{below-break}}
\end{align*}
in the boundary of $\overline{\mathcal{CT}}_m(t_{g}, \omega )$.

\subsubsection{Definition of the differential} \label{alg:sss:def-diff-n-ombas}

\begin{definition}[Differential]
The differential of a generating operation $(I,t_{g},\omega)$ is defined by reading the signs of the top dimensional strata in the boundary of the space $I \times \overline{\mathcal{CT}}_m(t_{g},\omega)$, endowed with its $\mathrm{dim}(I) - \ombas$ cell decomposition. It reads as
\[ \resizebox{\hsize}{!}{\begin{math} \begin{aligned}
\partial (I,t_{g},\omega) := &\sum_{l=0}^{\mathrm{dim}(I)} (-1)^l ( \partial^{sing}_l I ,t_{g},\omega) + (-1)^{|I|} \sum (-1)^{\dagger_{\Omega B As}} (I , \mathrm{int-collapse}(t_g,\omega )) \\
&+ (-1)^{|I|} \sum (-1)^{\dagger_{\Omega B As}} (I,\mathrm{gauge-vertex}(t_g,\omega )) + (-1)^{|I|} \sum (-1)^{\dagger_{\Omega B As}} (I,\mathrm{above-break}(t_g,\omega)) \\
&+ (-1)^{|I|} \sum_{I_1 \cup \cdots \cup I_b = I} (-1)^{\dagger_{\Omega B As}} ((I_1,\dots,I_b),\mathrm{below-break}(t_g,\omega)) \ ,
\end{aligned}
\end{math}} \]
where $b$ denotes the number of gauges of $\mathrm{below-break}(t_g)$ and the signs $(-1)^{\dagger_{\Omega B As}}$ denote the $\Omega B As - \mathrm{Morph}$ signs listed in the previous subsection. 
\end{definition}

For instance, choosing the orientation $e_1 \wedge e_2$ on 
\[ \orderingarbrebicoloresignesA \ , \] 
the signs in the computation of subsection~\ref{alg:sss:n-ombas-morph} are
\begin{align*}
 \partial \left( \exarbrebicoloreA , e_1 \wedge e_2 \right) = & \left( \exarbrebicoloreB , e_1 \wedge e_2 \right) - \left( \exarbrebicoloreC , e_1 \wedge e_2 \right) + \left( \exarbrebicoloreD , e_1 \wedge e_2 \right) \\ 
 &- \left( \exarbrebicoloreE , e_1 \wedge e_2 \right) - \left( \exarbrebicoloreF , e_1 \wedge e_2 \right) + \left( \exarbrebicoloreG , e_1 \wedge e_2 \right) \\
  & - \left( \exarbrebicoloreJ , \emptyset \right) - \left( \exarbrebicoloreK , \emptyset \right) - \left( \exarbrebicoloreL , \emptyset \right) \\
 & + \left( \exarbrebicoloreH , e_1 \right) - \left( \exarbrebicoloreI , e_2 \right) \ .
\end{align*}
This concludes the construction of the operadic bimodule $n - \Omega B As - \mathrm{Morph}$.

\subsubsection{The morphism of operadic bimodules $\infmorn \rightarrow n - \Omega B As - \mathrm{Morph}$} \label{alg:sss:morph-op-bimod-infmorn}

To conclude, it remains to define the morphism of operadic bimodules $\infmorn \rightarrow n - \Omega B As - \mathrm{Morph}$. It is enough to define this morphism on the generating operations of \infmorn\ and to check that it is compatible with the differentials. 

\begin{alg:prop:morph-op-bimod} 
The map $\infmorn \rightarrow n - \Omega B As - \mathrm{Morph}$ defined on the generating operations of \infmorn\ as
\[ f_{I,m} \longmapsto \sum_{t_g \in CBRT_m} (I , t_g, \omega_{can}) \]
is a morphism of $(\Ainf , \Ainf)$-operadic bimodules.
\end{alg:prop:morph-op-bimod} 
We refer to section I.5.3 of~\cite{mazuir-I} for the definition of the canonical orientations $\omega_{can}$. It is easy to check that this map is indeed compatible with the differentials : either making explicit signs computations, or noting that this morphism corresponds to the refinement of the $n-\Ainf$-cell decomposition of $n - J_m$ to its $n-\Omega B As$-cell decomposition.

\newpage

\begin{leftbar}
\part{The simplicial sets $\mathrm{HOM}_{\mathsf{\Ainf -Alg}}(A,B)_{\bullet}$} \label{p:simplicial}
\end{leftbar}

\setcounter{section}{0}

\section{$\infty$-categories, Kan complexes and cosimplicial resolutions}

\subsection{$\infty$-categories and Kan complexes} \label{alg:ss:inf-cat}

\subsubsection{Motivation} \label{alg:sss:motiv-inf-cat}

The operads \Ainf\ and $\Omega B As$ provide two equivalent frameworks to study the notion of "dg-algebras which are associative up to homotopy". See section III.2 of~\cite{mazuir-I} for a detailed account on the matter. In fact, the operad \Ainf\ can also be used to define the notion of "dg-categories whose composition is associative up to homotopy" : these categories are called \emph{\Ainf -categories}. We recall their definition in subsection~\ref{ss:n-functors}. They are of prime interest in symplectic topology for instance, where they appear as the Fukaya categories of symplectic manifolds. The notion of $\Omega B As$-categories could be defined similarly, but it has never appeared in the litterature to the author's knowledge.

\Ainf -categories are thus "categories" which are endowed with a collection of operations corresponding to all the higher coherent homotopies arising from the associativity up to homotopy of their composition. They are thus \emph{operadic in essence}. The notion of $\infty$-category that we are going to define below, provides another framework to study "categories whose composition is associative up to homotopy" but is, on the other hand, not operadic : it does not come with a specific set of operations encoding rigidly all the higher coherent homotopies.

\subsubsection{Intuition} \label{alg:sss:int-inf-cat}

A category can be seen as the data of a set of points, its objects, together with a set of arrows between them, the morphisms. The composition is then simply an operation which produces from two arrows $A \rightarrow B$ and $B \rightarrow C$ a new arrow $A \rightarrow C$. 

Part of the data of an $\infty$-category will also consist in a set of objects and arrows between them. The difference will lie in the notion of composition. Given two arrows $u : A \rightarrow B$ and $v : B \rightarrow C$, an $\infty$-category will have the property that there always exists a new arrow $A \rightarrow C$, which can be called \emph{a composition} of $u$ and $v$. But this arrow is not necessarily unique, and above all, it results from a property of the "category" and is not produced by an operation of composition. It is in this sense that an $\infty$-category is not operadic.

\subsubsection{Definition} \label{alg:sss:def-inf-cat}

The correct framework to formulate this paradigm is the one of simplicial sets. We write $\Delta^n$ for the simplicial set naturally realizing the standard $n$-simplex $\Delta^n$, and $\mathsf{\Lambda}^k_n$ for the simplicial set realizing the simplicial subcomplex obtained from $\Delta^n$ by removing the faces $[0 < \cdots < n]$ and  $[0 < \cdots <\widehat{k} < \cdots < n]$. The simplicial set $\mathsf{\Lambda}^k_n$ is called a \emph{horn}, if $ 0 < k <n$ it is called an \emph{inner horn}, and if $k=0$ or $k=n$ it is called an \emph{outer horn}.

An \emph{$\infty$-category} is then defined to be a simplicial set $X$ which has the left-lifting property with respect to all inner horn inclusions $\mathsf{\Lambda}^k_n \rightarrow \Delta^n$ : for each $n \geqslant 2$ and each $0 < k < n$, every simplicial map $ u : \mathsf{\Lambda}^k_n \rightarrow X$ extends to a simplicial map $\overline{u} : \Delta^n \rightarrow X$ whose restriction to $\mathsf{\Lambda}^k_n$ is $u$. This is illustrated in the diagram below.
\[ \begin{tikzcd}[row sep=large, column sep = large]
\mathsf{\Lambda}^k_n \arrow[hookrightarrow]{d} \arrow[r,"u"{above}] & X \\
\Delta^n \arrow[ur, dashed , "\exists \ \overline{u}" {below right}] &
\end{tikzcd} \]
The vertices of $X$ are then to be seen as objects, while its edges correspond to morphisms. An \emph{$\infty$-groupoid}, also called \emph{Kan complex}, is defined to be a simplicial set $X$ which has the left-lifting property with respect to all horn inclusions.

For an $\infty$-category, the left-lifting property with respect to $\mathsf{\Lambda}^1_2 \rightarrow \Delta^2$ ensures that the following diagram can always be filled by the dashed arrows
\[ \begin{tikzcd}[row sep=large, column sep = large]
0 \arrow[dashed]{dr}[name=U,above right,pos=0.6]{} \arrow{r} & 1 \arrow{d} \arrow[dashed,Rightarrow, to=U] \\
& 2
\end{tikzcd} \ . \] 
The $[ 0 < 2]$ edge will represent a composition of the morphisms associated to $[0<1]$ and $[1<2]$. For an $\infty$-groupoid, the left-lifting property with respect to the outer horns $\mathsf{\Lambda}^0_2 \rightarrow \Delta^2$ and $\mathsf{\Lambda}^2_2 \rightarrow \Delta^2$ ensures that every morphism is invertible up to homotopy (hence the name \emph{$\infty$-groupoid}). The intuition of subsection~\ref{alg:sss:int-inf-cat} is thus realized, and gives rise to a wide range of higher homotopies controlled by the combinatorics of simplicial algebra.

\subsubsection{Simplicial homotopy groups of a Kan complex (\cite{goerss-jardine})} \label{sss:simpl-hom-groups}

Let $\pmb{X} := \{ X_n \}_{n \geqslant 0}$ be a simplicial set. It is straighforward to define its \emph{set of path components} $\pi_0(X)$.
We define a \emph{simplicial homotopy} between two simplicial maps $f , g : \Delta^n \rightarrow X$ to be a simplicial map $h :\Delta^1 \times \Delta^n \rightarrow X$ such that $h \circ (\ide \times d_1) = g$ and $h \circ (\ide \times d_0) = f$, i.e. such that the following diagram commutes
\[ \begin{tikzcd}[column sep = large]
\Delta^n \arrow[d,right,"\ide \times d_0"] \arrow[dr,"f"] & \\
\Delta^1 \times \Delta^n \arrow[r,"h"] & X \\
\Delta^n \arrow[u,left,"\ide \times d_1"] \arrow[ur,below,"g"] & 
 \end{tikzcd} \ . \]

Suppose now that $\pmb{X}$ is a Kan complex and choose a vertex $x \in X_0$. One can associate to the pair $(\pmb{X},x)$ a sequence of groups called its \emph{simplicial homotopy groups}. For $n \geqslant 1$, consider the set of simplicial maps $\Delta^n \rightarrow \pmb{X}$ taking $\partial \Delta^n$ to $x$. We say that two such maps $f,g : \Delta^n \rightarrow \pmb{X}$ are equivalent if there exists a simplicial homotopy $h$ from $f$ to $g$, that maps $\Delta^1 \times \partial \Delta^n$ to $x$. We define $\pi_n(\pmb{X},x)$ to be the set of equivalence classes of such maps under this equivalence relation. It can be endowed with a composition law as follows. Given two representatives $f$ and $g$ in $\pi_n(\pmb{X},x)$, define the inner horn $\phi_{f,g} : \mathsf{\Lambda}^{n}_{n+1} \rightarrow \pmb{X}$ to send the $i$-th face to $x$ for $i = 0 , \dots , n-2$, the $(n-1)$-th face to $f$ and the $(n+1)$-th face to $g$. The simplicial set $\pmb{X}$ being a Kan complex, this horn can be filled to a $(n+1)$-simplex $\Phi : \Delta^{n+1} \rightarrow \pmb{X}$. We then define $[f] \cdot [g] \in \pi_n(\pmb{X},x)$ to be the equivalence class of the $n$-th face of $\Phi$. 

The assumption that $\pmb{X}$ is a Kan complex then ensures that this composition law is well-defined, and that the set $\pi_n(X,x)$ endowed with this composition law is indeed a group, called the \emph{$n$-th (simplicial) homotopy group of $\pmb{X}$ at $x$}. This group is abelian when $n \geqslant 2$. Moreover, it is naturally isomorphic to the classical homotopy group $\pi_n(|\pmb{X}|,x)$ of the geometric realization $|\pmb{X}|$ of $\pmb{X}$.

\subsection{Cosimplicial resolutions in model categories} \label{alg:ss:cosimplicial-resolutions}

One way to produce Kan complexes is through \emph{cosimplicial resolutions} in model categories. All the results stated in this section are drawn from~\cite{hirschhorn}. We refer to chapters 7 and 8 for basics on model categories, and will only list the technical details that we will need in the proof of Theorem~\ref{alg:th:infinity-gr}. 

We define the simplex category $\Delta$ to be the category whose objets are nonnegative integers $[n]$ and whose sets of morphisms $\Delta([n],[m])$ consists of the increasing maps from $\{ 0 , \dots , n \}$ to $\{ 0 , \dots , m \}$. This is the category encoding cosimplicial objects : a cosimplicial object in a category $\mathcal{C}$ corresponds to a functor $\Delta \rightarrow \mathcal{C}$. We denote $\mathcal{C}^\Delta$ the category of cosimplicial objects, whose morphisms are the morphisms of cosimplicial objects, i.e. the natural transformations between the associated functors $\Delta \rightarrow \mathcal{C}$. For an object $C \in \mathcal{C}$ we denote moreover $const^*C$ the constant cosimplicial object whose cofaces and codegeneracies are the identity maps of $C$.

Let now $\mathcal{C}$ be a model category. The category of cosimplicial objects $\mathcal{C}^\Delta$ can then also be endowed with a model category structure, called its \emph{Reedy model category structure}. Its weak equivalences are the maps of cosimplicial objects that are level-wise weak equivalences in $\mathcal{C}$. Its cofibrants objects are the cosimplicial objects $\pmb{C} := \{ C^n \}$ such that the latching maps $L_n \pmb{C} \rightarrow C^n$ are cofibrations in $\mathcal{C}$. We refer to chapters 15 and 16 of~\cite{hirschhorn} for a definition of latching objects and latching maps, together with a complete description of the Reedy model category structure on $\mathcal{C}^\Delta$.

Let $C \in \mathcal{C}$. A \emph{cosimplicial resolution} of $C$ is defined to be a cofibrant approximation $\pmb{C}$ of $const^*C$ in the model category $\mathcal{C}^\Delta$. In other words, it is the data of a cosimplicial object $\pmb{C} := \{ C^n \}_{n \geqslant 0 }$ of $\mathcal{C}$ together with a cosimplicial morphism $\pmb{C} \rightarrow const^*C$, such that the maps $C^n \rightarrow C$ are weak equivalences in $\mathcal{C}$ and the latching maps $L_n \pmb{C} \rightarrow C^n$ are cofibrations in $\mathcal{C}$.

\begin{lemma}[Lemma 16.5.3 of \cite{hirschhorn}] \label{alg:lemma:hirschhorne}
If $\pmb{C} \rightarrow const^*C$ is a cosimplicial resolution in $\mathcal{C}$ and $D$ is a fibrant object of $\mathcal{C}$, then the simplicial set $\mathcal{C}(\pmb{C} , D)$ is a Kan complex.
\end{lemma}
\noindent Following~\cite{dwyer-kan}, the simplicial set $\mathcal{C}(\pmb{C} , D)$ is called a \emph{function complex} or \emph{homotopy function complex} from $C$ to $D$, and its homotopy type is sometimes called the \emph{derived hom space} from $C$ to $D$.

\section{The $\mathrm{HOM}$-simplicial sets $\mathrm{HOM}_{\mathsf{\Ainf -Alg}}(A,B)_{\bullet}$} \label{alg:s:hom-are-inf-gr}

\subsection{The $\mathrm{HOM}$-simplicial set $\mathrm{HOM}_{\Ainf}(A,B)_\bullet$ is a Kan complex}

The $\mathrm{HOM}$-simplicial sets $\mathrm{HOM}_{\mathsf{\Ainf -Alg}}(A,B)_{\bullet}$ provide a satisfactory framework to study the higher algebra of \Ainf -algebras thanks to the following theorem : 
\begin{theorem} \label{alg:th:infinity-gr}
For $A$ and $B$ two \Ainf -algebras, the simplicial set $\mathrm{HOM}_{\Ainf}(A,B)_\bullet$ is a Kan complex.
\end{theorem}

\noindent The simplicial homotopy groups of this Kan complex are computed in subsection~\ref{alg:ss:homotopy-groups-HOM}. In fact, we can moreover give an explicit description of all inner horn fillers :
\begin{proposition} \label{alg:prop:algebraic-infty}
For every inner horn $\mathsf{\Lambda}^k_n \subset \Delta^n$, there is a one-to-one correspondence
\[ \left\{ \text{fillers} \ \ \ 
\begin{tikzcd}[row sep=large, column sep = large]
\mathsf{\Lambda}^k_n \arrow[d] \arrow[r] & \mathrm{HOM}_{\Ainf}(A,B)_\bullet \\
\Delta^n \arrow[ur, dashed] &
\end{tikzcd} \right\} 
\longleftrightarrow
\left\{ \begin{array}{c} \text{families of maps of degree $-n$ } \\ F^{(m)}_{\Delta^n} : (sA)^{\otimes m} \rightarrow sB , \ m \geqslant 1 \end{array} \right\} \ . \] 
In other words, the Kan complex $\mathrm{HOM}_{\Ainf}(A,B)_\bullet$ is in particular an \emph{algebraic $\infty$-category}.
\end{proposition}
\noindent Note that our choice of terminology \emph{algebraic $\infty$-category} is borrowed from~\cite{robertnicoud-higher}.

One aspect of this construction needs however to be clarified. The points of these $\infty$-groupoids are the \Ainf -morphisms, and the arrows between them are the \Ainf -homotopies. This can be misleading at first sight, but \emph{the points are the morphisms and NOT the algebras} and \emph{the arrows are the homotopies and NOT the morphisms}. 

\subsection{Proof of Theorem~\ref{alg:th:infinity-gr}} \label{alg:ss:proof-theorem-kan-complex}

\subsubsection{The model category structure on $\mathsf{dg-Cogc}$}

Let $C$ be a dg-coalgebra. Define for every $n \geqslant 2$,
\begin{align*}
\Delta^{(n)} &:= (\ide^{\otimes n -2} \otimes \Delta ) \circ (\ide^{\otimes n -3} \otimes \Delta ) \circ \cdots \circ \Delta \\
F_nC &:= \mathrm{Ker} (\Delta^{(n+1)}) \ .
\end{align*} 
We say that $C$ is cocomplete if $C = \cup_{n \geqslant 1} F_n C$. Every tensor coalgebra $\overline{T}V$ is cocomplete. Given any coalgebra $C$ and any cocomplete coalgebra $D$, their tensor product $C \otimes D$ is also a cocomplete dg-coalgebra. 

We denote $\mathsf{dg-Cogc} \subset \mathsf{dg-Cog}$ the full subcategory of cocomplete dg-coalgebras. We introduce moreover $\mathsf{dg-Alg}$, the category of dg-algebras with morphisms of dg-algebras between them. These two categories can then be related through the classical bar-cobar adjunction 
\[ \begin{tikzcd}
\Omega \ : \ \mathsf{dg-Cogc}
\arrow[r, {name=F}, yshift = 5 pt] &
\mathsf{dg-Alg} \ : \ B
\arrow[l, {name=G}, yshift = - 5 pt]
\arrow[phantom, from=F, to=G, "\dashv" rotate=-90, yshift = - 8 pt]
\end{tikzcd} \ . \]
Theorem~1.3.1.2 of~\cite{lefevre-hasegawa} states that the category $\mathsf{dg-Cogc}$ can be made into a model category with the three following classes of morphisms : 
\begin{enumerate}[label=(\roman*)]
\item the class of \emph{weak equivalences} is the class of morphisms $f :C \rightarrow C'$ such that $\Omega f : \Omega C \rightarrow \Omega C'$ is a quasi-isomorphism ;
\item the class of \emph{cofibrations} is the class of morphisms which are monomorphisms when seen as standard morphisms between cochain complexes  ;
\item the class of \emph{fibrations} is the class of morphisms which admit the right-lifting property with respect to trivial cofibrations.
\end{enumerate}
We point out that a weak equivalence between cocomplete dg-coalgebras is always a quasi-isomorphism, but the converse is not true. We list in Lemma~\ref{alg:lemma:model-cat-dgcog} some noteworthy properties of this model category structure on $\mathsf{dg-Cogc}$ that we will need in our upcoming proof of Theorem~\ref{alg:th:infinity-gr}. They can all be found in section~1.3 of~\cite{lefevre-hasegawa}.

Let $C$ be a dg-\Z -module. A \emph{filtration} of $C$ is defined to be a sequence of sub-dg-\Z -modules $C_i \subset C$ such that
\[ C_0 \subset C_1 \subset \cdots \subset C_i \subset C_{i+1} \subset \cdots \ . \] It is \emph{admissible} if $\mathrm{colim}(C_i) = C$ and $C_0 = 0$. Given two filtered dg-\Z -modules $C$ and $C'$, one can then define a \emph{filtered morphism} $f : C \rightarrow C'$ to be a dg-morphism such that $\forall i$, $f (C_i) \subset C_i'$. It is defined to be a \emph{filtered quasi-isomorphism} if $\forall i$, the induced morphism
\[ f_i : C_{i} / C_{i-1} \longrightarrow C'_{i} / C'_{i-1} \]
is a quasi-isomorphism. A \emph{filtered dg-coalgebra} is then defined to be a coalgebra in the category of filtered dg-\Z -modules, in other words a dg-coalgebra together with a filtration $C_i$ on its underlying dg-\Z -module and whose coproduct satisfies 
\[ \Delta_C ( C_i) \subset \bigoplus_{p+q=i} C_p \otimes C_q \ \ \forall i \ . \]

\begin{lemma}[\cite{lefevre-hasegawa}] \label{alg:lemma:model-cat-dgcog}
\begin{enumerate}
\item Every dg-coalgebra in $\mathsf{dg-Cogc}$ is cofibrant.
\item A dg-coalgebra in $\mathsf{dg-Cogc}$ is fibrant if and only if it is isomorphic as a graded coalgebra to a tensor coalgebra $\overline{T}V$.
\item Filtered quasi-isomorphisms between admissible filtered cocomplete dg-coalgebras are weak equivalences.
\end{enumerate}
\end{lemma}

\subsubsection{Proof of Theorem~\ref{alg:th:infinity-gr}}

Recall that the simplicial set $\mathrm{HOM}_{\mathsf{\Ainf -Alg}}(A,B)_{\bullet}$ is defined as 
\[ \mathrm{HOM}_{\mathsf{\Ainf -Alg}}(A,B)_{n} = \mathrm{Hom}_{\mathsf{dg-Cogc}}( \pmb{\Delta}^n \otimes \overline{T}(sA),\overline{T}(sB)) \ . \]
Following Lemma~\ref{alg:lemma:model-cat-dgcog}, the cocomplete dg-coalgebra $\overline{T}(sB)$ is fibrant. It is thus enough to prove that the cosimplicial cocomplete dg-coalgebra $\pmb{C} := \{ \pmb{\Delta}^n \otimes \overline{T}(sA) \}_{n \geqslant 0}$ is a cosimplicial replacement of $\overline{T}(sA)$ and then apply Lemma~\ref{alg:lemma:hirschhorne} in the model category $\mathsf{dg-Cogc}$, to conclude that $\mathrm{HOM}_{\mathsf{\Ainf -Alg}}(A,B)_{\bullet}$ is a Kan complex.
Following subsection~\ref{alg:ss:cosimplicial-resolutions}, we have to prove that~:
\begin{enumerate}[label=(\roman*)]
\item the latching maps $L_n \pmb{C} \rightarrow C^n = \pmb{\Delta}^n \otimes \overline{T}(sA)$ are cofibrations, i.e. they are injective ; \label{item:first-thing}
\item the maps $p \otimes \mathrm{Id}_{\overline{T}(sA)} : \pmb{\Delta}^n \otimes \overline{T}(sA) \rightarrow \pmb{\Delta}^0 \otimes \overline{T}(sA) = \overline{T}(sA)$ are weak equivalences in the model category $\mathsf{dg-Cogc}$, where $p : \pmb{\Delta}^n \rightarrow \pmb{\Delta}^0$ is the map collapsing the simplex $\Delta^n$ on one of its vertices. \label{item:second-thing}
\end{enumerate}
The latching map $L_n \pmb{C} \rightarrow C^n$ simply corresponds to the inclusion $\pmb{\partial \Delta}^n \otimes \overline{T}(sA) \hookrightarrow \pmb{\Delta}^n \otimes \overline{T}(sA)$, hence is injective. See chapters 15 and 16 of~\cite{hirschhorn} for details on how to compute $L_n \pmb{C}$. This proves point~\ref{item:first-thing}.

To prove point~\ref{item:second-thing}, Lemma~\ref{alg:lemma:model-cat-dgcog} states that it is enough to show that $p \otimes \mathrm{Id}_{\overline{T}(sA)}$ is in fact a filtered quasi-isomorphism. Endow $\pmb{\Delta}^n \otimes \overline{T}(sA)$ with the filtration
\[ F_i \left( \pmb{\Delta}^n \otimes \overline{T}(sA) \right) := \pmb{\Delta}^n \otimes \bigoplus_{j=1}^i (sA)^{\otimes j} \ . \]
This filtration is admissible. To prove that $p \otimes \mathrm{Id}_{\overline{T}(sA)}$ is a filtered quasi-isomorphism of admissible filtered dg-coalgebras, we have to prove that the maps 
\[ p \otimes \mathrm{Id}_{(sA)^{\otimes i}} : \pmb{\Delta}^n \otimes (sA)^{\otimes i} \longrightarrow (sA)^{\otimes i} \]
are quasi-isomorphisms.
This is a simple consequence of the fact that the dg-module $\pmb{\Delta}^n$ is a deformation retract of $\pmb{\Delta}^0$. Indeed, defining the degree 0 dg-morphism $i : \pmb{\Delta}^0 \rightarrow \pmb{\Delta}^n$ as $[0] \rightarrow [0]$ and the degree -1 map $h : \pmb{\Delta}^n \rightarrow \pmb{\Delta}^n$ as 
\begin{align*}
[i_0 < \cdots < i_k ] &\longmapsto 0   &\text{if $i_0 = 0$} \ , \\
[i_0 < \cdots < i_k ] &\longmapsto [0 < i_0 < \cdots < i_k ]  &\text{if $i_0 \neq 0$} \ , 
\end{align*}
we check that $p i = \mathrm{Id}$ and $ \mathrm{Id} - ip = [ \partial , h ]$.
This concludes the proof of Theorem~\ref{alg:th:infinity-gr}.

\subsection{Proof of Proposition~\ref{alg:prop:algebraic-infty}}

\subsubsection{Proof of Proposition~\ref{alg:prop:algebraic-infty}} \label{alg:sss:proof-proposition-algebraic}

Let $A$ and $B$ be two \Ainf -algebras. We now prove Proposition~\ref{alg:prop:algebraic-infty}, using the shifted bar construction framework, that is by defining an \Ainf -algebra to be a set of degree $+1$ operations $b_n : (sA)^{\otimes n} \rightarrow sA$ satisfying equations
\[ \sum_{i_1+i_2+i_3=n} b_{i_1+1+i_3} (\ide^{\otimes i_1} \otimes b_{i_2} \otimes \ide^{\otimes i_3} ) = 0 \ . \]
The proof will mainly consist of easy but tedious combinatorics. We recommend reading it in two steps : first ignoring the signs ; then adding them at the second reading stage and referring to section~\ref{alg:ss:signs-n-ainf-morph} for the sign conventions on the shifted \Ainf -equations.  

Consider an inner horn $\mathsf{\Lambda}^k_n \rightarrow \mathrm{HOM}_{\Ainf}(A,B)_\bullet$, where $0 < k < n$. It corresponds to a collection of degree $-\mathrm{dim}(I)$ morphisms
\[ F^{(m)}_I : (sA)^{\otimes m} \longrightarrow sB \]
for $I\subset \mathsf{\Lambda}^k_n$, which satisfy the \Ainf -equations
\[ \sum_{j=0}^{\mathrm{dim}(I)} (-1)^j F^{(m)}_{\partial_j I} + (-1)^{|I|} \sum_{i_1+i_2+i_3=m} F^{(i_1+1+i_3)}_I (\ide^{\otimes i_1} \otimes b_{i_2} \otimes \ide^{\otimes i_3})
= \sum_{\substack{i_1 + \cdots + i_s = m \\ I_1 \cup \cdots \cup I_s = I}} b_s ( F^{(i_1)}_{I_1} \otimes \cdots \otimes F^{(i_s)}_{I_s}) \ . \] 
Filling this horn amounts then to defining a collection of operations
\[ F^{(m)}_{[0 < \cdots < \widehat{k} < \cdots < n]} : (sA)^{\otimes m} \longrightarrow sB \ \ \text{and} \ \ F^{(m)}_{\Delta^n} : (sA)^{\otimes m} \longrightarrow sB \ , \]
of respective degree $-(n-1)$ and $-n$, and respectively satisfying the equations 
\begin{align*}
\sum_{l=0}^{n-1} (-1)^l F^{(m)}_{\partial_l [0 < \cdots < \widehat{k} < \cdots < n]} + &(-1)^{n-1} \sum_{i_1+i_2+i_3=m} F^{(i_1+1+i_3)}_{[0 < \cdots < \widehat{k} < \cdots < n]} (\ide^{\otimes i_1} \otimes b_{i_2} \otimes \ide^{\otimes i_3}) \\
= &\sum_{\substack{i_1 + \cdots + i_s = m \\ I_1 \cup \cdots \cup I_s = [0 < \cdots < \widehat{k} < \cdots < n]}} b_s ( F^{(i_1)}_{I_1} \otimes \cdots \otimes F^{(i_s)}_{I_s}) \ , \tag{i}
\end{align*}
and
\begin{align*}
\sum_{j=0}^{n} (-1)^j F^{(m)}_{\partial_j \Delta^n} + &(-1)^{n} \sum_{i_1+i_2+i_3=m} F^{(i_1+1+i_3)}_{\Delta^n} (\ide^{\otimes i_1} \otimes b_{i_2} \otimes \ide^{\otimes i_3}) \\
= &\sum_{\substack{i_1 + \cdots + i_s = m \\ I_1 \cup \cdots \cup I_s = \Delta^n}} b_s ( F^{(i_1)}_{I_1} \otimes \cdots \otimes F^{(i_s)}_{I_s}) \tag{ii} \ .
\end{align*}

We begin by pointing out that the operations $F^{(m)}_{\Delta^n}$ indeed completely determine the maps $F^{(m)}_{[0 < \cdots < \widehat{k} < \cdots < n]}$ under the formula 
\begin{align*}
F^{(m)}_{[0 < \cdots < \widehat{k} < \cdots < n]} = (-1)^k \left( \sum_{\substack{j= 0 \\ j \neq k}}^n (-1)^{j+1} F^{(m)}_{[0 < \cdots < \widehat{j} < \cdots < n]} \right. + & \left. \sum_{\substack{ i_1 + \cdots + i_s = m \\ I_1 \cup \cdots \cup I_s = \Delta^n}} b_s ( F^{(i_1)}_{I_1} \otimes \cdots \otimes F^{(i_s)}_{I_s}) \right. \\
+  (-1)^{n+1} & \left. \sum_{i_1+i_2+i_3=m} F^{(i_1+1+i_3)}_{\Delta^n} (\ide^{\otimes i_1} \otimes b_{i_2} \otimes \ide^{\otimes i_3}) \right) \ . 
\end{align*} 
To prove Proposition~\ref{alg:prop:algebraic-infty}, it remains to show that for any collection of operations $(F^{(m)}_{\Delta^n})_{m \geqslant 1}$, we can fill the inner horn $\mathsf{\Lambda}^k_n \rightarrow \mathrm{HOM}_{\Ainf}(A,B)_\bullet$ by defining the operations $F^{(m)}_{[0 < \cdots < \widehat{k} < \cdots < n]}$ as above. Note that the $F^{(m)}_{[0 < \cdots < \widehat{k} < \cdots < n]}$ are well-defined as all the morphisms $F^{(m)}_{I}$ appearing in their definition correspond to faces of the horn $\mathsf{\Lambda}^k_n$ or to the $F^{(m)}_{\Delta^n}$. 

It is clear that this choice of filler satisfies equations (ii), and we have now to verify that equations (i) are satisfied. For the sake of readability, we will only carry out the details of the proof in the case where $F^{(m)}_{\Delta^n} = 0$ for all $m$. In this regard, we will list one by one the terms of the left-hand side and right-hand side of this equality with their signs, and use the \Ainf -equations for the $b_i$ and the $F^{(m)}_{I}$ where $I \subset \mathsf{\Lambda}^k_n$, in order to show that the two sides are indeed equal.

The left-hand side consists of the following terms :
\begin{align*}
(-1)^l F_{\partial_l [0 < \cdots < \widehat{k} < \cdots < n]}^{(m)} \tag{A}
\end{align*}
for $l = 0 , \dots, n-1$ ;
\begin{align*}
(-1)^{n+k+j} F_{[0 < \cdots < \widehat{j} < \cdots < n]}^{(i_1+1+i_3)}(\ide^{\otimes i_1} \otimes b_{i_2} \otimes \ide^{\otimes i_3}) \tag{B}
\end{align*}
for $i_1+i_2+i_3=m$ and $j=0,\dots,\widehat{k},\dots,m$ ;
\begin{align*}
(-1)^{n-1+k} b_s(F_{I_1}^{(j_1)} \otimes \cdots \otimes F_{I_s}^{(j_s)})(\ide^{\otimes i_1} \otimes b_{i_2} \otimes \ide^{\otimes i_3}) \tag{C}
\end{align*}
for $i_1 + i_2 + i_3 = m$, $j_1 + \cdots + j_s = i_1 + 1 + i_3$ and $I_1 \cup \cdots \cup I_s = \Delta^n$ with $I_u \neq \Delta^n$ for all $u$.

The right-hand side has the following terms :
\begin{align*}
b_s (F_{I_1}^{(i_1)} \otimes \cdots \otimes F_{I_s}^{(i_s)}) \tag{D}
\end{align*}
for $i_1 + \cdots + i_s =m$ and $I_1 \cup \cdots \cup I_s = [ 0 < \cdots < \widehat{k} < \cdots < n]$ with $I_u \neq [ 0 < \cdots < \widehat{k} < \cdots < n]$ for all $u$ ;
\begin{align*}
(-1)^k b_s ( F_{I_1}^{(i_1)} \otimes \cdots \otimes F_{I_{t-1}}^{(i_{t-1})} \otimes b_q (F_{J_1}^{(j_1)} \otimes \cdots \otimes F_{J_q}^{(j_q)}) \otimes F_{I_{t+1}}^{(i_{t+1})} \otimes \cdots \otimes F_{I_s}^{(i_s)}) \tag{E}
\end{align*}
where, setting $I_t = J_1 \cup \cdots \cup J_q$, $i_t = j_1 + \cdots + j_q$, $i_1 + \cdots + i_s = m$ and $I_1 \cup \cdots \cup I_s = \Delta^n$, with $I_t = \Delta^n$ and $J_r \neq \Delta^n$ for all $r$ ;
\begin{align*}
(-1)^{j+k+1} b_s ( F_{I_1}^{(i_1)} \otimes \cdots \otimes F_{I_{t}}^{(i_{t})} \otimes \cdots \otimes F_{I_s}^{(i_s)}) \tag{F}
\end{align*}
for $j=0,\dots , \widehat{k} , \dots ,n$, where $i_1 + \cdots + i_s=m$ and $I_1 \cup \cdots \cup I_s = [ 0 < \cdots < \widehat{j} < \cdots < n]$ with $I_t = [ 0 < \cdots < \widehat{j} < \cdots < n]$.

Our goal is to prove that $ A + B + C = D + E + F$ or equivalently, that
\[ A + B + C - D - E - F = 0 \ . \]
Applying the \Ainf -equations for the $F^{(m)}_{[ 0 < \cdots < \widehat{j} < \cdots < n]}$, $j \neq k$, 
we have that 
\[ A + B - F = G \ , \]
the terms of the sum $G$ being of the form
\begin{align*}
(-1)^{j+k+1} b_s(F_{I_1}^{(i_1)} \otimes \cdots \otimes F_{I_s}^{(i_s)}) \tag{G}
\end{align*}
where $j=0 , \dots, \widehat{k} , \dots , n$, $i_1 + \cdots + i_s = m$ and $I_1 \cup \cdots \cup I_s = [ 0 < \cdots < \widehat{j} < \cdots < n]$ with $I_u \neq [ 0 < \cdots < \widehat{j} < \cdots < n]$ for all $u$.

Applying now the \Ainf -equations for the $F_{I_u}^{(i_u)}$, where $I_u \neq \Delta^n$, yields the equality
\[ C - D + G = H \ , \]
the terms of the sum $H$ having the form
\begin{align*}
(-1)^{n-1 + k + \sum_{u=t}^s |I_u|} b_s ( F_{I_1}^{(i_1)} \otimes \cdots \otimes F_{I_{t-1}}^{(i_{t-1})} \otimes b_q (F_{J_1}^{(j_1)} \otimes \cdots \otimes F_{J_q}^{(j_q)}) \otimes F_{I_{t+1}}^{(i_{t+1})} \otimes \cdots \otimes F_{I_s}^{(i_s)}) \tag{H}
\end{align*}
where, setting $I_t = J_1 \cup \cdots \cup J_q$ and $i_t = j_1 + \cdots + j_q$, $i_1 + \cdots + i_s = m$ and $I_1 \cup \cdots \cup I_s = \Delta^n$ with $I_u \neq \Delta^n$ for all $u$.

Finally, applying the \Ainf -equations for the $b_i$ proves the equality
\[ - E + H = 0 \ , \]
which concludes the proof.

\subsubsection{Remark on the proof} \label{alg:sss:remark-proof-inf-cat-ainf}

We point out that this proof does not adapt to the more general case of a $\mathrm{HOM}$-simplicial set $\mathrm{HOM}_{\mathsf{dg-Cog}}(C,C')_{\bullet}$. Indeed, while we can always solve the equation
\[ [ \partial , f_{\Delta^n} ] = \sum_{j=0}^n (-1)^j f_{[0 < \dots < \widehat{j} < \dots < n]} \ , \]
by setting $f_{\Delta^n} = 0 $ and $f_{[0 < \dots < \widehat{k} < \dots < n]} = (-1)^k \sum_{j=0 , \neq k}^n (-1)^{j+1} f_{[0 < \dots < \widehat{j} < \dots < n]}$, this choice of morphisms falls short to satisfy the equation
\[ \Delta_{C'} f_{\Delta^n} = \sum_{I_1 \cup I_2 = \Delta^n} (f_{I_1} \otimes f_{I_2}) \Delta_C \ . \]

\subsection{Homotopy groups} \label{alg:ss:homotopy-groups-HOM}

The simplicial set $\mathrm{HOM}_{\mathsf{\Ainf -Alg}}(A,B)_{\bullet}$ being a Kan complex, we can in particular compute its simplicial homotopy groups. We fix throughout the rest of this subsection an \Ainf -morphism $F$ from $A$ to $B$, i.e. a point of $\mathrm{HOM}_{\mathsf{\Ainf -Alg}}(A,B)_{\bullet}$. We will moreover work with the suspended definition of $n$-morphisms that we already used in subsection~\ref{alg:sss:proof-proposition-algebraic}.

\begin{proposition}
The set of path components $\pi_0 \left( \mathrm{HOM}_{\mathsf{\Ainf -Alg}}(A,B)_{\bullet} \right)$ corresponds to the set of equivalence classes of \Ainf -morphisms from $A$ to $B$ under the equivalence relation "being \Ainf -homotopic".
\end{proposition}

A simplicial map $\Delta^n \rightarrow \mathrm{HOM}_{\mathsf{\Ainf -Alg}}(A,B)_{\bullet}$ taking $\partial \Delta^n$ to $F$ corresponds to a $n$-morphism $(F_I^{(m)})_{I \subset \Delta^n}^{m \geqslant 1}$ such that $F_I^{(m)} = F^{(m)}$ for all $I$ such that $\mathrm{dim}(I) = 0$ and $F_I^{(m)} = 0$ for all $I$ such that $0 < \mathrm{dim}(I) < n$. In other words, this simplicial map simply corresponds to the data of maps $F_{\Delta^n}^{(m)} : (sA)^{\otimes m} \rightarrow sB$ of degree $-n$ such that
\begin{align*} \label{alg:eq:hom-groups}
&(-1)^{n} \sum_{i_1+i_2+i_3=m} F^{(i_1+1+i_3)}_{\Delta^n} \left( \ide^{\otimes i_1} \otimes b_{i_2} \otimes \ide^{\otimes i_3} \right) \\
= &\sum_{\substack{i_1 + \cdots + i_s  + l \\ + j_1 + \cdots + j_t = m}} b_{s+1+t} \left( F^{(i_1)} \otimes \cdots \otimes F^{(i_s)} \otimes F_{\Delta^n}^{(l)} \otimes F^{(j_1)} \otimes \cdots \otimes F^{(j_t)} \right) \ . \tag{$\star$}
\end{align*}

\begin{proposition} \label{alg:prop:simpl-hom}
Let $\mathcal{F} ,\mathcal{G} : \Delta^n \rightarrow \mathrm{HOM}_{\mathsf{\Ainf -Alg}}(A,B)_{\bullet}$ be two simplicial maps taking $\partial \Delta^n$ to $F$, that we will respectively denote $(F_{\Delta^n}^{(m)})$ and $(G_{\Delta^n}^{(m)})$. Two such maps are then equivalent under the simplicial homotopy relation if and only if there exists a collection of maps $H^{(m)} : (sA)^{\otimes m} \rightarrow sB$ of degree $-(n+1)$ such that
\begin{align*}
&G_{\Delta^n}^{(m)} - F_{\Delta^n}^{(m)} + (-1)^{n+1} \sum_{i_1+i_2+i_3=m} H^{(i_1+1+i_3)} (\ide^{\otimes i_1} \otimes b_{i_2} \otimes \ide^{\otimes i_3}) \\
= &\sum_{\substack{i_1 + \cdots + i_s  + l \\ + j_1 + \cdots + j_t = m}} b_{s+1+t} ( F^{(i_1)} \otimes \cdots \otimes F^{(i_s)} \otimes H^{(l)} \otimes F^{(j_1)} \otimes \cdots \otimes F^{(j_t)} ) \ .
\end{align*}
\end{proposition}

\begin{proof} 
Recall from subsection~\ref{sss:simpl-hom-groups} that a simplicial homotopy from $\mathcal{F}$ to $\mathcal{G}$ is defined to be a simplicial map $\mathcal{H} : \Delta^1 \times \Delta^n \rightarrow \mathrm{HOM}_{\mathsf{\Ainf -Alg}}(A,B)_{\bullet}$ such that $\mathcal{H} |_{[0] \times \Delta^n} = \mathcal{F}$, $\mathcal{H} |_{[1] \times \Delta^n} = \mathcal{G}$ and that maps $\Delta^1 \times \partial \Delta^n$ to $F$. Beware that the datum of a simplicial map $\mathcal{H} : \Delta^1 \times \Delta^n \rightarrow \mathrm{HOM}_{\mathsf{\Ainf -Alg}}(A,B)_{\bullet}$ is in general NOT equivalent to a morphism of dg-coalgebras $\pmb{\Delta}^1 \otimes \pmb{\Delta}^n \otimes \overline{T}(sA) \rightarrow \overline{T}(sB)$. To understand the map $\mathcal{H}$, we first have to make explicit the non-degenerate simplices of the simplicial set $\Delta^1 \times \Delta^n$.

Recall that the $k$-simplices of the simplicial set $\Delta^m$ are the monotone sequences of integers bounded by 0 and $m$
\[ \left( \begin{array}{cccc}
         i_0 & i_1 & \cdots & i_k 
  \end{array} \right) \ \text{where} \ 0 \leqslant i_0 \leqslant i_1 \leqslant \cdots \leqslant i_k \leqslant m  \ . \] 
Following~\cite{milnor-simplicial}, the non-degenerate $k$-simplices of the simplicial set $\Delta^1 \times \Delta^n$ are then labeled by all pairs composed of a $k$-simplex $\sigma$ of $\Delta^1$ and a $k$-simplex $\sigma '$ of $\Delta^n$ such that there does not exist $0 \leqslant j < k$ such that $\sigma_j = \sigma_{j+1}$ and $\sigma'_j = \sigma'_{j+1}$. For instance, the following two pairs of sequences label non-degenerate 3-simplices of $\Delta^1 \times \Delta^3$
\begin{align*} \left( \begin{array}{cccc}
         0&0&0&1 \\
         0&1&2&3
  \end{array} \right)  &&  \left( \begin{array}{cccc}
         0&0&1&1 \\
         0&1&1&2
  \end{array} \right)  \ ,
\end{align*}
while the following pair of sequences is a degenerate 3-simplex of $\Delta^1 \times \Delta^3$
\[ \left( \begin{array}{cccc}
         0&0&0&1 \\
         0&1&1&3
  \end{array} \right)  \ . \]
  
\begin{figure}[h]
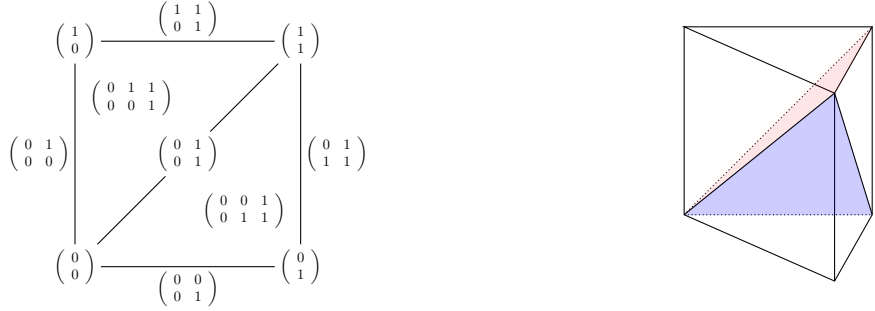
 
    \centering
    \begin{subfigure}{0.45\textwidth}
    \centering
       \triangprodunun
    \end{subfigure} ~
    \begin{subfigure}{0.45\textwidth}
    \centering
        \triangprodundeux
    \end{subfigure} 
    \caption{On the left, the labeling of the non-degenerate simplices of $\Delta^1 \times \Delta^1$. On the right, the (unlabeled) non-degenerate simplices of $\Delta^1 \times \Delta^2$. The two inner non-degenerate 2-simplices of $\Delta^1 \times \Delta^2$ are colored in red and in blue.} \label{alg:fig:triang-prod}
\end{figure}

We will use the following properties of the non-degenerate simplices of the simplicial set $\Delta^1 \times \Delta^n$ in our proof of proposition~\ref{alg:prop:simpl-hom}~: 
\begin{enumerate}[label=(\roman*)]
\item There are exactly $n+1$ non-degenerate $(n+1)$-simplices, labeled by the pairs of sequences
\[ \left( \begin{array}{ccccccccc}
         0 & 0 & 0 & \cdots & 0 & 1 & 1 & \cdots & 1 \\
         0 & 1 & 2 & \cdots & k & k & k+1 & \cdots & n
  \end{array} \right)  \ . \]
The non-degenerate $(n+1)$-simplex labeled by the above pair of sequences will be called the \emph{$k$-th non-degenerate $(n+1)$-simplex} of $\Delta^1 \times \Delta^n$.
\item All non-degenerate simplices of dimension $\leqslant n-1$ lie in $\Delta^1 \times \partial \Delta^n$.
\item All simplices of dimension $\geqslant n+2$ are degenerate.
\item There are exactly $n$ non-degenerate $n$-simplices lying in the interior of $\Delta^1 \times \Delta^n$. They are labeled by the pairs of sequences
\[ \left( \begin{array}{cccccccc}
         0 & 0 & \cdots & 0 & 1 & 1 & \cdots & 1 \\
         0 & 1 &  \cdots & k - 1 & k & k+1 & \cdots & n
  \end{array} \right) \ \text{for} \ 1 \leqslant k \leqslant n  \ . \]
The non-degenerate $n$-simplex labeled by the above pair of sequences will be called the \emph{$k$-th inner non-degenerate $n$-simplex} of $\Delta^1 \times \Delta^n$.
\end{enumerate}
We point out that taking the $l$-th face of a simplex of $\Delta^1 \times \Delta^n$ simply corresponds to deleting the $l$-th column of the array labeling it. For instance,
\begin{align*}
\partial_1 \left( \begin{array}{cccc}
         0&0&1&1 \\
         0&1&1&2
  \end{array} \right) 
  = \left( \begin{array}{ccc}
         0&1&1 \\
         0&1&2
  \end{array} \right) && \partial_3 \left( \begin{array}{cccc}
         0&0&1&1 \\
         0&1&1&2
  \end{array} \right) 
  = \left( \begin{array}{ccc}
         0&0&1 \\
         0&1&1
  \end{array} \right) \ .
\end{align*} 

A simplicial homotopy $\mathcal{H} : \Delta^1 \times \Delta^n \rightarrow \mathrm{HOM}_{\mathsf{\Ainf -Alg}}(A,B)_{\bullet}$ is equivalent to the data of maps $H_K^{(m)} : (sA)^{\otimes m} \rightarrow sB$ for every non-degenerate simplex $K$ of $\Delta^1 \times \Delta^n$, which moreover satisfy the \Ainf -equations for higher morphisms. 
According to the previous description of the non-degenerate simplices of $\Delta^1 \times \Delta^n$ :
\begin{enumerate}[label=(\roman*)]
\item The condition $\Delta^1 \times \partial \Delta^n \mapsto F$ implies that for every non-degenerate simplex $K$ of dimension $1 \leqslant \mathrm{dim}(K) \leqslant n-1$, $H_K^{(m)}=0$, and for every vertex $v$ of $\Delta^1 \times \Delta^n$, $H_v^{(m)}=F^{(m)}$. This also implies that all non-degenerate $n$-simplices $K$ lying in $\Delta^1 \times \partial \Delta^n$ are such that $H_K^{(m)}=0$.
\item The condition $\mathcal{H} |_{[0] \times \Delta^n} = \mathcal{F}$, $\mathcal{H} |_{[1] \times \Delta^n} = \mathcal{G}$ implies that the non-degenerate $n$-simplices $[0 ] \times \Delta^n$ and $[1] \times \Delta^n$ are respectively sent to the $F_{\Delta^n}^{(m)}$ and the $G_{\Delta^n}^{(m)}$.
\item For $K$ the $k$-th inner non-degenerate $n$-simplex, we will write $L_k^{(m)} := H_K^{(m)}$. For a fixed~$k$, the maps $L_k^{(m)}$ satisfy the same \Ainf -equations~(\ref{alg:eq:hom-groups}) as $F_{\Delta^n}^{(m)}$ and $G_{\Delta^n}^{(m)}$. We moreover set  $L_{n+1}^{(m)} := F_{\Delta^n}^{(m)}$ and $L_{0}^{(m)} := G_{\Delta^n}^{(m)}$.
\item Finally, we denote $H_k^{(m)}$ for the collection of maps associated to the $k$-th $(n+1)$-simplex. It satisfies the following \Ainf -equations
\begin{align*}
&(-1)^{k+1} L_{k+1}^{(m)} + (-1)^{k} L_{k}^{(m)} + (-1)^{n+1} \sum_{i_1+i_2+i_3=m} H_k^{(i_1+1+i_3)} (\ide^{\otimes i_1} \otimes b_{i_2} \otimes \ide^{\otimes i_3}) \\
= &\sum_{\substack{i_1 + \cdots + i_s  + l \\ + j_1 + \cdots + j_t = m}} b_{s+1+t} ( F^{(i_1)} \otimes \cdots \otimes F^{(i_s)} \otimes H_k^{(l)} \otimes F^{(j_1)} \otimes \cdots \otimes F^{(j_t)} ) \ .
\end{align*}
\end{enumerate}

Using this characterization of a simplicial homotopy from $\mathcal{F}$ to $\mathcal{G}$, we check that the collection of degree $-(n+1)$ maps 
\[ H^{(m)} := \sum_{k=0}^n (-1)^{k} H_k^{(m)} \]
is such that 
\begin{align*}
&G_{\Delta^n}^{(m)} - F_{\Delta^n}^{(m)} + (-1)^{n+1} \sum_{i_1+i_2+i_3=m} H^{(i_1+1+i_3)} (\ide^{\otimes i_1} \otimes b_{i_2} \otimes \ide^{\otimes i_3}) \\
= &\sum_{\substack{i_1 + \cdots + i_s  + l \\ + j_1 + \cdots + j_t = m}} b_{s+1+t} ( F^{(i_1)} \otimes \cdots \otimes F^{(i_s)} \otimes H^{(l)} \otimes F^{(j_1)} \otimes \cdots \otimes F^{(j_t)} ) \ .
\end{align*}
Conversely, we check that such a collection of maps can be arranged into a simplical homotopy from $\mathcal{F}$ to $\mathcal{G}$, by defining $L_0^{(m)} := G_{\Delta^n}^{(m)}$, $L_k^{(m)} := F_{\Delta^n}^{(m)}$ for $k \geqslant 1$, $H_0^{(m)} := H^{(m)}$ and $H_k^{(m)} := 0$ for $k \geqslant 1$. This concludes the proof of the proposition.
\end{proof}

We finally make explicit the composition law on these simplicial homotopy groups. Consider $( F_{\Delta^n}^{(m)} )^{m \geqslant 1}$ and $( G_{\Delta^n}^{(m)} )^{m \geqslant 1}$ two representatives in $\pi_n \left( \mathrm{HOM}_{\mathsf{\Ainf -Alg}}(A,B)_{\bullet} , F \right)$. Filling the cone $\phi_{ F_{\Delta^n} , G_{\Delta^n} } : \mathsf{\Lambda}_{n+1}^n \rightarrow  \mathrm{HOM}_{\mathsf{\Ainf -Alg}}(A,B)_{\bullet}$ defined in subsection~\ref{sss:simpl-hom-groups} with $\phi_{\Delta^{n+1}}^{(m)}=0$ as in the proof of Proposition~\ref{alg:prop:algebraic-infty}, we get that a representative for $[\mathcal{F}] \cdot [\mathcal{G}]$ is 
\[ \resizebox{\hsize}{!}{\begin{math} \begin{aligned}
G_{\Delta^1}^{(m)} + F_{\Delta^1}^{(m)} - \sum_{\substack{i_1 + \cdots + i_s  + l_1 \\ + j_1 + \cdots + j_t + l_2 \\ + k_1 + \cdots + k_u = m}} b_{s+t+u+2} ( F^{(i_1)} \otimes \cdots \otimes F^{(i_s)} \otimes F_{\Delta^1}^{(l_1)} \otimes F^{(j_1)} \otimes \cdots \otimes F^{(j_t)}  \otimes G_{\Delta^1}^{(l_2)} \otimes F^{(k_1)} \otimes \cdots \otimes F^{(k_u)}) \ .
\end{aligned}
\end{math}} \]
in the $n=1$ case, and
\[ G_{\Delta^n}^{(m)} + F_{\Delta^n}^{(m)} \]
if $n \geqslant 2$.
We get in particular that this composition law is indeed abelian when $n \geqslant 2$. All of our computations are summarized in the following theorem~:

\begin{theorem} \label{alg:th:simpl-hom-groups}
\begin{enumerate}[label=(\roman*)]
\item For $n \geqslant 1$, the set $\pi_n \left( \mathrm{HOM}_{\mathsf{\Ainf -Alg}}(A,B)_{\bullet} , F \right)$ corresponds to the equivalence classes of collections of degree $-n$ maps $F_{\Delta^n}^{(m)} : (sA)^{\otimes m} \rightarrow sB$ satisfying equations~\ref{alg:eq:hom-groups}, where two such collections of maps $( F_{\Delta^n}^{(m)} )^{m \geqslant 1}$ and $( G_{\Delta^n}^{(m)} )^{m \geqslant 1}$ are equivalent if and only if there exists a collection of degree $-(n+1)$ maps $H^{(m)} : (sA)^{\otimes m} \rightarrow sB$ such that
\begin{align*}
&G_{\Delta^n}^{(m)} - F_{\Delta^n}^{(m)} + (-1)^{n+1} \sum_{i_1+i_2+i_3=m} H^{(i_1+1+i_3)} (\ide^{\otimes i_1} \otimes b_{i_2} \otimes \ide^{\otimes i_3}) \\
= &\sum_{\substack{i_1 + \cdots + i_s  + l \\ + j_1 + \cdots + j_t = m}} b_{s+1+t} ( F^{(i_1)} \otimes \cdots \otimes F^{(i_s)} \otimes H^{(l)} \otimes F^{(j_1)} \otimes \cdots \otimes F^{(j_t)} ) \ .
\end{align*}
\item The composition law on $\pi_1 \left( \mathrm{HOM}_{\mathsf{\Ainf -Alg}}(A,B)_{\bullet} , F \right)$ is given by the formula
\[ \resizebox{\hsize}{!}{\begin{math} \begin{aligned}
G_{\Delta^1}^{(m)} + F_{\Delta^1}^{(m)} - \sum_{\substack{i_1 + \cdots + i_s  + l_1 \\ + j_1 + \cdots + j_t + l_2 \\ + k_1 + \cdots + k_u = m}} b_{s+t+u+2} ( F^{(i_1)} \otimes \cdots \otimes F^{(i_s)} \otimes F_{\Delta^1}^{(l_1)} \otimes F^{(j_1)} \otimes \cdots \otimes F^{(j_t)}  \otimes G_{\Delta^1}^{(l_2)} \otimes F^{(k_1)} \otimes \cdots \otimes F^{(k_U)}) \ .
\end{aligned}
\end{math}} \]
\item If $n \geqslant 2$, the composition law on $\pi_n \left( \mathrm{HOM}_{\mathsf{\Ainf -Alg}}(A,B)_{\bullet} , F \right)$ is given by the formula
\[ G_{\Delta^n}^{(m)} + F_{\Delta^n}^{(m)} \ . \]
\end{enumerate}
\end{theorem}

\subsection{A conjecture on the HOM-simplicial sets $\mathrm{HOM}_{\mathsf{\ombas -Alg}}(A,B)_{\bullet}$}

Given $A$ and $B$ two \ombas -algebras, we define the HOM-simplicial set
\[ \mathrm{HOM}_{\mathsf{\ombas -Alg}} (A,B)_n := \mathrm{Hom}_{\mathsf{(\ombas , \ombas )-op.
bimod.}} (n-\Omega B As - \mathrm{Morph} , \mathrm{Hom}(A,B) ) \ . \]
Drawing from Theorem~\ref{alg:th:infinity-gr}, we conjecture the following result : 
\begin{conjecture}
The simplicial sets $\mathrm{HOM}_{\mathsf{\ombas -Alg}} (A,B)_\bullet$ are $\infty$-categories.
\end{conjecture}
The proof without signs should follow the same lines as the proof without signs of Theorem~\ref{alg:th:infinity-gr}, working this time with stable ribbon trees and gauged stable ribbon trees instead of corollae. The sign computations will however be much more complicated, as we did not describe a construction analogous to the shifted bar construction which would yield ad hoc sign conventions. 

\section{Higher functors and pre-natural transformations between \Ainf -categories} \label{alg:s:a-inf-ainf-a-b}

\subsection{$n$-functors between \Ainf -categories} \label{ss:n-functors}

Recall that an \emph{\Ainf -category} $\mathcal{A}$ is defined to be the data
\begin{enumerate}[label=(\roman*)]
\item of a collection of objects $\mathrm{Ob}(\mathcal{A})$ ;
\item for every $A_0,A_1 \in \mathrm{Ob}(\mathcal{A})$ of a dg-module $\mathcal{A}(A_0,A_1)$ ; 
\item for every $A_0,\dots , A_n \in \mathrm{Ob}(\mathcal{A})$ of a degree $2-n$ map 
\[ m_n : \mathcal{A}(A_0,A_1) \otimes \cdots \otimes \mathcal{A}(A_{n-1},A_n) \longrightarrow \mathcal{A}(A_0,A_n) \ , \]
such that the maps $m_n$ satisfy a categorical version of the \Ainf -equations for \Ainf -algebras.
\end{enumerate}
The maps $m_n$ are called the \emph{higher compositions} of $\mathcal{A}$ and are to be thought of as the higher homotopies encoding the lack of associativity of the composition maps $m_2$. In particular, an \Ainf -category $\mathcal{A}$ induces an ordinary category $H^*(\mathcal{A})$ in cohomology. We refer to subsection \ref{ss:two-notions} for a discussion of the existence of identity morphisms in $H^*(\mathcal{A})$.

An \emph{\Ainf -functor} between two \Ainf -categories $\mathcal{F} : \mathcal{A} \rightarrow \mathcal{B}$ is then defined to be the data 
\begin{enumerate}[label=(\roman*)]
\item of a map $ \mathcal{F} : \mathrm{Ob}(\mathcal{A}) \rightarrow \mathrm{Ob}(\mathcal{B})$ ;
\item for every $A_0,\dots , A_n \in \mathrm{Ob}(\mathcal{A})$ of a degree $1-n$ map 
\[ f_n : \mathcal{A}(A_0,A_1) \otimes \cdots \otimes \mathcal{A}(A_{n-1},A_n) \longrightarrow \mathcal{B}(\mathcal{F}(A_0),\mathcal{F}(A_n)) \ , \]
such that the maps $f_n$ satisfy a categorical version of the \Ainf -equations for \Ainf -morphisms.
\end{enumerate}
\Ainf -functors correspond to functors between \Ainf -categories that preserve the composition up to higher coherent homotopies, and induce ordinary functors $H^*(\mathcal{F}) : H^*(\mathcal{A}) \rightarrow H^*(\mathcal{B})$ between the cohomological categories.

One can then similarly define a categorical generalization of $n$-morphisms between \Ainf -algebras given by \emph{$n$-functors between \Ainf -categories}. The sets of $n$-functors between two \Ainf -categories $\mathcal{A}$ and $\mathcal{B}$ then fit into a simplicial set 
\[ \mathrm{HOM}_{\mathsf{\Ainf -Cat}}(\mathcal{A},\mathcal{B})_{\bullet} \ . \]
It is straightforward from the proof of subsection \ref{alg:sss:proof-proposition-algebraic} that these simplicial sets are again algebraic $\infty$-categories. In analogy with Theorem \ref{alg:th:infinity-gr}, we expect that these simplicial sets are Kan complexes. The proof of this statement would rely on working out the homotopy theory of dg-cocategories. 

\subsection{The \Ainf -category of \Ainf -functors $\mathrm{Func}_{\mathcal{A},\mathcal{B}}$ and the simplicial nerve functor}

Given two \Ainf -categories $\mathcal{A}$ and $\mathcal{B}$, Fukaya constructed in \cite{fukaya-floer} an \Ainf -category $\mathrm{Func}_{\mathcal{A},\mathcal{B}}$ whose objects are \Ainf -functors from $\mathcal{A}$ to $\mathcal{B}$. See also~\cite{lefevre-hasegawa} and \cite{seidel-fukaya}. 
The goal of this section is to compare the construction of \cite{fukaya-floer} to the Kan complex $\mathrm{HOM}_{\mathsf{\Ainf -Cat}}(\mathcal{A},\mathcal{B})_{\bullet}$. 

We begin by defining the \Ainf -category $\mathrm{Func}_{\mathcal{A},\mathcal{B}}$. 
The objects of $\mathrm{Func}_{\mathcal{A},\mathcal{B}}$ are \Ainf -functors $\mathcal{A} \rightarrow \mathcal{B}$.
Given two \Ainf -functors $\mathcal{F}_0= \{ f_0^{(m)} \}$ and $\mathcal{F}_1 = \{ f_1^{(m)} \}$, an element $\mathcal{F}_{01} \in \mathrm{Func}_{\mathcal{A},\mathcal{B}}(\mathcal{F}_0,\mathcal{F}_1)$ is called a \emph{pre-natural transformation} and consists of a collection of morphisms 
\[ f_{01}^{(m)} : \mathcal{A}(A_0,A_1) \otimes \cdots \otimes \mathcal{A}(A_{n-1},A_n) \longrightarrow \mathcal{B}(\mathcal{F}_0(A_0),\mathcal{F}_1(A_n)) \]
for $m \geqslant 0$, where $f_{01}^{(0)}$ corresponds to an element of $\mathcal{B}(\mathcal{F}_0(A),\mathcal{F}_1(A))$ for all $A \in \mathcal{A}$. A pre-natural transformation has degree $r$ if each morphism $f_{01}^{(m)}$ has degree $r-m$. 
The differential $m_1$ on $\mathrm{Func}_{\mathcal{A},\mathcal{B}}(\mathcal{F}_0,\mathcal{F}_1)$ is then defined as 
\[ \left( m_1 (\mathcal{F}_{01}) \right)^{(m)} := \sum_{i_1+i_2+i_3=m } \pm f_{01}^{(i_1 + 1 + i_3)} (\ide^{\otimes i_1} \otimes m_{i_2} \otimes \ide^{\otimes i_3}) + \sum_{ |\pmb{i}_0|  + l + |\pmb{i}_1| = m} \pm m_{s} ( \pmb{f}_0^{\pmb{i}_0} \otimes f_{01}^{(l)} \otimes \pmb{f}_1^{\pmb{i}_1} ) \ ,  \]
where for a list $\pmb{i}_0 := (i^1_0,\dots,i^{k_0}_0)$ of indices, we denote
\begin{align*}
|\pmb{i}_0| := \sum_{j=1}^{k_0} i^j_0 && l(\pmb{i}_0) := k_0 && \pmb{f}_0^{\pmb{i}_0} := f_0^{(i_0^1)} \otimes \cdots \otimes f_0^{(i_0^{k_0})} \ ,
\end{align*} 
and where $s := l(\pmb{i}_0) + 1 + l(\pmb{i}_1)$ in the second sum. 
The \Ainf -operation 
\[ m_n := \mathrm{Func}_{\mathcal{A},\mathcal{B}}(\mathcal{F}_0,\mathcal{F}_1) \otimes \cdots \otimes \mathrm{Func}_{\mathcal{A},\mathcal{B}}(\mathcal{F}_{n-1},\mathcal{F}_n) \rightarrow \mathrm{Func}_{\mathcal{A},\mathcal{B}}(\mathcal{F}_0,\mathcal{F}_n) \]
evaluated on an element $\mathcal{F}_{01} \otimes \cdots \otimes \mathcal{F}_{n-1,n}$ is defined as
\[ \left( m_n( \mathcal{F}_{01} , \dots , \mathcal{F}_{n-1,n}) \right)^{(m)} := \sum_{ \sum_{r=0}^n |\pmb{i}_r|  + \sum_{r=0}^{n-1} l_{r,r+1} = m   } \pm m_{s} ( \pmb{f}_0^{\pmb{i}_0} \otimes f_{01}^{(l_{01})} \otimes \pmb{f}_1^{\pmb{i}_1}  \otimes \cdots \otimes f_{n-1,n}^{(l_{n-1,n})} \otimes \pmb{f}_n^{\pmb{i}_n}) \ ,  \]
where $s:= n + \sum_{r=0}^s l(\pmb{i}_r)$. 

In~\cite{faonte-simplicial}, Faonte defines the simplicial nerve $N_{\Ainf}$ of an \Ainf -category. Given an \Ainf -category $\mathcal{C}$, the simplicial nerve $N_{\Ainf}$ of $\mathcal{C}$ is a simplicial set $N_{\Ainf}(\mathcal{C})$ which has the property of being an $\infty$-category. A $n$-simplex in this simplicial set corresponds to the data for every $0 \leqslant i \leqslant n$ of an object $f_i \in \mathcal{C}$
and for every $0 \leqslant i_0 < \cdots < i_k \leqslant n$ with $k \geqslant 1$ of an element
$f_{i_0 \dots i_k} \in \mathcal{C}(f_{i_0} , f_{i_k})$ of degree $1-k$, such that 
\[ m_1 ( f_{i_0 \dots i_k}) = \sum_{j=1}^{k-1} (-1)^j f_{i_0 \dots \hat{i_j} \dots i_k} + \sum_{\substack{0 < j_1 < \cdots < j_{s-1} < k \\ s \geqslant 2 }} \pm m_s ( f_{i_0 \dots i_{j_1}} , \dots , f_{i_{j_{s-1}} \dots i_k} ) \ .  \]
\noindent One can thereby consider the simplicial set $N_{\Ainf}(\mathrm{Func}_{\mathcal{A},\mathcal{B}})$, which is an $\infty$-category. Its $n$-simplices correspond to the data of
\begin{enumerate}[label=(\roman*)]
\item an \Ainf -functor $\mathcal{F}_{[i]} = (f^{(m)}_{[i]} )_{m \geqslant 1}$ from $\mathcal{A}$ to $\mathcal{B}$ for every $0 \leqslant i \leqslant n$,
\item and of a pre-natural transformation $\mathcal{F}_{I} = (f^{(m)}_{I} )_{m \geqslant 0}$ of degree $1-m+|I|$ for every $I \subset \Delta^n$ such that $\mathrm{dim}(I) \geqslant 1$,
\end{enumerate}
which satisfy the following equations
\[ \resizebox{\hsize}{!}{\begin{math} \begin{aligned}
\left[ \partial , f^{(m)}_I \right] =  \sum_{j=1}^{\mathrm{dim}(I)-1} (-1)^j f^{(m)}_{\partial_jI} + \sum_{\substack{i_1+i_2+i_3=m \\ i_2 \geqslant 2}} \pm f^{(i_1+1+i_3)}_I (\ide^{\otimes i_1} \otimes m_{i_2} \otimes \ide^{\otimes i_3}) + \sum_{\substack{i_1 + \cdots + i_s = m \\ I_1 \cup \cdots \cup I_s = I \\ s \geqslant 2 }} \pm m_s ( f^{(i_1)}_{I_1} \otimes \cdots \otimes f^{(i_s)}_{I_s}) \ .
\end{aligned}
\end{math}} \]

These equations are almost but not exactly identical to the \Ainf -equations for $n$-functors defined in this article. Indeed, the sum for the simplicial differential now runs over $j=1, \dots , \mathrm{dim}(I)-1$ and the operations $f^{(m)}_{I}$ defining the $n$-simplex can have arity 0 when $\mathrm{dim}(I) \geqslant 1$.
These seemingly minor differences account for the fact that the simplicial sets $\mathrm{HOM}_{\mathsf{\Ainf -Cat}}(\mathcal{A},\mathcal{B})_{\bullet}$ and $N_{\Ainf}(\mathrm{Func}_{\mathcal{A},\mathcal{B}})$ differ fundamentally. Indeed, the 1-simplices of the simplicial set $\mathrm{HOM}_{\mathsf{\Ainf -Cat}}(\mathcal{A},\mathcal{B})_{\bullet}$ correspond to \Ainf -homotopies between two \Ainf -functors and its higher simplices are to be understood as the higher coherent homotopies generalizing \Ainf -homotopies. The simplices of the simplicial set $N_{\Ainf}(\mathrm{Func}_{\mathcal{A},\mathcal{B}})$ are to be interpreted differently. The equations computed in the previous paragraph show that a 1-simplex $\mathcal{F}_{01}$ of $N_{\Ainf}(\mathrm{Func}_{\mathcal{A},\mathcal{B}})$ corresponds exactly to an \emph{\Ainf -natural transformation} between two \Ainf -functors $\mathcal{F}_0$ and $\mathcal{F}_1$. A 1-simplex $\mathcal{F}_{01}$ corresponds indeed to a collection of operations from the \Ainf -category $\mathcal{A}$ to the \Ainf -category $\mathcal{B}$, and the arity 0 and 1 part of the equations they satisfy show that $\mathcal{F}_{01}$ descends to an ordinary natural transformation $H^*(\mathcal{F}_{01}) : H^*(\mathcal{F}_0) \Rightarrow H^*(\mathcal{F}_1)$. This is also the reason why the morphisms of the \Ainf -category $\mathrm{Func}_{\mathcal{A},\mathcal{B}}$ are called pre-natural transformations. The $n$-simplices of $N_{\Ainf}(\mathrm{Func}_{\mathcal{A},\mathcal{B}})$ are then to be understood as \emph{higher \Ainf -natural transformations} between \Ainf -functors.  
This interpretation explains in particular why the simplicial set $\mathrm{HOM}_{\mathsf{\Ainf -Cat}}(\mathcal{A},\mathcal{B})_{\bullet}$ is a Kan complex while $N_{\Ainf}(\mathrm{Func}_{\mathcal{A},\mathcal{B}})$ is an $\infty$-category but not necessarily a Kan complex : homotopies should always be invertible (up to homotopy), but this has no reason to hold in general for natural transformations.

\subsection{Two notions of homotopies between \Ainf -functors} \label{ss:two-notions}

The \Ainf -category $\mathrm{Func}_{\mathcal{A},\mathcal{B}}$ provides in fact an alternative framework to define a \emph{homotopy equivalence} relation between \Ainf -functors. Following \cite{fukaya-unobstructed}, we compare in this subsection this homotopy equivalence relation to our equivalence relation induced by \Ainf -homotopies between \Ainf -functors.

Define a unital \Ainf -algebra $B$ to be an \Ainf -algebra $B$ together with an element $e \in B$ such that $\partial e = 0$, $m_2(e,\cdot) = m_2(\cdot ,e) = \ide$ and $m_n(\cdots , e , \cdots) = 0$ when $n \geqslant 3$. A unital \Ainf -algebra yields in particular a unital algebra $H^*(B)$ in cohomology.
One defines the notion of a unital \Ainf -category $\mathcal{B}$ in a similar fashion. A unital \Ainf -category yields in fact an ordinary category $H^*(\mathcal{B})$ whose identity morphisms correspond to the cohomology classes of its unit morphisms. For two \Ainf -categories $\mathcal{A}$ and $\mathcal{B}$ such that $\mathcal{B}$ is unital, one can moreover check that the \Ainf -category $\mathrm{Func}_{\mathcal{A},\mathcal{B}}$ is again unital.

We suppose in the rest of this subsection that the \Ainf -category $\mathcal{B}$ is unital.
Following \cite{fukaya-unobstructed}, define a \emph{homotopy equivalence} between two \Ainf -functors $\mathcal{F}$ and $\mathcal{G}$ to be a degree 0 pre-natural transformation $\mathcal{T} \in \mathrm{Func}_{\mathcal{A},\mathcal{B}}(\mathcal{F},\mathcal{G})$ for which there exists a degree 0 pre-natural transformation $\mathcal{T}' \in \mathrm{Func}_{\mathcal{A},\mathcal{B}}(\mathcal{G},\mathcal{F})$ such that 
\begin{enumerate}[label=(\roman*)]
\item $m_1 (\mathcal{T}) = 0$ and $m_1 (\mathcal{T}') = 0$, i.e. the pre-natural transformations $\mathcal{T}$ and $\mathcal{T'}$ are \Ainf -natural transformations as defined in the previous subsection ;
\item $m_2 ( \mathcal{T} , \mathcal{T}') - \mathrm{Id}_{\mathcal{F}} \in \mathrm{Im}(m_1)$ and $m_2 ( \mathcal{T}' , \mathcal{T}) - \mathrm{Id}_{\mathcal{G}} \in \mathrm{Im}(m_1)$, where $\mathrm{Id}_{\mathcal{F}}$ denotes the unit of $\mathrm{Func}_{\mathcal{A},\mathcal{B}}(\mathcal{F},\mathcal{F})$.
\end{enumerate}
Two \Ainf -functors $\mathcal{F}$ and $\mathcal{G}$ are said to be \emph{homotopy equivalent} if there exists a homotopy equivalence between them. The \Ainf -natural transformations $\mathcal{T}$ and $\mathcal{T}'$ then induce natural equivalences $H^*(\mathcal{T}) : H^*(\mathcal{F}) \Rightarrow H^*(\mathcal{G})$ and $H^*(\mathcal{T}') : H^*(\mathcal{G}) \Rightarrow H^*(\mathcal{F})$ which are inverse to one another. In particular, if $\mathcal{F}$ and $\mathcal{G}$ are homotopy equivalent then $H^*(\mathcal{T}) \circ H^*(\mathcal{F}) \circ H^*(\mathcal{T})^{-1} = H^*(\mathcal{G})$.

We say that two \Ainf -functors $\mathcal{F}$ and $\mathcal{G}$ are \emph{homotopic} if there exists an \Ainf -homotopy between them. Two homotopic \Ainf -functors $\mathcal{F}$ and $\mathcal{G}$ define in particular the same functor $H^*(\mathcal{F}) = H^*(\mathcal{G})$ in cohomology. 
These two notions of homotopy on \Ainf -functors are related by the following proposition proven by Fukaya in \cite{fukaya-unobstructed} :
\begin{proposition}
Let $\mathcal{B}$ be a unital \Ainf -category and $\mathcal{F},\mathcal{G} : \mathcal{A} \rightarrow \mathcal{B}$ be two \Ainf -functors. If $\mathcal{F}$ and $\mathcal{G}$ are homotopic then they are homotopy equivalent. 
\end{proposition}
\noindent The converse is however not true in general.

\section{The $\infty$-category of \Ainf -algebras ?} \label{fd:s:simplic-enrich}

Given three \Ainf -algebras $A$, $B$ and $C$ together with two $n$-morphisms going respectively from $A$ to $B$ and from $B$ to $C$, we have not yet defined a way to compose them. In other words, we have not defined a simplicial enrichment of the category $\mathsf{\Ainf -Alg}$.

\subsection{Simplicially enriched categories} \label{fd:ss:simpl-enrich-cat}

A \emph{simplicially enriched category} $\mathsf{D}$, or \emph{simplicial category} for short, is the data of
\begin{enumerate}[label=(\roman*)]
\item a collection of objects $\mathrm{Ob}(\mathsf{D})$ ;
\item for every two objects $A$ and $B$ a simplicial set of morphisms between $A$ and $B$, that we write $\mathrm{HOM}_\mathsf{D}(A,B)_{n}$ ;
\item simplicial composition maps $$\mathrm{HOM}_\mathsf{D}(A,B)_{n} \times \mathrm{HOM}_\mathsf{D}(B,C)_{n} \longrightarrow \mathrm{HOM}_\mathsf{D}(A,C)_{\bullet} \ ; $$
\end{enumerate}
which satisfy the standard axioms of an ordinary category.
Defining a \emph{simplicial enrichment} of an ordinary category $\mathsf{C}$ consists then in defining a simplicial category $\mathsf{C}_\Delta$ having the same objects as~$\mathsf{C}$ and such that the sets of vertices of its HOM-simplicial sets are exactly the sets of morphisms of~$\mathsf{C}$, in other words
\[\mathrm{HOM}_{\mathsf{C}_\Delta} (A,B)_0 = \mathrm{Hom}_\mathsf{C}(A,B) \text{ \ for each $n$} . \]

In the particular case of the category $\mathsf{C} := \mathsf{\Ainf -Alg}$ we have already constructed the HOM-simplicial sets, and we would now like to define simplicial composition maps 
\[ \mathrm{HOM}_{\mathsf{\Ainf -Alg}}(A,B)_n \times \mathrm{HOM}_{\mathsf{\Ainf -Alg}} (B,C)_n \longrightarrow \mathrm{HOM}_{\mathsf{\Ainf -Alg}}(A,C)_n \ . \]
It is enough to construct these simplicial maps for $\mathsf{dg -Cog}$, i.e. to define simplicial composition maps
\[ \mathrm{HOM}_{\mathsf{dg -Cog}}(A,B)_n \times \mathrm{HOM}_{\mathsf{dg -Cog}} (B,C)_n \longrightarrow \mathrm{HOM}_{\mathsf{dg -Cog}}(A,C)_n \ , \]
which are associative, preserve the identity and lift the composition on $\mathrm{HOM}_0 = \mathrm{Hom}$.

\subsection{A natural candidate that fails to preserve the coproduct} \label{fd:ss:nat-candidate}

Let $F : \pmb{\Delta}^n \otimes C \rightarrow C'$ and $G : \pmb{\Delta}^n \otimes C' \rightarrow C''$ be two morphisms of dg-coalgebras. The only natural candidate to construct a composition is the Alexander-Whitney coproduct $\Delta_{\pmb{\Delta}^n}$, i.e. we define $G \circ F$ to be the following composite of maps
\[ \begin{tikzcd}[column sep = large]
\pmb{\Delta}^n \otimes C \arrow{r}[above = 2pt]{\Delta_{\pmb{\Delta}^n} \otimes \ide_C} 
& \pmb{\Delta}^n \otimes \pmb{\Delta}^n \otimes C \arrow{r}[above = 2pt]{\ide_{\pmb{\Delta}^n} \otimes F} 
& \pmb{\Delta}^n \otimes C' \arrow{r}[above = 2pt]{G} 
& C''
\end{tikzcd} \ . \]
Note that we use the word "map" and not "morphism" because we have yet to check that this composite is indeed a morphism of dg-coalgebras.

Before moving on, we point out that for the composition of continuous maps of topological spaces $\Delta^n \times X \rightarrow Y$ we use the diagonal map of $\Delta^n$,
\[ \Delta^n \times X \underset{diag_{\Delta^n} \times \ide_X}{\longrightarrow} \Delta^n \times \Delta^n \times X \underset{\ide_{\Delta^n} \times F}{\longrightarrow} \Delta^n \times Y \underset{G}{\longrightarrow} Z \ . \]
This construction cannot be reproduced in our case, as the diagonal map $ \pmb{\Delta}^n \rightarrow \pmb{\Delta}^n \otimes \pmb{\Delta}^n$ does not respect the gradings, nor does it respect the differentials.

Set $\pmb{\Delta}^n_1 := \pmb{\Delta}^n$, $\pmb{\Delta}^n_2 := \pmb{\Delta}^n$ and write $\Delta_{\pmb{\Delta}^n} : \pmb{\Delta}^n \rightarrow \pmb{\Delta}^n_1 \otimes \pmb{\Delta}^n_2$ for the Alexander-Whitney map seen as a map from the dg-coalgebra $\pmb{\Delta}^n$ to the product dg-coalgebra $\pmb{\Delta}^n_1 \otimes \pmb{\Delta}^n_2$. In the previous composition, it is sufficient to prove that $\Delta_{\pmb{\Delta}^n} : \pmb{\Delta}^n \rightarrow \pmb{\Delta}^n_1 \otimes \pmb{\Delta}^n_2$ is a morphism of dg-coalgebras to prove that $G \circ F$ is a morphism of dg-coalgebras. This map does preserve the differential, but it does not preserve the coproduct !
Indeed, consider the following diagram
\[
\begin{tikzcd}
\pmb{\Delta}^n \arrow[r,"\Delta_{\pmb{\Delta}^n}"] \arrow[d,"\Delta_{\pmb{\Delta}^n}"] & \pmb{\Delta}^n_1 \otimes \pmb{\Delta}^n_2 \arrow[r,"\Delta_{\pmb{\Delta}^n_1}\otimes \Delta_{\pmb{\Delta}^n_2}"] & \pmb{\Delta}^n_1 \otimes \pmb{\Delta}^n_1 \otimes \pmb{\Delta}^n_2 \otimes \pmb{\Delta}^n_2 \arrow[d,"\ide \otimes \tau \otimes \ide "] \\
\pmb{\Delta}^n \otimes \pmb{\Delta}^n \arrow[rr,"\Delta_{\pmb{\Delta}^n} \otimes \Delta_{\pmb{\Delta}^n}"]     &        & (\pmb{\Delta}^n_1 \otimes \pmb{\Delta}^n_2) \otimes (\pmb{\Delta}^n_1 \otimes \pmb{\Delta}^n_2)
\end{tikzcd}
\ . \]
Up to specifying the correct signs, the upper composite path of the square is the map
\[ I \longmapsto \sum_{I_1 \cup I_2 \cup I_3 \cup I_4 = I } I_1 \otimes I_3 \otimes I_2 \otimes I_4 \ , \]
where $I_1 \cup I_2 \cup I_3 \cup I_4$ denotes an overlapping partition of the face $I \subset \Delta^n$, while the lower composite path of the square is the map
\[ I \longmapsto \sum_{I_1 \cup I_2 \cup I_3 \cup I_4 = I } I_1 \otimes I_2 \otimes I_3 \otimes I_4 \ . \]
These two maps are not equal, the square does not commute. 
The map $G \circ F$ is in particular not a morphism of dg-coalgebras, and as a result does not belong to $\mathrm{HOM}_{\mathsf{dg-Cog}}(A,C)_n$. It ensues that the composition fails to be lifted to higher morphisms with this naive approach.

Still, something more can be said about the previous non-commutative square. Again, up to computing the correct signs, the map
\begin{align*}
\pmb{\Delta}^n &\longrightarrow (\pmb{\Delta}^n_1 \otimes \pmb{\Delta}^n_2) \otimes (\pmb{\Delta}^n_1 \otimes \pmb{\Delta}^n_2) \\
I &\longmapsto \sum_{I_1 \cup I_2 \cup I_3 \cup I_4 \cup I_5 = I} I_1 \otimes I_3 \otimes (I_2 \cup I_4)\otimes I_5 \ ,
\end{align*}
defines a homotopy between the upper composite path and the lower composite path of the square~: it fills the square to make it homotopy-commutative.
In the language introduced in~\cite{mcclure-smith}, the upper composite path is equal to $1324$, the lower one is equal to $1234$, and the filler is equal to $13234$. Using the results of~\cite{mcclure-smith}, the author proves in~\cite{mazuir-AW} that :
\begin{theorem} \label{alg:th:aw}
The Alexander-Whitney coproduct can be lifted to an \Ainf -morphism between the dg-coalgebras $\pmb{\Delta}^n$ and $\pmb{\Delta}^n_1 \otimes \pmb{\Delta}^n_2$, whose first higher homotopy is the map $13234$.
\end{theorem}

\subsection{A second approach using the tensor product of \Ainf -morphisms} \label{fd:ss:open-question-plus-result}

We proved in subsection~\ref{alg:ss:equivalent-definition} of part~\ref{p:algebra} that a $n$-morphism from $A$ to $B$ can equivalently be defined as a morphism of \Ainf -algebras $A \rightarrow \pmb{\Delta}_n \otimes B$. Using this definition, we can construct the composition of two $n$-morphisms $A \rightarrow \pmb{\Delta}_n \otimes B$ and $B \rightarrow \pmb{\Delta}_n \otimes C$ as 
\[ G \circ F := A \underset{F}{\longrightarrow} \pmb{\Delta}_n \otimes B \underset{\ide_ {\pmb{\Delta}_n}\otimes G}{\longrightarrow} \pmb{\Delta}_n \otimes \pmb{\Delta}_n \otimes C \underset{\cup \otimes \ide_c}{\longrightarrow} \pmb{\Delta}_n \otimes C \ . \]
In this composition, we write tensor products of \Ainf -morphisms $\ide_ {\pmb{\Delta}_n}\otimes G$ and $\cup \otimes \ide_c$ between tensor \Ainf -algebras. This requires some further explanations.

Given two \Ainf -algebras $A$ and $B$, it is not straightforward to define an \Ainf -algebra structure on the tensor dg-module $A \otimes B$. Indeed, if we define naively the operations $m_n^{A \otimes B}$ as \[ m_n^{A \otimes B}(a_1 \otimes b_1 , \cdots , a_n \otimes b_n) := \pm m_n^A(a_1 , \cdots , a_n) \otimes m_n^B (b_1 , \cdots , b_n) \ , \]
they fail to satisfy the \Ainf -equations and do not even have the right degree. As explained in~\cite{markl-assoc}, the definition of a natural tensor product of \Ainf -algebras can be done by constructing a morphism of operads $\Ainf \rightarrow \Ainf \otimes \Ainf$, where $\Ainf \otimes \Ainf (n) := \Ainf (n) \otimes \Ainf (n)$ denotes the Hadamard product of operads. In~\cite{masuda-diagonal-assoc}, the authors construct such a morphism of operads by constructing a polytopal diagonal on the associahedra $K_m$ and recover the formula originally computed on the dg-level by Markl and Shnider in~\cite{markl-assoc}. 
In the particular case of a dg-algebra $A$ and an \Ainf -algebra $B$, the \Ainf -structure on $A \otimes B$ deduced from a diagonal on the operad \Ainf\ is moreover exactly the one described at the beginning of subsection~\ref{alg:ss:equivalent-definition} of part~\ref{p:algebra}. The \Ainf -algebras appearing in the definition of the \Ainf -morphism $G \circ F : A \rightarrow \pmb{\Delta}_n \otimes C$ are all of this form. 

Given two \Ainf -morphisms $f^A : A_1 \rightarrow A_2$ and $f^B : B_1 \rightarrow B_2$, we would also like to define a morphism $f^A \otimes f^B : A_1 \otimes B_1 \rightarrow A_2 \otimes B_2$ between the tensor \Ainf -algebras $A_1 \otimes B_1$ and  $A_2 \otimes B_2$. This involves defining this time a morphism of operadic bimodules $\infmor \rightarrow \infmor \otimes \infmor$, compatible with the morphism of operads $\Ainf \rightarrow \Ainf \otimes \Ainf$ introduced in the previous paragraph. Guillaume Laplante-Anfossi together with the author define such a morphism in an upcoming article~\cite{masuda-diagonal-multipl}, following the method of~\cite{masuda-diagonal-assoc} by constructing an explicit polytopal diagonal on the multiplihedra $J_m$. See also the work of Lipshitz, Ozsv\'{a}th and Thurston in~\cite{lipshitz-diagonals}.

In the particular case when the \Ainf -algebras $A_1$ and $A_2$ are dg-algebras and the morphism $f_A$ is a morphism of dg-algebras, the datum of a diagonal on \infmor\ is not necessary to define the \Ainf -morphism $f^A \otimes f^B$. It can indeed simply be defined as
\[ (f_A \otimes f_B)_m (a_1 \otimes b_1 , \dots , a_m \otimes b_m) := \pm f^A_1 ( a_1 \cdot \cdots \cdot a_m) \otimes f^B_m (b_1, \dots , b_m) \ , \]
where $a_1 \cdot \cdots \cdot a_m$ denotes the product of the elements $a_1 , \dots , a_m$.
The map $\ide_ {\pmb{\Delta}_n}\otimes G$ in the composition $G \circ F$ is of this form. However, such a diagonal is necessary to define the tensor \Ainf -morphism $\cup \otimes \ide_C$, as the map $\cup$ is this time an \Ainf -morphism and not a mere morphism between dg-algebras. Here the map $\cup$ denotes indeed the \Ainf -morphism between the dg-algebras $\pmb{\Delta}_n \otimes \pmb{\Delta}_n$ and $\pmb{\Delta}_n$, deduced from the \Ainf -morphism between the dg-coalgebras $\pmb{\Delta}^n$ and $\pmb{\Delta}^n \otimes \pmb{\Delta}^n$ of theorem~\ref{alg:th:aw}.

Hence, the datum of a diagonal on the operadic bimodule \infmor , as constructed in~\cite{masuda-diagonal-multipl} or \cite{lipshitz-diagonals}, allows us to define the composition of two $n$-morphisms $A \rightarrow \pmb{\Delta}_n \otimes B$ and $B \rightarrow \pmb{\Delta}_n \otimes C$. It is however not immediately clear that this composition defines a map of simplicial sets
$$\mathrm{HOM}_\mathsf{\Ainf}(A,B)_{n} \times \mathrm{HOM}_\mathsf{\Ainf}(B,C)_{n} \longrightarrow \mathrm{HOM}_\mathsf{\Ainf}(A,C)_{\bullet} \ , $$
nor that this composition is associative.
It is thereby still an open question to know whether these $\mathrm{HOM}$-simplicial sets could fit into a simplicial enrichment of the category $\mathsf{\Ainf -Alg}$. This would then endow $\mathsf{\Ainf -Alg}$ with a structure of $\infty$-category, following Proposition 1.1.5.10. of~\cite{lurie-htt}. 
The author plans to inspect these questions in an upcoming paper.

\subsection{Enriching $\mathsf{\Ainf - Cat}$ using the \Ainf -categories $\mathrm{Func}_{\mathcal{A},\mathcal{B}}$}

In~\cite{faonte-infinity-two}, Faonte claims that the simplicial sets  $N_{\Ainf}(\mathrm{Func}_{\mathcal{A},\mathcal{B}})$ enhance the category $\mathsf{\Ainf - Cat}$ of \Ainf -categories with \Ainf -functors between them to an $(\infty , 2)$-category. The same difficulties that were tackled in this section and section~\ref{fd:ss:nat-candidate} seem however to arise when lifting the composition of \Ainf -functors to the level of the simplicial sets $N_{\Ainf}(\mathrm{Func}_{\mathcal{A},\mathcal{B}})$.

In the same vein, Lyubashenko constructs in \cite{lyubashenko} an \Ainf -bifunctor
\[ \mathrm{Func}_{\mathcal{A},\mathcal{B}} \times \mathrm{Func}_{\mathcal{B},\mathcal{C}} \longrightarrow \mathrm{Func}_{\mathcal{A},\mathcal{C}} \]
defined as the composition of \Ainf -functors on objects. We refer to his paper for a definition of an \Ainf -bifunctor and simply stress here that the notation $\mathrm{Func}_{\mathcal{A},\mathcal{B}} \times \mathrm{Func}_{\mathcal{B},\mathcal{C}}$ is a mere notation and does not refer to the tensor product of the \Ainf -categories $\mathrm{Func}_{\mathcal{A},\mathcal{B}}$ and $\mathrm{Func}_{\mathcal{B},\mathcal{C}}$. Fukaya then proves in \cite{fukaya-unobstructed} that this composition \Ainf -bifunctor is associative up to homotopy equivalence.

This suggests that the \Ainf -categories $\mathrm{Func}_{\mathcal{A},\mathcal{B}}$ should fit into an enrichment in \Ainf -categories of the category $\mathsf{\Ainf -Cat}$. This enrichment should in particular induce in cohomology the 2-category structure on the category $\mathsf{Cat}$, whose objects are ordinary categories, whose 1-morphisms are functors and whose 2-morphisms are natural transformations. The structure of a category enriched in \Ainf -categories has however not been defined to this day. Bottman is currently working on such a definition, which he calls an \emph{$(\Ainf , 2)$-category} structure. See for instance \cite{bottman-carmeli}. The \Ainf -categories $\mathrm{Func}_{\mathcal{A},\mathcal{B}}$ and the problem of defining the notion of a category enriched in \Ainf -categories arise in fact naturally in symplectic topology when considering moduli spaces of pseudo-holomorphic quilts defining operations on Fukaya categories $\mathrm{Fuk}(M)$ of symplectic manifolds $M$. See for instance \cite{mau-wehrheim-woodward}, \cite{fukaya-unobstructed}, \cite{bottman-figure-8} and \cite{bottman-witch} for more details on the subject.

\newpage

\begin{leftbar}
\part{Higher morphisms in Morse theory} \label{p:geo}
\end{leftbar}

\setcounter{section}{0}

\section{$n$-morphisms in Morse theory} \label{geo:s:n-morph-morse}

Let $M$ be a closed oriented Riemannian manifold endowed with a Morse function $f$ together with a Morse-Smale metric. In~\cite{mazuir-I}, we explored how to realize the moduli spaces of stable metric ribbon trees $\mathcal{T}_m$ and the moduli spaces of stable two-colored metric ribbon trees $\mathcal{CT}_m$ in Morse theory. It was proven that, upon choosing admissible perturbation data $\mathbb{X}^f$ on the moduli spaces $\mathcal{T}_m$ for the function $f$, the Morse cochains $C^*(f)$ can be endowed with an $\Omega B As$-algebra structure whose operations $m_t$ for $t \in SRT_m$ are defined by counting 0-dimensional moduli spaces $\mathcal{T}_t^{\mathbb{X}^f}(y ; x_1,\dots,x_m)$. Similarly, choose an additional Morse function $g$ together with admissible perturbation data $\mathbb{X}^g$ on the moduli spaces $\mathcal{T}_m$, and admissible perturbation data $\mathbb{Y}$ on the moduli spaces $\mathcal{CT}_m$ which are compatible with $\mathbb{X}^f$ and $\mathbb{X}^g$. We can then define an $\Omega B As$-morphism $\mu^\mathbb{Y} : (C^*(f),m_t^{\mathbb{X}^f}) \rightarrow (C^*(g),m_t^{\mathbb{X}^g})$, whose operations $\mu_{t_g}$ for $t_g \in SCRT_m$ are defined by counting the 0-dimensional moduli spaces $\mathcal{CT}_{t_g}^{\mathbb{Y}}(y;x_1,\dots,x_m)$.

The goal of this section is to realize the $n$-multiplihedra $n-J_m$ endowed with their $n-\ombas$-cell decomposition in Morse theory.
We first introduce the notion of $n$-simplices of perturbation data on the moduli spaces $\mathcal{CT}_m$ (definitions~\ref{geo:def:n-simpl-perturb-data}~and~\ref{geo:def:n-simpl-perturb-data-tot}), generalizing the notion of perturbation data on these moduli spaces defined in~\cite{mazuir-I}. We then use $n$-simplices of perturbation data to define the moduli spaces $\mathcal{CT}_{I,t_g}(y ; x_1,\dots,x_m)$, $I \subset \Delta^n$. Under generic assumptions on the simplices of perturbation data, these moduli spaces are orientable manifolds (Proposition~\ref{geo:prop:orientable-manifolds}). Requiring some additional compatibilities involving the maps $\mathrm{AW}_{\pmb{a}}$ on the simplices of perturbation data, the 1-dimensional moduli spaces $\mathcal{CT}_{I,t_g}(y ; x_1,\dots,x_m)$ can be compactified to 1-dimensional manifolds with boundary, whose boundary is modeled on the boundary of $\Delta^n \times \overline{\mathcal{CT}}_m$ endowed with its $n-\Omega B As$-cell decomposition (Theorems~\ref{geo:th:existence} and~\ref{geo:th:compactification}). We construct as a result a $n-\Omega B As$-morphism between the Morse cochains $C^*(f)$ and $C^*(g)$ (Theorem~\ref{geo:th:n-ombas-Morse}), by counting the signed points of the 0-dimensional oriented manifolds $\mathcal{CT}_{I,t_g}(y ; x_1,\dots,x_m)$. We finally prove a filling theorem for perturbation data parametrized by a simplicial subcomplex $S \subset \Delta^n$ (Theorem~\ref{geo:th:filler}), solving as a corollary the question that initially motivated this paper (corollary~\ref{geo:cor:initial}).

\subsection{Conventions} \label{geo:ss:conv-morse}

We will study Morse theory of the Morse function $f : M \rightarrow \R$ using its negative gradient vector field $-\nabla f$. Denote $d$ the dimension of the manifold $M$ and $\phi^s$ the flow of $-\nabla f$. For a critical point $x$ define its unstable and stable manifolds
\begin{align*}
W^U(x) &:= \{ z \in M , \ \lim_{s \rightarrow - \infty} \phi^s(z) = x  \} \\
W^S(x) &:= \{ z \in M , \ \lim_{s \rightarrow + \infty} \phi^s(z) = x  \} \ .
\end{align*}
Their dimensions are such that $\mathrm{dim} (W^U(x)) + \mathrm{dim} (W^S(x)) = d$. 
We then define the \emph{degree of a critical point $x$ }to be $|x| := \mathrm{dim} (W^S(x))$. This degree is often referred to as the \emph{coindex of $x$} in the litterature. 

We will moreover work with Morse cochains. For two critical point $x \neq y$, define 
\[ \mathcal{T} (y;x) := W^S(y) \cap W^U(x) / \R  \]
to be the moduli space of negative gradient trajectories connecting $x$ to $y$. Denote moreover $\mathcal{T}(x;x) = \emptyset$. Under the Morse-Smale assumption on $f$ and the Riemannian metric on $M$, for $x \neq y$ the moduli space $\mathcal{T} (y;x)$ has dimension $\mathrm{dim} \left( \mathcal{T} (y;x) \right) = |y| - |x| - 1$.
The Morse differential $\partial_{Morse} : C^*(f) \rightarrow C^*(f)$ is then defined to count descending negative gradient trajectories
\[  \partial_{Morse} (x) :=\sum_{|y| = |x| + 1} \# \mathcal{T} (y;x) \cdot y  \ . \]

\subsection{$n$-simplices of perturbation data on a stratum $\mathcal{CT}_m(t_g)$} \label{geo:ss:n-simpl-perturb-data}

Fix a gauged stable metric ribbon tree $T_g=(t_g,\lambda, \{ l_e \}_{e \in E(t)})$. Let $T_c=(t_c,L_{f_c})$ be its associated two-colored metric ribbon tree, $\overline{E}(t_c)$ the set of all edges of $t_c$ and $E(t_c) \subset \overline{E}(t_c)$ the set of internal edges of $t_c$. We point out that $L_{f_c}$ is a linear combination of the parameters $\lambda, \{ l_e \}_{e \in E(t)}$ and that we should in fact write $L_{f_c}(\lambda, \{ l_e \}_{e \in E(t)})$. Recall from~\cite{mazuir-I} that :

\begin{definition}[\cite{mazuir-I}]
A \emph{choice of perturbation data} on $T_g$ consists of the following data : 
\begin{enumerate}[label=(\roman*)]
\item a vector field
\[ [ 0 , L_{f_c} ] \times M \underset{\mathbb{X}_{f_c}}{\longrightarrow} T M \ , \]
that vanishes on $[ 1 , L_{f_c} -1 ]$, for every internal edge $f_c$ of $t_c$ ; 
\item a vector field 
\[ [ 0 , +\infty [ \times M \underset{\mathbb{X}_{f_0}}{\longrightarrow} T M \ , \]
that vanishes away from $[0,1]$, for the outgoing edge $f_0$ of $t_c$ ; 
\item a vector field 
\[ ] - \infty , 0 ] \times M \underset{\mathbb{X}_{f_i}}{\longrightarrow} T M \ , \]
that vanishes away from $[-1,0]$, for every incoming edge $f_i \ (1 \leqslant i \leqslant n)$ of $t_c$.
\end{enumerate}

In the rest of the paper, we will moreover write $D_{f_c}$ for all segments $[ 0 , L_{f_c} ]$, as well as for all semi-infinite segments $] - \infty , 0 ]$ and $[ 0 , +\infty [$.
\end{definition}

\begin{definition} \label{geo:def:n-simpl-perturb-data}
A \emph{$n$-simplex of perturbation data} for $T_g$ is defined to be a choice of perturbation data $\mathbb{Y}_{\delta,T_g}$ for every $\delta \in \mathring{\Delta}^n$. Equivalently, it is the datum of a vector field
\[  \mathring{\Delta}^n \times D_{f_c} \times M \underset{\mathbb{Y}_{\Delta^n,T_g,f_c}}{\longrightarrow} T M \]
for every edge $f_c \in \overline{E}(t_c)$, abiding by the previous vanishing conditions on $D_{f_c}$. We will denote it as $\mathbb{Y}_{\Delta^n,T_g} := \{ \mathbb{Y}_{\delta,T_g} \}_{\delta \in \mathring{\Delta}^n}$.
\end{definition}

Introduce the cone $C_{f_c} \subset \mathcal{CT}_m(t_g) \times \R$ defined as
\begin{enumerate}[label=(\roman*)]
\item $\{ ((\lambda, \{ l_e \}_{e \in E(t)}),s) \text{ such that } (\lambda, \{ l_e \}_{e \in E(t)}) \in \mathcal{CT}_m(t_g) \text{ and } 0 \leqslant s \leqslant L_{f_c} \}$ if $f_c$ is an internal edge~;
\item $\{ ((\lambda, \{ l_e \}_{e \in E(t)}),s) \text{ such that } (\lambda, \{ l_e \}_{e \in E(t)}) \in \mathcal{CT}_m(t_g) \text{ and } s \leqslant 0 \}$ if $f_c$ is an incoming edge~;
\item $\{ ((\lambda, \{ l_e \}_{e \in E(t)}),s) \text{ such that } (\lambda, \{ l_e \}_{e \in E(t)}) \in \mathcal{CT}_m(t_g) \text{ and } s \geqslant 0 \}$ if $f_c$ is the outgoing edge.
\end{enumerate}

\begin{definition} \label{geo:def:n-simpl-perturb-data-tot} 
A \emph{$n$-simplex of perturbation data on $\mathcal{CT}_m(t_g)$}, or \emph{choice of perturbation data on $\mathcal{CT}_m(t_g)$ parametrized by $\Delta^n$}, is defined to be the data of a $n$-simplex of perturbation data $\mathbb{Y}_{\Delta^n,T_g}$ for every $T_g \in \mathcal{CT}_m(t_g)$. A $n$-simplex of perturbation data $\mathbb{Y}_{\Delta^n,t_g}$ defines maps
\[ \mathbb{Y}_{\Delta^n,t_g,f_c} : \mathring{\Delta}^n \times D_{f_c} \times M \longrightarrow TM \ , \]
for every edge $f_c$ of $t_c$. It is said to be \emph{smooth} if all these maps are smooth.
\end{definition}

\subsection{The moduli spaces $\mathcal{CT}_{I,t_g}(y ; x_1,\dots,x_m)$} \label{geo:ss:mod-space-I-CTm-morse}

Recall from~\cite{mazuir-I} that given an admissible choice of perturbation data $\mathbb{Y}$ on the moduli spaces $\mathcal{CT}_m$, the moduli spaces $\mathcal{CT}_{t_g}^{\mathbb{Y}}(y;x_1,\dots,x_m)$ are defined as the inverse image of the thin diagonal $\Delta \subset M^{\times m +1}$ under the flow map
\[ \phi_{\mathbb{Y}_{t_g}} : \mathcal{CT}_m(t_g) \times W^S(y) \times W^U(x_1) \times \cdots \times W^U(x_m) \longrightarrow M^{\times m+1} \ . \]

\begin{definition} \label{geo:def:moduli-space-n-simplex}
Let $\mathbb{Y}_{\Delta^n,t_g}$ be a smooth $n$-simplex of perturbation data on $\mathcal{CT}_m(t_g)$. Given $y \in \mathrm{Crit}(g)$ and $x_1,\dots ,x_m \in \mathrm{Crit}(f)$, we define the moduli spaces
\[ \resizebox{\hsize}{!}{\begin{math} \begin{aligned}
\mathcal{CT}_{\Delta^n,t_g}^{\mathbb{Y}_{\Delta^n,t_g}}(y ; x_1,\dots,x_m) &:= \bigcup_{\delta \in \mathring{\Delta}^n} \mathcal{CT}_{t_g}^{\mathbb{Y}_{\delta,t_g}}(y;x_1,\dots,x_m) \\
&= \left\{\begin{array}{l}
         \text{( $\delta$ , two-colored perturbed Morse gradient tree associated to $(T_g,\mathbb{Y}_{\delta ,T_g})$ } \\
         \text{which connects $x_1,\dots,x_m$ to $y$), for  $T_g \in \mathcal{CT}_m(t_g)$ and $\delta \in \mathring{\Delta}^n$ } \\
  \end{array}\right\} \ .  
\end{aligned}
\end{math}} \]
\end{definition}

\begin{figure}[h]
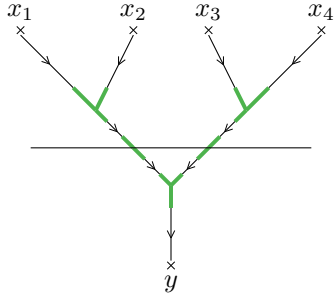

\centering
\arbredegradientbicoloreexempledelta
\caption{An example of a perturbed two-colored Morse gradient tree associated to the perturbation data $\mathbb{Y}_{\delta}$ for a $\delta \in \mathring{\Delta}^n$. The black segments above the gauge correspond to $-\nabla f$ and the green ones to $-\nabla f+\mathbb{Y}_{\delta}$. As for the segments below the gauge, replace $f$ by $g$ in these formulae.} \label{geo:fig:arbre-de-gradient-bicolore}
\end{figure}

An example of a perturbed two-colored Morse gradient tree associated to the perturbation data $\mathbb{Y}_{\delta}$ for a $\delta \in \mathring{\Delta}^n$ is represented on figure~\ref{geo:fig:arbre-de-gradient-bicolore}. 
Introduce the flow map
\[ \phi_{\mathbb{Y}_{\Delta^n ,t_g}} : \mathring{\Delta}^n \times \mathcal{CT}_m(t_g) \times W^S(y) \times W^U(x_1) \times \cdots \times W^U(x_m) \longrightarrow M^{\times m+1} \ , \]
whose restriction to every $ \delta \in \mathring{\Delta}^n$ is 
\[ \phi_{\mathbb{Y}_{\delta ,t_g}} : \mathcal{CT}_m(t_g) \times W^S(y) \times W^U(x_1) \times \cdots \times W^U(x_m) \longrightarrow M^{\times m+1} \ .  \]

\begin{proposition} \label{geo:prop:orientable-manifolds}
\begin{enumerate}[label=(\roman*)]
\item The moduli space $\mathcal{CT}_{\Delta^n,t_g}^{\mathbb{Y}_{\Delta^n,t_g}}(y ; x_1,\dots,x_m)$ can be rewritten as 
\[ \mathcal{CT}_{\Delta^n,t_g}^{\mathbb{Y}_{\Delta^n ,t_g}}(y;x_1,\dots,x_m) = \phi_{\mathbb{Y}_{\Delta^n,t_g}}^{-1}(\Delta) \ , \]
where $\Delta \subset M^{\times m +1}$ is the thin diagonal of $M^{\times m +1}$.
\item Given a $n$-simplex of perturbation data $\mathbb{Y}_{\Delta^n ,t_g}$ making $\phi_{\mathbb{Y}_{\Delta^n ,t_g}}$ transverse to $\Delta$, the moduli space $\mathcal{CT}_{\Delta^n,t_g}(y;x_1,\dots,x_m)$ is an orientable manifold of dimension \[ \dim \left( \mathcal{CT}_{\Delta^n,t_g}(y;x_1,\dots,x_m) \right) = - |t_{\Delta^n , g }| + |y| - \sum_{i=1}^m|x_i| \ . \]
\item $n$-simplices of perturbation data $\mathbb{Y}_{\Delta^n ,t_g}$ such that $\phi_{\mathbb{Y}_{\Delta^n ,t_g}}$ is transverse to $\Delta$ exist.
\end{enumerate}
\end{proposition}

Replacing $\Delta^n$ by any face $I \subset \Delta^n$, the moduli spaces $\mathcal{CT}_{I,t_g}^{\mathbb{Y}_{I ,t_g}}(y;x_1,\dots,x_m)$ can be defined in the same way and made into orientable manifolds of dimension 
\[ \dim \left( \mathcal{CT}_{I,t_g}(y;x_1,\dots,x_m) \right) = - | t_{I,g} | + |y| - \sum_{i=1}^m|x_i| \ . \]
We refer to section~\ref{geo:s:trans-or-signs} for the details on transversality and orientability.

\subsection{Compactifications} \label{geo:ss:compact-morse}

\subsubsection{The compactified moduli spaces $\overline{\mathcal{CT}}_{I,t_g}(y;x_1,\dots,x_m)$} \label{geo:sss:compact-mod-space-morse}

We now would like to compactify the 1-dimensional moduli spaces $\mathcal{CT}_{I,t_g}(y;x_1,\dots,x_m)$ to 1-dimensional manifolds with boundary.
They are defined as the inverse image in $\mathring{I} \times \mathcal{CT}_m(t_g) \times W^S(y) \times W^U(x_1) \times \cdots \times W^U(x_m)$ of the thin diagonal $\Delta \subset M^{\times m+1}$ under the flow map $\phi_{\mathbb{Y}_{I, t_g}}$. The boundary components in the compactification should come from those of $W^S(y)$, of the $W^U(x_i)$ and of $\mathring{I} \times \mathcal{CT}_m(t_g)$. However, rather than considering the boundary components coming from the separate compactifications of $\mathring{I}$ and $\mathcal{CT}_m(t_g)$, we will consider the $n-\Omega B As$-decomposition of $I \times \overline{\mathcal{CT}}_m(t_g)$ and model the remaining boundary components on this decomposition.

Choose admissible perturbation data $\mathbb{X}^f$ and $\mathbb{X}^g$ for the functions $f$ and $g$. Choose moreover smooth simplices of perturbation data $\mathbb{Y}_{I,t_g}$ for all $t_g \in SCRT_i, \ 1 \leqslant i \leqslant m$ and $I \subset \Delta^n$. We denote $(\mathbb{Y}_{I,m})_{I \subset \Delta^n} := (\mathbb{Y}_{I,t_g})_{I\subset \Delta^n}^{t_g \in SCRT_m}$, and call it a choice of perturbation data on $\mathcal{CT}_m$ parametrized by~$\Delta^n$. Fixing a two-colored stable ribbon tree type $t_g \in SCRT_m$ and $I \subset \Delta^n$ we want to compactify the moduli space $\mathcal{CT}_{I,t_g}^{\mathbb{Y}_{I,t_g}}(y;x_1,\dots,x_m)$ using the perturbation data $\mathbb{X}^f$, $\mathbb{X}^g$ and $(\mathbb{Y}_{I,k})_{I \subset \Delta^n}^{k \leqslant m}$. The boundary will be described by the following phenomena :
\begin{enumerate}[label=(\roman*)]
\item the parameter $\delta \in I$ tends towards the codimension 1 boudary of $I$ ($\partial^{sing} I$) ;
\item an external edge breaks at a critical point (Morse) ;
\item an internal edge of the tree $t$ collapses (int-collapse) :
$$\mathcal{CT}_{I,t_g'}^{\mathbb{Y}_{I,t_g'}}(y;x_1,\dots,x_m)$$ where $t_g' \in SCRT_n$ are all the two-colored trees obtained by collapsing exactly one internal edge, which does not cross the gauge ;
\item the gauge moves to cross exactly one additional vertex of the underlying stable ribbon tree (gauge-vertex) :
$$\mathcal{CT}_{I,t_g'}^{\mathbb{Y}_{I,t_g'}}(y;x_1,\dots,x_m)$$
where $t_g' \in SCRT_n$ are all the two-colored trees obtained by moving the gauge to cross exactly one additional vertex of $t$ ;
\item an internal edge located above the gauge or intersecting it breaks or, when the gauge is below the root, the outgoing edge breaks between the gauge and the root (above-break) : $$ \mathcal{CT}_{I,t^1_g}^{\mathbb{Y}_{I,t^1_g}}(y ; x_1,\dots,x_{i_1},z,x_{i_1+i_2+1},\dots,x_m ) \times \mathcal{T}^{\mathbb{X}^f_{t^2}}_{t^2}(z;x_{i_1+1},\dots,x_{i_1+i_2}) \ ,$$ where the tree resulting from grafting the outgoing edge of $t^2$ to the $i_1+1$-th incoming edge of $t^1_g$  is $t_g$ ; \label{geo:item:above-break}
\item edges (internal or incoming) that are possibly intersecting the gauge, break below it, such that there is exactly one edge breaking in each non-self crossing path from an incoming edge to the root (below-break) ;  the simplex of perturbation data $\mathbb{Y}_{I,t_g}$ then "breaks" according to the combinatorics of the Alexander-Whitney coproduct : $$\mathcal{T}^{\mathbb{X}^g_{t^0}}_{t^0}(y ; y_{1},\dots,y_{s}) \times \mathcal{CT}^{\mathbb{Y}_{I_1,t^1_g}}_{I_1,t^1_g}(y_1 ; x_1,\dots ) \times \cdots \times \mathcal{CT}^{\mathbb{Y}_{I_s,t^s_g}}_{I_s,t^s_g}(y_s ; \dots , x_m ) $$ where the tree resulting from grafting for each $r$ the outgoing edge of $t^r_g$ to the $r$-th incoming edge of $t^0$ is $t_g$, and $I_1 \cup \cdots \cup I_s = I$ is an overlapping partition of $I$. \label{geo:item:below-break}
\end{enumerate}
Note that the (Morse) boundaries are a simple consequence of the fact that external edges are Morse trajectories away from a length 1 segment. 

\subsubsection{Smooth choice of perturbation data $\mathbb{Y}_{I \subset \Delta^n ,m}$} \label{geo:sss:smooth-choice-perturb-data}

We begin by tackling the conditions coming with the ($\partial^{sing}I$), (int-collapse) and (gauge-vertex) boundaries.
Let $t_g \in SCRT_m$ and denote $coll \cup g-v (t_g) \subset SCRT_m$ the set consisting of all stable gauged trees obtained by collapsing internal edges of $t$ and/or moving the gauge to cross additional vertices of $t$. In particular, $t_g \in coll \cup g-v (t_g)$. We define 
$$\underline{\mathcal{CT}_m}(t_g) := \bigcup_{t'_g \in coll \cup g-v (t_g)} \mathcal{CT}_m(t_g') $$
for the stratum $\mathcal{CT}_m(t_g) \subset \mathcal{CT}_m$ together with its inner boundary components.  
A choice of perturbation data $(\mathbb{Y}_{I,t'_g})_{t_g' \in coll \cup g-v (t_g) }$ for a fixed $I \subset \Delta^n$ corresponds to a $\mathrm{dim}(I)$-simplex of perturbation data on $\underline{\mathcal{CT}_m}(t_g)$. Following section~\ref{geo:ss:n-simpl-perturb-data}, such a choice of perturbation data is equivalent to a map 
\[ \tilde{\mathbb{Y}}_{I,t_g,f_c} : \mathring{I} \times \tilde{C}_{f_c} \times M \longrightarrow TM \ , \]
for every edge $f_c$ of $t_c$, where $\tilde{C}_{f_c} \subset \underline{\mathcal{CT}_m}(t_g) \times \R$ is defined in a similar fashion to $C_{f_c}$.

\begin{definition} \label{geo:def:smooth-choice-of}
A choice of perturbation data $(\mathbb{Y}_{I,m})_{I \subset \Delta^n}$ is said to be \emph{smooth} if all maps  
\[ \tilde{\mathbb{Y}}_{\Delta^n,t_g,f_c} : \Delta^n \times \tilde{C}_{f_c} \times M \longrightarrow TM \ , \]
are smooth, where we extended $\mathring{\Delta}^n$ to $\Delta^n$ by defining $\tilde{\mathbb{Y}}_{\Delta^n,t_g,f_c} := \tilde{\mathbb{Y}}_{I,t_g,f_c}$ on a face $I \subset \Delta^n$. 
\end{definition}

\subsubsection{The (above-break) boundary} \label{geo:sss:above-break-bound}

The (above-break) conditions are tackled as in~\cite{mazuir-I}. Write $t_c$ for the two-colored ribbon tree associated to $t_g$. The (above-break) boundary corresponds to the breaking of an internal edge $f_c$ of $t_c$ located above the set of colored vertices. Denote $t^1_c$ and $t^2$ the trees obtained by breaking $t_c$ at the edge $f_c$, where $t^2$ is seen to lie above $t^1_c$. We have to specify for each edge $e_c \in \overline{E}(t_c)$ and each $\delta \in \mathring{I}$, what happens to the perturbation $\mathbb{Y}_{\delta,t_c,e_c}$ at the limit.
\begin{enumerate}[label=(\roman*)]
\item For $e_c \in \overline{E}(t^2)$ and $\neq f_c$, we require that $$\lim \mathbb{Y}_{\delta,t_c,e_c} = \mathbb{X}^f_{t^2,e_c} \ .$$
\item For $e_c \in \overline{E}(t^1_c)$ and $\neq f_c$, we require that $$\lim \mathbb{Y}_{\delta,t_c,e_c} = \mathbb{Y}_{\delta,t^1_c,e_c} \ .$$
\item For $f_c=e_c$, $\mathbb{Y}_{\delta,t_c,f_c}$ yields two parts at the limit : the part corresponding to the outgoing edge of $t^2$ and the part corresponding to the incoming edge of $t^1_c$. We then require that they coincide respectively with the perturbation $\mathbb{X}^f_{t^2}$ and $\mathbb{Y}_{\delta,t^1_c}$.
\end{enumerate}
An example of each case is illustrated in figure~\ref{geo:fig:perturbation-above-break}.

\begin{figure}[h]
    \centering
    \begin{subfigure}{\textwidth}
    \centering
       \exampleperturbationCTbreakundelta
       \caption*{(above-break) case (i)}
    \end{subfigure} ~
    \begin{subfigure}{\textwidth}
    \centering
        \exampleperturbationCTbreakdeuxdelta
       \caption*{(above-break) case (ii)}
    \end{subfigure} ~
    \begin{subfigure}{\textwidth}
    \centering
        \exampleperturbationCTbreaktroisdelta
       \caption*{(above-break) case (iii)}
    \end{subfigure}
    \caption{} \label{geo:fig:perturbation-above-break}
\end{figure}

\subsubsection{The (below-break) boundary} \label{geo:sss:below-break-bound}

Denote $t^1_c,\dots,t^s_c$ and $t^0$ the trees obtained by breaking $t_c$ below the gauge, where the trees $t^r_c$ for $r=1, \dots ,s$ are seen to lie above $t^0$ and are ordered from left to right. We write $i_r$ for the arity of $t_c^r$ and introduce the dividing sequence $\pmb{a}$ defined as
\[ \frac{i_1 + \cdots + i_{s-1}}{m} > \frac{i_1 + \cdots + i_{s-2}}{m} > \cdots > \frac{i_1}{m} \ , \]
as in subsection \ref{alg:sss:n-multiplihedra} of part \ref{p:algebra}. 
Consider now the map $\mathrm{AW}_{\pmb{a}} : I \rightarrow I^s $. It comes with $s$ maps $$\mathrm{pr}_r \circ \mathrm{AW}_{\pmb{a}} : I \longrightarrow I$$ for $1 \leqslant r \leqslant s$ corresponding to the projection on the $r$-th factor of $I^s$. For the sake of readability we will simply denote them $\mathrm{pr}_r$. 

We have to specify for each edge $e_c \in \overline{E}(t_c)$ and each $\delta \in \mathring{I}$, what happens to the perturbation $\mathbb{Y}_{\delta,t_c,e_c}$ at the limit. The maps $\mathrm{pr}_r$ will allow us to produce the overlapping partitions combinatorics on the parameter $\delta$.
\begin{enumerate}[label=(\roman*)]
\item For $e_c \in \overline{E}(t^r_c)$ and not among the breaking edges, we require that $$\lim \mathbb{Y}_{\delta ,t_c ,e_c} = \mathbb{Y}_{\mathrm{pr}_r(\delta),t^r_c,e_c} \ .$$
\item For $e_c \in \overline{E}(t^0)$ and not among the breaking edges, we require that $$\lim \mathbb{Y}_{\delta ,t_c ,e_c} = \mathbb{X}^g_{t^0,e_c} \ .$$
\item For $f_c$ among the breaking edges, $\mathbb{Y}_{\delta ,t_c ,f_c}$ yields two parts at the limit : the part corresponding to the outgoing edge of $t^r_c$ and the part corresponding to the incoming edge of $t^0$. We then require that they coincide respectively with the perturbations $\mathbb{Y}_{\mathrm{pr}_r(\delta),t^r_c,e_c}$ and $\mathbb{X}^g_{t^0,e_c}$.
\end{enumerate}
This is again illustrated in figure~\ref{geo:fig:perturbation-below-break}. We also point out that Proposition~\ref{alg:prop:coassoc-a-b-c} ensures that the limit condition (iii) on the perturbation $\mathbb{Y}_{\delta,t_c,e_c}$ is consistent.

\begin{figure}[h]
    \centering
    \begin{subfigure}{\textwidth}
    \centering
       \exampleperturbationCTbreakquatredelta
       \caption*{(below-break) case (i)}
    \end{subfigure} ~
    \begin{subfigure}{\textwidth}
    \centering
        \exampleperturbationCTbreakcinqdelta
       \caption*{(below-break) case (ii)}
    \end{subfigure} ~
    \begin{subfigure}{\textwidth}
    \centering
        \exampleperturbationCTbreaksixdelta
       \caption*{(below-break) case (iii)}
    \end{subfigure}
    \caption{} \label{geo:fig:perturbation-below-break}
\end{figure}

\subsubsection{Admissible $n$-simplices of perturbation data} \label{geo:sss:admissible-choice-perturb}

\begin{definition} \label{geo:def:admissible}
A smooth choice of perturbation data $(\mathbb{Y}_{I,m})_{I \subset \Delta^n}^{m \geqslant 1}$ is said to be \emph{gluing-compatible w.r.t. $\mathbb{X}^f$ and $\mathbb{X}^g$} if it satisfies the (above-break) and (below-break) conditions described in subsections~\ref{geo:sss:above-break-bound} and~\ref{geo:sss:below-break-bound}. Smooth and gluing-compatible perturbation data $(\mathbb{Y}_{I,m})_{I \subset \Delta^n}^{m \geqslant 1}$ such that all maps $\phi_{\mathbb{Y}_{I,t_g}}$ are transverse to the diagonal $\Delta$ are called \emph{admissible w.r.t. $\mathbb{X}^f$ and $\mathbb{X}^g$} or simply \emph{admissible}.
\end{definition}

\begin{theorem} \label{geo:th:existence}
Admissible choices of perturbation data $(\mathbb{Y}_{I,m})_{I \subset \Delta^n}^{m \geqslant 1}$ exist.
\end{theorem}

\begin{theorem} \label{geo:th:compactification}
Let $(\mathbb{Y}_{I,m})_{I \subset \Delta^n}^{m \geqslant 1}$ be an admissible choice of perturbation data. The 0-dimensional moduli spaces $\mathcal{CT}_{I,t_g}(y ; x_1,\dots,x_m)$ are compact. The 1-dimensional moduli spaces $\mathcal{CT}_{I,t_g}(y ; x_1,\dots,x_m)$ can be compactified to 1-dimensional manifolds with boundary $\overline{\mathcal{CT}}_{I,t_g}(y ; x_1,\dots,x_m)$, whose boundary is described in subsection~\ref{geo:sss:compact-mod-space-morse}.
\end{theorem}

The proof of Theorem~\ref{geo:th:existence} is postponed to subsection~\ref{geo:sss:proof-th-three} and will proceed as in~\cite{mazuir-I}. Theorem~\ref{geo:th:compactification} is a direct consequence of the analysis carried out in chapter 6 of~\cite{mescher-morse}. For this reason, we will not give details of its proof. We only point out that all spaces 
\[ \mathcal{T}^{\mathbb{X}^g_{t^0}}_{t^0}(y ; y_{1},\dots,y_{s}) \times \mathcal{CT}^{\mathbb{Y}_{I_1,t^1_g}}_{I_1,t^1_g}(y_1 ; x_1,\dots ) \times \cdots \times \mathcal{CT}^{\mathbb{Y}_{I_s,t^s_g}}_{I_s,t^s_g}(y_s ; \dots , x_m ) \]
where $I_1 \cup \cdots \cup I_s = I$ is an $i$-overlapping $s$-partition of $I$, could a priori appear in the boudary of $\mathcal{CT}_{I,t_g}(y ; x_1,\dots,x_m)$. The assumption that our choice of perturbation data is admissible ensures however in particular that whenever $I_1 \cup \cdots \cup I_s = I$ is not an $(s-1)$-overlapping $s$-partition of $I$ the previous space is empty, as at least one of its factors then has negative dimension.

Theorem~\ref{geo:th:compactification} implies moreover the existence of gluing maps 
\begin{align*}
\#^{above-break}_{ T_{I,g}^{1,Morse} , T^{2,Morse} } : \ &[ R , + \infty ] \longrightarrow \overline{\mathcal{CT}}_{I,t_g}(y ; x_1,\dots ,x_n) \ , \\
\#^{below-break}_{ T^{0,Morse} ,T_{I_1,g}^{1,Morse}, \dots , T_{I_s,g}^{s,Morse} } : \ &[ R , + \infty ] \longrightarrow \overline{\mathcal{CT}}_{I,t_g}(y ; x_1,\dots ,x_n) \ ,
\end{align*}
whenever the perturbed Morse trees $T_{I,g}^{1,Morse} , T^{2,Morse}$ and $T^{0,Morse} ,T_{I_1,g}^{1,Morse}, \dots , T_{I_s,g}^{s,Morse}$ respectively lie in a 0-dimensional moduli space, and where notations are as in items~\ref{geo:item:above-break} and~\ref{geo:item:below-break} of subsection~\ref{geo:sss:compact-mod-space-morse}. The constructions of explicit gluing maps in subsections II.4.4.3 and II.4.5.4 of~\cite{mazuir-I} in the case of the moduli spaces $\mathcal{CT}_{t_g}(y_1 ; x_1,\dots ,x_n )$ can be adapted without problems to the present setting.

\subsection{$n-\Omega B As$-morphisms between Morse cochains} \label{geo:ss:n-ombas-morph-morse}

Let $\mathbb{X}^f$ and $\mathbb{X}^g$ be admissible choices of perturbation data for the Morse functions $f$ and $g$. Denote $(C^*(f),m_t^{\mathbb{X}^f})$ and $(C^*(g),m_t^{\mathbb{X}^g})$ the Morse cochains endowed with their $\Omega B As$-algebra structures constructed in~\cite{mazuir-I}.

\begin{theorem} \label{geo:th:n-ombas-Morse}
Let $(\mathbb{Y}_{I,m})_{I \subset \Delta^n}^{m \geqslant 1}$ be a choice of perturbation that is admissible w.r.t. $\mathbb{X}^f$ and $\mathbb{X}^g$.
Defining for every $m$ and $t_g \in SCRT_m$, and every $I \subset \Delta^n$ the operations $\mu_{I,t_g}$ as
\begin{align*}
\mu_{I,t_g} : C^*(f) \otimes \cdots \otimes C^*(f) &\longrightarrow C^*(g) \\
x_1 \otimes \cdots \otimes x_m &\longmapsto \sum_{|y|= \sum_{i=1}^m|x_i| + |t_{I,g}|} \# \mathcal{CT}_{I,t_g}^{\mathbb{Y}_{I,t_g}}(y ; x_1,\cdots,x_m) \cdot y \ ,
\end{align*} 
they fit into a $n - \ombas$-morphism $(C^*(f),m_t^{\mathbb{X}^f}) \rightarrow (C^*(g),m_t^{\mathbb{X}^g})$
\end{theorem}

The proof is postponed to section~\ref{geo:ss:twisted-n-ombas-morse}. It boils down to counting the boundary points of the 1-dimensional oriented compactified moduli spaces $\overline{\mathcal{CT}}_{I,t_g}^\mathbb{Y}(y ; x_1,\cdots,x_m)$ whose boundary is described in the subsection~\ref{geo:sss:compact-mod-space-morse}.
As a matter of fact, the set of operations $\{ \mu_{I,t_g} \}$ does not exactly define a $n - \Omega B As$-morphism. One of the two differentials $\partial_{Morse}$ in the bracket $[ \partial_{Morse} , \mu_{I,t_g} ]$ appearing in the $n-\Omega B As$-equations has to be twisted by a specific sign for the $n-\Omega B As$-equations to hold. We will speak about a \emph{twisted $n - \Omega B As$-morphism} between twisted $\Omega B As$-algebras. In the case where $M$ is odd-dimensional, this twisted $n-\Omega B As$-morphism is a standard $n-\Omega B As$-morphism. 

As explained in subsection~\ref{alg:sss:n-ombas-to-n-ainf} of part~\ref{p:algebra}, if we want moreover to go back to the algebraic framework of \Ainf -algebras, a  $n-\Ainf$-morphism between the induced $\Ainf$-algebra structures on the Morse cochains can simply be obtained under the morphism of operadic bimodules $\infmorn \rightarrow n-\Omega B As - \mathrm{Morph}$. 

\subsection{Filling properties in Morse theory} \label{geo:ss:inf-cat-morse}

Consider a simplicial subcomplex $S \subset \Delta^n$. Definitions~\ref{geo:def:n-simpl-perturb-data-tot} and~\ref{geo:def:admissible} can be straightforwardly extended to define an \emph{admissible choice of perturbation data parametrized by $S$} on the moduli spaces $\mathcal{CT}_m$, that we will denote $\mathbb{Y}_{S} := (\mathbb{Y}_{I,m})_{I \subset S}^{m \geqslant 1}$ . The following theorem is proven in section~\ref{geo:ss:proofs-existence-filler} : 

\begin{theorem} \label{geo:th:filler}
For every admissible choice of perturbation data $\mathbb{Y}_{S}$ parametrized by a simplicial subcomplex $S \subset \Delta^n$, there exists an admissible $n$-simplex of perturbation data $\mathbb{Y}_{\Delta^n}$ extending $\mathbb{Y}_{S}$.
\end{theorem}

We define for every $n \geqslant 0$, 
\[ \mathrm{HOM}^{geom}_{\ombas}(C^*(f),C^*(g))_n \subset \mathrm{HOM}_{\ombas}(C^*(f),C^*(g))_n \]
to be the set of $n$-\ombas -morphisms $\mu$ from $C^*(f)$ to $C^*(g)$ for which there exists an admissible $n$-simplex of perturbation data $\mathbb{Y}_{\Delta^n}$ such that $\mu = \mu^{\mathbb{Y}_{\Delta^n}}$. 

\begin{theorem} \label{geo:th:contractible}
The sets $\mathrm{HOM}^{geom}_{\ombas}(C^*(f),C^*(g))_n$ define a simplicial subset of the simplicial set $\mathrm{HOM}_{\ombas}(C^*(f),C^*(g))_\bullet$. The simplicial set $\mathrm{HOM}^{geom}_{\ombas}(C^*(f),C^*(g))_\bullet$ has the property of being a Kan complex which is contractible.
\end{theorem}

\begin{proof}
We first prove that the face and degeneracy maps of $\mathrm{HOM}_{\ombas}(C^*(f),C^*(g))_\bullet$ preserve the sets $\mathrm{HOM}^{geom}_{\ombas}(C^*(f),C^*(g))_\bullet$. This is clear for the face maps. Consider a $n$-simplex $\mu^{\mathbb{Y}_{\Delta^n}} \in \mathrm{HOM}^{geom}_{\ombas}(C^*(f),C^*(g))_n$ and a degeneracy map
\[ \sigma_i : \mathrm{HOM}_{\ombas}(C^*(f),C^*(g))_n \longrightarrow \mathrm{HOM}_{\ombas}(C^*(f),C^*(g))_{n+1} , \ 1 \leqslant i \leqslant n+1 \ . \]
We have to construct an admissible $(n+1)$-simplex of perturbation data $\mathbb{Y}'$ such that $\sigma_i(\mu^{\mathbb{Y}_{\Delta^n}}) = \mu^{\mathbb{Y}'}$. Using the realizations 
\[ \Delta^n = \{ (z_1,\dots,z_n) \in \R^n | 1 \geqslant z_1 \geqslant \cdots \geqslant z_n \geqslant 0 \} \ , \]
we define $s_i : \Delta^{n+1} \rightarrow \Delta^n$ as $s_i ( z_1, \dots, z_{n+1}) := (z_1 , \dots , \hat{z_i} , \dots, z_{n+1})$.
The $(n+1)$-simplex of perturbation data defined as $\mathbb{Y}'_{\delta} :=  (\mathbb{Y}_{\Delta^n})_{s_i(\delta)}$ for $\delta \in \Delta^{n+1}$ is then an admissible simplex of perturbation data which has the desired property.

It is clear from Theorem~\ref{geo:th:filler} that the simplicial set $\mathrm{HOM}^{geom}_{\ombas}(C^*(f),C^*(g))_\bullet$ is a Kan complex. A Kan complex is contractible if and only if all its simplicial homotopy groups are trivial. One can moreover check on the definition of the homotopy relation in subsection~\ref{sss:simpl-hom-groups} of part~\ref{p:algebra} that if a Kan complex $X_\bullet$ has the property that each simplicial subcomplex $S \subset \Delta^n$ can be filled in $X_\bullet$, then its homotopy groups are trivial. In particular, Theorem~\ref{geo:th:filler} implies that $\mathrm{HOM}^{geom}_{\ombas}(C^*(f),C^*(g))_\bullet$ has trivial homotopy groups hence is contractible.
\end{proof}

Shifting from the \ombas\ to the \Ainf\ viewpoint, we can define in a similar fashion the simplicial subset 
\[ \mathrm{HOM}^{geom}_{\Ainf}(C^*(f),C^*(g))_\bullet \subset \mathrm{HOM}_{\Ainf}(C^*(f),C^*(g))_\bullet \ . \]
The simplicial set $\mathrm{HOM}^{geom}_{\Ainf}(C^*(f),C^*(g))_\bullet$ is then again a Kan complex which is contractible. Given an admissible horn of perturbation data $\mathbb{Y}_{\mathsf{\Lambda}_n^k}$, Theorem~\ref{alg:th:infinity-gr} implies that the induced horn $\mathsf{\Lambda}_n^k \rightarrow \mathrm{HOM}_{\Ainf}(C^*(f),C^*(g))_\bullet$ can always be filled algebraically. The fact that $\mathrm{HOM}^{geom}_{\Ainf}(C^*(f),C^*(g))_\bullet$ is a Kan complex implies something stronger : this horn can not only be filled algebraically, but also geometrically. 
We moreover point out that we should in fact work with twisted $n-\Ainf$ and $n-\ombas$-morphisms, as explained in section~\ref{geo:ss:twisted-n-ombas-morse}. However, the constructions of this section still hold in that context. 

The following proposition is a direct corollary to Theorem~\ref{geo:th:contractible} and solves the motivational question formulated in the introduction :

\begin{corollary} \label{geo:cor:initial}
Let $\mathbb{Y}$ and $\mathbb{Y}'$ be two admissible choices of perturbation data on the moduli spaces $\mathcal{CT}_m$. The \ombas -morphisms $\mu^{\mathbb{Y}}$ and $\mu^{\mathbb{Y}'}$ are then \ombas -homotopic
\[ \begin{tikzcd}[row sep=large, column sep = large]
C^*(f) \arrow[bend left=20]{r}[above]{\mu^{\mathbb{Y}}}[name=U,below,pos=0.5]{}
\arrow[bend right=20]{r}[below]{\mu^{\mathbb{Y}'}}[name=D,pos=0.5]{}
& C^*(g)
\arrow[Rightarrow, from=U, to=D]
\end{tikzcd} \ . \]
\end{corollary}

\section{Transversality, signs and orientations} \label{geo:s:trans-or-signs}

\subsection{Proof of theorems~\ref{geo:th:existence} and~\ref{geo:th:filler}}  \label{geo:ss:proofs-existence-filler}

\subsubsection{Proof of theorem~\ref{geo:th:existence}} \label{geo:sss:proof-th-three}

We detailed in section~II.3. of~\cite{mazuir-I} how to build an admissible choice of perturbation data $(\mathbb{X}_n)_{n \geqslant 2}$ on the moduli spaces $\mathcal{T}_m$. Drawing from this construction, we provide a sketch of the proof of Theorem~\ref{geo:th:existence} in this subsection : admissible $n$-simplices of perturbation data $(\mathbb{Y}_{I,m})_{I \subset \Delta^n}^{m \geqslant 1}$ on the moduli spaces $\mathcal{CT}(t_g)$ exist. The proof proceeds again by induction on the integer $N = \mathrm{dim} ( \mathcal{CT}(t_g) ) + \mathrm{dim} ( I )$.

If $N = 0$, $\mathrm{dim}(I)=0$ and the gauged tree $t_g$ is a corolla whose gauge intersects its root. Let $y \in \mathrm{Crit}(g)$ and $x_1, \cdots , x_m \in \mathrm{Crit}(f)$ and fix an integer $l$ such that 
\[ l \geqslant \mathrm{max} \left( 1 , |y| - \sum_{i=1}^m|x_i| + 1 \right) \ . \]
Define the parametrization space
\[ \mathfrak{X}_{t_{g}}^l := \{ \text{$C^l$-perturbation data $\mathbb{Y}_{t_g}$ on $\mathcal{CT}_m(t_g)$} \} \ , \]
and introduce the $C^l$-map
\[ \phi_{t_g} : \mathfrak{X}_{t_g}^l \times \mathcal{CT}_m(t_g) \times W^S(y) \times W^U(x_1) \times \cdots \times W^U(x_m) \longrightarrow M^{\times m+1} \ , \]
such that for every $\mathbb{Y}_{t_g} \in \mathfrak{X}_{t_g}^l$, $\phi_{t_g} ( \mathbb{Y}_{t_g} , \cdot ) = \phi_{\mathbb{Y}_{t_g}}$. Note that we should in fact write $\phi_{t_g}^{y,x_1,\dots,x_n}$ as the domain of $\phi_{t_g}$ depends on $y,x_1,\dots,x_n$. The space $\mathfrak{X}_{t_g}^l$ is then a Banach space and the map $\phi_{t_g}$ is a submersion. The map $\phi_{t_g}$ is in particular transverse to the diagonal $\Delta \subset M^{\times m+1}$. The parametric transversality lemma implies that there exists a subset $\mathfrak{Y}_{t_g}^{l;y , x_1 , \dots x_m} \subset \mathfrak{X}_{t_g}^l$ which is residual in the sense of Baire, and such that for every choice of perturbation data $\mathbb{Y}_{t_g} \in \mathfrak{Y}_t^{l;y , x_1 , \dots x_m}$ the map $\phi_{\mathbb{Y}_{t_g}}$ is transverse to the diagonal $\Delta \subset M^{\times m+1}$. Any $\mathbb{Y}_{t_g}$ in the intersection
\[ \mathfrak{Y}_{t_g}^l := \bigcap_{y,x_1,\dots,x_m} \mathfrak{Y}_{t_g}^{l;y , x_1 , \dots x_m} \subset \mathfrak{X}_{t_g} \]
then yields a $C^l$-choice of perturbation data on $\mathcal{CT}(t_g)$ such that all maps $\phi_{\mathbb{Y}_{t_g}}$ are transverse to the diagonal $\Delta \subset M^{\times m+1}$. Using an argument à la Taubes we prove that one can in fact construct a residual set $\mathfrak{Y}_{t_g} \subset \mathfrak{X}_{t_g}$, where $\mathfrak{X}_{t_g}$ is the Fréchet space defined by replacing "$C^l$" by "smooth" in the definition of $\mathfrak{X}_{t_g}^l$, and such that any $\mathbb{Y}_{t_g} \in \mathfrak{Y}_{t_g}$ yields a smooth choice of perturbation data such that all maps $\phi_{\mathbb{Y}_{t_g}}$ are transverse to the diagonal $\Delta \subset M^{\times m+1}$. See subsection II.3.2.2 of~\cite{mazuir-I} for more details on that last point. This wraps up the first step of the induction.

Let $N \geqslant 0$ and suppose that we have constructed an admissible choice of perturbation data $(\mathbb{Y}_{I,t_g}^0)$, where $I \subset \Delta^n$ and $t_g \in SCRT_m$ are such that $\mathrm{dim} ( \mathcal{CT}(t_g) ) + \mathrm{dim} ( I ) \leqslant N$. Let $I \subset \Delta^n$ and $t_g \in SCRT_m$ be such that $\mathrm{dim} ( \mathcal{CT}(t_g) ) + \mathrm{dim} ( I ) = N+1$. 
Let $y \in \mathrm{Crit}(g)$ and $x_1, \cdots , x_m \in \mathrm{Crit}(f)$ and fix an integer $l$ such that 
\[ l \geqslant \mathrm{max} \left( 1 , |y| - \sum_{i=1}^m|x_i| - |t_{I,g}| + 1 \right) \ . \]

We introduce the parametrization space 
\[  \mathfrak{X}_{I,t_g}^l := 
  \left\{\begin{array}{l}
         \text{$\mathrm{dim}(I)$-simplices of perturbation data $\mathbb{Y}_{I,t_g}$ on $\mathcal{CT}_m(t_g)$ such that the perturbation} \\
         \text{data $\{ \mathbb{Y}_{I,t_g} \} \cup (\mathbb{Y}_{J,t'_g}^0 )_{J \subset I}^{t_g' \in coll \cup g-v(t_g)}$ are of class $C^l$ in the sense of definition~\ref{geo:def:smooth-choice-of}, } \\
         \text{and such that $\mathbb{Y}_{I,t_g}$ is gluing-compatible w.r.t. the perturbation data $( \mathbb{Y}_{I,t_g}^0 )$}
  \end{array} \right\} \ . \]
This parametrization space is a Banach affine space.
Define again the $C^l$-map
\[ \phi_{I,t_g} : \mathfrak{X}_{I,t_g}^l \times \mathring{I} \times \mathcal{CT}_m(t_g) \times W^S(y) \times W^U(x_1) \times \cdots \times W^U(x_m) \longrightarrow M^{\times m+1} \ . \]
The map $\phi_{I,t_g}$ is then transverse to the diagonal $\Delta \subset M^{\times m+1}$. 
Applying the parametric transversality theorem and proceeding as in the case $N=0$, there exists a residual set $\mathfrak{Y}_{I,t_g}^l \subset \mathfrak{X}_{I,t_g}^l$ such that for every choice of perturbation data $\mathbb{Y}_{I,t_g} \in \mathfrak{Y}_{I,t_g}^l$ the map $\phi_{\mathbb{Y}_{I,t_g}}$ is transverse to the diagonal $\Delta \subset M^{\times m+1}$. Resorting again to an argument à la Taubes, we can prove the same statement in the smooth context. By definition of the parametrization spaces $\mathfrak{X}_{I,t_g}$ this construction yields an admissible choice of perturbation data $(\mathbb{Y}_{I,t_g})$, where the indices $I$ and $t_g$ are such that $\mathrm{dim} ( \mathcal{CT}(t_g) ) + \mathrm{dim} ( I ) \leqslant N+1$. This concludes the proof of Theorem~\ref{geo:th:existence} by induction.

\subsubsection{Proof of theorem~\ref{geo:th:filler}}

The proof of Theorem~\ref{geo:th:filler} proceeds exactly as the previous proof, by replacing the requirements in the definition of $\mathfrak{X}^l_{I,t_g}$ by the conditions prescribed by the simplicial subcomplex $S \subset \Delta^n$.

\subsection{Orientation and transversality} \label{geo:ss:or-transv}

\subsubsection{Signed short exact sequences}

Consider a short exact sequence of vector spaces
\[ 0 \longrightarrow V_2 \longrightarrow W \longrightarrow V_1 \longrightarrow 0 \ . \]
It induces a direct sum decomposition $W = V_1 \oplus V_2 $. Suppose that the vector spaces $W$, $V_1$ and $V_2$ are oriented. We denote $(-1)^{\varepsilon}$ the sign obtained by comparing the orientation on $W$ to the one induced by the direct sum $V_1 \oplus V_2$. We will then say that the short exact sequence has sign $(-1)^{\varepsilon}$. In particular, when $(-1)^{\varepsilon}= 1$, we will say that the short exact sequence is \emph{positive}.

\subsubsection{Orientation and transversality} \label{geo:sss:or-transver}

Given now two manifolds $M,N$, a codimension $k$ submanifold $S \subset N$ and a smooth map
\[ \phi : M \longrightarrow N \]
which is tranverse to $S$, the inverse image $\phi^{-1}(S)$ is a codimension $k$ submanifold of $M$. Moreover, choosing a complementary $\nu_S$ to $TS$, the transversality assumption yields the following short exact sequence of vector bundles
\[ 0 \longrightarrow T \phi^{-1} (S) \longrightarrow T M |_{ \phi^{-1}(S)} \underset{d \phi}{\longrightarrow} \nu_S \longrightarrow 0 \ . \]
Suppose now that $M$, $N$ and $S$ are oriented. The orientations on $N$ and $S$ induce an orientation on $\nu_S$. The submanifold $\phi^{-1}(S)$ is then oriented by requiring that the previous short exact sequence be positive. We will refer to this choice of orientation as the \emph{natural orientation on $\phi^{-1}(S)$}. 

For instance, the moduli space $\mathcal{T}_{t}^{\mathbb{X}}(y ; x_1 , \dots , x_m)$ is defined as the inverse image of the diagonal $\Delta \subset M^{\times m+1}$ under the map
\[ \phi_{\mathbb{X}_t} : \mathcal{T}_m(t) \times W^S(y) \times W^U(x_1) \times \cdots \times W^U(x_m) \longrightarrow M^{\times m+1} \ . \]
Orienting the domain and codomain of $\phi_{\mathbb{X}_t}$ by taking the product orientation, and orienting the diagonal $\Delta \subset M^{\times m+1}$ as $M$, defines a natural orientation on $\mathcal{T}_{t}(y ; x_1 , \dots , x_m)$.

\subsection{Algebraic preliminaries} \label{geo:ss:algebraic-prelim}

\subsubsection{Reformulating the $n-\Omega B As$-equations} \label{geo:sss:reform-ombas-eq}

We set for the rest of this section an orientation $\omega$ for each $t_g \in SCRT_n$, which endows each moduli space $\mathcal{CT}_n(t_g)$ with an orientation. We write moreover $\mu_{I,t_g}$ for the operations $(I,t_g, \omega)$ of $n- \Omega B As - \mathrm{Morph}$. The $\Omega B As$-equations for a $n-\Omega B As$-morphism then read as
\[ \resizebox{\hsize}{!}{\begin{math} \begin{aligned}
[ \partial , \mu_{I,t_g} ] = 
&\sum_{l=0}^{\mathrm{dim}(I)} (-1)^l \mu_{\partial_l^{sing} I,t_g} + (-1)^{|I|} \left( \sum_{\substack{t^0 \# (t_g^1 , \dots , t_g^s) = t_g \\ I_1 \cup \cdots \cup I_s = I}} (-1)^{\dagger_{\Omega B As}} m_{t^0} \circ (\mu_{I_1,t^1_g} \otimes \cdots \otimes \mu_{I_s,t^s_g} ) \right. \\
&\left. + \sum_{t'_g \in coll(t_g)} (-1)^{\dagger_{\Omega B As}} \mu_{I,t'_g} + \sum_{t'_g \in g-vert(t_g)} (-1)^{\dagger_{\Omega B As}} \mu_{I,t'_g} + \sum_{t^1_g \#_i t^2 = t_g} (-1)^{\dagger_{\Omega B As}} \mu_{I,t^1_g} \circ_i m_{t^2} \right) \ .
\end{aligned}
\end{math}} \]
The signs $(-1)^{\dagger_{\Omega B As}}$ need not be made explicit, but can be computed as in section I.5.2 of~\cite{mazuir-I}. 

\subsubsection{Twisted $n - \Ainf$-morphisms and twisted $n-\Omega B As$-morphisms} \label{geo:sss:twisted-n-ainf}

\begin{definition}
\begin{enumerate}[label=(\roman*)]
\item A \emph{twisted \Ainf -algebra} is a dg-\Z -module $A$ endowed with two different differentials $\partial_1$ and $\partial_2$, and a collection of degree $2-m$ operations $m_m : A^{\otimes m} \rightarrow A$ such that
\[ [ \partial , m_m ] = - \sum_{\substack{i_1+i_2+i_3=m \\ 2 \leqslant i_2 \leqslant m-1}} (-1)^{i_1 + i_2i_3} m_{i_1+1+i_3} (\ide^{\otimes i_1} \otimes m_{i_2} \otimes \ide^{\otimes i_3} ) \ , \]
where $[ \partial , \cdot ]$ denotes the bracket for the maps $ (A^{\otimes m} , \partial_1) \rightarrow (A , \partial_2)$.
\item Let $(A, \partial_1 , \partial_2 , m_m)$ and $(B, \partial_1 , \partial_2 , m_m)$ be two twisted \Ainf -algebras. A \emph{twisted $n-\Ainf$-morphism} from $A$ to $B$ is defined to be a sequence of degree $1-m+|I|$ operations $f_I^{(m)} : A^{\otimes m} \rightarrow B$ such that
\begin{align*}
\left[ \partial , f^{(m)}_I \right] = \sum_{j=0}^{\mathrm{dim}(I)} (-1)^j f^{(m)}_{\partial_jI} &+ (-1)^{|I|} \sum_{\substack{i_1+i_2+i_3=n \\ i_2 \geqslant 2}} (-1)^{i_1 + i_2i_3} f^{(i_1+1+i_3)}_I (\ide^{\otimes i_1} \otimes m_{i_2} \otimes \ide^{\otimes i_3})  \\ 
&- \sum_{\substack{i_1 + \cdots + i_s = m \\ I_1 \cup \cdots \cup I_s = I \\ s \geqslant 2 }} (-1)^{\epsilon_B} m_s ( f^{(i_1)}_{I_1} \otimes \cdots \otimes f^{(i_s)}_{I_s}) \ ,
\end{align*}
where $[ \partial , \cdot ]$ denotes the bracket for the maps $ (A^{\otimes m} , \partial_1) \rightarrow (B , \partial_2)$.
\item  A\emph{ twisted $\Omega B As$-algebra} and a \emph{twisted $n-\Omega B As$-morphism} between twisted $\Omega B As$-algebras are defined similarly. 
\end{enumerate}
\end{definition}

The explicit formulae obtained by evaluating the $n-\Omega B As$-equations of a twisted $n-\Omega B As$-morphism on $A^{\otimes m}$ then read as follows :
\begin{align*}
&- \partial_2 \mu_{I,t_g} (a_1 , \dots , a_m) + (-1)^{|I|+|t_g| + \sum_{j=1}^{i-1}|a_j|} \mu_{I,t_g} ( a_1 , \dots , a_{i-1} , \partial_1 a_i , a_{i+1} , \dots , a_m) \\
&+ \sum_{t^1_g \# t^2 = t} (-1)^{|I|+\dagger_{\Omega B As} + |t^2| \sum_{j=1}^{i_1} |a_j|} \mu_{I,t^1_g} (a_1 , \dots , a_{i_1} , m_{t^2} (a_{i_1 + 1} , \dots , a_{i_1 + i_2} ) , a_{i_1 + i_2 + 1} , \dots , a_m) \\
&+ \sum_{\substack{t^0 \# (t_g^1 , \dots , t_g^s) = t_g \\ I_1 \cup \cdots \cup I_s = I}} (-1)^{|I|+\dagger_{\Omega B As} + \dagger_{Koszul} } m_{t^0} ( \mu_{I_1,t_g^1} (a_1 , \dots ,a_{i_1} ) , \dots , \mu_{I_s,t_g^s} (a_{i_1 + \cdots + i_{s-1} +1} , \dots , a_m)) \\
&+ \sum_{t'_g \in coll(t_g)} (-1)^{|I|+\dagger_{\Omega B As}} \mu_{I,t'_g} (a_1 , \dots , a_m) + \sum_{t'_g \in g-vert(t_g)} (-1)^{|I|+\dagger_{\Omega B As}} \mu_{I,t'_g} (a_1 , \dots , a_m) \\
&+ \sum_{l=0}^{\mathrm{dim}(I)} (-1)^l \mu_{\partial_l^{sing} I,t_g} (a_1 , \dots , a_m) = 0 \ ,
\end{align*}
where
\[ \dagger_{Koszul} = \sum_{r=1}^s(|I_r| + |t_g^r|) \left( \sum_{t=1}^{r-1} \sum_{j=1}^{i_t} |a_{i_1 + \cdots + a_{i_{t-1}} + j}| \right) \ . \]
As explained in~\cite{mazuir-I}, these definitions cannot be phrased using an operadic viewpoint. However, a twisted $n-\Omega B As$-morphism between twisted $\Omega B As$-algebras still always descends to a twisted $n-\Ainf$-morphism between twisted \Ainf -algebras.

\subsection{Proof of Theorem~\ref{geo:th:n-ombas-Morse}} \label{geo:ss:twisted-n-ombas-morse}

\subsubsection{Recollections on twisted $\Omega B As$-algebra structures on the Morse cochains} \label{geo:sss:recoll-twisted-ombas-alg-Morse}

We prove in \cite{mazuir-I} that given a Morse function $f$ and an admissible choice of perturbation data $\mathbb{X}$ on the moduli spaces $\mathcal{T}_m$, the Morse cochains $C^*(f)$ can be endowed with a twisted $\Omega B As$-algebra structure by counting the 0-dimensional moduli spaces $\mathcal{T}_t^{\mathbb{X}_t}(y ; x_1,\dots,x_n)$. 

We twist to this end the natural orientation on the moduli spaces $\mathcal{T}_{t}^\mathbb{X}(y ; x_1,\dots,x_m)$ defined in subsection~\ref{geo:sss:or-transver}, by a sign of parity
\[ \sigma (t ; y ; x_1 , \dots , x_m) := dm ( 1 + |y| + |t| ) + |t| |y| + d \sum_{i=1}^m |x_i| (m-i) \ , \]
and the orientation on the moduli spaces $\mathcal{T}(y;x)$ by a sign of parity
\[ \sigma (y ; x) := 1 \ , \]
where $d$ denotes the dimension of the manifold $M$.
The moduli spaces $\mathcal{T}_{t}^\mathbb{X}(y ; x_1,\dots,x_m)$ and $\mathcal{T}(y;x)$ endowed with these new orientations are then respectively written $\widetilde{\mathcal{T}}_{t}^\mathbb{X}(y ; x_1,\dots,x_m)$ and $\widetilde{\mathcal{T}}(y;x)$. 

The operations $m_t$ and the differential on $C^*(f)$ are then defined as
\begin{align*}
m_{t} (x_1 , \dots , x_m) &=  \sum_{|y|= \sum_{i=1}^m|x_i| + |t| } \# \widetilde{\mathcal{T}}_{t}^\mathbb{X}( y ; x_1,\dots,x_m ) \cdot y \ , \\
\partial_{Morse} (x) &=  \sum_{|y|= |x|+1 } \# \widetilde{\mathcal{T}}(y;x) \cdot y \ .
\end{align*} 
Counting the signed points in the boundary of the oriented 1-dimensional manifolds $\widetilde{\mathcal{T}}_{t}( y ; x_1,\dots,x_m )$ proves that the operations $m_t$ define a twisted $\Omega B As$-algebra structure on $(C^*(f), \partial_{Morse}^{Tw},\partial_{Morse})$, where
\[ (\partial_{Morse}^{Tw})^k = (-1)^{(d+1)k}\partial_{Morse}^k \ . \]
In particular, either working with coefficients in $\Z /2$, or with coefficients in \Z\ and an odd-dimensional manifold $M$, the operations $m_t$ define an $\Omega B As$-algebra structure on the Morse cochains.

\subsubsection{Twisted $n-\Omega B As$-morphisms between the Morse cochains} \label{geo:sss:twisted-n-ombas-morph}

Let $\mathbb{X}^f$ and $\mathbb{X}^g$ be admissible choices of perturbation data for the Morse functions $f$ and $g$. Denote $(C^*(f),m_t^{\mathbb{X}^f})$ and $(C^*(g),m_t^{\mathbb{X}^g})$ the Morse cochains endowed with their $\Omega B As$-algebra structures. Given an admissible $n$-simplex of perturbation data $(\mathbb{Y}_{I,m})_{I \subset \Delta^n}^{m \geqslant 1}$, we now construct a twisted $n-\Omega B As$-morphism $$\mu_{I,t_g} : (C^*(f), \partial_{Morse}^{Tw}, \partial_{Morse}) \longrightarrow (C^*(g), \partial_{Morse}^{Tw}, \partial_{Morse}) \ , \ I \subset \Delta^n , \ t_g \in SCRT \ ,$$ 
which completes the proof of Theorem~\ref{geo:th:n-ombas-Morse}.

The moduli space $\mathcal{CT}_{I,t_g}^{\mathbb{Y}_{I,t_g}}( y ; x_1,\dots,x_m )$ is defined as the inverse image of the diagonal $\Delta \subset M^{\times m+1}$ under the map
\[ \phi_{\mathbb{Y}_{I,t_g}} : \mathring{I} \times \mathcal{CT}_m(t_g) \times W^S(y) \times W^U(x_1) \times \cdots \times W^U(x_m) \longrightarrow M^{\times m+1} \ . \]
Orienting the domain and codomain of $\phi_{\mathbb{Y}_{I,t_g}}$ with the product orientation, and orienting the diagonal $\Delta \subset M^{\times m+1}$ as $M$, defines a natural orientation on $\mathcal{CT}_{I,t_g}( y ; x_1,\dots,x_m )$ as explained in subsection~\ref{geo:sss:or-transver}.

\begin{definition}
We define $\widetilde{\mathcal{CT}}_{I,t_g}^{\mathbb{Y}_{I,t_g}}(y ; x_1,\dots,x_m )$ to be the oriented manifold $\mathcal{CT}_{I,t_g}^{\mathbb{Y}_{I,t_g}}( y ; x_1,\dots,x_m )$ whose natural orientation has been twisted by a sign of parity
\[ \sigma (t_{I,g} ; y ; x_1 , \dots , x_m) := dm ( 1 + |y| + |t_{I,g}| ) + |t_{I,g}| |y| + d \sum_{i=1}^m |x_i| (m-i) \ . \]
\end{definition}

\begin{proposition}
If the moduli space $\widetilde{\mathcal{CT}}_{I,t_g}( y ; x_1,\dots,x_m )$ is 1-dimensional, its boundary decomposes as the disjoint union of the following components
\begin{enumerate}[label=(\roman*)]
\item $(-1)^{|y| + |I| + \dagger_{\Omega B As} + |t^2| \sum_{i=1}^{i_1} |x_i|} \widetilde{\mathcal{CT}}_{I,t^1_g}(y ; x_1,\dots,x_{i_1 } , z , x_{i_1+i_2+1} , \dots , x_m ) \times \widetilde{\mathcal{T}}_{t^2}(z ; x_{i_1 +1 },\dots,x_{i_1 + i_2});$
\item $(-1)^{|y| + |I| + \dagger_{\Omega B As} + \dagger_{Koszul}} \widetilde{\mathcal{T}}_{t^0}(y ; y_1,\dots,y_s) \times \widetilde{\mathcal{CT}}_{I_1,t^1_g}(y_1 ; x_{1},\dots ) \times \cdots \times \widetilde{\mathcal{CT}}_{I_s,t^s_g}(y_s ; \dots,x_m)$ ;
\item $(-1)^{|y| + |I| + \dagger_{\Omega B As}} \widetilde{\mathcal{CT}}_{I,t'_g}(y ; x_1,\dots,x_m)$ for $t_g' \in coll(t)$ ;
\item $(-1)^{|y| + |I| + \dagger_{\Omega B As}} \widetilde{\mathcal{CT}}_{I,t'_g}(y ; x_1,\dots,x_m)$ for $t_g' \in g-vert(t)$ ;
\item $(-1)^{|y| + \dagger_{Koszul}+(m+1)|x_i|} \widetilde{\mathcal{CT}}_{I,t_g}( y ; x_1, \dots , z , \dots , x_m ) \times \widetilde{\mathcal{T}}(z ; x_i)$ where we have set $\dagger_{Koszul} = |I| + |t_g| + \sum_{j=1}^{i-1}|x_j|$ ;
\item $(-1)^{|y|+1} \widetilde{\mathcal{T}}(y ; z) \times  \widetilde{\mathcal{CT}}_{I,t_g}( z ; x_1, \dots ,  x_m )$ ;
\item $(-1)^{|y|+l} \widetilde{\mathcal{CT}}_{\partial_l^{sing}I,t_g}( y ; x_1, \dots ,  x_m )$.
\end{enumerate} 
\end{proposition}

Define the operations $ \mu_{I,t_g} : C^*(f)^{\otimes m} \rightarrow C^*(g)$ as
\[ \mu_{I,t_g} (x_1 , \dots , x_m) =  \sum_{|y|= \sum_{i=1}^m|x_i| + |t_{I,g}| } \# \widetilde{\mathcal{CT}}_{I,t_g}^\mathbb{Y}(y ; x_1,\dots,x_m) \cdot y \ . \]
Counting the points in the boundary of the oriented 1-dimensional manifolds $\widetilde{\mathcal{CT}}_{I,t_g}( y ; x_1,\dots,x_m )$ finally proves that :
\begin{geo:th:n-ombas-Morse}
The operations $\mu_{I,t_g}$ define a twisted $n-\Omega B As$-morphism between the Morse cochains $(C^*(f), \partial_{Morse}^{Tw}, \partial_{Morse})$ and $(C^*(g), \partial_{Morse}^{Tw}, \partial_{Morse})$.
\end{geo:th:n-ombas-Morse}

We send the reader back to~\cite{mazuir-I} for the complete check of signs in the case of the operations $m_t$, which easily transports to the case of the operations $\mu_{I,t_g}$. Again, either working with coefficients in $\Z /2$, or with coefficients in \Z\ and an odd-dimensional manifold $M$, the operations $\mu_{I,t_g}$ fit into a standard $n-\Omega B As$-morphism between $\Omega B As$-algebras.

\section{Towards the problem of the composition} \label{geo:s:towards}

At the end of~\cite{mazuir-I} we stated two main questions. The first was the motivational question solved in this article and the second one came as follows :

\paragraph{\textbf{Problem 2.}} Given three Morse functions $f_0,f_1,f_2$, choices of perturbation data $\mathbb{X}^i$, and choices of perturbation data $\mathbb{Y}^{ij}$ defining morphisms 
\begin{align*}
\mu^{\mathbb{Y}^{01}} : (C^*(f_0),m_t^{\mathbb{X}^0}) \longrightarrow (C^*(f_1),m_t^{\mathbb{X}^1}) \ , \\
\mu^{\mathbb{Y}^{12}} : (C^*(f_1),m_t^{\mathbb{X}^1}) \longrightarrow (C^*(f_2),m_t^{\mathbb{X}^2}) \ , \\
\mu^{\mathbb{Y}^{02}} : (C^*(f_0),m_t^{\mathbb{X}^0}) \longrightarrow (C^*(f_2),m_t^{\mathbb{X}^2}) \ ,
\end{align*}
can we construct an \ombas -homotopy such that $\mu^{\mathbb{Y}^{12}} \circ \mu^{\mathbb{Y}^{01}} \simeq \mu^{\mathbb{Y}^{02}}$ through this homotopy ? That is, can the following cone be filled in the $\Omega B As$ realm
\[ \begin{tikzcd}[row sep=large, column sep = large]
C^*(f_0) \arrow{dr}[below left]{\mu^{\mathbb{Y}^{02}}}[name=U,above right,pos=0.6]{} \arrow{r}[above]{\mu^{\mathbb{Y}^{01}}}
& C^*(f_1) \arrow{d}[right]{\mu^{\mathbb{Y}^{12}}} \arrow[Rightarrow, to=U] \\
& C^*(f_2) 
\end{tikzcd} \ ? \]

The author plans to prove in an upcoming article that the answer to this question is positive. This simple problem will in fact again generalize to a wider range of constructions in Morse theory, involving the $n$-morphisms introduced in this article as well as some new interesting combinatorics.

\newpage

\bibliographystyle{alpha}

\bibliography{ha-morse-II-biblio}

\end{document}